\documentclass{amsart}
\makeindex

\usepackage[all]{xy}
\usepackage{amssymb}

\newcommand{\W}{\mathcal W}

\newcommand{\midd}{\overleftrightarrow{\mathrm{mid}}}

\newcommand{\Polish}{\mathbf{Polish}}
\newcommand{\Set}{\mathbf{Set}}
\newcommand{\Pow}{\mathrm{Pow}}

\newcommand{\C}{\mathcal C}
\newcommand{\M}{\mathcal M}
\newcommand{\N}{\mathcal N}
\newcommand{\I}{\mathcal I}
\newcommand{\J}{\mathcal J}
\newcommand{\K}{\mathcal K}
\newcommand{\U}{\mathcal U}
\newcommand{\F}{\mathcal F}
\newcommand{\E}{\mathcal E}
\newcommand{\A}{\mathcal A}
\newcommand{\V}{\mathcal V}
\newcommand{\Tau}{\mathcal T}

\newcommand{\NF}{\mathcal{NF}}
\newcommand{\pr}{\mathrm{pr}}
\newcommand{\dist}{\mathrm{dist}}

\newcommand{\mc}{\mathcal}

\newcommand{\e}{\varepsilon}

\newcommand{\IR}{\mathbb R}
\newcommand{\IT}{\mathbb T}

\newcommand{\IZ}{\mathbb Z}

\newcommand{\IN}{\mathbb N}
\newcommand{\w}{\omega}
\newcommand{\Ra}{\Rightarrow}
\newcommand{\on}{\operatorname}
\newcommand{\add}{\mathrm{add}}
\newcommand{\cof}{\mathrm{cof}}

\newcommand{\GH}{\mathcal{GH}}
\newcommand{\GHN}{\mathcal{GHN}}
\newcommand{\GHM}{\mathcal{GHM}}
\newcommand{\GHI}{\mathcal{GHI}}
\newcommand{\EHI}{\mathcal{EHI}}

\newcommand{\EHM}{\mathcal{EHM}}
\newcommand{\SHM}{\mathcal{SHM}}
\newcommand{\HN}{\mathcal{HN}}
\newcommand{\HM}{\mathcal{HM}}
\newcommand{\HI}{\mathcal{HI}}

\newcommand{\HT}{\mathcal{HT}}
\newcommand{\EHT}{\mathcal{EHT}}

\newcommand{\diam}{\mathrm{ diam}\,}

\newcommand{\supp}{\mathrm{ supp}\,}

\newcommand{\rng}{\mathrm{ rng}\,}

\newtheorem{theorem}{Theorem}[section]

\newtheorem{example}[theorem]{Example}
\newtheorem{problem}[theorem]{Problem}
\newtheorem{corollary}[theorem]{Corollary}
\newtheorem{claim}[theorem]{Claim}
\newtheorem{proposition}[theorem]{Proposition}
\newtheorem{lemma}[theorem]{Lemma}

\theoremstyle{definition}
\newtheorem{definition}[theorem]{Definition}
\newtheorem{remark}[theorem]{Remark}

\begin{document}
\title[Looking at Polish groups through compact glasses]{Haar-$\I$ sets: looking at small sets in Polish groups\\ through compact glasses}
\author{Taras Banakh, Szymon G\l\c ab, Eliza Jab\l o\'nska, Jaros\l aw Swaczyna}
\thanks{The research of the fourth author has been supported by the Polish Ministry of Science and Higher Education (a part of the project ``Diamond grant'', 2014-17, No. 0168/DIA/2014/43).}
\address{T.Banakh: Ivan Franko University of Lviv (Ukraine) and Jan Kochanowski University in Kielce (Poland)}
\email{t.o.banakh@gmail.com}
\address{E.Jab\l o\'nska: Department of Mathematics, Pedagogical University of Cracow, Podchor\c a\.zych 2, 30-084 Krak\'ow}
\email{eliza.jablonska@up.krakow.pl}
\address{S.G\l \c{a}b, J.Swaczyna: Institute of Mathematics, \L \'od\'z University of Technology, ul. W\'olcza\'nska 215, 93-005 \L \'od\'z}
\email{szymon.glab@p.lodz.pl, jswaczyna@wp.pl}

\begin{abstract}
Generalizing Christensen's notion of a Haar-null set and Darji's notion of a Haar-meager set, we introduce and study the notion of a Haar-$\I$ set in a Polish group. Here $\I$ is an ideal of subsets of some compact metrizable space $K$. A Borel subset $B\subset X$ of a Polish group $X$ is called {\em Haar-$\I$} if there exists a continuous map $f:K\to X$ such that $f^{-1}(B+x)\in\I$ for all $x\in X$. Moreover, $B$ is {\em generically Haar-$\I$} if the set of witness functions $\{f\in C(K,X):\forall x\in X\;\;f^{-1}(B+x)\in\I\}$ is comeager in the function space $C(K,X)$. We study (generically) Haar-$\I$ sets in Polish groups for many concrete and abstract ideals $\I$, and construct the corresponding distinguishing examples. We prove some results on Borel hull of Haar-$\I$ sets, generalizing results of Solecki, Elekes, Vidny\'anszky, Dole\v{z}al, Vlas\v{a}k on Borel hulls of Haar-null and Haar-meager sets.   Also we establish various Steinhaus properties of the families of (generically) Haar-$\I$ sets in Polish groups for various ideals $\I$.
\end{abstract}
\maketitle

\tableofcontents
\newpage

\section{Introduction}

The study of various $\sigma$-ideals in topological groups is a popular and important area of mathematics, situated on the border of Real Analysis, Topological Algebra, General Topology and Descriptive Set Theory. Among the most important $\sigma$-ideals let us mention the $\sigma$-ideal $\M$ of meager subsets in a Polish group and the $\sigma$-ideal $\N$ of sets of Haar measure zero in a locally compact group. Both ideals are well-studied, have many applications, and also have many similar properties.

For example, by the classical Steinhaus-Weil Theorem \cite{Stein}, \cite[\S11]{Weil} for any measurable sets $A,B$ of positive Haar measure in a locally compact group $X$ the sum $A+B=\{a+b:a\in A,\;b\in B\}$ has nonempty interior in $X$ and the difference $A-A$ is a neighborhood of zero in $X$. A similar result for the ideal $\M$ is attributed to Piccard \cite{Pic} and Pettis \cite{Pet}. They proved that for any non-meager Borel sets $A,B$ in a Polish group $X$ the sum $A+B$ has nonempty interior and the difference $A-A$ is a neighborhood of zero in $X$.


In contrast to the ideal $\M$, which is well-defined on each Polish group, the ideal
$\N$ of sets of Haar measure zero is defined only on locally compact groups (since its definition requires the Haar measure, which does not exist in non-locally compact groups). In order to bypass this problem, Christensen \cite{Ch} introduced the important notion of a Haar-null set, which can be defined for any Polish group. Christensen defined a Borel subset $A$ of a Polish group $X$ to be \index{subset!Haar-null}{\em Haar-null} if there exists a $\sigma$-additive Borel probability measure $\mu$ on $X$ such that $\mu(A+x)=0$ for all $x\in X$. Christensen proved that Haar-null sets share many common properties with sets of Haar measure zero in locally compact groups. In particular, the family $\HN$ of subsets of Borel Haar-null sets in a Polish group $X$ is an invariant $\sigma$-ideal on $X$. For any locally compact Polish group $X$ the ideal $\HN$ coincides with the ideal $\N$ of sets of Haar measure zero. Christensen also generalized the Steinhaus difference theorem to Haar-null sets, proving that if a Borel subset $B$ of a Polish group $X$ is not Haar-null, then the difference $B-B$ is a neighborhood of zero in $X$.

Since probability measures on Polish spaces are supported by $\sigma$-compact sets, the witness measure $\mu$ in the definition of a Haar-null set can be chosen to have compact support. This means that for defining Haar-null sets it suffices to look at the Polish group through a ``compact window" (coinciding with the support of the witness measure). Exactly this ``compact window" philosophy was used by Darji \cite{D} who introduced the notion of a Haar-meager set in a Polish group. Darji defined a Borel subset $B$ of a Polish group $X$ to be \index{subset!Haar-meager}{\em Haar-meager} if there exists a continuous map $f:K\to X$ defined on a compact metrizable space $K$ such that for any $x\in X$ the preimage $f^{-1}(B+x)$ is meager in $K$. Haar-meager sets share many common properties with Haar-null sets: the family $\HM$ of subsets of Borel Haar-meager sets in a Polish group $X$ is an invariant $\sigma$-ideal on $X$; if the Polish group $X$ is locally compact, then the ideal $\HM$ coincides with the ideal $\M$ of meager sets in $X$; if a Borel subset $B\subset X$ is not Haar-meager, then $B-B$ is a neighborhood of zero in $X$ (see \cite{J}).

In this paper we shall show that the ideals $\HN$ and $\HM$ are particular instances of the ideals of Haar-$\I$ sets in Polish groups. Given a compact metrizable space $K$ and an ideal $\I$ of subsets of $K$, we define a Borel subset $B$ of a Polish group $X$ to be \index{subset!Haar-$\I$}{\em Haar-$\I$} if there exists a continuous map $f:K\to X$ such that $f^{-1}(B+x)\in\I$ for any $x\in X$. For a Polish group $X$ by $\HI$ we denote the family of subsets of Borel Haar-$\I$ sets in $X$. For the ideal $\I=\M$ of meager sets in the Cantor cube $K=\{0,1\}^\w$ the family $\HI$ coincides with the ideal $\HM$ of Haar-meager sets in $X$ and for the ideal $\I=\N$ of sets on product measure null on $K=\{0,1\}^\w$ the family $\HI$ coincides with the ideal $\HN$ of Haar-null sets in $X$.

In Sections~\ref{s4}--\ref{s11} we study Haar-$\I$ sets in Polish groups for many concrete and abstract ideals $\I$. In particular, we find properties of the ideal $\I$, under which the family $\HI$ is an ideal or $\sigma$-ideal. We present many examples distinguishing Haar-$\I$ sets for various ideals $\I$. In Section~\ref{s8} we characterize ideals $\I$ for which the family $\HI$ has the Steinhaus property in the sense that for any Borel set $B\notin\HI$ in a Polish group $X$ the difference $B-B$ is a neighborhood of zero in $X$. 

In Section~\ref{s:hull} we prove some results on Borel hulls of Haar-$\I$ sets, which generalize the results of  Elekes, Vidny\'anszky \cite{EVid}, \cite{EVidN} of Borel hulls of Haar-null sets and Dole\v{z}al, Vlas\v{a}k \cite{DV} on Borel hulls of Haar-meager sets. In Section~\ref{s:cc} we evaluate the additivity and cofinality of the semi-ideals of Haar-$\I$ sets in non-locally compact Polish groups, generalizing some results of Elekes and Po\'or \cite{EP}.

In Sections~\ref{s12}, \ref{s13} we study generically Haar-$\I$ sets in Polish groups. We define a Borel subset $A$ of a Polish group $X$ to be \index{subset!generically Haar-$\I$}{\em generically Haar-$\I$} if the set of witnessing functions $\{f\in C(K,X):\forall x\in X\;\;f^{-1}(A+x)\in\I\}$ is comeager in the space $C(K,X)$ of continuous functions endowed with the compact-open topology. For the ideal $\I=\N$ of sets of measure zero in the Cantor cube $K=\{0,1\}^\w$ our notion of a generically Haar-$\I$ set is equivalent to that introduced and studied by Dodos \cite{D1}--\cite{D3}. Extending a result of Dodos, we show that for any Polish group $X$ the family $\GHI$ of subsets of Borel generically Haar-$\I$ sets in $X$ has the weak Steinhaus property: for any Borel set $B\notin\GHI$ the difference $B-B$ is not meager and hence $(B-B)-(B-B)$ is a neighborhood of zero in $X$.

In Section~\ref{s14} we detect various smallness properties for some simple subsets in countable products of finite groups and in Sections~\ref{s16}--\ref{s19} we study smallness properties of additive and mid-convex sets in Polish groups and Polish vector spaces.  In the final Section~\ref{s20} we collect the obtained results and open problems on (weak and strong) Steinhaus properties of various families of small sets in Polish groups.

\section{Preliminaries}

{\em All groups considered in this paper are Abelian}. A \index{Polish group}{\em Polish group} is a topological group whose underlying topological space is Polish (i.e., homeomorphic to a separable complete metric space).

For a Polish group $X$ by $\theta$ we shall denote the neutral element of $X$ and by $\rho:X\times X\to\mathbb{R}_+$ a complete invariant metric generating the topology of $X$.

By $\w$ we denote the set of finite ordinals and by $\IN=\w\setminus\{0\}$ the set positive integer numbers. Each number $n\in\w$ is identified with the set $\{0,\dots,n-1\}$ of smaller numbers.

For every $n\in\IN$ we denote by $C_n$ a cyclic group of order $n$, endowed with the discrete topology. For two subsets $A,B$ of a group let $A+B:=\{a+b:a\in A,\;b\in B\}$ and $A-B:=\{a-b:a\in A,\;b\in B\}$ be the algebraic sum and difference of the sets $A,B$, respectively.

For a metric space $(X,\rho)$, a point $x_0\in X$ and a positive real number $\e$ let $S(x_0,\e):=\{x\in X:\rho(x,z_0)=\e\}$ and $B(x_0;\e):=\{x\in X:\rho(x,x_0)<\e\}$ be the $\e$-sphere and open $\e$-ball centered at $x_0$, respectively. Also let $B(A;\e)=\bigcup_{a\in A}B(a;\e)$ be the $\e$-neighborhood of a set $A\subset X$ in the metric space $(X,\rho)$.
For a nonempty subset $A\subset X$ by $\diam A:=\sup\{\rho(a,b):a,b\in A\}$ we denote the {\em diameter} of $A$ in the metric space $(X,\rho)$.

By \index{$2^\w$}$2^{\w}=\{0,1\}^\omega$ we denote the Cantor cube, endowed with the Tychonoff product topology. The Cantor cube carries the standard product measure denoted by $\lambda$. The measure $\lambda$ coincides with the Haar measure on $2^\w$ identified with the countable power $C_2^\w$ of the two-element cyclic group $C_2$. We shall refer to the measure $\lambda$ as the \index{Haar measure}{\em Haar measure} on $2^\w$.

A subset $A$ of a topological space $X$
\begin{itemize}
\item is \index{meager subset}\index{subset!meager}{\em meager} if $A$ can be written as the countable union $A=\bigcup_{n\in\w}A_n$ of nowhere dense sets in $X$;
\item is  \index{comeager subset}\index{subset!comeager}{\em comeager} if the complement $X\setminus A$ is meager in $X$;
\item has \index{Baire property}{\em the Baire property} if there exists an open set $U\subset X$ such that the symmetric difference $A\triangle U:=(A\setminus U)\cup(U\setminus A)$ is meager in $X$.
\end{itemize}
It is well-known \cite[11.5]{K} that each Borel subset of any topological space has the Baire property. By \cite[21.6]{K}, each analytic subspace of a Polish space $X$ has the Baire property in $X$. We recall that a topological space is \index{analytic space}{\em analytic} if it is a continuous image of a Polish space.

A topological space $X$ is \index{Baire space}{\em Baire} if for any sequence $(U_n)_{n\in\w}$ of open dense subsets in $X$ the intersection $\bigcap_{n\in\w}U_n$ is dense in $X$. It is well-known that a topological space $X$ is Baire if and only if each nonempty open set $U\subset X$ is not meager in $X$.

An indexed family of sets $(A_i)_{i\in I}$ is \index{disjoint family}\index{family!disjoint}{\em disjoint} if for any distinct indices $i,j\in I$ the sets $A_i$ and $A_j$ are disjoint. A family of sets $\mathcal A$ is {\em disjoint} if any distinct sets $A,B\in\mathcal A$ are disjoint. A subset of a topological space is \index{subset!clopen}\index{clopen subset}{\em clopen} if it is closed and open.

In Sections~\ref{s12} and \ref{s13} we shall need the following known fact on the preservation of meager and comeager sets by open maps.

\begin{lemma}\label{l:M-coM} Let $f:X\to Y$ be an open continuous surjective map between Polish spaces.
\begin{enumerate}
\item[\textup{1)}] A subset $A\subset Y$ (with the Baire property) is meager in $Y$ (if and) only if its preimage $f^{-1}(A)$ is meager in $X$.
\item[\textup{2)}] A subset $B\subset Y$ is comeager in $Y$ if and only if its preimage $f^{-1}(B)$ is comeager in $X$.
\end{enumerate}
\end{lemma}

\begin{proof} 1. If $A$ is meager, then $A=\bigcup_{n\in\w}A_n$ for some nowhere dense sets $A_n\subset Y$. The openness of $f$ guarantees that for every $n\in\w$ the closed set $f^{-1}(\bar A_n)$ has empty interior in $X$, so $f^{-1}(A_n)$ is nowhere dense and $f^{-1}(A)=\bigcup_{n\in\w}f^{-1}(A_n)$ is meager in $X$.

Now assume that the set $A$ has the Baire property and is not meager in $Y$. The Baire property of $A$ yields an open set $U\subset Y$ such that the symmetric difference $U\triangle A$ is meager. Since $A$ is not meager, the open set $U$ is not empty. As we have already proved, the preimage $f^{-1}(U\triangle A)=f^{-1}(U)\triangle f^{-1}(A)$ of the meager set $U\triangle A$ is meager. Since the open set $f^{-1}(U)$ is not empty, the set $f^{-1}(A)$ is not meager in $X$.
\smallskip

2. If a set $B\subset Y$ is comeager in $Y$, then $Y\setminus B$ is meager, its preimage $f^{-1}(Y\setminus B)=X\setminus f^{-1}(B)$ is meager in $X$ and $f^{-1}(B)$ is comeager in $X$.

If $B$ has comeager preimage $f^{-1}(B)$ in the Polish space $X$, then $f^{-1}(B)$ contains a dense $G_\delta$-set $G$ whose image $f(G)\subset B$ is an analytic subspace of the Polish space $Y$. By \cite[21.16]{K}, the set $f(G)$ has the Baire property in $Y$. So there exists an open set $U\subset Y$ such that $U\triangle f(G)$ is meager in $Y$. We claim that the set $U$ is dense in $Y$. In the opposite case, we can find a nonempty open set $V\subset Y$ such that $V\cap U=\emptyset$ and hence $V\cap f(G)\subset f(G)\setminus U\subset U\triangle f(G)$ is meager in $Y$. By the first item the preimage $f^{-1}(V\cap f(G))$ is meager in $f^{-1}(V)$, which is not possible as it contains the comeager subset $G\cap f^{-1}(V)$ of the nonempty Polish space $f^{-1}(V)$. So the set $U$ is dense in $Y$ and the sets $f(G)\subset B$ are comeager in $Y$.
\end{proof}

For a topological space $X$ by $\F(X)$ we denote the space of all closed subsets of $X$. The space $\F(X)$ is endowed with the \index{Fell topology}{\em Fell topology}, which is generated
by the subbase consisting of sets $U^+:=\{F\in\F(X):F\cap U\ne\emptyset\}$ and $K^-=\{F\in\F(X):F\cap K=\emptyset\}$ where $U$ and $K$ run over open and compact subsets of $X$, respectively. The Borel $\sigma$-algebra on $\F(X)$ generated by the Fell topology of $\F(X)$ is called the \index{Effros-Borel structure}{\em Effros-Borel structure} of $\F(X)$; see \cite[\S12.C]{K}.

For a topological space $X$ by \index{$\K(X)$}$\K(X)$ we denote the space of all nonempty compact subsets of $X$, endowed with the Vietoris topology. If the topology of $X$ is generated by a (complete) metric $\rho$, then the Vietoris topology of $\K(X)$ is generated by the (complete) metric
$$\rho_H(A,B)=\max\{\max_{a\in A}\rho(a,B),\max_{b\in B}\rho(A,b)\},$$
called the \index{Hausdorff metric}{\em Hausdorff metric} on $\K(X)$.

For a compact Hausdorff space $X$ the Vietoris topology on $\K(X)$ coincides with the subspace topology inherited from the Fell topology of the space $\F(X)$.

For a topological space $X$ by \index{$P(X)$}$P(X)$ we denote the space of all probability $\sigma$-additive Borel measures on $X$, endowed with the topology, generated by the subbase consisting of sets $\{\mu\in P(X):\mu(U)>a\}$, where $U$ runs over open sets in $X$ and $a\in[0,1)$. It is known \cite[17.23]{K} that for any Polish space $X$ the space $P(X)$ is Polish, too.

 A measure $\mu\in P(X)$ on a topological space $X$ is called
\begin{itemize}
\item \index{measure!continuous}\index{continuous measure}{\em continuous} if $\mu(\{x\})=0$ for all $x\in X$;
\item \index{measure!strictly positive}\index{strictly positive measure}{\em strictly positive} if each nonempty open set $U\subset X$ has measure $\mu(U)>0$;
\item \index{Radon measure}\index{measure!Radon}{\em Radon} if for every $\e>0$ there exists a compact set $K\subset X$ of measure $\mu(K)>1-\e$.
\end{itemize}
It is well-known \cite[7.14.22]{Bog} that each measure $\mu\in P(X)$ on a Polish space $X$ is Radon.

Each continuous map $f:X\to Y$ between topological spaces induces a continuous map $P\!f:P(X)\to P(Y)$ assigning to each probability measure $\mu\in P(X)$ the measure $\nu\in P(Y)$ defined by $\nu(B)=\mu(f^{-1}(B))$ for any Borel set $B\subset Y$.
For a measure $\mu\in P(X)$ on a topological space $(X,\tau)$ the set $$\supp(\mu):=\big\{x\in X:\forall U\in\tau\;\big(x\in U\Ra\mu(U)>0\big)\big\}$$is called the \index{support of a measure}\index{measure!support of}{\em support} of $\mu$.

\smallskip

For a compact metrizable space $K$ and a topological space $Y$ by \index{$\C(K,Y)$}$\C(K,Y)$ we denote the space of all continuous functions from $K$ to $Y$ endowed with the compact-open topology. It is well-known that for a Polish space $Y$ the function space $\C(K,Y)$ is Polish. Moreover, for any complete metric $\rho$ generating the topology of $Y$, the compact-open topology on $\C(K,Y)$ is generated by the sup-metric $$\hat\rho(f,g):=\sup_{x\in K}\rho(f(x),g(x)).$$ By \index{$\E(K,Y)$}$\E(K,Y)$ we denote the subspace of $\C(K,Y)$ consisting of injective continuous maps from $K$ to $Y$.

A topological space $X$ is called
\begin{itemize}
\item \index{topological space!scattered}\index{scattered space}{\em scattered} if each nonempty subspace of $X$ has an isolated point;
\item \index{topological space!crowded}\index{crowded space}{\em crowded} if it has no isolated points.
\end{itemize}
So, a topological space is not scattered if and only if it contains a nonempty crowded subspace.

The following lemma is well-known but we could not find a proper reference, so we provide a proof for the convenience of the reader.

\begin{lemma}\label{E-in-C} For any (zero-dimensional) compact metrizable space $K$ and any (crowded) Polish space $Y$ the set $\E(K,Y)$ is a (dense) $G_\delta$-set in the function space $\C(K,Y)$.
\end{lemma}

\begin{proof} Lemma 1.11.1 in \cite{End} implies that $\E(K,Y)$ is a $G_\delta$-set in $\C(K,Y)$. Next, assuming that the compact space $K$ is zero-dimensional and the Polish space $Y$ is crowded, we show that the set $\E(K,Y)$ is dense in $\C(K,Y)$. Fix a metric $\rho$ generating the topology of the Polish space $Y$. Given any $f\in\C(K,Y)$ and $\e>0$, we should find a function $g\in\E(K,Y)$ such that $\hat \rho(g,f)<\e$. By the continuity of $f$, the zero-dimensional compact space $K$ admits a finite disjoint open cover $\U$ such that for every $U\in\U$ the set $f(U)$ has diameter $<\frac\e3$ in the metric space $(Y,\rho)$. In each $U\in\U$ choose a point $x_U$. Since the space $Y$ has no isolated points, we can choose a family $(y_U)_{U\in\U}$ of pairwise distinct points in $Y$ such that $\rho(y_U,f(x_U))<\frac\e3$. Choose a $\delta>0$ such that $\delta<\frac\e3$ and $\delta\le\frac12\min\{\rho(y_U,y_V):U,V\in\U,\;U\ne V\}$. For every $U\in \U$, we can apply Theorem 6.2 from \cite{K} and find an injective continuous map $g_U:U\to B(y_U;\delta)$.
Observe that for every $x\in U$ we have
\begin{multline*}\rho(g_U(x),f(x))\le \rho(g_U(x),y_U)+\rho(y_U,f(x_U))+\rho(f(x_U),f(x))\\ <\delta+\frac\e3+\diam(f(U))<\frac\e3+\frac\e3+\frac\e3=\e.
\end{multline*}
Since the family of balls $\big(B(y_U;\delta)\big)_{U\in\U}$ is disjoint, the family $(g_U(U))_{U\in\U}$ is disjoint, too. Then the map $g:K\to Y$ defined by $g|U=g_U$ for $U\in\U$ is injective and satisfies $\hat \rho(g,f)<\e$.
\end{proof}




\section{Steinhaus properties of ideals and semi-ideals}

A nonempty family $\I$ of subsets of a set $X$ is called
\begin{itemize}
\item \index{family!directed}\index{directed family}{\em directed} if for any sets $A,B\in\I$ there exists a set $C\in\I$ such that $A\cup B\subset C$;
\item \index{family!$\sigma$-continuous}\index{$\sigma$-continuous family}{\em $\sigma$-continuous} if for any countable directed subfamily $\mathcal D\subset\I$ the union $\bigcup\mathcal D\in \I$;
\item a \index{semi-ideal}{\em semi-ideal} if $\I$ is closed under taking subsets;
\item an \index{ideal}{\em ideal} if $\I$ is a directed semi-ideal;
\item a \index{$\sigma$-ideal}{\em $\sigma$-ideal} if $\I$ is a $\sigma$-continuous ideal.
\end{itemize}
It is clear that

\centerline{$\sigma$-continuous $\Leftarrow$ $\sigma$-ideal $\Ra$ ideal $\Ra$ semi-ideal.}
\smallskip

A semi-ideal $\I$ on a set $X$ is \index{proper semi-ideal}{\em proper} if $X\notin\I$.

Each nonempty family $\F$ of subsets of a set $X$ generates the semi-ideal \index{${\downarrow}\F$}$${\downarrow}\F=\{E\subset X:\exists F\in\F\;\;E\subset F\},$$
the ideal\index{$\vee\F$}
$$\vee\F=\{E\subset X:\exists \,\E\in[\F]^{<\w},\;\;E\subset\textstyle{\bigcup}\E\},$$
and the $\sigma$-ideal\index{$\sigma\F$}
$$\sigma\F=\{E\subset X:\exists \,\E\in[\F]^{\le\w},\;\;E\subset\textstyle{\bigcup}\E\},$$
where by \index{$[\F]^{<\w}$}, \index{$[\F]^{\le \w}$}$[\F]^{<\w}$ and $[\F]^{\leq\w}$ we denote the families of finite and at most countable subfamilies of $\F$, respectively.

For a family $\F$ of subsets of a topological space $X$ by \index{$\overline\F$}$\overline\F$ we denote the family of all closed subsets of $X$ that belong to the family $\F$. Then \index{$\sigma\overline{\F}$}$\sigma\overline\F$ is the $\sigma$-ideal generated by closed sets in $\F$.

For a semi-ideal $\I$ on a set $X=\bigcup\I\notin\I$, consider the four standard cardinal characteristics:
$$\begin{aligned}
&\index{$\on{add}(\I)$}\on{add}(\I):=\min\{|\J|:\J\subset\I\;\;(\cup\J\notin \I)\};\\
&\index{$\on{cov}(\I)$}\on{cov}(\I):=\min\{|\J|:\J\subset\I\;\;(\cup\J=X)\};\\
&\index{$\on{non}(\I)$}\on{non}(\I):=\min\{|A|:A\subset X\;\;(A\notin\I)\};\\
&\index{$\on{cof}(\I)$}\on{cof}(\I):=\min\{|\J|:\J\subset \I\;\;(\forall I\in\I\;\exists J\in\J\;\;I\subset J)\}.
\end{aligned}
$$

A semi-ideal $\I$ on a topological group $X$ is called
\begin{itemize}
\item \index{semi-ideal!translation-invariant}\index{translation-invariant semi-ideal}{\em translation-invariant} if for any $I\in\I$ and $x\in X$ the set $I+x$ belongs to $\I$;
\item \index{semi-ideal!invariant}\index{invariant semi-ideal}{\em invariant} if $\I$ is translation-invariant and for any topological group isomorphism $f:X\to X$ the family $\{f(I):I\in\I\}$ coincides with $\I$.
\end{itemize}

Let $\I$ be a semi-ideal on a topological space $X$ and $\J$ be a semi-ideal on a topological space $Y$. We say that the semi-ideals $\I,\J$ are \index{semi-ideals!topologically isomorphic}\index{topologically isomorphic semi-ideals}{\em topologically isomorphic} if there exists a homeomorphism $f:X\to Y$ such that $\J=\{f(A):A\in \I\}$.

\begin{definition} We shall say that a family $\F$ of subsets of a Polish group $X$ has
\begin{itemize}
\item \index{strong Steinhaus property}\index{property!strong Steinhaus}{\em strong Steinhaus property} if for any Borel subsets $A,B\notin\F$ of $X$ the sum $A+B$ has nonempty interior and the difference $A-A$ is a neighborhood of zero in $X$;
\item the \index{Steinhaus property}\index{property!Steinhaus}{\em Steinhaus property} if for any Borel subset $A\notin\F$ the difference $A-A$ is a neighborhood of zero;
\item the \index{weak Steinhaus property}\index{property!weak Steinhaus}{\em weak Steinhaus property} if for any Borel subset $A\notin\F$ the difference $A-A$ is not meager in $X$.
\end{itemize}
\end{definition}

It is clear that
$$
\mbox{strong Steinhaus property $\Ra$ Steinhaus property $\Ra$ weak Steinhaus property}.
$$

\begin{theorem}\label{t:S+} Let $\I$ be a semi-ideal of subsets of a Polish group $X$ such that for any closed subsets $A,B\notin\I$ of $X$ the set $A-B$ is not nowhere dense in $X$. Then for any analytic subspaces $A,B\notin\sigma\overline{\I}$ of $X$ the set $A-B$ is not meager in $X$. Consequently, the $\sigma$-ideal $\sigma{\overline{\I}}$ on $X$ has the weak Steinhaus property.
\end{theorem}

\begin{proof} Given two analytic subspaces $A,B\notin\sigma\overline{\I}$ of $X$ we should prove that $A-B$ is non-meager in $X$. The spaces $A,B$, being analytic, are images of Polish spaces $P_A$ and $P_B$ under continuous surjective maps $f_A:P_A\to A$ and $f_B:P_B\to B$. Let $W_A$ be the union of all open subsets $U\subset P_A$ whose images $f_A(U)$ belong to the $\sigma$-ideal $\sigma\overline{\I}$. Since $W_A$ is Lindel\"of, the image $f_A(W_A)$ belongs to $\sigma\overline{\I}$. By the maximality of $W_A$, for every nonempty relatively open subset $V\subset P_A\setminus W_A$ the image $f_A(V)$ does not belong to $\sigma\overline{\I}$. Replacing $A$ by $f_A(P_A\setminus W_A)$, we can assume that $W_A$ is empty, so for every nonempty open subset $V\subset P_A$ the image $f_A(V)\notin\sigma\overline{\I}$.

By analogy, we can replace $B$ by a smaller analytic subset and assume that for any nonempty open set $V\subset P_B$ the image $f_B(V)\notin\sigma\overline{\I}$. We claim that $A-B$ is not meager in $X$. To derive a contradiction, assume that $A-B$ is meager and hence $A-B\subset \bigcup_{n\in\w}F_n$ for a sequence $(F_n)_{n\in\w}$ of closed nowhere dense subsets $F_n\subset X$.

Now consider the map $\partial:P_A\times P_B\to A-B$, $\partial:(x,y)\mapsto f_A(x)-f_B(y)$. By the continuity of $\partial$, for any $n\in\w$ the preimage $\partial^{-1}(F_n)$ is a closed subset of the Polish space $P_A\times P_B$. Since $P_A\times P_B=\bigcup_{n\in\w}\partial^{-1}(F_n)$, by the Baire Theorem, for some $n\in\w$ the closed set $\partial^{-1}(F_n)$ has nonempty interior and hence contains a nonempty open subset $V_A\times V_B$ of $P_A\times P_B$.

By our assumption, the images $f_A(V_A)$ and $f_B(V_B)$ do not belong to $\sigma\overline{\I}$. Then their closures $\overline{f_A(V_A)}$ and $\overline{f_B(V_B)}$ do not belong to the ideal $\I$. By our assumption, the difference $\overline{f_A(V_A)}-\overline{f_B(V_B)}$ is not nowhere dense in $X$.

On the other hand, by closedness of $F_n$ and inclusion $f_A(V_A) - f_B(V_B) \subset F_n$ we get $$f_A(V_A) - f_B(V_B) \subset \overline{f_A(V_A)} - \overline{f_B(V_B)} \subset \overline{f_A(V_A) - f_B(V_B)} \subset F_n$$ which yields a desired contradiction.
\end{proof}

Now we consider some standard examples of ideals and $\sigma$-ideals.

For a set $X$ and a cardinal $\kappa$ let
\begin{itemize}
\item \index{$[X]^{\le \kappa}$}$[X]^{\le \kappa}:=\{A\subset X:|A|\le\kappa\}$ and
\item  \index{$[X]^{<\kappa}$}$[X]^{<\kappa}:=\{A\subset X:|A|<\kappa\}$.
\end{itemize}
For an infinite cardinal $\kappa$ the family $[X]^{<\kappa}$ is an ideal and the family $[X]^{\le k}$ is a $\sigma$-ideal. For any finite cardinal $\kappa$ the families $[X]^{\le \kappa}$ and $[X]^{<\kappa}$ are $\sigma$-continuous semi-ideals.

For a topological space $X$ by \index{$\M_X$}$\M_X$ we denote the $\sigma$-ideal of meager subsets of $X$.
If $X$ is also endowed with a $\sigma$-additive Borel measure $\mu$, then by \index{$\M_\mu$}$\N_\mu$ we denote the $\sigma$-ideal of sets of zero $\mu$-measure in $X$ (more precisely, the family of subsets of Borel sets of zero $\mu$-measure in $X$). If the topological space $X$ or the measure $\mu$ is clear from the context, then we omit the subscript and write \index{$\M$}$\M$ and \index{$\N$}$\N$ instead of $\M_X$ and $\N_\mu$. For a locally compact topological group $X$ by $\N_X$ or just $\N$ we denote the ideal $\N_\lambda$ of null sets of any Haar measure $\lambda$ on $X$ (which is unique up to a multiplicative constant). By \cite[2.12]{Akin99}, for any strictly positive continuous measures $\mu,\nu\in P(2^\w)$ on the Cantor cube $2^\w$ and any $\e>0$ there exists a homeomorphism $h:2^\w\to 2^\w$ such that
 $$(1-\e)\mu(B)\le\nu(h(B))\le(1+\e)\mu(B)$$
 for any Borel set $B\subset 2^\w$. This implies that the $\sigma$-ideals $\N_\mu$ and $\N_\nu$ are topologically isomorphic.


Observe that for each topological space $X$ we get $\sigma\overline{\M}=\M$. On the other hand, for any strictly positive continuous measure $\mu$ on a Polish space $X$, the $\sigma$-ideal $\sigma\overline{\N}_\mu$ is a proper subideal of $\M_X\cap\N_\mu$.
In \cite{BRZ}, \cite{BaJu} the $\sigma$-ideal \index{$\sigma\overline{\N}$}$\sigma\overline{\N}$ on the Cantor cube is denoted by $\E$.

Now we recall some known Steinhaus-type properties of the ideals $\M$ and $\N$ on topological groups.

\begin{theorem}[Steinhaus \cite{Stein}, Weil {\cite[\S11]{Weil}}]\label{Stein} For every locally compact Polish group the ideal $\N$ has the strong Steinhaus property.
\end{theorem}

\begin{theorem}[Piccard \cite{Pic}, Pettis \cite{Pet}]\label{t:PP} For every Polish group the ideal $\M$ has the strong Steinhaus property. 
\end{theorem}

\begin{corollary}\label{c:MN-} For every locally compact Polish group the ideal $\M\cap\N$ has the Steinhaus property.
\end{corollary}

\begin{remark}\label{r:MN+} The ideal $\M\cap\N$ on a non-discrete locally compact Polish group $X$ does not have the strong Steinhaus property. Indeed, for any dense $G_\delta$-set $A\subset X$ of Haar measure zero and any countable dense set $D\subset X$ the set $B=\bigcap_{x\in D}(x-X\setminus A)$ is meager of full Haar measure. So, $A,B\notin\M\cap\N$ but $A+B$ is disjoint with $D$ and hence has empty interior in $X$. This example is presented in Example 1.3 and Theorem 1.11 of \cite{BFN}. Since $\sigma\overline{\N}\subset\M\cap\N$, the $\sigma$-ideal $\sigma\overline{\N}$ does not have the strong Steinhaus property, too. On the other hand, this ideal has the weak Steinhaus property, established in the following corollary of Theorems~\ref{t:S+}, \ref{Stein} and \ref{t:PP}.
\end{remark}

\begin{corollary}\label{c:sN->S+} For any analytic subsets $A,B\notin\sigma\overline{\N}$ of a locally compact Polish group $X$ the set $A-B$ is non-meager in $X$ and hence $(A-B)-(A-B)$ is a neighborhood of zero in $X$.
Consequently, the ideal $\sigma\overline{\N}$ in a locally compact Polish group $X$ has the weak Steinhaus property.
\end{corollary}

\begin{example}\label{e:sN+-} For the compact metrizable topological group $X=\prod_{n\in\w} C_{2^n}$ the ideal $\sigma\overline{\N}$ does not have the Steinhaus property.
\end{example}

\begin{proof} In the compact topological Abelian group $X=\prod_{n\in\w} C_{2^n}$ consider the closed subset
$A=\prod_{n\in\w} C^*_{2^n}$ of Haar measure $\prod_{n\in\w}\frac{2^n-1}{2^n}>0$.
Here $C_{2^n}^*$ stands for the set of non-zero elements of the cyclic group $C_{2^n}$.
Fix a countable dense set $D\subset X$ consisting of points $x\in X$ such that $x(n)\ne \theta(n)$ for all but finitely many numbers $n$. Observe that for any such $x\in D$ the intersection $A\cap(x+A)$ is nowhere dense in $A$. Then the set $B=A\setminus\bigcup_{x\in D}(x+A)$ is a dense $G_\delta$-set in $A$.
Taking into account that restriction of Haar measure to $\overline{B}=A$ is strictly positive, and applying the Baire Theorem, we can prove that the $G_\delta$-subset $B$ of $X$ does not belong to the $\sigma$-ideal $\sigma\overline{\N}$. On the other hand, $B\cap (x+B)\subset B\cap (x+A)=\emptyset$ for any $x\in D$ and hence $B-B$ is disjoint with the dense set $D$. So, $B-B=B+(-B)$ has empty interior in $X$ and the ideal $\sigma\overline{\N}$ does not have the Steinhaus property.
\end{proof}

In this paper we shall often apply the following (known) corollary of the Piccard-Pettis Theorem \ref{t:PP}. We provide a short proof for completeness.

\begin{corollary}[Piccard, Pettis]\label{c:PP} For any non-meager analytic subset $A$ of a Polish group $X$ the set $A-A$ is a neighborhood of $\theta$ in $X$.
\end{corollary}

\begin{proof} By \cite[21.6]{K}, the analytic set $A$ has the Baire property in $X$. So, there exists an open set $U\subset X$ such that the symmetric difference $U\triangle A$ is meager and hence is contained in some meager $F_\sigma$-set $M\subset X$. It follows that $G:=U\setminus M$ is a $G_\delta$-set in $X$ such that $G\subset A\cup M$. Since the set $A$ is not meager, the open set $U$ is not empty and the $G_\delta$-set $G=U\setminus M$ is not meager in $X$. By the Piccard-Pettis Theorem~\ref{t:PP}, the set $G-G$ is a neighborhood of $\theta$ in $X$ and so is the set $A-A\supset G-G$.
\end{proof}

Corollary~\ref{c:PP} implies the following (known) version of the Open Mapping Principle for Polish groups (cf. \cite[4.27]{Luk}).

\begin{corollary}[Open Mapping Principle]\label{c:OMP} Any continuous surjective homomorphism $h:X\to Y$ from an analytic topological group $X$ to a Polish group $Y$ is open.
\end{corollary}

\begin{proof} Given an open neighborhood $U\subset X$ of the neutral element $\theta_X$ of the group $X$, we need to prove that $h(U)$ is a neighborhood of the neutral element $\theta_Y$ in the topological group $Y$. Choose an open neighborhood $V\subset X$ of $\theta_X$ such that $V-V\subset U$. Fix any countable dense set $D\subset X$ and observe that $X=D+V$. Then $Y=h(X)=h(D)+h(V)$.  By the Baire Theorem, there exists $x\in D$ such that the set $h(x)+h(V)$ is not meager in $Y$. The open subspace $V$ of the analytic space $X$ is analytic and so is its image $h(V)$ in $Y$. By Corollary~\ref{c:PP}, the difference $(h(x)+h(V))-(h(x)+h(V))=h(V)-h(V)$ is a neighborhood of $\theta_Y$ in $Y$. Since $h(V)-h(V)=h(V-V)\subset h(U)$, the set $h(U)$ is a neighborhood of $\theta_Y$ in $Y$ and the homomorphism $h$ is open.
\end{proof}

\section{Haar-null sets}\label{s4}

Christensen \cite{Ch} defined a subset $A$ of a Polish group $X$ to be \index{subset!Haar-null}\index{Haar-null subset}{\em Haar-null} provided there exist a Borel set $B\subset X$ containing $A$ and a measure $\mu\in P(X)$ such that $\mu(B+x)=0$ for each $x\in X$. The measure $\mu$ (called the \index{witness measure}{\em witness measure} for $A$) can be assumed to have compact support, see e.g., \cite[Theorem 17.11]{K}.


The following classical theorem is due to Christensen \cite{Ch}.

 \begin{theorem}[Christensen]\label{t:HN} Let $X$ be a Polish group.
\begin{enumerate}
\item[\textup{1)}] The family \index{$\HN$}$\HN$ of all Haar-null sets in $X$ is an invariant $\sigma$-ideal on $X$.
\item[\textup{2)}] If $X$ is locally compact, then $\HN=\N$.
\item[\textup{3)}] The $\sigma$-ideal $\HN$ has the Steinhaus property.
\end{enumerate}
\end{theorem}

Solecki \cite{S01} defined a subset $A$ of a Polish group $X$ to be \index{subset!openly Haar-null}\index{openly Haar-null subset}{\em openly Haar-null} if there exists a probability measure $\mu\in P(X)$ such that for every $\e>0$ there exists an open set $U_\e\subset X$ such that $A\subset U_\e$ and $\mu(U_\e+x)<\e$ for all $x\in X$. Solecki \cite{S01} observed that openly Haar-null sets in a Polish group form a $\sigma$-ideal \index{$\HN^\circ$}$\HN^\circ$, contained in the $\sigma$-ideal $\HN$ of Haar-null sets. It is clear that each openly Haar-null set $A\subset X$ in a Polish group $X$ is contained in a Haar-null $G_\delta$-set $B\subset X$. By a result of Elekes and Vidny\'anszky \cite{EVid} (see also Corollary~\ref{c:hull-MN}), each non-locally compact Polish group $X$ contains a Borel Haar-null set $B\subset X$ which cannot be enlarged to a Haar-null $G_\delta$-set in $X$. Such set $B$ is Haar-null but not openly Haar-null. This result of Elekes and Vidny\'anszky implies the following characterization.

\begin{theorem} For any Polish group $X$ we have $\HN^\circ\subset\HN$. Moreover, $X$ is locally compact if and only if $\HN^\circ=\HN$.
\end{theorem}




Our next theorem gives a function characterization of Haar-null sets. In this theorem by $\N$ we denote the ideal of sets of Haar measure zero in the Cantor cube.

\begin{theorem}\label{Haarnull}
For a Borel subset $A$ of a Polish group $X$ the following conditions are equivalent:
\begin{enumerate}
\item[\textup{1)}] $A$ is Haar-null in $X$;
\item[\textup{2)}] there exists an injective continuous map $f:2^\w\to X$ such that $f^{-1}(A+x)\in\N$ for all $x\in X$;
\item[\textup{3)}] there exists a continuous map $f:2^\w\to X$ such that $f^{-1}(A+x)\in\N$ for all $x\in X$.
\end{enumerate}
\end{theorem}

\begin{proof} The implication $(2)\Ra(3)$ is trivial. To see $(3)\Ra(1)$, let $\mu:=P\!f(\lambda)$ be the image of the product measure $\lambda$ under the map $f$, and observe that $\mu(A+x)=\lambda(f^{-1}(A+x))$ for any $x\in X$. So, $\mu$ witnesses that $A$ is Haar-null in $X$.

The implication $(1)\Ra(2)$ follows immediately from Lemmas~\ref{l:Cantorsupp} and \ref{l:a} proved below.
\end{proof}

\begin{lemma}\label{l:Cantorsupp} For any nonempty Haar-null subset $A$ of Polish group $X$ there exists a measure $\nu\in P(X)$ such that
\begin{enumerate}
\item[\textup{1)}] $\nu(A+x)=0$ for every $x\in X$;
\item[\textup{2)}] the support $\supp(\nu)$ is homeomorphic to the Cantor cube.
\end{enumerate}
\end{lemma}

\begin{proof} Since $A$ is Haar-null, there exists a measure $\mu\in P(X)$ such that $\mu(A+x)=0$ for all $x\in X$. Observe that $\mu$ is continuous in the sense that $\mu(\{x\})=0$ for all $x\in X$. We claim that $\mu(K)>0$ for some subset $K\subset X$, homeomorphic to the Cantor cube.

Fix any metric $\rho$ generating the topology of $X$. Let $D$ be a countable dense subset of $X$. For any $x\in X$ let $L_x$ be the set of all positive real numbers $r>0$ such that the sphere $S(x,r)=\{y\in X:\rho(x,y)=r\}$ has positive measure $\mu(S(x,r))$. The countable additivity of the measure $\mu$ ensures that $L_x$ is at most countable. Fix a countable dense set $R\subset (0,+\infty)\setminus\bigcup_{x\in D}L_x$.

By the $\sigma$-additivity of $\mu$, the union $S:=\bigcup_{x\in D}\bigcup_{r\in R}S(x,r)$ has measure $\mu(S)=0$. Then the complement $X\setminus S$ is a zero-dimensional $G_\delta$-set of measure $\mu(X\setminus S)=1$. By the regularity of the measure $\mu$, there exists a compact (zero-dimensional) subset $K\subset X\setminus S$ of positive measure $\mu(K)>0$. Replacing $K$ by a suitable closed subset, we can assume that each nonempty open subset $U$ of $K$ has positive measure $\mu(U)$. Since the measure $\mu$ is continuous, the compact space $K$ has no isolated points and being zero-dimensional, is homeomorphic to the Cantor cube $2^\w$.

It is easy to see that the measure $\nu\in P(X)$ defined by $\nu(B)=\frac{\mu(B\cap K)}{\mu(K)}$ for any Borel subset $B\subset X$ has the required properties.
\end{proof}

\begin{lemma}\label{l:a} For every continuous probability measure $\mu\in P(X)$ on a Polish space $X$ and every positive $a<1$ there exists an injective continuous map $f:2^\w\to X$ such that for any Borel set $B\subset 2^\w$ we get $\mu(f(B))=a\cdot\lambda(B)$.
\end{lemma}

\begin{proof} Fix a decreasing sequence $(a_n)_{n\in\w}$ of real numbers with $a=\inf_{n\in\w}a_n<a_0=1$.

For a point $x\in X$ let $\tau_x$ be the family of open neighborhoods of $x$ in $X$. A subset $P\subset X$ will be called \index{subset!$\mu$-positive}\index{$\mu$-positive subset}{\em $\mu$-positive} if for any $x\in P$ and neighborhood $O_x\subset X$ of $x$ the set $P\cap O_x$ has positive measure $\mu(P\cap O_x)$.

Observe that for every Borel set $B\subset X$ the subset
$$P=\{x\in B:\forall O_x\in\tau_x\;\;\mu(B\cap O_x)>0\}$$
of $B$ is closed in $B$, is $\mu$-positive, and has measure $\mu(P)=\mu(B)$.

By induction we shall construct an increasing number sequence $(n_k)_{k\in\w}$ and a family $(X_s)_{s\in 2^{<\w}}$ of $\mu$-positive compact subsets of $X$ such that for every $k\in\w$ and $s\in 2^{n_k}$ the following conditions are satisfied:
\begin{enumerate}
\item[\textup{(1)}] $\frac{a_{2k+1}}{2^{n_k}}<\mu(X_s)\le\frac{a_{2k}}{2^{n_k}}$;
\item[\textup{(2)}] the family $\{X_{t}:t\in 2^{n_{k+1}},\;t|n_k=s\}$ is disjoint and consists of $\mu$-positive compact subsets of $X_s$ of diameter $<\frac1{2^k}$;
\item[\textup{(3)}] for any $n\in\w$ with $n_{k}<n\le n_{k+1}$ and any $t\in 2^n$ we get $X_t=\bigcup\{X_{\sigma}:\sigma\in 2^{n_{k+1}},\;\sigma|n=t\}$.
\end{enumerate}
To start the inductive construction put $n_0=0$ and repeating the proof of Lemma~\ref{l:Cantorsupp}, choose any $\mu$-positive zero-dimensional compact set $X_\emptyset\subset X$ of measure $\mu(X_\emptyset)>a_1$. Assume that for some $k\in\w$ the numbers $n_0<n_1<\dots<n_k$ and families $(X_s)_{s\in 2^{n_i}}$, $i\le k$, satisfying the inductive conditions (1)--(3) have been constructed. The condition (2) implies that for every $s\in 2^{n_k}$ the set $X_s$ is a subset of the zero-dimensional compact space $X_\emptyset$. So, there exists a finite disjoint cover $\W_s$ of $X_s$ by clopen subsets of diameter $<\frac1{2^{k}}$.

Choose a number $n_{k+1}>n_k$ such that $$\frac{|\W_s|}{2^{n_{k+1}-n_k}}<\frac{a_{2k+1}}{a_{2k+2}}-1\mbox{ \ for every $s\in 2^{n_k}$}.$$
This inequality and the inductive assumption (1) implies that for every $s\in 2^{n_k}$ we have
$$\sum_{W\in\W_s}\frac{2^{n_{k+1}}}{a_{2k+2}}\cdot\mu(W)=
\frac{2^{n_{k+1}}}{a_{2k+2}}\cdot \mu(X_s)> \frac{2^{n_{k+1}}}{a_{2k+2}}\frac{a_{2k+1}}{2^{n_k}}=
2^{n_{k+1}-n_k}\frac{a_{2k+1}}{a_{2k+2}}>
2^{n_{k+1}-n_k}+|\W_s|.$$
Then for every $W\in \W_s$ we can choose a non-negative integer number $m_W\le \frac{2^{n_{k+1}}}{a_{2k+2}}\cdot\mu(W)$ such that $\sum_{W\in\W_s}m_W=2^{n_{k+1}-n_k}$.
For every $W\in\W_s$ the inequality $m_W\frac{a_{2k+2}}{2^{n_{k+1}}}\le\mu(W)$ allows us to choose a disjoint family $(X_{W,i})_{i\in m_W}$ of $\mu$-positive compact subset of $W$ of measure $\frac{a_{2k+3}}{2^{n_{k+1}}}<\mu(X_{W,i})\le \frac{a_{2k+2}}{2^{n_{k+1}}}$ for all $i\in m_W$.
Let $\{X_t\}_{t\in 2^{n_{k+1}}}$ be any enumeration of the family $\bigcup_{s\in 2^{n_k}}\bigcup_{W\in\W_s}\{X_{W,i}:i\in m_W\}$ such that for every $s\in 2^{n_k}$ we get
$\{X_t:t\in 2^{n_{k+1}},\;t|2^{n_k}=s\}=\bigcup_{W\in\W_s}\{X_{W,i}\}_{i\in m_W}$.
It is clear that the family $(X_t)_{t\in 2^{n_{k+1}}}$ satisfies the conditions (1), (2). For every $n\in\w$ with $n_{k}<n\le n_{k+1}$ and any $t\in 2^n$ put $X_t=\bigcup\{X_{\sigma}:\sigma\in 2^{n_{k+1}},\;\sigma|n=t\}$.
This completes the inductive step.

After completing the inductive construction, we obtain a family $(X_s)_{s\in 2^{<\w}}$ of $\mu$-positive compact sets in $X$ satisfying the conditions (1)--(3). By the condition (2), for every $s\in 2^\w$ the intersection $\bigcap_{n\in\w}X_{s|n}$ contains a unique point $x_s$. So, the map $f:2^\w\to X$, $f:s\mapsto x_s$, is well-defined. By the condition (2) this map is continuous and injective.

Now given any $k\in \w$ and $s\in 2^{n_k}$, consider the basic open set
$U_s=\{t\in 2^\w:t|n_k=s\}\subset 2^\w$ of measure $\lambda(U_s)=\frac1{2^{n_k}}$. Next, for every $i\ge k$ consider the set $X_{s,i}=\bigcup\{X_t:t\in 2^{n_i},\;t|n_k=s\}$. Applying the condition (1), we conclude that
$$\frac{a_{2i+1}}{2^{n_k}}=\frac{a_{2i+1}}{2^{n_i}}\cdot 2^{n_i-n_k}<\mu(X_{s,i})\le \frac{a_{2i}}{2^{n_i}}\cdot 2^{n_i-n_k}=\frac{a_{2i}}{2^{n_k}}.$$ Taking into account that $f(U_s)=\bigcap_{i=k}^\infty X_{s,i}$, we conclude that
$$\mu(f(U_s))=\lim_{i\to\infty}\mu(X_{s,i})=\frac{a}{2^{n_k}}=a\cdot\lambda(U_s).$$
Since each clopen subset of $2^\w$ is a finite disjoint union of basic open sets $U_s$, $s\in 2^{<\w}$, we obtain that $\mu(f(U))=a\cdot\lambda(U)$ for any clopen set $U\subset 2^\w$. Now the countable additivity of the measures $\mu$ and $\lambda$ implies that $\mu(f(B))=a\cdot\lambda(B)$ for any Borel subset $B\subset 2^\w$.
\end{proof}

\section{Haar-meager sets}\label{s5}

In \cite{D} Darji defined a subset $A$ of a Polish group $X$ to be \textit{Haar-meager} if there are a Borel set $B\subset X$ with $A\subset B$, and a continuous function $f:K\to X$ defined on a compact metrizable space $K$ such that $f^{-1}(B+x)$ is meager in $K$ for each $x\in X$. If $B$ is not empty, then the compact space $K$ cannot contain isolated points.

In fact, we can assume that the compact space $K$ in this definition is the Cantor cube.

\begin{proposition}\label{omega}
A Borel set $B\subset X$ is Haar-meager if and only if there is a continuous function $f:2^{\omega}\to X$ such that $f^{-1}(B+x)$ is meager in $2^{\omega}$ for all $x\in X.$
\end{proposition}

\begin{proof} The ``if'' part is trivial. To prove the ``only if'' part, assume that $B$ is Haar-meager and fix a continuous function $h:K\to X$ defined on a compact metrizable space $K$ such that $h^{-1}(B+x)\in\mathcal{M}$ in $K$ for every $x\in X$. By \cite[Lemma~2.10]{D}, there is a surjective continuous function $f:2^\w\to K$ such that $f^{-1}(M)\in\mathcal{M}$ in $2^\w$ for each meager set $M\subset K$. Hence $h\circ f:2^{\w}\to X$ is continuous and $(h\circ f)^{-1}(B+x)\in\mathcal{M}$ for each $x\in X$.
\end{proof}

The above proposition shows that Haar-meager sets can be equivalently defined as follows.

\begin{definition} A Borel subset $B$ of a Polish group $X$ is called
\begin{itemize}
\item \index{subset!Haar-meager}\index{Haar-meager subset}{\em Haar-meager} if there exists a continuous function $f:2^{\w}\to X$ such that $f^{-1}(B+x)$ is meager in $2^\w$ for each $x\in X$;
\item \index{subset!injectively Haar-meager}\index{injectively Haar-meager subset}{\em injectively Haar-meager} if there exists an injective continuous function $f:2^{\w}\to X$ such that $f^{-1}(B+x)$ is meager in $2^\w$ for each $x\in X$;
\item\index{subset!strongly Haar-meager}\index{strongly Haar-meager subset} {\em strongly Haar-meager} if there exists a nonempty compact subset $K\subset X$ such that the set $K\cap (B+x)$ is meager in $K$ for each $x\in X$.
\end{itemize}
\end{definition}

For a Polish group $X$ by \index{$\HM$}$\HM$ (resp. \index{$\EHM$}\index{$\SHM$}$\EHM$, $\SHM$) we denote the family of subsets of (injectively, strongly) Haar-meager Borel sets in $X$.
It is clear that $$\EHM\subset \SHM\subset\HM.$$

The following theorem is proved in \cite[Theorems 2.2 and 2.9]{D}.

\begin{theorem}[Darji]\label{D1} For any Polish group $X$ the family $\HM$ is a $\sigma$-ideal, contained in $\M$.
\end{theorem}

The Steinhaus property of the ideal $\HM$ was established by Jab\l o\'nska \cite{J}.

\begin{theorem}[Jab\l o\'nska]\label{t:Jab}
For each Polish group the ideal $\HM$ has the Steinhaus property.
\end{theorem}

\begin{remark} In Corollary~\ref{c:EHT-} we shall improve Theorem~\ref{t:Jab} showing that for each Polish group the semi-ideal $\EHM$ has the Steinhaus property.
\end{remark}

\begin{example}\label{e:notHN+} The closed subset $A=\w^\w$ of the Polish group $X=\IZ^\w$ is neither Haar-null nor Haar-meager. Yet, $A+A=A$ is nowhere dense in $X$. This implies that the ideals $\HN$ and $\HM$ on $\IZ^\w$ do not have the strong Steinhaus property.
\end{example}

\begin{proof} It is easy to see that $A+A=A$ and for any compact subset $K\subset X$ there is $x\in X$ such that $K+x\subset A$. This implies that $A\notin \HN\cup\HM$.
\end{proof}

The equivalence of the first two items of the following characterization was proved by Darji \cite{D}.

\begin{theorem}\label{t:HM=M}
The following conditions are equivalent:
\begin{enumerate}
\item[\textup{1)}] a Polish group $X$ is locally compact;
\item[\textup{2)}] $\HM=\M$;
\item[\textup{3)}] $\overline{\HM}=\overline{\M}$.
\end{enumerate}
\end{theorem}

\begin{proof} The equivalence $(1)\Leftrightarrow(2)$ was proved by Darji \cite[Theorems 2.2 and 2.4]{D} and $(2)\Ra(3)$ is trivial.
To see that $(3)\Ra(1)$, assume that $X$ is not locally compact. Then, by Solecki's Theorem \cite{S}, there is a closed set $F\subset X$ and a continuous function $f:F\to 2^{\omega}$
such that for each $a\in 2^{\omega}$ the set $f^{-1}(a)$ contains a translation of every compact subset of $X$. Since $X$ is separable, there is $a_0\in 2^{\omega}$ such that $f^{-1}(a_0)$ is nowhere dense and thus it belongs to $\overline{\M}\setminus\overline{\HM}$.\end{proof}

The second part of the following problem was posed by Darji \cite[Problem 3]{D}.

\begin{problem}\label{Problem1}
Is $\EHM$ a $\sigma$-ideal? Is $\SHM$ a $\sigma$-ideal?
\end{problem}

We recall that a topological space $X$ is \index{topological space!totally disconnected}\index{totally disconnected space}{\em totally disconnected} if for any distinct points $x,y\in X$ there exists a clopen set $U\subset X$ such that $x\in U$ and $y\notin U$.

\begin{remark}\label{r:EHM=SHM}
For any totally disconnected Polish group $X$ we have
$\EHM=\SHM$.
\end{remark}

\begin{proof}
It suffices to prove that $\SHM \subset \EHM$. If $B\subset X$ is Borel and strongly Haar-meager, then there is a compact subset $K\subset X$
(without isolated points) such that $K\cap (B+x)$ is meager in $K$ for each $x\in X$.
Since $X$ is totally disconnected, $K$ is
zero-dimensional and hence is homeomorphic to the Cantor cube $2^\w$. So, we
can choose a homeomorphism $f:2^{\omega}\to K$ and observe that $f^{-1}(B+x)$ is meager in $2^{\omega}$ for each $x\in X$.
\end{proof}

\begin{remark}\label{r:prom} In Example~\ref{SHMnotsubsetEHM} we shall prove that the Polish group $\IR^\w$ contains a closed subset $F\in \SHM\setminus\EHM$. By a recent result of Elekes, Nagy, Po\'or and Vidny\'anszky \cite{ENPV}, the Polish group $\IZ^\w$ contains a $G_\delta$-set $G\in \HM\setminus\SHM$.\end{remark}


On the other hand, for hull-compact Polish groups we have the equality $\SHM=\HM$.

\begin{definition} A topological group $X$ is called \index{topological group!hull-compact}\index{hull-compact topological group}{\em hull-compact} if each compact subset of $X$ is contained in a compact subgroup of $X$.
\end{definition}

\begin{theorem}\label{t:HM=SHM}
Each hull-compact Polish group has $\HM=\SHM$.
\end{theorem}

\begin{proof}
Let $A$ be a Borel Haar-meager set in a hull-compact Polish group $X$. Take a continuous function $f:2^\w\to X$ such that $f^{-1}(A+x)$ is meager in $2^{\w}$ for each $x\in X$. By the hull-compactness of $X$ there is a compact subgroup $Y\subset X$
containing the compact set $f(2^{\w})$. We prove that $Y\cap (A+x)$ is meager in $Y$ for each $x\in X$.
Contrary suppose that this is not the case. By Theorem~\ref{D1}, $\mathcal{HM}\subset \mathcal{M}$, so $Y\cap (A+x)$ is not Haar-meager in $Y$. Consequently, there exists $y\in Y$ such that $f^{-1}((Y\cap (A+x))+y)\subset f^{-1}(A+x+y)$ is not meager in $2^{\w}$, which is a desired contradiction.
\end{proof}

We recall that a group $X$ is \index{group!locally finite}\index{locally finite group}{\em locally finite} if each finite subset of $X$ is contained in a finite subgroup of $X$.

\begin{example} The Tychonoff product $X=\prod_{n\in\w}X_n$ of infinite locally finite discrete groups $X_n$ is Polish, hull-compact, but not locally compact.
For this group we have $\EHM=\SHM=\HM\ne\M$.
\end{example}
\begin{proof}
Clearly, $X$ is Polish but not locally compact. To prove that $X$ is hull-compact, fix a compact set $K \subset X$. Then for every $n\in\w$ the projection $\pi_n (K)$ of $K$ onto the factor $X_n$ is compact. Since the group $X_n$ is discrete and locally finite, the compact subset $\pi_n(K)$ of $X_n$ is finite and generates a finite subgroup $F_n$ of $X_n$. Then $F=\prod_{n \in \omega} F_n$ is a compact subgroup of $X$, containing $K$ and witnessing that the topological group $X$ is hull-compact.

For the group $X$ we have $\EHM=\SHM=\HM\ne \M$ according to Remark~\ref{r:EHM=SHM} and Theorems~\ref{t:HM=SHM} and \ref{t:HM=M}.
\end{proof}

\section{A universal counterexample}\label{s6}
In this section we shall present a construction of a closed subset of the topological group $\IR^\w$ allowing to distinguish many ideals of small sets in $\IR^\w$.

We recall that by $\K(\IR^\w)$ we denote the hyperspace of nonempty compact subsets of $\IR^\w$ endowed with the Vietoris topology. It is well-known that $\K(\IR^\w)$ is a Polish space. A subspace $\K\subset\K(\IR^\w)$ is {\em analytic} if $\K$ is a continuous image of a Polish space.

\begin{theorem}\label{t:univer} For any nonempty analytic subspace $\K\subset\K(\IR^\w)$ there exists a closed set $F\subset\IR^\w$ such that
\begin{enumerate}
\item[\textup{1)}] for any $K\in\K$ there exists $x\in\IR^\w$ such that $K+x\subset F$;
\item[\textup{2)}] for any $x\in\IR^\w$ the intersection $F\cap(x+[0,1]^\w)$ is contained in $K+d$ for some $K\in\K$ and $d\in\IR^\w$.
\end{enumerate}
\end{theorem}

\begin{proof}
Let $I_n=[2^{n+1},2^{n+1}+n+1]\subset \IR$ and $L_m=\bigcup_{n \in \omega} I_{i(n,m)}$ for $n,m \in \omega$, where $i: \w \times \w \to \IN$ is an injective function, which is increasing with respect to the first coordinate (for example, we can take $i(x,y):=2^x3^y$ for $(x,y)\in\w\times\w$).

In the following obvious claim we put $\dist(A,B)=\inf\{|a-b|:a\in A,\;b\in B\}$ for two nonempty subsets $A,B$ of the real line.

\begin{claim}\label{dist2} For any distinct numbers $n,m\ge 1$ we have
$\dist(I_n,I_m)\ge 2$ and hence $\dist(L_n,L_m)\ge 2$.
\end{claim}

By $\IR_+$ we denote the closed half-line $[0,+\infty)$ and $\IR^{\le 2}=\{(x,y)\in\IR^2:x\le y\}$ the closed half-plane. Let also $[\IR]^2$ be the family of 2-elements subsets of $\IR$, endowed with the Vietoris topology.

\begin{lemma}\label{lemreal}
For every $m \in \omega$ there exists a continuous function $D_m: \IR^{\le 2}\to [\IR]^2$ such that for every $(x,y) \in \IR^{\le 2}$ the doubleton $D_m(x,y)$ has diameter $\ge 2+(y-x)$ and $D_m(x,y)$ contains a point $d$ such that $d+[x,y]\subset L_m$.
\end{lemma}

\begin{proof}
The function $D_m$ will be defined as $D_m(x,y)=\{a_m(x,y),b_m(x,y)\}$ for suitable continuous functions $a_m,b_m: \IR^{\le 2}\to \IR$. Let $\bar a_m,\bar b_m:\IR_+\to\IR_+$ be the piecewise linear functions such that
\begin{itemize}
\item $\bar a_m(0)=\min I_{i(1, m)}$, $\bar b_m(0)=\min I_{i(3, m)}$;
\item $\bar a_m(i(n,m))=\min I_{i(n+1,m)}$, $\bar b_m(i(n,m))=\min I_{i(n+4, m)}$ for each even $n\in\w$, and
\item $\bar a_m(i(n,m))=\min I_{i(n+2,m)}$, $\bar b_m(i(n,m))=\min I_{i(n+3, m)}$ for each odd $n\in\w$.
\end{itemize}

\begin{claim}\label{cl:diff} $\bar b_m(x)-\bar a_m(x)\ge 2+x$ for any $x\in\IR^+$.
\end{claim}

\begin{proof} Since the functions are piecewise linear, it suffices to check the inequality for $x\in\{0\}\cup\{i(n,m)\}_{n\in\w}$. First observe that
\begin{multline*}
\bar b_m(0)-\bar a_m(0)=\min I_{i(3,m)}-\min I_{i(1,m)}=2^{i(3,m)+1}-2^{i(1,m)+1}\\
=2^{i(1,m)+1}(2^{i(3,m)-i(1,m)}-1)\ge 2^2(2^2-1)=12>2+0.
\end{multline*}
For every even $n\in\w$ we have
\begin{multline*}
\bar b_m(i(n,m))-\bar a_m(i(n,m))=2^{i(n+4,m)+1}-2^{i(n+1,m)+1}\\
=2^{i(n+1,m)+1}(2^{i(n+4,m)-i(n+1,m)}-1)
\ge2^{i(n,m)+2}(2^3-1)\ge i(n,m)+2.
\end{multline*}
Finally, for every odd $n\in\w$ we have
\begin{multline*}
\bar b_m(i(n,m))-\bar a_m(i(n,m))=2^{i(n+3,m)+1}-2^{i(n+2,m)+1}\\
=2^{i(n+2,m)+1}(2^{i(n+3,m)-i(n+2,m)}-1)
\ge 2^{i(n,m)+3}\ge i(n,m)+2.
\end{multline*}
\end{proof}

Put $a_m(x,y):=\bar a_m(y-x)-x$ and $b_m(x,y):=\bar b_m(y-x)-x$ for every $(x,y)\in \IR^{\le2}$. Claim~\ref{cl:diff} implies that $b_m(x,y)-a_m(x,y)=\bar b_m(y-x)-\bar a_m(y-x)\ge 2+y-x$ for all $(x,y)\in\IR^{\le 2}$.

Now, define the continuous function $D_m:\IR^{\le2}\to[\IR]^2$ letting $$D_m(x,y):=\big\{a_m(x,y),b_m(x,y)\big\}$$ for $(x,y)\in\IR^{\le2}$. It is clear that for every $(x,y)\in\IR^{\le2}$ the doubleton $D_m(x,y)$ has diameter $b_m(x,y)-a_m(x,y)\ge 2+y-x$. Let us show that $D_m(x,y)$ contains a point $d$ such that $d+[x,y]\subset L_m$.

If $y-x<i(0,m)$, then the point $a_m(x,y)\in F_m(x,y)$ has the required property: $$a_m(x,y)+[x,y]=\bar a_m(y-x)+[0,y-x]=[\min I_{i(1,m)}, \min I_{i(1,m)}+y-x]\subset I_{i(1,m)}\subset L_m.$$ If $y-x\ge i(0,m)$ then there exists a unique $n \in \omega$ such that $i(n,m)\le y-x<i(n+1,m)$.
If $n$ is
\begin{itemize}
\item even, then $b_m(x,y)+[x,y]=[\min I_{i(n+4,m)}, \min I_{i(n+4,m)}+y-x]\subset I_{i(n+4,m)} \subset L_m$,
\item odd, then $a_m(x,y) + [x,y] = [ \min I_{i(n+2,m)}, \min I_{i(n+2,m)} + y-x] \subset I_{i(n+2,m)} \subset L_m$.
\end{itemize}
\end{proof}

Since the nonempty subspace $\K$ of $\K(\IR^\w)$ is analytic, there exists a continuous surjection $f:\w^{\w}\to\K$. For every $\alpha\in\w^\w$ consider the compact set $K_{\alpha}:=f(\alpha)\in\K$. For $m\in\w$ by $\pi_m:\IR^\w\to\IR$, $\pi_m:x\mapsto x(m)$, we denote the projection of $\IR^\w$ onto the $m$-th coordinate.

For every $\alpha\in\w^\w$ consider the subsets $$D_\alpha=\prod_{i\in\w}D_{\alpha(i)}(\min\pi_i(K_\alpha),\max\pi_i(K_\alpha))\quad\mbox{ and }\quad L_\alpha=\prod_{i\in\w}L_{\alpha(i)}$$of $\IR^\w$.

We claim that the union $F=\bigcup_{\alpha \in \omega^{\omega}} A_{\alpha}$ of the compact sets
$A_{\alpha}=(K_\alpha+D_\alpha)\cap L_\alpha$
satisfies our requirements. This will be proved in the following three lemmas.

\begin{lemma} The set $F$ is closed in $\IR^\w$.
\end{lemma}

\begin{proof} Fix a sequence $(x_n)_{n\in\w}$ of elements of $F$ convergent to some $x\in \IR^{\w}$. Since $A_{\alpha}$'s are pairwise disjoint, we obtain that for each $n\in \omega$ there exists an exactly one $\alpha_n \in \w^{\w}$ with $x_n \in A_{\alpha_n}$.

We claim that the sequence $(\alpha_n)_{n\in\w}$ is convergent in $\w^\w$.
The convergence of the sequence $(x_n)_{n\in\w}$ implies that
for any $i\in\w$ there exists $n\in\w$ such that $|x_n(i)-x_m(i)|<2$ for any $m\ge n$. We claim that $\alpha_m(i)=\alpha_n(i)$ for any $m\ge n$. Assuming that $\alpha_m(i)\ne\alpha_n(i)$ for some $m\ge n$ and taking into account that $x_m\in A_{\alpha_m}\subset \prod_{i\in\w}L_{\alpha_m(i)}$ and $x_n\in A_{\alpha_n}\subset \prod_{i\in\w}L_{\alpha_n(i)}$, we conclude that
$$|x_m(i)-x_n(i)|\ge\dist(L_{\alpha_m(i)},L_{\alpha_n(i)})\ge2,$$
which contradicts the choice of $n$. This contradiction shows that $\alpha_m(i)=\alpha_n(i)$ for all $m\ge n$ and hence the sequence $(\alpha_m(i))_{m\in\w}$ converges in $\w$.
 Then the sequence $(\alpha_m)_{m\in\w}$ converges in $\w^\w$ to some element $\alpha\in\w^\w$.

Given any $i\in\w$, find $n\in\w$ such that $\alpha_m(i)=\alpha(i)$ for all $m\ge n$ and conclude that $x_m(i)\in L_{\alpha_m(i)}=L_{\alpha(i)}$ for all $m\ge n$. Taking into account that the set $L_{\alpha(i)}$ is closed in $\IR$, we conclude that $x(i)=\lim_{m\to\infty}x_m(i)\in L_{\alpha(i)}$, and hence $x\in\prod_{i\in\w}L_{\alpha(i)}=L_\alpha$.

The continuity of the function $f:\w^\w\to\K$ guarantees that $\lim_{n\to\infty}K_{\alpha_n}=K_\alpha$ and $\lim_{n\to\infty}D_{\alpha_n}=D_\alpha$ in the hyperspace $\K(\IR^\w)$. Then
$x=\lim_{n\to\infty}x_n\in \lim_{n\to\infty}(K_{\alpha_n} + D_{\alpha_n})=K_{\alpha} + D_\alpha$ and finally $x\in (K_\alpha+D_\alpha)\cap L_\alpha\subset F$.
\end{proof}

\begin{lemma} For any $K\in\K$ there exists $x\in \IR^\w$ such that $K+x\subset F$.
\end{lemma}

\begin{proof} Given any $K\in\K$, find $\alpha\in\w^\w$ such that $K=K_\alpha$. Lemma~\ref{lemreal} implies that the Cantor set $D_\alpha$ contains a point $x$ such that $K_\alpha+x\subset L_{\alpha}$.
Then $K+x=K_\alpha+x\subset (K_\alpha+D_\alpha)\cap L_\alpha\subset F$.
\end{proof}

\begin{lemma} For any $x\in\IR^\w$ the intersection $F\cap(x+[0,1]^\w)$ is contained in the set $K+d$ for some $K\in\K$ and $d\in\IR^\w$.
\end{lemma}

\begin{proof} We can assume that the intersection $F\cap(x+[0,1]^\w)$ is not empty (otherwise it is contained in any set $K+d$).

Claim~\ref{dist2} implies that there exists a unique $\alpha\in \w^\w$ such that $(x+[0,1]^\w)\cap L_\alpha$ is not empty.
Then $$(x+[0,1]^\w)\cap F=(x+[0,1]^\w)\cap A_\alpha=(x+[0,1]^\w)\cap (K_\alpha+D_\alpha)\cap L_\alpha$$and we can find a point $d\in D_\alpha$ such that the intersection $(x+[0,1]^\w)\cap(K_\alpha+d)$ is not empty and hence contains some point $z$. We claim that for any point $d'\in D_\alpha\setminus\{d\}$ the intersection $(x+[0,1]^\w)\cap (K_\alpha+d')$ is empty.

To derive a contradiction, assume that $(x+[0,1]^\w)\cap (K_\alpha+d')$
is not empty and hence contains some point $z'$.
Since $d'\ne d$, there exists a coordinate $i\in\w$ such that $d(i)\ne d(i')$. Without loss of generality, we can assume that $d(i)<d'(i)$. Taking into account that doubleton $\{d(i),d'(i)\}=D_{\alpha(i)}(\min\pi_i(K_\alpha),\max\pi_i(K_\alpha))$ has diameter $\ge 2+\max\pi_i(K_\alpha)-\min\pi_i(K_\alpha)$, we conclude that $$d'(i)-d(i)\ge 2+\max\pi_i(K_\alpha)-\min\pi_i(K_\alpha)\ge2+(z(i)-d(i))-(z'(i)-d'(i)),$$ which implies $z'(i)-z(i)\ge 2$. But this is impossible as $z(i),z'(i)\in x(i)+[0,1]$ and hence $|z(i)-z'(i)|\le1$.

Therefore, $(x+[0,1]^\w)\cap F=(K_\alpha+d)\cap L_\alpha\subset K_\alpha+d=K+d$.
\end{proof}
\end{proof}

We shall apply Theorem~\ref{t:univer} to produce the following example promised in Remark~\ref{r:prom}.

\begin{example}\label{SHMnotsubsetEHM} The topological group $\IR^\w$ contains a closed subset $F\in\SHM\setminus\EHM$.
\end{example}

\begin{proof} In the hyperspace $\mathcal K(\IR^\w)$ consider the subspace $\mc{K}_0$ consisting of topological copies of the Cantor set (i.e., zero-dimensional compact subsets of $\IR^\w$ without isolated points).

By Lemma~\ref{E-in-C}, the set $\E(2^\w,\IR^\w)$ of injective maps is Polish (being a $G_\delta$-set of the Polish space $\C(2^\w,\IR^\w)$). Then the space $\K_0$ is analytic, being the image of the Polish space $\E(2^\w,\IR^\w)$ under the continuous map $\rng:\E(2^\w,\IR^\w)\to\K_0$, $\rng:f\mapsto f(2^\w)$.

By Theorem~\ref{t:univer}, the topological group $\IR^\w$ contains a closed subset $F$ such that
\begin{enumerate}
\item[\textup{(1)}] for any $K\in\K_0$ there exists $x\in \IR^\w$ with $x+K\subset F$;
\item[\textup{(2)}] for $x\in \IR^\w$ the intersection $[0,1]^\w\cap (x+F)$ belongs to $\K_0$ and hence is nowhere dense in $[0,1]^\w$.
\end{enumerate}
The last condition witnesses that $F\in\SHM$ and the first condition that $F\notin\EHM$.
\end{proof}

\section{Haar-open sets}\label{s7}

In this section we discuss Haar-open sets and their relation to Haar-meager and Haar-null sets.

\begin{definition} A subset $A$ of a topological group $X$ is called \index{subset!Haar-open}\index{Haar-open subset}{\em Haar-open} in $X$ if for any compact subset $K\subset X$ and point $p\in K$ there exists $x\in X$ such that $K\cap(A+x)$ is a neighborhood of $p$ in $K$.
\end{definition} 


\begin{theorem}\label{t:HO} A subset $A$ of a complete metric group $X$ is Haar-open if and only if for any nonempty compact set $K\subset X$ there exists $x\in X$ such that the set $K\cap(A+x)$ has nonempty interior in $K$.
\end{theorem}

 \begin{proof} The ``only if" part is trivial. To prove the ``if'' part, take any non-empty compact space $K\subset X$ and a point $p\in K$. Let $\rho$ be an invariant  complete metric generating the topology of the group $X$.
 
For every $n\in\w$ consider the compact neighborhood $K_n:=\{x\in K:\rho(x,p)\le\frac1{2^n}\}$ of the point $p$ in $K$. Next, consider the compact space $\Pi:=\prod_{n\in\w}K_n$ and the map 
$$f:\Pi\to X,\;\;f:(x_n)_{n\in\w}\mapsto\sum_{n=0}^\infty x_n,$$
which is well-defined and continuous because of the completeness of the metric $\rho$ and the convergence of the series $\sum_{n=0}^\infty\frac1{2^n}$. 

If for the compact set $f(\Pi)$ in $X$ there exists $x\in X$ such that the intersection $f(\Pi)\cap(A+x)$ contains some non-empty open subset $U$ of $f(\Pi)$, then we can choose any point $(x_i)_{i\in\w}\in f^{-1}(U)$ and find $n\in\w$ such that $\{(x_i)\}_{i<n}\times \prod_{i\ge n}K_i\subset f^{-1}(U)$. Then for the point $s=\sum_{i<n}x_i\in \sum_{i<n}K_i\subset X$, we get $s+K_n\subset U\subset A+x$ and hence the intersection $K\cap (A+x-s)\supset K_n$ is a neighborhood of $p$ in $K$.
\end{proof}

Haar-open sets are related to (finitely or openly) thick sets.

\begin{definition} A subset $A$ of a topological group $X$ is called
\begin{itemize}
\item \index{subset!thick}\index{thick subset}{\em thick} if for any compact subset $K\subset X$ there is $x\in X$ such that $K\subset x+A$;
\item \index{subset!countably thick}\index{countably thick subset}{\em countably thick} if for any countable subset $E\subset X$ there is $x\in X$ such that $E\subset x+A$;
\item \index{subset!finitely thick}\index{finitely thick subset}{\em finitely thick} if for any finite subset $F\subset X$ there is $x\in X$ such that $F\subset x+A$;
\item \index{subset!openly thick}\index{openly thick subset}{\em openly thick} if there exists a nonempty open set $W\subset X$ such that for any finite family $\U$ of nonempty open subsets of $W$ there exists $x\in X$ such that $A+x$ intersects each set $U\in\U$.
\end{itemize}
\end{definition}

For any subset of a topological group we have the implications:
$$\mbox{openly thick}\Leftarrow\mbox{finitely thick}\Leftarrow\mbox{thick}\Ra\mbox{Haar-open}\Ra\mbox{not (Haar-null or Haar-meager)}.$$

The following proposition shows that Haar-open sets are close to being thick.

\begin{proposition}\label{p:HO=>prethick} If a subset $A$ of a topological group $X$ is Haar-open, then for any compact set $K\subset X$ there exists a finite set $F\subset X$ such that $K\subset F+A$.
\end{proposition}

\begin{proof} Since $A$ is Haar-open, for every point $p\in K$ there exists an open neighborhood $U_p\subset K$ of $p$ in $K$ that is contained in some shift $x_p+A$ of the set $A$. By the compactness of $K$, the open cover $\{U_p:p\in K\}$ of $K$ admits a finite subcover $\{U_p:p\in E\}$. Then the finite set $F=\{x_p:p\in E\}$ has the desired property as $K=\bigcup_{p\in E}U_p\subset \bigcup_{p\in E}(A+x_p)=A+F$.
\end{proof} 

\begin{theorem}\label{t:prethick} For a closed subset $A$ of a Polish group $X$ the following conditions are equivalent:
\begin{enumerate}
\item[\textup{1)}] $A$ is Haar-meager in $X$;
\item[\textup{2)}] $A$ is strongly Haar-meager in $X$;
\item[\textup{3)}] $A$ is not Haar-open.
\end{enumerate}
\end{theorem}

\begin{proof} The implication $(2)\Ra(1)$ is trivial. To see that $(1)\Ra(3)$, assume that $A$ is Haar-meager and find a continuous map $f:K\to X$ defined on a compact metrizable space $K$ such that for every $x\in X$ the preimage $f^{-1}(x+A)$ is meager in $K$. Then the compact space $f(K)$ and any point $p\in f(K)$ witness that $A$ is not Haar-open in $X$.

To prove that $(3)\Ra(2)$, assume that $A$ is not Haar-open. By Theorem~\ref{t:HO}, there exists a nonempty compact set $K\subset X$ such that for every $x\in X$ the intersection $K\cap(A+x)$ has empty interior in $K$ and being a closed set, is nowhere dense and hence meager in $K$. 
\end{proof}

\begin{remark} Theorem~\ref{t:prethick} implies that for any Polish group the semi-ideal ${\downarrow}\overline{\SHM}={\downarrow}\overline{\HM}$ is an ideal.
\end{remark}

The proof of Theorem~17 in the paper by Dole\v{z}al and Vlas\v{a}k \cite{DV} yields the following fact.

\begin{theorem}\label{fthick} If a subset $A$ of a Polish group is not openly thick, then $\bar A\in\EHM$.
\end{theorem}


The following proposition shows that Theorem~\ref{fthick} cannot be reversed.

\begin{proposition}\label{finthicknotEHM} Each non-compact Polish group $X$ contains a closed discrete openly thick set $A$. If the group $X$ is uncountable, then $A$ is Haar-null and injectively Haar-meager.
\end{proposition}

\begin{proof} Fix a complete invariant metric $\rho$ generating the topology of the Polish group $X$. Taking into account that $X$ is not compact, we conclude that the complete metric space $(X,\rho)$ is not totally bounded. So, there exists $\e>0$ such that for any finite subset $F\subset X$ there exists $x\in X$ such that $\rho(x,F)>\e$, where $\rho(x,F)=\min\{\rho(x,y):y\in F\}$.

Let $D$ be a countable dense subset of $X$. Let $\{D_n\}_{n\in\w}$ be an enumeration of all nonempty finite subsets of $D$. By induction, construct a sequence $(x_n)_{n\in\w}$ of points in $X$ such that
$\rho(x_n,\bigcup_{k<n}D_k+x_k-D_n)>\e$ and hence $\rho(D_n+x_n,\bigcup_{k<n}D_k+x_k)>\e$ for all $n\in\w$.

The choice of the points $x_n$, $n\in\w$, ensures that the set $A=\bigcup_{n\in\w}(D_n+x_n)$ is closed and discrete.

To prove that the set $A$ is openly thick in $X$, it suffices to show that for any finite family $\U$ of nonempty open sets in $X$ there exists $x\in X$ such that the set $x+A$ intersects each set $U\in\mathcal U$. Given a finite family $\U$ of nonempty open sets in $X$, choose a finite subset $F\subset D$ intersecting each set $U\in\U$. Next, find a number $n\in\w$ such that $F=D_n$ and observe that the set $A-x_n\supset D_n=F$ intersects each set $U\in\U$.

If the Polish group $X$ is uncountable, then it contains a subset $C\subset X$, homeomorphic to the Cantor cube. Observe that for every $x\in X$ the intersection $C\cap (A+x)$ is finite (being a closed discrete subspace of the compact space $C$). This implies that $A$ is Haar-null and injectively Haar-meager.
\end{proof}


\begin{proposition}\label{p:HN=>HM} Each closed Haar-null subset of a Polish group $X$ is injectively Haar-meager. This yields the inclusions $\overline{\HN}\subset\overline{\EHM}$ and $\sigma\overline{\HN}\subset\sigma\overline{\EHM}\subset\sigma\overline{\HM}$.
\end{proposition}

\begin{proof}
Let $A$ be a closed Haar-null set in $X$. By Lemma~\ref{l:Cantorsupp}, there exists $\mu \in P(X)$ with $\supp(\mu)$ homeomorphic to the Cantor set. Since $\mu(U)>0$ for every open set $U \subset \supp(\mu)$, the set $(A+x)\cap \supp(\mu)$ is closed and has empty interior in $\supp(\mu)$ for any $x\in X$. But this shows that the set $A$ is injectively Haar-meager with any homeomorphism $h:2^\w\to\supp(\mu)$ witnessing this fact.
\end{proof}

\begin{proposition}\label{p:2thick} Each non-locally compact Polish group $X$ contains a closed thick set $A$ and a thick $G_\delta$-set $B$ such that the sum $A+B$ has empty interior in $X$.
\end{proposition}

\begin{proof} By \cite{MZ}, each non-locally compact Polish group $X$ contains two closed thick sets $A,C$ such that for every $x\in X$ the intersection $C\cap(x-A)$ is compact. Fix a countable dense set $D\subset X$ and consider the $\sigma$-compact set $S=\bigcup_{x\in D}C\cap(x-A)$. By Lemma~\ref{l:thick}, the $G_\delta$-set $B=C\setminus S$ is thick in $X$. On the other hand, the sum $B+A$ is disjoint with $D$ and hence $A+B=B+A$ has empty interior.
\end{proof}

\begin{lemma}\label{l:thick} For any thick set $A$ in a non-locally compact Polish group $X$ and any $\sigma$-compact set $S\subset X$ the set $A\setminus S$ is thick in $X$.
\end{lemma}

\begin{proof} Given a compact set $K\subset X$, we need to find an element $x\in X$ such that $x+K\subset A\setminus S$. Let $H$ be the subgroup of $X$ generated by the $\sigma$-compact set $S\cup K$. Since the Polish group $X$ is not locally compact, the $\sigma$-compact subgroup $H$ is meager in $X$. Take any point $z\in X\setminus H$. Since the set $A$ is thick, there exists $x\in X$ such that $x+\{\theta,z\}+K\subset A$. If $x+K$ is disjoint with $H$, then $x+K\subset A\setminus H\subset A\setminus S$ and we are done. If $x+K$ intersects $H$, then $x\in H-K=H$ and $x+z+K\subset z+H$ is disjoint with $H$. In this case $x+z+K\subset A\setminus H\subset A\setminus S$.
\end{proof}

On the other hand, the ideal $\HN$ has the following ``Haar-open'' version of the strong Steinhaus property.

\begin{theorem}\label{t:Sum-HO} Let $X=\prod_{n\in\w}X_n$ be the Tychonoff product of locally compact topological groups and $A,B\subset X$ be two Borel sets in $X$. If $A,B$ are not Haar-null, then the sum-set $A+B$ is Haar-open in $X$.
\end{theorem}

This theorem follows from Theorems~\ref{l7} and \ref{l8} proved below. In the proofs we shall use special measures, introduced in Definitions~\ref{d:coher} and \ref{d:subinv}. For a  measure $\mu\in P(X)$ on a topological space $X$, a Borel set $B\subset X$ is called \index{subset!$\mu$-positive}\index{$\mu$-positive subset}{\em $\mu$-positive} if $\mu(B)>0$.

\begin{definition}\label{d:coher} A measure $\mu\in P(X)$ on a topological group $X$ is called \index{measure!coherent}\index{coherent measure}{\em coherent} if for any $\mu$-positive compact sets $A,B\subset X$ there are points $a,b\in X$ such that the set $(A+a)\cap (B+b)$ is $\mu$-positive.
\end{definition}

\begin{definition}\label{d:subinv} Let $K$ be a compact set in a  topological group $X$. A  measure $\mu\in P(X)$ is called \index{measure!$K$-shiftable}{\em $K$-shiftable} if for every $\mu$-positive compact set $B\subset X$ there exists a neighborhood $U\subset X$ of $\theta$ such that $\inf_{x\in U\cap K}\mu(B+x)>0$.
\end{definition}

\begin{theorem}\label{l7} Let $X$ be a topological group such that for any compact set $K\subset X$ there exists a coherent $K$-shiftable Radon measure $\mu\in P(X)$. Let $A,B$ be Borel sets in $X$. If $A,B$ are not Haar-null, then the set $A-B=\{a-b:a\in A,\;b\in B\}$ is Haar-open in $X$.
\end{theorem}

\begin{proof} To show that $A-B$ is Haar-open, fix any compact set $K$ in $X$ and a point $p\in K$. We need to find a point $s\in X$ such that $(A-B+s)\cap K$ is a neighborhood of $p$ in $K$. By our assumption, for the compact set $K-p$, there exists a coherent $(K-p)$-shiftable measure $\mu\in P(X)$. Since $A,B$ are not Haar-null, there exist points $a,b\in X$ such that $\mu(A+a)>0$ and $\mu(B+b)>0$.  
Since the measure $\mu$ is Radon and coherent, there exist points $a',b'\in X$ such that the set $(A+a+a')\cap (B+b+b')$ is $\mu$-positive and hence contains some $\mu$-positive compact set $C$. We claim that the set $V:=\{x\in K-p:\mu\big(C\cap(C+x)\big)>0\}$ is a neighborhood of $\theta$ in $K-p$.

Since $\mu$ is $(K-p)$-shiftable, for the compact set $C$ there exist a neighborhood $U\subset X$ of $\theta$ and positive $\e>0$ such that $\mu(C+x)>2\e$ for all $x\in U\cap (K-p)$. Using the regularity of the measure $\mu$, we can replace the neighborhood $U$ by a smaller neighborhood of $\theta$ and assume that $\mu(C+U)<\mu(C)+\e$. We claim that  $U\cap (K-p)\subset V$. Indeed, for any $x\in U\cap (K-p)$, the set $C+x$ has measure $\mu(C+x)>2\e$. Since $C+x\subset C+U$, we conclude that
$$
\mu\big(C\cap (C+x)\big)=\mu(C)+\mu(C+x)-\mu\big(C\cup(C+x)\big)>\mu(C)+2\e-\mu(C+U)>\e>0,
$$which means that $x\in V$ and $V\supset U\cap (K-p)$ is a neighborhood of $\theta$ in $K-p$. Then $W=V+p$ is a neighborhood of $p$ in $K$. It follows that for every $x\in W$, we get $C\cap (C+x-p)\ne\emptyset$ and hence $x\in C+p-C\subset (A+a+a')+p-(B+b+b')$. So, for the point $s:=a+a'+p-b-b'$ the set $A-B+s$ contains each point $x$ of the set $W\subset K$ and $K\cap(A-B+s)\supset W$ is a neighborhood of $p$ in $K$.
\end{proof}


\begin{theorem}\label{l8}  Let $X=\prod_{n\in\w}X_n$ be the Tychonoff product of locally compact topological groups. Then for every compact subset $K\subset X$, there exists a coherent $K$-shiftable Radon measure $\mu\in P(X)$ such that $\mu=-\mu$ and $\mu$ has compact support.
\end{theorem}

\begin{proof} Fix a compact set $K\subset X$. Replacing $K$ by $K\cup\{\theta\}\cup(-K)$ we can assume that $\theta\in K=-K$.
 
For every $n\in\w$ we identify the products $\prod_{i< n}X_i$ and $\prod_{i\ge n}X_i$ with the subgroups $X_{<n}:=\{(x_i)_{i\in\w}\in X:\forall i\ge n\;\;x_i=\theta_i\}$ and $X_{\ge n}:=\{(x_i)_{i\in\w}\in X:\forall i<n\;\;x_i=\theta_i\}$ of $X$, respectively. Here by $\theta_i$ we denote the neutral element of the group $X_i$.

Let $\pr_n:X\to X_n$ and $\pr_{le n}:X\to X_{\le n}$ be the natural coordinate projections. For every $n\in\w$ let $K_n$ be the projection of the compact set $K$ onto the locally compact group $X_n$.

Fix a decreasing sequence $(O_{i,n})_{i\in\w}$ of open neighborhoods of zero in the group $X_n$ such that  $\bar O_{0,n}$ is compact and $O_{i+1,n}+O_{i+1,n}\subset O_{i,n}=-O_{i,n}$ for every $i\in\w$. For every $m\in\w$ consider the neighborhood $U_{m,n}:=O_{1,n}+O_{2,n}+\dots+O_{m,n}$ of zero in the group $X_n$. By induction it can be shown that $-U_{m,n}=U_{m,n}\subset U_{m+1,n}\subset O_{0,n}$ for all $m\in\w$. Let $U_{\w,n}:=\bigcup_{m\in\w}U_{m,n}\subset O_{0,n}$. 

Fix a Haar measure $\lambda_n$ in the locally compact group $X_n$.
Being Abelian, the locally compact group $X_n$ is amenable. So, we can apply the F\o lner Theorem \cite[4.13]{Pat} and find a compact set $\Lambda_n=-\Lambda_n\subset X$ such that $\lambda_n\big((\Lambda_n+K_n+U_{\w,n})\setminus \Lambda_n\big)<\frac1{2^n}\lambda_n(\Lambda_n)$. Multiplying the Haar measure $\lambda_n$ by a suitable positive constant, we can assume that the set $\Omega_n:=\Lambda_n+K_n+U_{\w,n}$ has measure $\lambda_n(\Omega_n)=1$.
Then $\tfrac1{2^n}\lambda_n(\Lambda_n)<\lambda_n(\Omega_n-\Lambda_n)=
1-\lambda_n(\Lambda_n)$ implies $\lambda_n(\Lambda_n)>\frac{2^n}{2^n+1}$.

 Now consider the probability  measure $\mu_n$ on $X_n$ defined by $\mu_n(B)=\lambda_n(B\cap \Omega_n)$ for any Borel subset of $X$. If follows from $\Omega_n=-\Omega_n$ that $\mu_n=-\mu_n$.

Let $\mu:=\otimes_{n\in\w}\mu_n\in P(X)$ be the product measure of the  
probability measures $\mu_n$. It is clear that $\mu=-\mu$ and $\mu$ has compact support contained in the compact subset $\prod_{n\in\w}\bar\Omega_n$ of $X$. We claim that the measure $\mu$ is $K$-shiftable and coherent.

\begin{claim} The measure $\mu$ is $K$-shiftable. 
\end{claim}

\begin{proof} For every $k\in\w$ consider the set 
$M_k:=\prod_{n<k}(\Lambda_n+K_n+U_{k,n})\times \prod_{n\ge k}\Lambda_n$. We claim that $\lim_{k\to\infty}M_k=1$.
Indeed, for every $\e>0$ we can find $m\in\IN$ such that $\prod_{n\ge m}\frac{2^n}{2^n+1}>1-\e$ and then for any $n<m$ by the $\sigma$-additivity of the Haar measure $\lambda_n$, find $i_n> m$ such that $\lambda_n(\Lambda_n+K_n+U_{i_n,n})>(1-\e)^{1/m}$. Then for any $k\ge\max\limits_{n<m}i_n> m$, we obtain the lower bound 
$$
\begin{aligned}
\mu(M_k)&=\prod_{n<k}\lambda_n(\Lambda_n+K_n+U_{k,n})\cdot\prod_{n\ge k}\lambda_n(\Lambda_n)\\
&\ge \prod_{n<k}\lambda_n(\Lambda_n+K_n+U_{k,n})\cdot\prod_{n\ge m}\lambda_n(\Lambda_n)\\
&>\prod_{n<m}(1-\e)^{\frac1m}\cdot\prod_{n\ge m}\tfrac{2^n}{2^n+1}>(1-\e)^2.
\end{aligned}
$$

Now take any $\mu$-positive compact set $B\subset X$. Since $\lim_{k\to\infty}M_k=1$, there exists $k\in\w$ such that $\mu(B\cap M_k)>0$.  Consider the neighborhood $U:=\prod_{n<m}O_{k+1,n}\times\prod_{n\ge m}X_n$ of $\theta$ in $X$, and observe that for any $x\in U\cap K$, we have $M_k+x\subset\prod_{n\in\w}\Omega_n$ and hence $\mu(B+x)\ge \mu((B\cap M_k)+x)=\mu(B\cap M_k)$ by the definition of the measure $\mu$.
\end{proof}

In the following claim we prove that the measure $\mu=-\mu$ is coherent.

\begin{claim} For any $\mu$-positive compact sets $A,B\subset X$ there are points $a,b\in X$ such that $\mu((A+a)\cap (B+b))>0$.
\end{claim}

\begin{proof} We lose no generality assuming that the sets $A,B$ are contained in $\prod_{n\in\w}\Omega_n$. By the regularity of the measure $\mu$, there exists a neighborhood $U\subset X$ of $\theta$ such that $\mu(A+U)<\frac65\mu(A)$ and $\mu(B+U)<\frac65\mu(B)$. Replacing $U$ by a smaller neighborhood, we can assume that $U$ is of the basic form $U=V+X_{\ge n}$ for some $n\in\w$ and some open set $V\subset X_{<n}$. Consider the tensor product $\lambda=\lambda'\otimes\mu'$ of the measures $\lambda'=\otimes_{k<n}\lambda_k$ and $\mu':=\otimes_{k\ge n}\mu_n$ on the subgroups $X_{<n}$ and $X_{\ge n}$ of $X$, respectively. It follows from $A\cup B\subset \prod_{n\in\w}\Omega_n$ that $\lambda(A)=\mu(A)$ and $\lambda(B)=\mu(B)$. Observe also that the measure $\lambda$ is $X_{<n}$-invariant in the sense that $\lambda(C)=\lambda(C+x)$ for any Borel set $C\subset X$ and any $x\in X_{<n}$.

For every $y\in X_{\ge n}$ let $A_y=\{x\in X_{<n}:x+y\in A\}$ be the $y$th section of the set $A$. By the Fubini Theorem, 
$\lambda(A)=\int_{\mu'}\lambda'(A_y)dy$.
Consider the projection $A'$ of $A$ onto $X_{<n}$. 
Taking into account that $A_y\subset A'\subset\prod_{k<n}\Omega_k$, we conclude that $$\lambda'(A_y)\le \lambda'(A')=\mu(A'+X_{\ge n})=\mu(A+X_{\le n})\le\mu(A+U)<\tfrac65\mu(A)$$and hence $0<\frac56\lambda'(A')<\mu(A)$.

Consider the set  $$L_A:=\big\{y\in \prod_{k\ge n}\Omega_k:\lambda'(A_y)>\tfrac13\lambda'(A')\big\}.$$ By the Fubini Theorem,
$$\tfrac56\lambda'(A')<\mu(A)=\lambda(A)=\int_{\mu'}\lambda'(A_y)dy\le\mu'(L_A)\cdot \lambda'(A')+\tfrac13\lambda'(A')\cdot\big(1-\mu'(L_A)\big)$$ and hence $\mu'(L_A)>\frac34$.

By analogy, for the set $B$ consider the projection $B'$ of $B$ onto $X_{<n}$ and for every $y\in B'$ consider the $y$th section $B_y:=\{x\in X_{<n}:x+y\in B\}$ of $B$. Repeating the above argument we can show that the set $L_B:=\{y\in \prod_{k\ge n}\Omega_k:\mu'(B_x)>\frac13\lambda'(B')\}$ has measure $\mu'(L_B)>\frac34$.
Then $\mu'(L_A\cap L_B)>\frac12$.

Now we are ready to find $s\in X_{<n}$ such that $A\cap (B+s)\ne\emptyset$.
Observe that a point $z\in X$ belongs to $A\cap (B+s)$ is and only if $\chi_A(z)\cdot\chi_{B+s}(z)>0$, where $\chi_A$ and $\chi_{B+s}$ denote the characteristic functions of the sets $A$ and $B+s$ in $X$.

So, it suffices to show that the function $\chi_A\cdot \chi_{B+s}$ has a non-zero value. For this write $z$ as a pair $(x,y)\in X_{<n}\times X_{\ge n}$ and  for every $s\in X_{<n}$ consider the function 
$$g(s):=\int_{\lambda}\chi_A(z)\cdot\chi_{B+s}(z)dz=
\int_{\lambda'}\int_{\mu'}\chi_A(x,y)\cdot\chi_{B+s}(x,y)\,dy\,dx$$and its integral, transformed with help of Fubini's Theorem:
$$
\begin{aligned}
\int_{\lambda'}g(s)\,ds&=\int_{\lambda'}
\int_{\mu'}\int_{\lambda'}\chi_A(x,y)\cdot\chi_{B+s}(x,y)\,dx\, dy\,ds\\
&=\int_{\mu'}\int_{\lambda'}\int_{\lambda'}\chi_A(x,y)\cdot\chi_{B}(x-s,y)\,ds\, dx\, dy\\
&=\int_{\mu'}\int_{\lambda'}\chi_A(x,y)\int_{\lambda'}\chi_B(x-s,y)\,ds\,dx\, dy\\
&=\int_{\mu'}\int_{\lambda'}\chi_A(x,y)\lambda'(B_y)\,dx \,dy=\int_{\mu'}\lambda'(B_y)\int_{\lambda'}\chi_A(x,y)\,dx\, dy\\
&=\int_{\mu'}\lambda'(B_y)\cdot \lambda'(A_y)\,dy>\frac19\lambda'(A')\cdot\lambda'(B')\cdot \mu'(L_A\cap L_B)>0.
\end{aligned}
$$
Now we see that $g(s)>0$ for some $s\in X_{<n}$ and hence $A\cap (B+s)\ne\emptyset$.
\end{proof}
\end{proof}

Theorem~\ref{t:Sum-HO} rises two open problems.

\begin{problem} Let $A,B$ be Borel subsets of a (reflexive) Banach space $X$ such that  $A,B$ are not Haar-null in $X$. Is the set $A+B$ (or at least $A+A$) Haar-open in $X$?
\end{problem}

\begin{problem} Let $A,B$ be Borel subsets of a Polish group $X$  such that $A,B$ are not Haar-meager in $X$. Is the set $A+B$ (or at least $A+A$) Haar-open in $X$?
\end{problem}

\section{An example of a ``large'' Borel subgroup in $\IZ^\w$}

In the following example by $\IZ$ we denote the group of integer numbers endowed with the discrete topology.

\begin{example}\label{Banakh} The Polish group $\IZ^\w$ contains a meager Borel subgroup $H$, which does not belong to the ideal $\sigma\overline{\HM}$ in $\IZ^\w$. Consequently, the $\sigma$-ideals $\sigma\overline{\HM}$ and $\sigma\overline{\HN}$ on $\IZ^\w$ does not have the weak Steinhaus property.
\end{example}

\begin{proof} The Borel group $H$ will be defined as the group hull of some special $G_\delta$-set $P\subset\IZ^\w$. In the construction of this set $P$ we shall use the following lemma.

\begin{lemma}\label{l1} There exists an infinite family $\Tau$ of thick subsets of $\IZ$ and an increasing number sequence $(\Xi_m)_{m\in\w}$ such that for any positive numbers $n\le m$, non-zero integer numbers $\lambda_1,\dots,\lambda_n\in[-m,m]$, pairwise distinct sets $T_1,\dots,T_n\in\Tau$ and points $x_i\in T_i$, $i\le n$, such that $\{x_1,\dots,x_n\}\not\subset [-\Xi_m,\Xi_m]$ the sum $\lambda_1x_1+\dots+\lambda_n x_n$ is not equal to zero.
\end{lemma}

\begin{proof} For every $m\in\w$ let $\xi_m\in\w$ be the smallest number such that $$2^{2^{x}}-x>m^2(2^{2^{x-1}}+x)$$ for all $x\ge \xi_m$, and put $\Xi_m=2^{2^{\xi_m}}+\xi_m$. Choose an infinite family $\A$ of pairwise disjoint infinite subsets of $\IN$ and for every $A\in\A$ consider the thick subset $$T_A:=\bigcup_{a\in A}[2^{2^a}-a,2^{2^a}+a]$$ of $\IZ$. Here by $[a,b]$ we denote the segment $\{a,\dots,b\}$ of integers. Taking into account that the families $([2^{2^n}-n,2^{2^n}+n])_{n=1}^\infty$ and $\mathcal A$ are disjoint, we conclude that so is the family $(T_A)_{A\in\A}$.

We claim that the disjoint family $\Tau=\{T_A\}_{A\in\A}$ and the sequence $(\Xi_m)_{m\in\w}$ have the required property.
Take any positive numbers $n\le m$, non-zero integer numbers $\lambda_1,\dots,\lambda_n\in[-m,m]$, pairwise distinct sets $T_1,\dots,T_n\in\Tau$ and points $x_1\in T_1,\dots,x_n\in T_n$ such that $\{x_1,\dots,x_n\}\not\subset [-\Xi_m,\Xi_m]$. For every $i\le n$ find an integer number $a_i$ such that $x_i=2^{2^{a_i}}+\e_i$ for some $\e_i\in[-a_i,a_i]$. Since the family $\A$ is disjoint and the sets $T_1,\dots,T_n$ are pairwise distinct, the points $a_1,\dots,a_n$ are pairwise distinct, too. Let $j$ be the unique number such that $a_j=\max\{a_i:1\le i\le n\}$. Taking into account that $\{x_1,\dots,x_n\}\not\subset [0,\Xi_m]=[0,2^{2^{\xi_m}}+\xi_m]$, we conclude that $2^{2^{a_j}}+a_j\ge 2^{2^{a_j}}+\e_j=x_j>2^{2^{\xi_m}}+\xi_m$ and hence $a_j>\xi_m$. The definition of the number $\xi_m$ guarantees that $2^{2^{a_j}}-a_j>m^2(2^{2^{a_j-1}}+a_j)$ and hence
$$|\lambda_jx_j|\ge x_j=2^{2^{a_j}}+\e_j\ge 2^{2^{a_j}}-a_j>m^2(2^{2^{a_j-1}}+a_j)> \sum_{i\ne j}|\lambda_i|(2^{2^{a_i}}+a_i)\ge |\sum_{i\ne j}\lambda_ix_i|,$$and hence $\sum_{i=1}^n\lambda_ix_i\ne0$.
\end{proof}

Now we are ready to start the construction of the $G_\delta$-set $P\subset\IZ^\w$. This construction will be done by induction on the tree $\w^{<\w}=\bigcup_{n\in\w}\w^n$ consisting of finite sequences $s=(s_0,\dots,s_{n-1})\in\w^n$ of finite ordinals. 
For a sequence $s=(s_0,\dots,s_{n-1})\in \w^n$ and a number $m\in\w$ by $s\hat{\;}m=(s_0,\dots,s_{n-1},m)\in \w^{n+1}$ we denote the {\em concatenation} of $s$ and $m$.

For an infinite sequence $s=(s_n)_{n\in\w}\in\IZ^\w$ and a natural number $l\in\w$ let $s|l=(s_0,\dots,s_{l-1})$ be the restriction of the function $s:\w\to\IZ$ to the subset $l=\{0,\dots,l-1\}$.
Observe that the topology of the Polish group $\IZ^\w$ is generated by the ultrametric $$\rho(x,y)=\inf\{2^{-n}:n\in\w,\;x|n=y|n\},\;\;x,y\in\IZ^\w.$$Observe also that for every $z\in\IZ^\w$ and $n\in\w$ the set $U(z|n)=\{x\in\IZ^\w:x|n=z|n\}$ coincides with the closed ball $\bar B(z;2^{-n})=\{x\in\IZ^\w:\rho(x,z)\le 2^{-n}\}$ centered at $z$.

Using Lemma~\ref{l1}, choose a number sequence $(\Xi_m)_{m\in\w}$ and a sequence $(T_s)_{s\in\w^{<\w}}$ of thick sets in the discrete group $\IZ$ such that for every positive integer numbers $n\le m$, finite set $F\subset \w^{<\w}$ of cardinality $|F|\le n$, function $\lambda:F\to [-m,m]\setminus\{0\}$, and numbers $z_s\in T_s$, $s\in F$, such that $\{z_s\}_{s\in F}\not\subset[-\Xi_m,\Xi_m]$ the sum $\sum_{s\in F}\lambda(s)\cdot z_s$ is not equal to zero.

For every $s\in\w^{<\w}$ and $n\in\w$ choose any point $t_{s,n}\in T_s\setminus[-n,n]$ and observe that the set $T_{s,n}=T_s\setminus(\{t_{s,n}\}\cup[-n,n])$ remains thick in $\IZ$.

By induction on the tree $\w^{<\w}$ we shall construct a sequence $(z_s)_{s\in\w^{<\w}}$ of points of $\IZ^\w$ and a sequence $(l_s)_{s\in\w^{<\w}}$ of finite ordinals satisfying the following conditions for every $s\in \w^{<\w}$:
\begin{itemize}
\item[$(1_s)$] $U(z_{s\hat{\;} i}|l_{s\hat{\;} i})\cap U(z_{s\hat{\;} j}|l_{s\hat{\;} j})=\emptyset$ for any distinct numbers $i,j\in\w$;
\item[$(2_s)$] $l_{s\hat{\;} i}> l_s+i$ for every $i\in\w$;
\item[$(3_s)$] the closure of the set $\{z_{s\hat{\;}i}\}_{i\in\w}$ contains the Haar-open set $\mathbb T_s=\{z_s|l_s\}\times\prod_{n\ge l_s}T_{s,n}$ and is contained in the set $\hat{\mathbb T}_s=\{z_s|l_s\}\times \prod_{n\ge l_s}(T_{s,n}\cup\{t_{s,n}\})\subset U(z_s|l_s)$.
\end{itemize}
We start the inductive construction letting $z_0=0$ and $l_0=0$. Assume that for some $s\in\w^{<\w}$ a point $z_s\in\IZ^\w$ and a number $l_s\in\w$ have been constructed. Consider the Haar-open sets $\mathbb T_s$ and $\hat\IT_s$ defined in the condition $(3_s)$. Since $\mathbb T_s$ is nowhere dense in $\hat{\mathbb T}_s$, we can find a sequence $(z_{s\hat{\;}i})_{i\in\w}$ of pairwise distinct points of $\hat{\mathbb T}_s$ such that the space $D_s=\{z_{s\hat{\;}i}\}_{i\in\w}$ is discrete and contains $\mathbb T_s$ in its closure. Since $D_s$ is discrete, for every $i\in\w$ we can choose a number $l_{s\hat{\;}i}>l_s+i$ such that the open sets $U(z_{s\hat{\;}i}|l_{s\hat{\;}i})$, $i\in\w$, are pairwise disjoint. Observing that the sequences $(z_{s\hat{\;}i})_{i\in\w}$ and $(l_{s\hat{\;}i})_{i\in\w}$ satisfy the conditions $(1_s)$--$(3_s)$, we complete the inductive step.

We claim that the $G_\delta$-subset $P=\bigcap_{n\in\w}\bigcup_{s\in\w^n}U(z_s|l_s)=\bigcup_{s\in\w^\w}\bigcap_{n\in\w}U(z_{s|n}|l_{s|n})$ of $\IZ^\w$ has required properties. First observe that the map $h:\w^\w\to P$ assigning to each infinite sequence $s\in\w^\w$ the unique point $z_s$ of the intersection $\bigcap_{n\in\w}U(z_{s|n}|l_{s|n})$ is a homeomorphism of $\w^\w$ onto $P$. Then the inverse map $h^{-1}:P\to\w^\w$ is a homeomorphism too.

\begin{claim} For every nonempty open subset $U\subset P$ its closure $\overline{U}$ in $\IZ^\w$ is Haar-open in $\IZ^\w$.
\end{claim}

\begin{proof} Pick any point $p\in U$ and find a unique infinite sequence $t\in\w^\w$ such that $\{p\}=\bigcap_{m\in\w}U(z_{t|m}|l_{t|m})$. Since the family $\{U(z_{t|m}|l_{t|m})\}_{m\in\w}$ is a neighborhood base at $p$, there is $m\in\w$ such that $U(z_{t|m}|l_{t|m})\subset U$. Consider the finite sequence $s=t|m$. By the definition of $P$, for every $i\in\w$ the intersection $P\cap U(z_{s\hat{\;}i}|l_{s\hat{\;}i})$ contains some point $y_{s\hat{\;}i}$. Taking into account that $$\rho(z_{s\hat{\;}i},y_{s\hat{\;}i})\le\diam\, U(z_{s\hat{\;}i}|l_{s\hat{\;}i})\le 2^{-l_{s\hat{\;}i}}\le 2^{-i}$$ and $\mathbb T_s$ is contained in the closure of the set $\{z_{s\hat{\;}i}\}_{i\in\w}$, we conclude that the Haar-open set $\IT_s$ is contained in the closure of the set $\{y_{s\hat{\;}i}\}_{i\in\w}\subset P\cap U(z_s|l_s)\subset U$, which implies that $\overline{U}$ is Haar-open.
\end{proof}

\begin{claim} The subgroup $H\subset\IZ^\w$ generated by $P$
cannot be covered by countably many closed Haar-meager sets in $\IZ^\w$.
\end{claim}

\begin{proof} To derive a contradiction, assume that $H\subset\bigcup_{n\in\w}H_n$ where each set $H_n$ is closed and Haar-meager in $\IZ^\w$. Since the Polish space $P=\bigcup_{n\in\w}P\cap H_n$ is Baire, for some $n\in\w$ the set $P\cap H_n$ contains a nonempty open subset $U$ of $P$. Taking into account that closure $\overline{U}_n\subset H_n$ is Haar-open in $\IZ^\w$, we conclude that the set $H_n$ is Haar-open and not Haar-meager (according to Theorem~\ref{t:prethick}).
\end{proof}

It remains to prove that the subgroup $H\subset\IZ^\w$ generated by the $G_\delta$-set $P$ is Borel and meager in $\IZ^\w$. 

We recall that by $h:\w^\w\to P$ we denote the homeomorphism assigning to each infinite sequence $s\in\w^\w$ the unique point $z_s$ of the intersection $\bigcap_{m\in\w}U(z_{s|m}|l_{s|m})$.

For every finite subset $F\subset\w^{<\w}$ let $l_F:=\max_{s\in F} l_s$.
For a finite subset $E\subset \w^\w$ let $r_E\in\w$ be the smallest number such that the restriction map $E\to\w^{r_E}$, $s\mapsto s|r_E$, is injective.
Let $l_E:=\max\{l_{s|r_E}:s\in E\}$. Taking into account that $l_{s|m}\ge m$ for every $s\in\w^\w$ and $m\in\w$, we conclude that $l_E\ge r_E$.

\begin{claim}\label{cl2} For every $m\in\IN$, nonempty finite set $E\subset \w^\w$ of cardinality $|E|\le m$ and function $\lambda:E\to[-m,m]\setminus\{0\}$ we get $$\sum_{s\in E}\lambda(s)\cdot z_s\in\big\{y\in\IZ^\w:\forall k> \max\{l_E,\Xi_m\}\;\;y(k)\ne 0\big\}.$$
\end{claim}

\begin{proof} Given any number $k>\max\{l_E,\Xi_m\}$, consider the projection $\pr_k:\IZ^\w\to \IZ$, $\pr_k:x\mapsto x(k)$, to the $k$-th coordinate. The claim will be proved as soon as we check that the element $y=\sum_{s\in E}\lambda(s)\cdot z_s$ has non-zero projection $\pr_k(y)$. Observe that $\pr_k(y)=\sum_{s\in F}\lambda(s)\cdot\pr_k(z_s)$. For every $s\in F$ the equality $\{z_s\}=\bigcap_{m\in\w}U(z_{s|m}|l_{ s|m})$ implies that
$\pr_k(z_s)=\pr_k(z_{ s|m})$ for any $m\in\w$ such that $l_{ s|m}>k$.

 Taking into account that $k>l_E=\max_{s\in E}l_{s|r_E}$, for every $s\in E$ we can find a number $m_s\ge r_E$ such that $k\in[l_{s|m_s},l_{s|(m_s+1)})$.
Then $$\pr_k(z_s)\!=\!\pr_k(z_{s|(m_s{+}1)}|l_{s|(m_s{+}1)})\in T_{s|m_s\!,k}\cup\{t_{s|m_s,k}\}
\subset T_{s|m_s}\setminus [-k,k]\subset T_{s|m_s}\!\setminus [-\Xi_m,\Xi_m]$$ by the condition $(3_{ s|m_s})$ of the inductive construction.

It follows from $r_E\le\min_{s\in F}m_s$ that the sequences
$s|m_s$, $s\in E$, are pairwise distinct.
Then $\pr_k(y)=\sum_{s\in E}\lambda(s)\pr_k(z_s)\ne0$ by the choice of the family $(T_s)_{s\in\w^{<\w}}$.
\end{proof}

For every $n\in\w$ and a finite sequence $s\in\w^n$ consider the basic clopen subset
$V_s=\{v\in\w^\w:v|n=s\}$ in $\w^\w$ and observe that $h(V_s)=P\cap U(z_s|l_s)$.
For any nonempty finite set $F\subset \w^n$ put $$V_F=\prod_{s\in F}V_{s}.$$ Given any function $\lambda:F\to\IZ\setminus\{0\}$ we shall prove that the function
$$\Sigma_\lambda:V_F\to H,\;\;\Sigma_\lambda:v\mapsto \sum_{s\in F}\lambda(s)\cdot z_{v(s)},$$is injective. 
Let $\|\lambda\|=\max_{s\in F}|\lambda(s)|$.

\begin{claim} The function $\Sigma_\lambda:V_F\to H$ is injective.
\end{claim}

\begin{proof} Choose any distinct sequences $u,v\in V_F$ and consider the nonempty finite set $D:=\{s\in F:u(s)\ne v(s)\}$ and the finite set $E=\{u(s)\}_{s\in D}\cup\{v(s)\}_{s\in D}\subset\w^\w$. Put $m=\max\{|E|,\|\lambda\|\}$ and $k=1+\max\{l_E,\Xi_{m}\}$. Claim~\ref{cl2} guarantees that the element
$$y=\Sigma_\lambda(u)-\Sigma_\lambda(v)=\sum_{s\in D}\lambda(s)z_{u(s)}-\sum_{s\in D}\lambda(s)z_{v(s)}$$ has $y(k)\ne 0$, which implies that $y\ne 0$ and $\Sigma_\lambda(u)\ne\Sigma_\lambda(v)$.
\end{proof}

Claim~\ref{cl2} implies

\begin{claim}\label{cl:meg} The set $\Sigma_\lambda(V_F)\subset\{y\in\IZ^\w:\exists m\in\w\;\forall k\ge m\;\;y(k)\ne0\}$ is Borel and meager in $\IZ^\w$.
\end{claim}

\begin{proof} By Lusin-Suslin Theorem \cite[15.1]{K}, the image $\Sigma_\lambda(V_F)$ of the Polish space $V_F$ under the continuous injective map $\Sigma_\lambda:V_F\to\IZ^\w$ is a Borel subset of $\IZ^\w$. By Claim~\ref{cl2}, the set $\Sigma_\lambda(V_F)$ is contained in the meager subset $\{y\in\IZ^\w:\exists n\;\forall m\ge n\;\;y(m)\ne 0\}$ and hence is meager in $\IZ^\w$.
\end{proof}

Denote by $[\w^{<\w}]^{<\w}$ the family of finite subsets of $\w^{<\w}$ and let $\IZ^*=\IZ\setminus\{0\}$. For the empty set $F=\emptyset$ and the unique map $\lambda\in(\IZ^*)^F$ we put $\Sigma_\lambda(V_F)=\{0\}$. Claim~\ref{cl:meg} and the obvious equality
$$H=\bigcup_{n\in\w}\bigcup_{F\in[\w^n]^{<\w}}\bigcup_{\lambda\in (\IZ^*)^F}\Sigma_\lambda(V_F)$$imply that the subgroup $H$ is Borel and is contained in the meager subset $$\{0\}\cup\bigcup_{m\in\w}\{z\in \IZ^\w:\forall k\ge m\;\;z(k)\ne0\}$$ of $\IZ^\w$.
\end{proof}

\begin{problem} What is the Borel complexity of the subgroup $H$ in Example~\ref{Banakh}?
\end{problem}

\begin{remark} Example~\ref{Banakh} shows that Corollary~\ref{c:sN->S+} does not extend to arbitrary Polish groups.
\end{remark}

\section{Haar-thin sets in topological groups}\label{s8}

In this section we introduce thin sets in the Cantor cube, define a new notion of a Haar-thin set in a topological group, and prove that for any Polish group the semi-ideal of Haar-thin sets has the Steinhaus property.

\begin{definition}\label{d:thin} A subset $T$ of the Cantor cube $2^\w$ is called \index{subset!thin}\index{thin subset}{\em thin} if for every number $n\in\w$ the restriction $\pr|T$ of the projection $\pr:2^\w\to 2^{\w\setminus\{n\}}$, $\pr:x\mapsto x|\w\setminus\{n\}$, is injective.
\end{definition}

The family of thin subsets of $2^\omega$ is a semi-ideal containing all singletons, however it is not an ideal as it does not contain all two element sets.

\begin{example} The closed uncountable subset $$T=\{x\in 2^\w:\forall n\in\w\;\;\;x(2n)+x(2n+1)=1\}$$ of\/ $2^\w$ is thin.
\end{example}

\begin{proposition}\label{p:thin=>HN+HM} Each thin Borel subset of the Cantor cube is meager and has Haar measure zero.
\end{proposition}

\begin{proof} Identify the Cantor cube $2^\w$ with the countable power $C_2^\w$ of the two-element cyclic group. For every $n\in\w$ let $z_n\in 2^\w$ be the characteristic function of the singleton $\{n\}$ (i.e., the unique function such that $z_n^{-1}(1)=\{n\}$). Observe that the sequence $(z_n)_{n\in\w}$ converges to the neutral element of the group $2^\w$.

Let $T$ be a Borel thin subset of $2^\w$. Definition~\ref{d:thin} implies that the set $T-T=\{x-y:x,y\in T\}$ is disjoint with the set $\{z_n:n\in\w\}$ and hence is not a neighborhood of zero in the topological group $2^\w$. By the Steinhaus and Piccard-Pettis Theorems~\ref{Stein} and \ref{t:PP}, the Borel set $T$ is meager and has Haar measure zero in $2^\w$.
\end{proof}

\begin{definition} A subset $A$ of a topological group $X$ is called \index{subset!Haar-thin}\index{Haar-thin subset}\index{subset!injectively Haar-thin}\index{injectively Haar-thin subset}({\em injectively}) {\em Haar-thin} if there exists an (injective) continuous map $f:2^\w\to X$ such that for every $x\in X$ the preimage $f^{-1}(A+x)$ is thin in $2^\w$.
\end{definition}

For a topological group $X$ by \index{$\HT$}$\HT$ (resp. \index{$\EHT$}$\EHT$) we denote the semi-ideal consisting of subsets of Borel (injectively) Haar-thin sets in $X$.

Proposition~\ref{p:thin=>HN+HM} implies the following corollary.

\begin{corollary} Each (injectively) Haar-thin Borel subset of a Polish group is Haar-null and (injectively) Haar-meager. Moreover, for any Polish group we have the inclusions$$\HT=\EHT\subset\HN\cap\EHM\subset\EHM\subset\SHM\subset\HM\subset\M.$$
\end{corollary}

The equality $\HT=\EHT$ is proved in the following characterization of
(injectively) Haar-thin sets.

\begin{theorem}\label{t:thin} For a subset $A$ of a Polish group $X$ the following conditions are equivalent:
\begin{enumerate}
\item[\textup{1)}] $A$ is Haar-thin;
\item[\textup{2)}] $A$ is injectively Haar-thin;
\item[\textup{3)}] $A-A$ is not a neighborhood of zero in $X$.
\end{enumerate}
\end{theorem}

\begin{proof} We shall prove the implications $(1)\Ra(3)\Ra(2)\Ra(1)$.
\smallskip

$(1)\Ra(3)$ Assume that the set $A$ is Haar-thin in $X$. Then there exists a continuous map $f:2^\w\to X$ such that for every $x\in X$ the set $f^{-1}(A+x)$ is thin in $2^\w$.

To derive a contradiction, assume that $A-A$ is a neighborhood of zero in $X$. For every $n\in\w$ consider the characteristic function $z_n\in 2^\w$ of the singleton $\{n\}$ (which means that $z_n^{-1}(1)=\{n\}$) and observe that the sequence $(z_n)_{n\in\w}$ converges to the zero function $z_\w:\w\to\{0\}\subset 2$. Since $f(z_\w)+A-A$ is a neighborhood of $f(z_\w)$, there exists $n\in\w$ such that $f(z_n)\in f(z_\w)+A-A$ and hence $f(z_n)-f(z_\w)=b-a$ for some elements $a,b\in A$. It follows that $\{f(z_n),f(z_\w)\}\subset f(z_\w)-a+A$ and hence $\{z_n,z_\w\}\subset f^{-1}(f(z_\w)-a+A)$, which implies that the set $f^{-1}(f(z_\w)-a+A)$ is not thin in $2^\w$. But this contradicts the choice of $f$.
\smallskip

$(3)\Ra(2)$ Assuming that $A-A$ is not a neighborhood of zero in $X$, we shall prove that the set $A$ is injectively Haar-thin in $X$.

Fix a complete metric $\rho$ generating the topology of the Polish group $X$ and for every $x\in X$ put $\|x\|=\rho(x,\theta)$. By induction choose a sequence $(x_n)_{n\in\w}$ of points $x_n\in X\setminus(A-A)$ such that $\|x_{n+1}\|<\frac12\|x_n\|$ for every $n\in\w$. The latter condition implies that the map
$$\Sigma:2^\w\to X,\;\Sigma:a\mapsto\sum_{n\in a^{-1}(1)}x_n,$$
is well-defined, continuous and injective. We claim that for every $x\in X$ the set $T=\Sigma^{-1}(A+x)$ is thin. Assuming that $T$ is not thin, we could find a number $n\in\w$ and two points $a,b\in T$ such that $a(n)=0$, $b(n)=1$ and $a|\w\setminus\{n\}=b|\w\setminus\{n\}$. Then $$x_n=\Sigma(b)-\Sigma(a)\in (A+x)-(A+x)=A-A,$$
which contradicts the choice of $x_n$.
\smallskip

The implication $(2)\Ra(1)$ is trivial.
\end{proof}

Proposition~\ref{p:thin=>HN+HM} and Theorem~\ref{t:thin} imply

\begin{corollary}\label{c:EHT-} For any Polish group the semi-ideals $\HT=\EHT$, $\EHM$, $\SHM$, $\HM$, $\HN$, $\M$ have the Steinhaus property.
\end{corollary}

\section{Null-finite, Haar-finite, and Haar-countable sets}\label{s9}

In this section we introduce and study six new notions of smallness.

In the following definition we endow the ordinal $\w{+}1$ with the order topology, turning $\w{+}1$ into a compact metrizable space with a unique non-isolated point.

\begin{definition}
A subset $A$ of a Polish group $X$ is called
\begin{itemize}
\item \index{subset!Haar-finite}\index{Haar-finite subset}{\em Haar-finite} if there exists a continuous function $f:2^\w\to X$ such that for every $x\in X$ the set $f^{-1}(x+A)$ is finite;
\item \index{subset!Haar-scattered}\index{Haar-scattered subset}{\em Haar-scattered} if there exists a continuous function $f:2^\w\to X$ such that for every $x\in X$ the subspace $f^{-1}(x+A)$ of $2^\w$ is scattered  (i.e., each its nonempty subspace has an isolated point);
\item \index{subset!Haar-countable}\index{Haar-countable subset}{\em Haar-countable} if there exists a continuous function $f:2^\w\to X$ such that for every $x\in X$ the set $f^{-1}(x+A)$ is countable;
\item \index{subset!Haar-$n$}\index{Haar-$n$ subset}{\em Haar-$n$} for $n\in\IN$ if there exists a continuous function $f:2^\w\to X$ such that for every $x\in X$ the set $f^{-1}(x+A)$ has cardinality at most $n$;
\item \index{subset!null-finite}\index{null-finite subset}{\em null-finite} if
there exists a continuous function $f:\w{+}1\to X$ such that for every $x\in X$ the set $f^{-1}(x+A)$ is finite;
\item \index{subset!null-$n$}\index{null-$n$ subset}{\em null-$n$} for $n\in\IN$ if
there exists a continuous function $f:\w{+}1\to X$ such that for every $x\in X$ the preimage $f^{-1}(x+A)$ has cardinality at most $n$.
\end{itemize}
\end{definition}

Null-finite sets were introduced and studied in \cite{BJ}, Haar-finite and Haar-$n$ sets in the real line were explored by Kwela \cite{Kwela}. 

A sequence $(x_n)_{n\in\w}$ in a topological group $X$ is called a \index{null-sequence}{\em null-sequence} if it converges to the neutral element of $X$. The following characterization of null-finite and null-$n$ sets (proved in \cite[Proposition 2.2]{BJ}) justifies the choice of terminology.

\begin{proposition}\label{p:null-char} A subset $A$ of a topological group $X$ is null-finite (resp. null-$n$ for some $n\in\IN$) if and only if $X$ contains a null-sequence $(z_k)_{k\in\w}$ such that for every $x\in X$ the set $\{k\in\w:z_k\in x+A\}$ is finite (resp. has cardinality at most $n$).
\end{proposition}


For any Borel subset of a Polish group we have the following implications:
$$
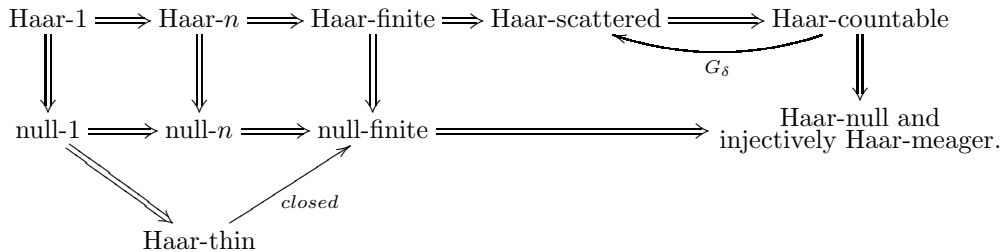
\xymatrixcolsep{15pt}
\xymatrix{
\mbox{Haar-1}\ar@{=>}[d]\ar@{=>}[r]&\mbox{Haar-}n\ar@{=>}[d]\ar@{=>}[r]
&\mbox{Haar-finite}\ar@{=>}[d]\ar@{=>}[r]&\mbox{Haar-scattered}\ar@{=>}[r]&\mbox{Haar-countable}\ar@{=>}[d]\ar@/^12px/^{G_\delta}[l]\\
\mbox{null-1}\ar@{=>}[r]\ar@{=>}[rd]&\mbox{null-}n\ar@{=>}[r]&\mbox{null-finite}\ar@{=>}[rr]
&&\mbox{Haar-null and}\atop\mbox{injectively Haar-meager.}\\
&\mbox{Haar-thin}\ar_{closed}[ru]
}$$

Non-trivial implications in this diagram are proved in Theorem~\ref{t:NF=>HN+EHM}, Propositions \ref{p:scattered}, \ref{p:n1} and Corollary~\ref{c:HT=>NF}.

\begin{theorem}\label{t:NF=>HN+EHM} Each Borel null-finite subset $A$ in a Polish group $X$ is Haar-null and injectively Haar-meager.
\end{theorem}

\begin{proof} By Theorems~6.1 and 5.1 in \cite{BJ}, each Borel null-finite subset of $X$ is Haar-null and Haar-meager. By a suitable modification of the proof of Theorem 5.1 in \cite{BJ}, we shall prove that each null-finite Borel set in $X$ is injectively Haar-meager.
By Proposition~\ref{p:null-char}, the Polish group $X$ contains a null-sequence $(z_n)_{n\in\w}$ such that for any $x\in X$ the set $\{n\in\w:z_n\in x+A\}$ is finite. Fix a complete invariant metric $\rho$ generating the topology of the Polish group $X$ and put $\|x\|=\rho(x,\theta)$ for every $x\in X$.

Replacing $(z_n)_{n\in\w}$ by a suitable subsequence, we can assume that $$\frac1{2^n}>\|z_n\|>\sum_{i=n+1}^\infty \|z_i\|$$for every $n\in\w$. 

Fix a sequence $(\Omega_n)_{n\in\w}$ of pairwise disjoint infinite subsets of $\w$ such that $\Omega_n\subset[n,\infty)$ for all $n\in\w$.

For every $n\in\w$ consider the compact set $S_k:=\{\theta\}\cup\{z_n\}_{n\in\Omega_k}\subset X$. The choice of the sequence $(z_n)_{n\in\w}$ ensures that the function $$\Sigma:\prod_{n\in\w}S_n\to X,\;\;\Sigma:(x_n)_{n\in\w}\mapsto\sum_{n\in\w}x_n,$$ is well-defined and continuous.
Repeating the argument from the proof of Theorem 5.1 \cite{BJ}, we can show that for every $x\in X$ the set $\Sigma^{-1}(x+A)$ is meager in $\prod_{n\in\w}S_n$.

We claim that the function $\Sigma$ is injective. Given two distinct sequences $(x_i)_{i\in\w}$ and $(y_i)_{i\in\w}$ in $\prod_{i\in\w}S_i$, we should prove that $\sum_{i\in\w}x_i\ne\sum_{i\in\w}y_i$. For every $i\in\w$ find numbers $n_i,m_i\in\Omega_i$ such that $x_i=z_{n_i}$ and $y_i=z_{m_i}$. Let $\Lambda=\{i\in\w:n_i\ne m_i\}$ and observe that
$\sum_{k\in\w}x_k\ne\sum_{k\in\w}y_k$ if and only if $\sum_{i\in\Lambda}z_{n_i}\ne\sum_{i\in\Lambda}z_{m_i}$.
It follows that the sets $\{n_i:i\in\Lambda\}$ and $\{m_i:i\in\Lambda\}$ are disjoint. Let $s$ be the smallest element of the set $\{n_i:i\in\Lambda\}\cup\{m_i:i\in\Lambda\}$. Without loss of generality, we can assume that $s=n_k$ for some $k\in\Lambda$.
Then $$
\Big\|\sum_{i\in\Lambda}z_{n_i}-\sum_{i\in\Lambda}z_{m_i}\Big\|\ge
\|z_{n_k}\|-\sum_{i\in\Lambda\setminus\{k\}}\|z_{n_i}\|-\sum_{i\in\Lambda}\|z_{m_i}\|\ge
\|z_s\|-\sum_{i=s+1}^\infty\|z_i\|>0,
$$
which implies $\sum_{i\in\Lambda}z_{n_i}\neq\sum_{i\in\Lambda}z_{m_i}$ and $\sum_{n\in\w}x_n\neq\sum_{n\in\w}y_n$.

The product $\prod_{n\in\w}S_n$ is homeomorphic to the Cantor cube $2^\w$ (being a zero-dimen\-sional compact metrizable space without isolated points). Then the injective continuous function $\Sigma:\prod_{n\in\w}S_n\to X$ witnesses that the set $A$ is injectively Haar-meager and hence $A\in\EHM$.
\end{proof}

We recall that a topological space $X$ is \index{topological space!scattered}\index{scattered space}{\em scattered} if each nonempty subspace of $X$ has an isolated point.  Each second-countable scattered space is countable (being hereditarily  Lindel\"of). On the other hand, the space of rational numbers is countable but not scattered. 
The Baire Theorem implies that a Polish space is scattered if and only if it is countable. This fact implies the following characterization.

\begin{proposition}\label{p:scattered} A $G_\delta$-subset $A$ of a Polish group $X$ is Haar-countable if and only if it is Haar-scattered.
\end{proposition}

\begin{remark} In Example~\ref{ex:M} we shall construct an $F_\sigma$-subset $M$ of the Banach space $X:=C[0,1]$ such that $M$ is Haar-countable but not Haar-scattered in $X$.
\end{remark}

\begin{proposition}\label{p:n1} If a subset $A$ of a topological group $X$ is null-$1$, then $A-A$ is not a neighborhood of zero and hence $A$ is Haar-thin in $X$.
\end{proposition}

\begin{proof} To derive a contradiction, assume that $A-A$ is a neighborhood of zero in $X$. Since $A$ is null-1, there exists an infinite compact subset $K\subset X$ such that $|K\cap(A+x)|\le 1$ for every $x\in X$. Choose any non-isolated point $x$ of the compact space $K$. Since $O_x:=x+(A-A)$ is a neighborhood of $x$, we can choose a point $y\in K\cap O_x\setminus\{x\}$ and conclude that $y-x=b-a$ for some $a,b\in A$. Then $\{x,y\}\subset K\cap (A-a+x)$ and hence $|K\cap(A-a+x)|\ge 2$, which contradicts the choice of the compact set $K$.

This contradiction shows that $A-A$ is not a neighborhood of zero. By Theorem~\ref{t:thin}, the set $A$ is Haar-thin in $X$.
\end{proof}

Moreover, for the semi-ideals of null-finite and null-1 sets we have the following Steinhaus-like properties (cf. \cite[3.1]{BJ}).

\begin{theorem}\label{t:St-NF} Let $A$ be a subset of a topological group $X$.
\begin{enumerate}
\item[\textup{1)}] If $A$ is not null-finite, then $A-\bar A$ is a neighborhood of zero in $X$.
\item[\textup{2)}] If $X$ is first-countable and $A$ is not null-$1$, then each neighborhood $U\subset X$ of zero contains a finite set $F\subset X$ such that $F+(A-A)$ is a neighborhood of zero.
\end{enumerate}
\end{theorem}

\begin{proof} 1. Assuming that $A$ is not null-finite, we shall show that $A-\bar A$ is a neighborhood of zero. In the opposite case, we could find a null-sequence $(x_n)_{n\in\w}$ contained in $X\setminus(A-\bar A)$. Since $A$ is not null-finite, there exists $a\in X$ such that the set $\Omega=\{n\in\w:a+x_n\in A\}$ is infinite. Then $a\in\overline{\{a+x_n\}}_{n\in\Omega}\subset\bar A$ and hence $x_n=(a+x_n)-a\in A-\bar A$ for all $n\in\Omega$, which contradicts the choice of the sequence $(x_n)_{n\in\w}$.
\smallskip

2. Assume that $X$ is first-countable and $A$ is not null-$1$. Fix a decreasing neighborhood base $(U_n)_{n\in\w}$ at zero in $X$ such that $U_0\subset U$. For the proof by contradiction, suppose that for any finite set $F\subset U$ the set $F+(A-A)$ is not a neighborhood of zero. Then we can inductively construct a null-sequence $(x_n)_{n\in\w}$ such that $x_n\in U_n\setminus \bigcup_{0\le i<n}(x_i+A-A)$ for all $n\in\w$. Observe that for each $z\in X$ the set $\{n\in\w:z+x_n\in A\}$ contains at most one point. Indeed, in the opposite case we could find two numbers $k<n$ with $z+x_k,z+x_n\in A$ and conclude that $z\in -x_k+A$ and hence $x_n\in -z+A\subset x_k-A+A$, which contradicts the choice of the number $x_n$. The sequence $(x_n)_{n\in\w}$ witnesses that the set $A$ is null-$1$ in $X$, which is a desired contradiction.
\end{proof}

Theorem~\ref{t:St-NF} suggests the following open problem.

\begin{problem}\label{prob:NF-St} Assume that a Borel subset $A$ of a (locally compact) Polish group $X$ is not null-finite. Is $A-A$ a neighborhood of zero in $X$?
\end{problem}

Theorems~\ref{t:St-NF} and \ref{t:thin} imply:

\begin{corollary}\label{c:HT=>NF} Each closed Haar-thin set in a Polish group is null-finite.
\end{corollary}

\begin{problem} Is each closed Haar-thin subset of a Polish group Haar-finite?
\end{problem}

\begin{example} The Polish group $\IR^\w$ contains a closed subset $F$, which is Haar-2 but not Haar-thin.
\end{example}

\begin{proof}
 In the hyperspace $\K(\IR^\w)$ consider the closed subset
$$\K=\{K\in\K(\IR^\w):|K|\le 2\}.$$
By Theorem~\ref{t:univer}, there exists a closed subset $F\subset\IR^\w$ such that
\begin{enumerate}
\item[\textup{(1)}] for every $K\in\K$ there exists $d\in \IR^\w$ such that $K+d\subset F$;
\item[\textup{(2)}] for any $x\in \IR^\w$ the intersection $[0,1]^\w\cap (F+x)$ is contained in $K+d$ for some $K\in\K$ and $d\in\IR^\w$.
\end{enumerate}

To see that $F$ is Haar-2, take any injective continuous map $f:2^\w\to[0,1]^\w\subset\IR^\w$ and observe that the condition (2) ensures that for any $x\in\IR^\w$ the intersection $f(2^\w)\cap(F+x)\subset [0,1]^\w\cap(F+x)$ has cardinality $\le 2$. By the injectivity of $f$, the preimage $f^{-1}(F+x)$ also has cardinality $\le2$, witnessing that $F$ is Haar-2.

Next, we show that the set $F$ is not Haar-thin. To derive a contradiction, assume that $F$ is Haar-thin and find a continuous map $g:2^\w\to \IR^\w$ such that for every $x\in \IR^\w$ the preimage $g^{-1}(F+x)$ is thin in $2^\w$. Choose any doubleton $D\subset 2^\w$, which is not thin and consider the set $K=g(D)\in\K$. By the condition (1), there exists $d\in \IR^\w$ such that $K\subset F+d$. Then the set $g^{-1}(F+d)\supset g^{-1}(K)\supset D$ is not thin, which contradicts the choice of the function $g$.
\end{proof}

\begin{remark} By Example~\ref{ex:13.11} and Theorem~\ref{t:thin}, the closed subset $(C_3^*)^\w$ of the compact Polish group $C_3^\w$ is Haar-2 but not Haar-thin.
\end{remark}

\begin{example} The Polish group $\IR^\w$ contains a closed subset $F$, which is null-1 but not Haar-countable.
\end{example}

\begin{proof} Fix any sequence $(z_n)_{n\in\w}$ of pairwise distinct points of the set $(0,1]^\w\subset\IR^\w$ that converges to the point $z_\w=\theta\in\IR^\w$.

 In the hyperspace $\K(\IR^\w)$ consider the $G_\delta$-subset
$$\K=\bigcap_{n\ne m}\big\{K\in\K(\IR^\w):(K-z_n)\cap (K-z_m)=\emptyset\big\},$$
where $n,m\in\w\cup\{\w\}$. By Theorem~\ref{t:univer}, there exists a closed subset $F\subset\IR^\w$ such that
\begin{enumerate}
\item[\textup{(1)}] for every $K\in\K$ there exists $d\in \IR^\w$ such that $K+d\subset F$;
\item[\textup{(2)}] for any $x\in \IR^\w$ the intersection $(x+[0,1]^\w)\cap F$ is contained in $K+d$ for some $K\in\K$ and $d\in\IR^\w$.
\end{enumerate}

We claim that the infinite compact set $S=\{\theta\}\cup\{z_n\}_{n\in\w}$ witnesses that the set $F$ is null-1. Assuming the opposite, we could find $x\in\IR^\w$ such that $(x+S)\cap F$ contains two distinct points $x+z_n$ and $x+z_m$ with $n,m\in\w\cup\{\w\}$. By the condition (2), the intersection $(x+S)\cap F\subset (x+[0,1]^\w)\cap F$ is contained in the sum $K+d$ for some $K\in\K$ and some $d\in\IR^\w$. It follows that $\{x+z_n,x+z_m\}\subset (x+S)\cap F\subset K+d$ and hence $x-d\in (K-z_n)\cap (K-z_m)$, which contradicts $K\in\K$. This contradiction shows that the set $F$ is null-1.

Now we show that $F$ is not Haar-countable. Assuming that $F$ is Haar-countable, we can find a continuous map $f:2^\w\to \IR^\w$ such that for every $x\in \IR^\w$ the set $f^{-1}(x+F)$ is at most countable. It follows that for any $y\in \IR^\w$ the preimage $f^{-1}(y)$ is at most countable, which implies that the compact set $C=f(2^\w)$ has no isolated points.

Observe that for any distinct ordinals $n,m\in\w\cup\{\w\}$ the set $$\U_{n,m}=\{K\in\K(C):(K-z_n)\cap (K-z_m)=\emptyset\}$$is open and dense in the hyperspace $\K(C)$ of the compact space $C$. Then the intersection $\U=\bigcap_{n\ne m}\U_{n,m}$ is a dense $G_\delta$-set in $\K(C)$
and hence $\U$ contains an uncountable set $K\in\U$. By the definitions of $\U$ and $\K$, the set $K$ belongs to the family $\K$ and by the condition (1), $x+K\subset F$ for some $x\in \IR^\w$. Then $f^{-1}(F-x)\supset f^{-1}(K)$ is uncountable, which contradicts the choice of $f$.
\end{proof}

\begin{problem} Let $A$ be a null-finite compact subset of a Polish group $X$. Is $A$ Haar-countable \textup{(}Haar-finite\textup{)}?
\end{problem}

\begin{remark} By Example~\ref{ex:hard}, there exists compact Polish group $X$ containing a Haar-thin null-finite $G_\delta$-subset $B\subset X$ such that $B$ is not Haar-countable and not null-$n$ for every $n\in\IN$.
\end{remark}

\begin{remark} In \cite{Kwela} Kwela constructed two Haar-finite compact subsets $A,B\subset\IR$ whose union $A\cup B$ is not null-finite in the real line. Kwela also constructed an example of a compact subset of the real line, which is Haar-finite but not Haar-$n$ for every $n\in\IN$.
\end{remark}

\section{Haar-$\I$ sets in topological groups}\label{s10}

In this section we introduce the notion of a Haar-$\I$ set which generalizes the notions of small sets, considered in the preceding sections.

\begin{definition}\label{defHsmall}
 Let $\I$ be a semi-ideal of subsets of some nonempty topological space $K=\bigcup\I$. A subset $A$ of a topological group $X$ is called \index{subset!Haar-$\I$}\index{Haar-$\I$ subset}\index{subset!injectively Haar-$\I$}\index{injectively Haar-$\I$ subset}({\em injectively}) {\em Haar-$\I$} if there exists an (injective) continuous map $f:K\to X$ such that $f^{-1}(A+x)\in\I$ for all $x\in X$.
\end{definition}

For a Polish group $X$ by \index{$\HI$}$\HI$ (resp. \index{$\EHI$}$\EHI$) we denote the semi-ideal consisting of subsets of (injectively) Haar-$\I$ Borel sets in $X$. Definition~\ref{defHsmall} implies the following simple but important fact.

\begin{proposition} Let $\I$ be a semi-ideal on a compact space $K$. For any Polish group $X$ the semi-ideals $\HI$ and $\EHI$ are invariant.
\end{proposition}

\begin{theorem}\label{t:NoSt+} Let $\I$ be a proper semi-ideal on a compact space $K$. For any non-locally compact Polish group $X$ the semi-ideal $\HI$ does not have the strong Steinhaus property.
\end{theorem}

\begin{proof} By Proposition~\ref{p:2thick}, the group $X$ contains two Borel thick sets $A,B\subset X$ whose sum $A+B$ has empty interior. We claim that $A,B\notin\HI$. Indeed, since $A$ is thick, for any continuous map $f:K\to X$ there is a point $x\in X$ such that $x+f(K)\subset A$. Then $f^{-1}(A-x)=K\notin\I$ and hence $A$ is not Haar-$\I$. By the same reason the set $B$ is not Haar-$\I$ in $X$.
\end{proof}

On the other hand, we have the following corollary of Theorem~\ref{t:thin}.

\begin{theorem} If a semi-ideal $\I$ on $2^\w$ contains all Borel thin subsets of $2^\w$, then for any Polish group the semi-ideals $\EHI$ and $\HI$ have the Steinhaus property.
\end{theorem}

Now we show that the notion of a Haar-$\I$ set generalizes many notions of smallness, considered in the preceding sections. The items (1) and (2) of the following characterization can be derived from Theorem~\ref{Haarnull}, Lemma~\ref{l:a} and Proposition~\ref{omega}; the items (3)--(7) follow from the definitions.

\begin{theorem}\label{t:Haar} A subset $A$ of a Polish group $X$ is:
\begin{enumerate}
\item[\textup{1)}] Haar-null if and only if $A$ is Haar-$\N$ if and only if $A$ is injectively Haar-$\N$ where $\N$ is the $\sigma$-ideal of null subsets of $2^\w$ with respect to any continuous measure $\mu\in P(2^\w)$;
\item[\textup{2)}] (injectively) Haar-meager if and only if $A$ is (injectively) Haar-$\M$ for the $\sigma$-ideal $\M$ of meager sets in $2^\w$;
\item[\textup{3)}] Haar-countable if and only if $A$ is Haar-$[2^\w]^{\le\w}$ for the $\sigma$-ideal $[2^\w]^{\le\w}$ of countable subsets of the Cantor cube;
\item[\textup{4)}] Haar-scattered if and only if $A$ is Haar-$\I$ for the ideal $\I$ of scattered subsets of the Cantor cube;
\item[\textup{5)}] Haar-finite if and only if $A$ is Haar-$[2^\w]^{<\w}$ for the ideal $[2^\w]^{<\w}$ of finite subsets of the Cantor cube;
\item[\textup{6)}] Haar-$n$ if and only if $A$ is Haar-$[2^\w]^{\le n}$ for the semi-ideal of subsets of cardinality $\le n$ in the Cantor cube;
\item[\textup{7)}] null-finite if and only if it is Haar-$[\w{+}1]^{<\w}$ for the ideal $[\w{+}1]^{<\w}$ of finite subsets of the ordinal $\w+1$ endowed with the order topology;
\item[\textup{8)}] null-$n$ if and only if it is Haar-$[\w{+}1]^{\le n}$ for the semi-ideal $[\w{+}1]^{\le n}$ of subsets of cardinality $\le n$ in the ordinal $\w+1$.
\end{enumerate}
\end{theorem}

In many cases closed Haar-$\I$ sets in Polish groups are Haar-meager.

\begin{theorem}\label{t:HI=HM} Let $\I$ be a proper ideal on a compact topological space $K$. Any closed Haar-$\I$ set $A$ in a Polish group $X$ is (strongly) Haar-meager, which yields the inclusions $\overline{\HI}\subset\overline{\HM}$ and $\sigma\overline{\HI}\subset\sigma\overline{\HM}$.
\end{theorem}

\begin{proof} Given a closed Haar-$\I$ set $A\subset X$, find a continuous map $f:K\to X$ such that $f^{-1}(A+x)\in\I$ for all $x\in X$.

Fix a complete invariant metric $\rho$ generating the topology of the Polish group $X$. For every $n\in\w$ let $\U_n$ be the cover of $K$ by open subsets $U\subset K$ such that $\diam\,f(U)<\frac1{2^n}$. By the compactness of $K$, the cover $\U_n$ has a finite subcover $\U_n'$. Since $\bigcup\U_n'=K\notin\I$, there exists a set $U_n\in\U_n'$ such that $U_n\notin \I$. Then the closure $K_n$ of $U_n$ in $K$ does not belong to the ideal $\I$ and its image has $\diam f(K_n)\le\frac1{2^n}$. Choose any point $y_n\in f(K_n)$.

It follows that the map $$g:\prod_{n\in\w}K_n\to X,\;\;g:(x_n)_{n\in\w}\mapsto\sum_{n=0}^\infty (f(x_n)-y_n)$$is well-defined and continuous. We claim that for every $x\in X$ the preimage $g^{-1}(A+x)$ is meager in $\prod_{n\in\w}K_n$. Assuming that for some $x\in X$ the closed set $g^{-1}(A+x)$ is not meager in $\prod_{n\in\w}K_n$, we would conclude that $g^{-1}(A+x)$ has nonempty interior $W$, which contains some point $(w_n)_{n\in\w}$. Then for some number $m\in\w$ we have the inclusion $\{(w_i)_{i<m}\}\times\prod_{n\ge m}K_n\subset W\subset g^{-1}(A+x)$. Moreover, for the point $y=\sum_{n\in\w\setminus\{m\}}(f(w_n)-y_n)\in X$ we have
\begin{multline*}f(K_m)-y_m+y=f(K_m)-y_m+\sum_{n\in\w\setminus\{m\}}(f(w_n)-y_n)\\
=g\big(\{(w_i)_{i<m}\}\times K_m\times\{(w_i)_{i>m}\}\big)\subset x+A
\end{multline*}and hence
$\I\not\ni K_m\subset f^{-1}(A+x+y_m-y)$, which contradicts the choice of the map $f$. This contradiction shows that for every $x\in X$ the preimage $g^{-1}(A+x)$ is meager in $K$, witnessing that the set $A$ is Haar-meager.

By Theorem~\ref{t:prethick}, the closed Haar-meager set $A$ is strongly Haar-meager.
\end{proof}




(Injectively) Haar-$\I$ sets behave nicely under homomorphisms.

\begin{theorem}\label{GrouphomoHI} Let $\I$ be a semi-ideal in a zero-dimensional compact space $K$.
Let $h:X\to Y$ be a continuous surjective homomorphism between Polish groups.
For any (injectively) Haar-$\I$ set $A\subset Y$, the preimage $h^{-1}(A)$ is an (injectively) Haar-$\I$ set in $X$.
\end{theorem}

\begin{proof} Given an (injectively) Haar-$\I$ set $A\subset Y$, we can find an (injective) continuous map $f:K\to Y$ such that $f^{-1}(A+y)\in\I$ for every $y\in Y$. By \cite[Theorem 1.2.6]{BecKec}, the continuous surjective homomorphism $h:X\to Y$ is open.
By the zero-dimensional Michael Selection Theorem \cite[Theorem 2]{Mich}, there exists a continuous map $\varphi:K\to X$ such that $f=h\circ \varphi$. Observe that the map $\varphi$ is injective if so is the map $f$. We claim that the map $\varphi$ witnesses that the set $B=h^{-1}(A)$ is (injectively) Haar-$\I$ in $X$. Since $h$ is a homomorphism, for any $x\in X$ we have $B+x=h^{-1}(A+y)$ where $y=h(x)$. Then $$\varphi^{-1}(B+x)=\varphi^{-1}(h^{-1}(A+y))=
(h\circ\varphi)^{-1}(A+y)=f^{-1}(A+y)\in\I.$$
\end{proof}

The analog of the above theorem for images through homomorphisms does not hold, as shows the following easy example.

\begin{example}
Let $X:= \IR^2$ and denote by $\pi :\IR^2 \to \IR$ the projection onto the first coordinate. Then $\pi$ is a continuous surjective homomorphism and the set $A:=\IR \times \{0\}$ is Haar-1, but its image is not Haar-$\I$ for any proper semi-ideal $\I$ on a compact topological space $K$.
\end{example}


\begin{remark}
Let $\I$ be a semi-ideal on a compact space $K$. For any Polish groups $X$, $Y$ and (injectively) Haar-$\I$ set $A\subset X$ the set $A\times Y$ is (injectively) Haar-$\I$ in a Polish group $X \times Y$.
\end{remark}

It looks like the ideals $\I$ on the Cantor space $2^\omega$ play critical role for whole theory of Haar-$\I$ sets. In such setting we provide another proposition for $\sigma$-ideals.

\begin{proposition}
Let $\I$ be a proper $\sigma$-ideal on the Cantor cube $2^\w=\bigcup \I$. Then there exists a proper $\sigma$-ideal $\mc{J}$ with $\bigcup \J = 2^\w$ such that $\mc{HI} = \mc{HJ}$ and $\mathcal J$ contains no open nonempty subsets of $2^\w$.
\end{proposition}

\begin{proof}
The union $W$ of all open sets in the ideal $\I$ belongs to the $\sigma$-ideal $\I$ by the Lindel\"of property of $W$. Since $\I$ is a proper ideal, $W\ne 2^\w$. Moreover, the compact subset $2^\w\setminus W$ does not contain isolated points due to $\bigcup \I = 2^\w$, hence there exists homeomorphism $h:2^\w \to 2^\w \setminus W$. Obviously $\mc{J} := \{ A \subset 2^\w : h(A) \in \I \}$ has all required properties.
\end{proof}

Next, we evaluate the Borel complexity of closed Haar-$\I$ sets in the space $\F(X)$ of all closed subsets of a Polish group $X$, endowed with the Fell topology.

A subset $D\subset\w^\w$ is called \index{subset!dominating}\index{dominating set}{\em dominating} if for any $x\in \w^\w$ there exists $y\in D$ such that $x\le^* y$ (which means that $x(n)\le y(n)$ for all but finitely many numbers $n\in\w$).

In each non-locally compact Polish group, Solecki \cite{S01} constructed a closed subset $F$ admitting an open perfect map $f:F\to\w^\w$ possessing the following properties.

\begin{theorem}\label{t:Solecki} For any non-locally compact Polish group $X$ there exists a closed set $F\subset X$ and an open perfect surjective map $f:F\to\w^\w$ having the following properties:
\begin{enumerate}
\item[\textup{1)}] for any non-dominating set $H\subset\w^\w$ the preimage $f^{-1}(H)$ is openly Haar-null in $X$;
\item[\textup{2)}] for any compact set $K\subset X$ there is $x\in X$ such that $x+K\subset f^{-1}(y)$ for some $y\in\w^\w$;
\item[\textup{3)}] for any $x,y\in\w^\w$ with $x\le^* y$ there exists a countable set $C\subset X$ such that $f^{-1}(x)\subset C+f^{-1}(y)$;
\item[\textup{4)}] for any dominating set $D\subset \w^\w$ the preimage $f^{-1}(D)$ is not Haar-$\I$ for any proper $\sigma$-ideal $\I$ on a compact topological space $K$.
\end{enumerate}
\end{theorem}

\begin{proof} The properties (1)--(3) were established by Solecki \cite[p.208]{S01}.
To see that (4) also holds, take any dominating set $D\subset \w^\w$ and assume that the preimage $A:=f^{-1}(D)$ is Haar-$\I$ for some proper $\sigma$-ideal $\I$ on a compact topological space $K$. Then we can find a continuous map $\varphi:K\to X$ such that $\varphi^{-1}(A+x)\in\I$ for all $x\in X$. By the property (2), for the compact subset $\varphi(K)$ of $X$ there exists $x\in X$ such that $\varphi(K)\subset x+f^{-1}(y)$ for some $y\in \w^\w$. The set $D$, being dominating, contains an element $z\in D$ such that $y\le^* z$. Now the property (3) yields a countable set $C\subset X$ such that $f^{-1}(y)\subset C+f^{-1}(z)$. Then $$\varphi(K)\subset x+f^{-1}(y)\subset x+C+f^{-1}(z)\subset x+C+f^{-1}(D)=x+C+A$$ and $K=\bigcup_{c\in C}\varphi^{-1}(x+c+A)\in\I$ as $\I$ is a $\sigma$-ideal. This means that the ideal $\I$ is not proper, which contradicts our assumption.
\end{proof}

By \index{$\mathrm{CD}$}\index{$\mathrm{CND}$} $\mathrm{CD}$ and $\mathrm{CND}$ we denote the families of subsets of $\F(\w^\w)$ consisting of closed dominating and closed non-dominating subsets of $\w^\w$, respectively.

\begin{corollary}\label{c:CND} For each non-locally compact Polish group $X$, there exists a continuous map $\Phi:\F(\w^\w)\to\F(X)$ such that $\Phi^{-1}(\HN^\circ)=\Phi^{-1}(\HI)=\mathrm{CND}$ for any proper $\sigma$-ideal $\I$ on a compact topological space $K$ with $\F(X)\cap\HN^\circ\subset\HI$. In particular, $\Phi^{-1}(\HN^\circ)=\Phi^{-1}(\mathcal{HJ})=\mathrm{CND}$ for any proper $\sigma$-ideal $\J$ on $2^\w$ with $\overline{\N}\subset\J$.
\end{corollary}

\begin{proof} Let $f:F\to\w^\w$ be the open perfect map from Theorem~\ref{t:Solecki}. By \cite[3.1]{S01}, the map
$$\Phi:\F(\w^\w)\to\F(X),\;\;\Phi:D\mapsto f^{-1}(D),$$is continuous.
Theorem~\ref{t:Solecki}(1,4) implies that $\Phi^{-1}(\HN^\circ)=\Phi^{-1}(\HI)=\mathrm{CND}$ for any proper $\sigma$-ideal $\I$ on a compact topological space $K$ with $\F(X)\cap\HN^\circ\subset\HI$.

Now assume that $\J$ is a proper $\sigma$-ideal on $2^\w$ with $\overline{\N}\subset \J$. By Theorem~\ref{Haarnull},
$$\F(X)\cap\HN^\circ\subset\F(X)\cap\HN=\F(X)\cap\mathcal H\overline{\N}\subset\F(X)\cap\mathcal{HJ}$$and hence $\Phi^{-1}(\HN^\circ)=\Phi^{-1}(\mathcal{HJ})=\mathrm{CND}$.
\end{proof}

Let $\Gamma$ be a class of metrizable separable spaces. A subset $H$ of a standard Borel space $X$ is called \index{subset!$\Gamma$-hard}\index{$\Gamma$-hard subset}{\em $\Gamma$-hard} in $X$ if for any Polish space $P$ and a subspace $G\in\Gamma$ in $P$ there exists a Borel map $f:P\to X$ such that $f^{-1}(H)=G$. We recall that $\mathbf{\Sigma_1^1}$ and $\mathbf{\Pi_1^1}$ denote the classes of analytic and coanalytic spaces, respectively. By \cite[3.2]{S01} (and an unpublished result of Hjorth \cite{Hjorth} mentioned in \cite{S01} and proved in \cite[4.8]{TV}), the set $\mathrm{CND}$ of $\F(\w^\w)$ is $\mathbf{\Pi_1^1}$-hard (and $\mathbf{\Sigma^1_1}$-hard) and hence $\mathrm{CND}$ is neither analytic nor coanalytic in $\F(\w^\w)$.

This facts combined with Corollary~\ref{c:CND} imply another corollary.

\begin{corollary}\label{c:hard} For any proper $\sigma$-ideal $\I$ on $2^\w$ with $\overline{\N}\subset\I$ and any non-locally compact Polish group $X$ the subset $\F(X)\cap\HI$ of $\F(X)$ is $\mathbf{\Pi_1^1}$-hard and $\mathbf{\Sigma_1^1}$-hard. Consequently, the set $\F(X)\cap\HI$ is neither analytic nor coanalytic in the standard Borel space $\F(X)$. 
\end{corollary}

Since $\overline{\N}\subset\sigma\overline{\N}\subset\M$ in $2^\w$, Corollary~\ref{c:hard} implies:

\begin{corollary}\label{c:hard-M} For any non-locally compact Polish group $X$ the subsets $\F(X)\cap\mathcal H\sigma\overline{\N}$, $\F(X)\cap\HN$, and
$\F(X)\cap\HM$ of $\F(X)$ are $\mathbf{\Pi_1^1}$-hard and $\mathbf{\Sigma_1^1}$-hard.
\end{corollary}

\begin{remark} For any non-locally compact Polish group the $\mathbf{\Pi_1^1}$-hardness and $\mathbf{\Sigma_1^1}$-hardness of the set $\F(X)\cap\HN$ in $\F(X)$ was proved by Solecki \cite{S01} and the proof of Corollary~\ref{c:hard} is just a routine modification of the proof of Solecki.
\end{remark}

\section{Fubini ideals}\label{s11}

In this section we give conditions on an ideal $\I$ of subsets of a topological space $K$ under which for every Polish group $X$ the semi-ideal $\HI$ is an ideal or a $\sigma$-ideal.

Let $K$ be a topological space. Let $n\le\w$ be a non-zero countable cardinal and $i\in n$. For every $a\in K^{n\setminus\{i\}}$, consider the embedding $e_a:K\to K^n$ assigning to each $x\in K$ the function $y:n\to K$ such that $y(i)=x$ and $y|n\setminus\{i\}=a$.

Each family $\I$ of subsets of the space $K$ induces the family
$$\I^n_i=\{A\subset K^n:\forall a\in K^{n\setminus\{i\}}\;\;e_a^{-1}(A)\in\I\}$$on $K^n$.

\begin{definition} A family $\I$ of subsets of the space $K$ is defined to be \index{family!$n$-Fubini}\index{$n$-Fubini family}{\em $n$-Fubini} for $n\in\IN\cup\{\w\}$ if there exists a continuous map $h:K\to K^n$ such that for any $i\in n$ and any Borel set $B\in\I^n_i$ the preimage $h^{-1}(B)$ belongs to the family $\I$.

We shall say that a family $\I$ of subsets of the space $K$ is \index{family!Fubini}\index{Fubini family} {\em Fubini} if it satisfies the equivalent conditions of the following theorem.
\end{definition}


\begin{theorem}\label{t:Fubini} For any family $\I$ of subsets of a topological space $K$ the following conditions are equivalent:
\begin{enumerate}
\item[\textup{1)}] $\I$ is $n$-Fubini for every $n$ with $2\le n\le\w$;
\item[\textup{2)}] $\I$ is $n$-Fubini for some $n$ with $2\le n\le\w$;
\item[\textup{3)}] $\I$ is $2$-Fubini.
\end{enumerate}
\end{theorem}

\begin{proof} The implication $(1)\Ra(2)$ is trivial. To see that $(2)\Ra(3)$, assume that $\I$ is $n$-Fubini for some $n$ with $2\le n\le \w$. Then there exists a continuous map $h:K\to K^n$ such that for any $i\in n$ and any Borel subset $B\in\I^n_i$ of $K^n$ the preimage $h^{-1}(B)$ belongs to the family $\I$. Let $p:K^n\to K^2$, $p:x\mapsto (x(0),x(1))$, be the projection of $K^n$ onto the first two coordinates.

We claim that the map $h_2:=p\circ h:K\to K^2$ witnesses that the ideal $\I$ is 2-Fubini.
Given any $i\in 2$ and any Borel subset $B\in\I^2_i$ of $K^2$, observe that the preimage $p^{-1}(B)$ is a Borel subset of $X^n$ that belongs to the family $\I^n_i$. Now the choice of the map $h$ guarantees that $h_2^{-1}(B)=h^{-1}(p^{-1}(B))\in\I$.
\smallskip

$(3)\Ra(1)$ Assume that $\I$ is 2-Fubini, which means that there exists a continuous map $h:K\to K^2$ such that for any $i\in 2=\{0,1\}$ and any Borel subset $B\in\I^2_i$ of $K^2$ the preimage $h^{-1}(B)$ belongs to the family $\I$. For every $i\in 2$ let $\pi_i:K^2\to K$, $\pi_i:x\mapsto x(i)$, be the coordinate projection.

For every integer $n\in\IN$ consider the map
$$H_n:K^n\to K^{n+1},\;\,H_n:(x_0,\dots,x_{n-2},x_{n-1})\mapsto (x_0,\dots,x_{n-2},\pi_0\circ h(x_{n-1}),\pi_1\circ h(x_{n-1})).$$
For $n=1$ the map $H_1$ coincides with $h$.

Define a sequence of continuous maps $(h_n:K\to K^n)_{n=2}^\infty$ by the recursive formula: $h_2=h$ and $h_{n+1}=H_n\circ h_n$ for $n\ge 2$.

We claim that for every $n\ge 2$ the map $h_n$ witnesses that the family $\I$ is $n$-Fubini. For $n=2$ this follows from the choice of the map $h_2=h$. Assume that for some $n\ge2$ we have proved that for every $i\in n$ and every Borel set $B\in \I^n_i$ in $K^n$ we have $h^{-1}_n(B)\in\I$.

Fix $i\in n+1$ and a Borel subset $B\in\I^{n+1}_i$ in $K^{n+1}$. Observe that $h^{-1}_{n+1}(B)=h_n^{-1}(H_n^{-1}(B))$. If $i<n-1$, then the definition of the map $H_n$ ensures that
$H_n^{-1}(B)\in\I^{n}_i$. If $i\in \{n-1,n\}$, then the choice of the map $h$ guarantees that $H_n^{-1}(B)\in \I^n_{n-1}$. In both cases $h_{n+1}^{-1}(B)=h_n^{-1}(H_n^{-1}(B))\in\I$ by the inductive assumption.

Next, we prove that the family $\I$ is $\w$-Fubini. This is trivial if $K=\emptyset$. So, we assume that $K\ne\emptyset$ and fix any point $p\in K$. For every $n\in\IN$ identify the power $K^n$ with the subset $K^n\times \{p\}^{\w\setminus n}$ of $K^\w$. For every $k\in\w$ let $\pi_k:K^\w\to K$, $\pi_k:x\mapsto x(k)$, be the projection onto the $k$-th coordinate.

Observe that for every $k<n-1$ we have
$$\pi_{k}\circ h_{n+1}=\pi_{k}\circ H_n\circ h_n=\pi_{k}\circ h_n,$$
which implies that the sequence of maps $h_n:K\to K^n\subset K^\w$ converges to the continuous map $h_\w:K\to K^\w$ defined by equality $\pi_k\circ h_\w=\pi_k\circ h_{k+2}$ for $k\in\w$.

We claim that the map $h_\w:K\to K^\w$ witnesses that the family $\I$ is $\w$-Fubini. Fix any $i\in\w$ and a Borel subset $B\in\I^\w_i$ in $K^\w$. Let $g_{i+2}=H_{i+2}:K^{i+2}\to K^{i+3}$ and for every $m\ge i+2$ let $g_{m+1}=H_{m+1}\circ g_m:K^{i+2}\to K^{m+2}$. Observe that for any $k<m$ we get
$$\pi_{k}\circ g_{m+1}=\pi_{k}\circ H_{m+1}\circ g_m=\pi_k\circ g_m,$$ which implies that the sequence $(g_m)_{m=i+2}^\infty$ converges to the continuous map $g_\w:K^{i+2}\to K^\w$ uniquely determined by the conditions
$$\pi_{k}\circ g_\w=\pi_k\circ g_{2+\max\{i,k\}}\mbox{ for $k\in\w$}.$$
We claim that $h_{m+1}=g_{m}\circ h_{i+2}$ for every $m\ge i+2$. This equality holds for $m=i+2$ by the definitions of $g_{i+2}:=H_{i+2}$ and $h_{i+3}:=H_{i+2}\circ h_{i+2}$. Assume that for some $m\ge i+2$ we have proved that $h_{m+1}=g_{m}\circ h_{i+2}$. Then $h_{m+2}=H_{m+1}\circ h_{m+1}=H_{m+1}\circ g_{m}\circ h_{i+2}=g_{m+1}\circ h_{i+2}$.

By the Principle of Mathematical Induction, the equality $h_{m+1}=g_{m}\circ h_{i+2}$ holds for every $m\ge i+2$. This implies that $h_\w=\lim_{m\to\infty}h_m=\lim_{m\to\infty}g_m\circ h_{i+2}=g_\w\circ h_{i+2}$. Consider the Borel set $D=g_\w^{-1}(B)\subset K^{i+2}$. We claim that $D\in\I^{i+2}_i$. Take any $a\in K^{(i+2)\setminus\{i\}}$ and consider the embedding $e_a:K\to K^{i+2}$. Since $\pi_{i}\circ g_\w=\pi_{i}\circ g_{i+2}=\pi_{i}\circ H_{i+2}=\pi_{i}$, the composition $g_\w\circ e_a:K\to K^\w$ is equal to the embedding $e_b$ for some $b\in K^{\w\setminus\{i\}}$. Then $e_a^{-1}(D)=e_a^{-1}(g_\w^{-1}(B))=(g_\w\circ e_a)^{-1}(B)=e_b^{-1}(B)\in \I$ as $B\in\I^\w_i$. This completes the proof of the inclusion $D\in\I^{i+2}_i$.

Taking into account that the map $h_{i+2}:K\to K^{i+2}$ witnesses that the family $\I$ is $(i+2)$-Fubini, we conclude that $$h_\w^{-1}(B)=(g_\w\circ h_{i+2})^{-1}(B)=h_{i+2}^{-1}(g_\w^{-1}(B))=h_{i+2}^{-1}(D)\in\I,$$which means that the map $h_\w:K\to K^\w$ witnesses that the family $\I$ is $\w$-Fubini.
\end{proof}

\begin{theorem}\label{t:Fubi}
If $\I$ is a Fubini ideal on a compact space $K$, then for any Polish group $X$ the semi-ideal $\HI$ of Haar-$\I$ sets in $X$ is an ideal.
\end{theorem}

\begin{proof} Since the ideal $\I$ is Fubini, there exists a continuous map $h:K\to K\times K$ such that for every $i\in\{0,1\}$ and any Borel set $B\in\I^2_i$ the preimage $h^{-1}(B)$ belongs to the ideal $\I$.

To show that the semi-ideal $\HI$ is an ideal, we shall prove that the union $A\cup B$ of any Haar-$\I$ sets $A,B$ in the Polish group $X$ is Haar-$\I$. By definition, the sets $A,B$ are contained in Borel Haar-$\I$ sets in $X$. So, we can assume that the Haar-$\I$ sets $A$ and $B$ are Borel. Let $f_A,f_B:K\to X$ be continuous maps such that for every $x\in X$ the sets $f_A^{-1}(A+x)$ and $f_B^{-1}(B+x)$ belong to the ideal $\I$. Consider the continuous map $f_{AB}:K\times K\to X$, $f_{AB}:(x,y)\mapsto f_A(x)+f_B(y)$ and the continuous map $f=f_{AB}\circ h:K\to X$.

It is easy to see that for every $x\in X$ the Borel set $f_{AB}^{-1}(A+x)$ belongs to the family $\I^2_0$. The choice of the map $h$ guarantees that the Borel set
$$f^{-1}(A+x)=h^{-1}(f_{AB}^{-1}(A+x))$$belongs to the ideal $\I$. By analogy we can prove that the Borel set $f^{-1}(B+x)$ belongs to the ideal $\I$. Then $f^{-1}((A\cup B)+x)=f^{-1}(A+x)\cup f^{-1}(B+x)\in\I$, which means that the map $f:K\to X$ witnesses that the union $A\cup B$ is Haar-$\I$ in $X$.
\end{proof}

\begin{theorem}\label{t:Fubi2}
If $\I$ is a Fubini $\sigma$-ideal on a zero-dimensional compact space $K$, then for any Polish group $X$ the semi-ideal $\HI$ is a $\sigma$-ideal.
\end{theorem}

\begin{proof} Let $\rho$ be a complete invariant metric generating the topology of the Polish group $X$. 
To prove that the semi-ideal $\HI$ is a $\sigma$-ideal, take a countable family $\{A_n\}_{n\in\w}$ of Haar-$\I$ Borel sets in $X$.

\begin{claim}\label{cl:ideal} For every $n\in\w$ there exists a continuous map $f_n:K\to X$ such that $\diam f_n(K)\le\frac1{2^n}$ and for every $x\in X$ the set $f_n^{-1}(A_n+x)$ belongs to the $\sigma$-ideal $\I$.
\end{claim}

\begin{proof} Since $A_n$ is Haar-$\I$, there exists a continuous map $g_n:K\to X$ such that for every $x\in X$ the set $g_n^{-1}(A_n+x)$ belongs to the $\sigma$-ideal $\I$. Since $K$ is zero-dimensional and Lindel\"of, there exists a disjoint cover $\U$ of $K$ by nonempty open sets $U\subset K$ such that $\diam g_n(U)\le\frac1{2^{n+1}}$. In each set $U\in\U$ choose a point $z_U$ and consider the continuous map $f_n:K\to X$ defined by
$f_n(x)=g_n(x)-g_n(z_U)$ for any $x\in U\in\U$. It follows that $\diam f_n(K)\le\diam\bigcup_{U\in\U}(g_n(U)-g_n(z_U))\le \frac2{2^{n+1}}=\frac1{2^n}$. Also, for every $x\in X$ we get $$f_n^{-1}(A+x)\subset \bigcup_{U\in\U}g_n^{-1}(A+x+g_n(z_U))\in\I$$since $\I$ is a $\sigma$-ideal.
\end{proof}

The choice of the maps $f_n:K\to X$ with $\diam f_n(K)\le\frac1{2^n}$ guarantees that the map $$f:K^\w\to X,\;\;f:(x_n)_{n\in\w}\mapsto \sum_{n\in\w}f_n(x_n),$$
is well-defined and continuous. By Theorem~\ref{t:Fubini}, the ideal $\I$ is $\w$-Fubini. So there exists a continuous map $h:K\to K^\w$ such that for every $n\in\w$ and every Borel set $B\in \I^\w_n$ in $K^\w$ the preimage $h^{-1}(B)$ belongs to the ideal $\I$.
We claim that the continuous map $g=f\circ h:K\to X$ witnesses that the union $A=\bigcup_{n\in\w}A_n$ is Haar-$\I$ in $X$. We need to show that for every $x\in X$ the Borel set $g^{-1}(A+x)$ belongs to the $\sigma$-ideal $\I$.

The choice of the sequence $(f_n)_{n\in\w}$ and the definition of the map $f$ ensure that for every $n\in \IN$ the Borel set $f^{-1}(A_n+x)$ belongs to the family $\I^\w_n$. By the choice of the map $h$, the preimage $h^{-1}(f^{-1}(A_n+x))=g^{-1}(A_n+x)$ belongs to $\I$. Since $\I$ is a $\sigma$-ideal, the set $g^{-1}(A+x)=\bigcup_{n\in\w}g^{-1}(A_n+x)$ belongs to $\I$.
\end{proof}

\begin{theorem}\label{t:Fubi3I}
If $\I$ is a Fubini $\sigma$-ideal on a zero-dimensional compact space $K$, then for any Polish group $X$ we have the inclusions
$$\sigma\overline{\HI}\subset \mathcal H\sigma\overline{\I}\subset\HI.$$
\end{theorem}

\begin{proof} The inclusion $\mathcal H\sigma\overline{\I}\subset\HI$ follows from the inclusion $\sigma\overline{\I}\subset\I$. To see that
$\sigma\overline{\HI}\subset\mathcal H\sigma\overline{\I}$, take any set $A\in \sigma\overline{\HI}$ in the Polish group $X$. By definition, $A\subset \bigcup_{n\in\w}F_n$ for some closed Haar-$\I$ sets $F_n\subset X$. Since $\HI$ is a $\sigma$-ideal, the countable union $F=\bigcup_{n\in\w}F_n$ is Haar-$\I$. By the definition of a Haar-$\I$ set, there exists a continuous map $f:K\to X$ such that for every $x\in X$ the $F_\sigma$-set $f^{-1}(F+x)$ belongs to the ideal $\I$, and being of type $F_\sigma$, belongs to the $\sigma$-ideal $\sigma\overline{\I}$. This means that $F$ is Haar-$\sigma{\overline{\I}}$ and so is the subset $A$ of $F$.
\end{proof}

The Kuratowski-Ulam Theorem \cite[8.41]{K} implies the following fact.

\begin{theorem}\label{t:KU} For any Polish space $K$ homeomorphic to its square $K\times K$, the $\sigma$-ideal $\M_K$ of meager subsets of $K$ is Fubini. In particular, the $\sigma$-ideal $\M$ of meager sets in the Cantor cube $2^\w$ is Fubini.
\end{theorem}

Theorems~\ref{t:Haar}(2), \ref{t:KU} and \ref{t:Fubi2} imply the following result of Darji \cite[Theorem 2.9]{D}.

\begin{corollary}[Darji]\label{c:Dar} For each Polish group the semi-ideal $\HM$ is a $\sigma$-ideal.
\end{corollary}

On the other hand, the classical Fubini Theorem \cite[p.104]{K} implies:

\begin{theorem}\label{t:Fub} For any continuous strictly positive Borel $\sigma$-additive measure $\mu$ on the Cantor cube $2^\w$, the $\sigma$-ideal $\N_\mu$ of sets of $\mu$-measure zero in $2^\w$ is Fubini. In particular, the $\sigma$-ideal $\N$ of sets of Haar measure zero in $2^\w$ is Fubini.
\end{theorem}

Theorems~\ref{t:Haar}(1), \ref{t:Fub} and \ref{t:Fubi2} imply the following result of Christensen \cite[Theorem 1]{Ch}.

\begin{corollary}[Christensen] For each Polish group the semi-ideal $\HN$ is a $\sigma$-ideal.
\end{corollary}

\begin{problem} Is the $\sigma$-ideal $\sigma\overline{\N}$ on $2^\w$ Fubini?
\end{problem}

\begin{remark}\label{rem:N} Theorems~\ref{t:Fubi3I}, \ref{t:Fub}, \ref{t:HI=HM} and Corollary~\ref{c:Dar} imply that $$\sigma\overline{\HN}\subset\mathcal H\sigma\overline{\N}\subset\HN\cap\HM\subset\HN\cap\M$$ for an uncountable Polish group.
For a locally compact Polish group some inclusions in this chain turn into equalities:
$$\sigma\overline{\HN}=\sigma\overline{\N}\subset\mathcal H\sigma\overline{\N}\subset\HN\cap\HM=\N\cap\M.$$
In Example~\ref{ex:hard} we shall observe that $\mathcal H\sigma\overline{\N}\ne \M\cap \N$ for some compact Polish group.
\end{remark}

\begin{problem} Is $\sigma\overline{\HN}=\mathcal H\sigma\overline{\N}$ for any (locally compact) Polish group?
\end{problem}

\begin{problem} Is the semi-ideal $\mathcal H\sigma\overline{\N}$ an ideal for any (locally compact) Polish group?
\end{problem}

\begin{definition} A pair $(\I,\J)$ of families $\I,\J$ of subsets of a topological space $K$ is defined to be \index{Fubini families}{\em Fubini} if there exists a continuous map $h:K\to K^2$ such that for any Borel set $A\in\I^2_0$ in $K^2$ the preimage $h^{-1}(A)$ belongs to $\I$ and for any Borel set $B\in\J^2_1$ in $K^2$ the preimage $h^{-1}(B)$ belongs to the family $\J$.
\end{definition}

Observe that a family $\I$ of subsets of $K$ is Fubini if and only if the pair $(\I,\I)$ is Fubini.

\begin{theorem}\label{t:IJ}
Let $(\I,\J)$ be a Fubini pair of semi-ideals $\I,\J$ on a topological space $K$. Then for any Polish group $X$ we get $\mathcal{H}(\I \cap \J) = \HI \cap \mc{HJ}$.
\end{theorem}

\begin{proof} The inclusion $\mathcal{H}(\I \cap \J)\subset\HI \cap \mc{HJ}$ is obvious. To see that $\HI \cap \mc{HJ}\subset \mathcal{H}(\I \cap \J)$, fix any set $A \in \HI \cap \mc{HJ}$. By definition, $A$ is contained in Borel sets $A_0\in\HI$ and $A_1\in \mc{HJ}$.

Let $f_0,f_1:K \to X$ be continuous maps such that $f_0^{-1}(A_0+x)\in\I$ and $f_1^{-1}(A_1+x)\in\J$ for all $x\in X$.
Consider the continuous map $f:K\times K\to X$, $f:(x,y)\mapsto f_0(x_0)+f_1(x_1)$, and observe that for every $x\in X$ we get $f^{-1}(A_0+x)\in\I^2_0$ and $f^{-1}(A_1+x)\in\J^2_1$ for all $x\in X$.

 Since the pair $(\I,\J)$ is Fubini, there exists a continuous map $h:K\to K\times K$ such that for any Borel sets $B_0\in\I^2_0$ and $B_1\in\J^2_1$ in $K^2$ the preimages $h^{-1}(B_0)$ and $h^{-1}(B_1)$ belong to the semi-ideals $\I$ and $\J$, respectively.

Consider the continuous map $g=f\circ h:K\to X$ and observe that
for every $x\in X$ we have
$$g^{-1}(A_0+x)=h^{-1}(f^{-1}(A_0+x))\in\I\mbox{ \ and \ } g^{-1}(A_1+x)=h^{-1}(f^{-1}(A_1+x))\in\J.$$ Consequently, for the Borel set $\tilde A=A_0\cap A_1\supset A$ in $X$ we get $$g^{-1}(A+x)\subset g^{-1}(\tilde A+x)= g^{-1}(A_0+x)\cap g^{-1}(A_1+x)\in\I\cap\J,$$
witnessing that the set $A$ belongs to the semi-ideal $\mc{H}(\I\cap\J)$.
\end{proof}

Kuratowski-Ulam Theorem \cite[8.41]{K} and Fubini Theorem~\cite[p.104]{K} imply the following fact.

\begin{theorem}\label{t:KU-F} The pair $(\M,\N)$ of the $\sigma$-ideals $\M$ and $\N$ on the Cantor cube $2^\w$ is Fubini.
\end{theorem}

This theorem combined with Theorem~\ref{t:IJ} imply:

\begin{corollary} For any Polish group we have $\HN\cap\HM=\mc{H}(\M\cap\N)$.
\end{corollary}


\section{Borel hulls of Haar-$\I$ sets in Polish groups}\label{s:hull}

In this section we will unify results obtained in papers \cite{EVid}, \cite{EVidN}, \cite{DV}, \cite{DVVR}, \cite{S} using the notion of Haar-$\I$ sets. Most of proofs  in this sections are mutatis mutandis the same as in papers mentioned above, however we put modified parts here for the sake of completness, as they are usually quite complicated. 
 
Let $\Polish$ be the category of Polish spaces and their continuous maps, and $\Set$ be the category of sets and functions between sets. Let $\Pow:\Polish\to\Set$ be the contravariant functor assigning to each set $X$ the family $\Pow(X)$ of subsets of $X$ and to each continuous function $f:X\to Y$ between Polish spaces the function $\Pow(f):\Pow(Y)\to\Pow(X)$, $\Pow(f):A\mapsto f^{-1}(A)$. 

By a \index{pointclass}{\em pointclass} we understand any subfunctor $\Gamma$ of the contravariant functor $\Pow$ such that for any Polish space $X$ and sets $A,B\in\Gamma(X)$ we have $\{\emptyset,X\}\subset\Gamma(X)$, $A\cap B\in\Gamma(X)$, and $A\cup B\in\Gamma(X)$. For a pointclass $\Gamma$ let $\Gamma^c$ be the pointclass assigning to each Polish space $X$ the family $\Gamma^c(X)=\{A\in\Pow(X):X\setminus A\in\Gamma(X)\}$.

By \index{$\Delta^1_1$}\index{$\Sigma^1_1$}\index{$\Pi^1_1$}${\Delta^1_1}$, ${\Sigma^1_1}$ and ${\Pi^1_1}$ we denote the pointclasses assigning to each Polish space $X$ the families $\Delta^1_1(X)$, $\Sigma^1_1(X)$ and $\Pi^1_1(X)$ of Borel, analytic and coanalytic subsets of $X$, respectively.

\begin{definition} Let $\Gamma,\Lambda$ be two pointclasses. We say that a family $\I$ of subsets of a Polish space $Z$ is \index{family!$\Lambda$-on-$\Gamma$}\index{$\Lambda$-on-$\Gamma$ family}{\em $\Lambda$-on-$\Gamma$}, if for any Polish space $Y$ and set $A \in \Gamma(Z\times Y)$, we have  $\{ y \in Y : A_y\in \I \} \in \Lambda(Y)$, where $A_y:=\{z\in Z:(z,y)\in A\}$ is the $y$-section of the set $A$. 
\end{definition}

Families which are $\Delta^1_1$-on-$\Delta^1_1$ will be called \index{family!Borel-on-Borel}\index{Borel-on-Borel family}{\em Borel-on-Borel}.

In the following theorem we collect some known facts on the $\Lambda$-on-$\Gamma$ properties of the ideals $\M$, $\mathcal N$, and $\M\cap\mathcal N$ on the Cantor set $2^\w$.

\begin{theorem}\label{t:MN-on} The ideals $\M$ and $\mathcal N$ are $\Pi^1_1$-on-$\Sigma^1_1$,   $\Sigma^1_1$-on-$\Pi^1_1$, and $\Delta^1_1$-on-$\Delta^1_1$.
\end{theorem}

\begin{proof} The $\Pi^1_1$-on-$\Sigma^1_1$ property of the ideals $\mathcal M$ and  $\mathcal N$ follow from \cite[29.22]{K} and \cite[29.26]{K}, respectively. The $\Sigma^1_1$-on-$\Pi^1_1$ property of the ideals $\mathcal M$ and $\mathcal N$ follows from \cite[36.24]{K}. Observe that for any Polish space $Y$ and any subset $A\subset 2^\w\times Y$ we have the equality
$$\{y\in Y:A_y\in\M\cap\mathcal N\}=\{y\in Y:A_y\in\M\}\cap\{y\in Y:A_y\in\mathcal N\}.$$Combining this equality with the $\Sigma^1_1$-on-$\Pi^1_1$ and $\Pi^1_1$-on-$\Sigma_1^1$ properties of the ideals $\M$, $\mathcal N$, we conclude that the ideal $\M\cap\mathcal N$ is  $\Sigma^1_1$-on-$\Pi^1_1$ and $\Pi^1_1$-on-$\Sigma_1^1$.

By the famous Souslin Theorem (see, \cite[14.11]{K}), $\Delta^1_1(Y)=\Sigma^1_1(Y)\cap\Pi^1_1(Y)$ for any Polish space $Y$. This equality and the  $\Sigma^1_1$-on-$\Pi^1_1$ and $\Pi^1_1$-on-$\Sigma^1_1$ properties of the  ideals $\mathcal M$, $\mathcal N$, $\M\cap\mathcal N$ imply that these ideals are $\Delta^1_1$-on-$\Delta^1_1$.
\end{proof}


We start by proving a common generalization of Theorem 4.1 \cite{EVid} and Theorem 10 \cite{DV} treating Haar-null sets and Haar-meager sets, respectively.

\begin{definition} Let $\Lambda,\Gamma$ be two pointclasses, and $\I$ be an semi-ideal on a zero-dimensional compact metrizable space $Z=\bigcup\I$. We say that the triple $(\I, \Gamma,\Lambda)$ has the property $\textbf{(H)}$ if the following conditions are satisfied:
\begin{itemize}
\item[\bf{(H1)}] For any $A\in\I$ the set $Z\setminus A$ does not belong to $\I$;
\item[\bf(H2)] For any Polish space $P$ there exists a set $U \in \Gamma(2^\w\times P)$, which is $\Lambda^c(P)$-universal in the sense that for every $A \in \Lambda^c(P)$ there exists $x \in 2^\w$ such that $U_x = A$;
\item[\bf(H3)] $\I$ is $\Gamma^c$-on-$\Gamma$;
\item[\bf(H4)] If a set $A \in \Gamma(C(Z,X) \times Y \times X)$ is such that for every $(h,y) \in C(Z,X)\times Y$ the section $A_{(h,y)}$ is either empty or $h^{-1}(A_{h,y}) \notin \I$, then the set $A$ has a uniformization $\phi\in \Gamma(C(Z,X)\times Y\times X)$, (which means  such that $\phi\subset A$ and for any $(h,y) \in C(Z,X) \times Y$ with $A_{(h,y)}\ne\emptyset$ the section $\phi_{(h,y)}$ is a singleton);
\item[\bf(H5)] For any Borel injection $f:X \to Y$ between Polish spaces $\{f(A):A\in\Gamma(X)\}\subset\Gamma(Y)$.
\end{itemize}
\end{definition}

\begin{theorem}\label{brakpowlok} Let $\I$ be a semi-ideal on a  compact metrizable space $Z=\bigcup\I$ and $\Gamma,\Lambda$ are two pointclasses such that the triple $(\I, \Gamma,\Lambda)$ has the property $\mathbf{(H)}$. Let $X$ be a Polish group containing a topological copy of the compact space $Z$. If $X$ is not locally compact, then there exists a Haar-$1$ (and hence Haar-$\I$) set $E \in \Gamma(X)$, which cannot be enlarged to a Haar-$\I$ set $H\in \Lambda(X)$.  
\end{theorem}

\begin{proof} The proof of this theorem is based on two lemmas.
The first one is a counterpart of \cite[Theorem 3.1]{EVid} and \cite[Proposition 7]{DV}.

\begin{lemma}\label{budulec}
There exists a set $\phi\in \Gamma(C(Z,X)\times 2^\w\times X)$ such that
\begin{itemize}
\item[\textup{1)}] for any triples $(h,x,y),(h,x,y')\in\phi$ we have $y=y'$;
\item[\textup{2)}] for any triple $(h,x,y) \in \phi$ we have $y \in h(Z)$;
\item[\textup{3)}] for any $h\in C(Z,X)$ and any set $S \in \Lambda(2^\w \times X)$ containing the section $\phi_h$, there exists $x\in 2^\w$ such that $h^{-1}(S_x) \notin \I$.
\end{itemize}
\end{lemma}

\begin{proof}Consider the closed set 
$$F:=\{(h,x,g)\in C(Z,X)\times 2^\w\times X:g\in h(Z)\}$$ in the Polish space $C(Z,X)\times 2^\w\times X$ and observe that $F\in\Gamma(F)$ by the definition of a pointclass. Since the identity inclusion $F\to C(Z,X)\times 2^\w\times X$ is a continuous injection, the set $F$ belongs to the class $\Gamma(C(Z,X)\times 2^\w\times X)$ according to the condition $\mathbf{(H5)}$.

By the condition $\mathbf{(H2)}$, there exists a $\Lambda^c(2^\w\times X)$-universal set $U \in \Gamma(2^\w \times 2^\w \times X)$. Consider the set
$$
U':=\{(h,x,g) \in F: (x,x,g) \in U \mbox{ and } h^{-1}(U_{(x,x)}) \notin \I \}.
$$

\begin{claim} $U'\in \Gamma(C(Z,X)\times 2^\w\times X)$.
\end{claim}

\begin{proof} Observe that $U'=F\cap U'_1\cap U'_2$ where
$$
\begin{aligned}
U'_1&=\{(h,x,g) \in C(K ,X) \times 2^\w \times X : (x,x,g) \in U\} \mbox{ and }\\
U'_2&=\{(h,x,g) \in C(K ,X) \times 2^\w \times X : h^{-1}(U_{(x,x)}) \notin \I \}.
\end{aligned}
$$
The continuity of the map $\xi_1:C(Z,X)\times 2^\w\times X\to 2^\w\times 2^\w\times X$, $\xi_1:(h,x,g)\mapsto (x,x,g)$, and definition of a pointclass as a subfunctor of the contravariant functor $\Pow$ ensure that the set $U'_1=\xi_1^{-1}(U)$ belongs to the class $\Gamma(C(Z,X)\times 2^\w\times X)$.
 
The continuity of the map $Z\times C(Z,X)\times 2^\w\to 2^\w\times 2^\w\times X$, $(z,h,x) \mapsto (x,x,h(z))$, and the definition of a pointclass ensure that the set $\ddot U:=\{ (z,h,x)\in Z\times C(Z,X)\times 2^\w: (x,x,h(z)) \in U\}$ belongs to the class $\Gamma(Z\times C(K, X)\times 2^\w)$. Since the semi-ideal $\I$ is $\Gamma^c$-on-$\Gamma$, the set $$L:=\{(h,x)\in C(Z,X)\times 2^\w:\ddot U_{(h,x)}\notin \I\}$$ belongs to the class $\Gamma(C(Z,X)\times 2^\w)$. Observe that for every $(h,x)\in C(Z,X)\times 2^\w$ we have 
\begin{multline*}
h^{-1}(U_{(x,x)})=\{z\in Z:h(z)\in U_{(x,x)}\}=\{z\in Z:(x,x,h(z))\in U\}\\
=\{z\in Z:(z,h,x)\in\ddot U\}=\ddot U_{(h,x)}
\end{multline*}
and hence
$$\{(h,x)\in C(Z,X)\times 2^\w:h^{-1}(U_{(x,x)})\notin\I\}=L\in \Gamma(C(Z,X)\times 2^\w).$$
The definition of a pointclass  implies that $$
U'_2=\{(h,x,g)\in C(Z,X)\times 2^\w\times X:h^{-1}(U_{(x,x)})\notin\I\}=L\times X\in\Gamma(C(Z,X)\times 2^\w\times X)$$and $U'=F\cap U'_1\cap U'_2\in\Gamma(C(Z,X)\times 2^\w\times X)$.
\end{proof}
We claim that for every $(h,x)\in C(Z,X)\times 2^\w$ the section $U'_{(h,x)}:=\{g\in X:(h,x,g)\in U'\}$ is either empty or $U'_{(h,x)}\notin\I$. Indeed, if $U'_{(h,x)}$ contains some element $g\in X$, then $(h,x,g)\in U'$,  $(x,x,g)\in U$ and $h^{-1}(U_{(x,x)})\notin \I$. Observe that for every $y\in h^{-1}(U_{(x,x)})$ we get $(h,x,y)\in U'$ and hence $y\in U'_{(h,x)}$. Therefore, $h^{-1}(U_{(x,x)})\subset U'_{(h,x)}$ and  then $h^{-1}(U_{(x,x)})\notin\I$ implies $U'_{(h,x)}\notin\I$.

By the condition $\textbf{(H4)}$, the set $U'$ has a uniformization $\phi\in \Gamma(C(Z,X)\times 2^\w\times X)$. Since $\phi$ is a partial function, the condition (1) of Lemma~\ref{budulec} is satisfied. The condition (2) follows from the inclusion $\phi\subset U'\subset F$.

It remains to check the condition (3). Assume towards a contradiction that there exists $h \in C(Z,X)$ and $S \in \Lambda(2^\w \times X)$ such that $\phi_h \subset S$ and $h^{-1}(S_x) \in \I$ for every $x \in 2^\w$. Let $B=(2^\w \times X)\setminus S$ and observe that for every $x\in 2^\w$ we have $B_x\cup S_x=X$. Since $h^{-1}(S_x)\in\I$, the condition $\mathbf{(H1)}$ implies that $h^{-1}(B_x)\notin\I$.

By the $\Lambda^c(2^\w\times X)$-universality of the set $U\in\Gamma(2^\w\times(2^\w\times X))$, there exists $x\in 2^\w$ such that $U_x=B$. Then $U_{(x,x)}=B_x$ and hence $h^{-1}(U_{(x,x)})=h^{-1}(B_x)\notin \I$. Then $(h,x,h(x))\in U'$ and for a unique triple $(h,x,g)\in\phi\subset U'$  we have $g\in\phi_{(h,x)}\subset S_x$.
On the other hand, $g\in \phi_{(h,x)}\subset U'_{(h,x)}\subset U_{(x,x)}=B_x$, which is not possible as $B_x\cap S_x=\emptyset$. 
\end{proof}

The second ingredient of the proof of Theorem~\ref{brakpowlok} is the following lemma whose proof can be found in \cite[Proposition 3.5]{EVid}.

\begin{lemma}\label{rozsuwaniezwartych} For any non-locally compact Polish group $X$ and any nonempty compact set $C\subset X$, there exists a Borel map $t:\mc{K}(X) \times 2^\omega \times 2^\omega \to X$ such that
\begin{itemize}
\item[\textup{1)}] for any distinct triples $(K,x,y), (K',x',y')\in \mc{K}(X) \times 2^\omega \times 2^\omega$, we have 
$$
\left( K-C+t(K,x,y) \right) \cap \left( K'-C+t(K',x',y') \right)= \emptyset;
$$
\item[\textup{2)}] for every $K \in \mc{K}(X)$ and $y \in 2^\omega$ the map $t(K, \cdot , y):2^\w\to X$, $t(K,\cdot,y):x\mapsto t(K,x,y)$, is continuous.
\end{itemize}
\end{lemma}

No we are ready to present {\em the proof of Theorem \ref{brakpowlok}}. By the assumption of the theorem, there exists a continuous injection $h_0:Z\to X$ such that $\theta\in h_0(Z)$. Fix a Borel injection $c:C(K,X) \to 2^\w$ from the Polish space $C(K,X)$ to the Cantor cube. Let $t$ be a function obtained by Lemma \ref{rozsuwaniezwartych} for $C=h_0(Z)$. Define a function $\Psi:C(K,X)\times 2^\w \times X \to X$ by the formula $$\Psi(h,x,g)= g+t(h(K),x,c(h))$$and put  $E:= \Psi(\phi)$ where $\phi\in\Gamma(C(K,X)\times 2^\w\times X)$ is the set given by Lemma~\ref{budulec}.

\begin{claim}
$E \in \Gamma$.
\end{claim}

\begin{proof} It is easy to see that the function $\Psi$ is Borel. We claim that it is injective on the closed subset $$F=\{(h,x,g)\in C(Z,X)\times 2^\w \times X: g \in h(Z)\}$$ of $C(Z,X)\times 2^\w\times X$. We need to check that $\Psi(h,x,g)\ne\Psi(h',x',g')$ for any distinct triples  $(h,x,g),(h',x',g')\in F$. The case $(h,x)=(h',x')$ is obvious. If $(h,x)\neq (h',x')$, then the injectivity of the function $c$ ensures that $(h(Z),x,c(h))\ne (h'(Z),x',c(h'))$ and then Lemma~\ref{rozsuwaniezwartych}(1) ensures that $\Psi(h,x,g)\ne \Psi(h',x',g')$ as $\Psi(h,x,g)=g +t(h(Z),x,c(h)) \subset h(Z)- C +t(h(Z),x,c(h))$ (recall that $g \in h(Z)$ and $\theta \in C$). Since $\phi \subset U'\subset F$, the definition of a pointclass implies $\phi\in \Gamma(F)$. Applying the condition $\textbf{(H5)}$, we obtain $E \in \Gamma(X)$. 
\end{proof}

\begin{claim}
$E$ is Haar-$1$.
\end{claim}

\begin{proof} It suffices to prove that $|h_0^{-1}(E+g)| \leq 1$ for any $g  \in X$, or equivalently (since $h_0$ is injective)  $|(C+g) \cap E| \leq 1$. To derive a contradiction, assume that for some $g \in X$ the intersection $(C+g)\cap E$ contains two distinct elements. Then we can find two distinct triples $(h,x,y),(h',x',y')\in\phi$ such that $\Psi(h,x,y)=y+t(h(Z),x,c(h))$ and $\Psi(h',x',y')=y'+t(h'(Z),x',c(h'))$ are two points of the set $C+g$. It follows that
\begin{multline*}
g\in \big(y-C+t(h(Z),x,c(h))\big)\cap\big(y'-C+t(h'(Z),x',c(h')\big)\\ 
\subseteq\big(h(Z)-C+t(h(Z),x,c(h))\big)\cap \big(h'(Z)-C+t(h'(Z),x',c(h'))\big).
\end{multline*}
Lemma~\ref{rozsuwaniezwartych}(1) and the injectivity of the function $c$ ensure that $x=x'$ and $h=h'$. Now Lemma~\ref{budulec}(1) implies that $y=y'$ and hence $(h,x,y)=(h',x',y')$,  which contradicts the choice of these two triples.  
\end{proof}

\begin{claim}
There is no Haar-$\I$ set $H \in \Lambda(X)$ containing $E$.
\end{claim}

\begin{proof}
Suppose $H \in \Lambda(X)$ is  a Haar-$\I$ set containing the set $E$. Choose a function $h \in C(Z,X)$ witnessing that $H$ is Haar-$\I$ in $X$.  Lemma~\ref{rozsuwaniezwartych}(2) and the definition of the function $\Psi$ imply that the map $\Psi_h:2^\w\times X\to X$, $\Psi_h:(x,g)\mapsto g+t(h(Z),x,c(h))$, is continuous, so by the definition of a pointsclass, the preimage $S:=\Psi_h^{-1}(H)$ belongs to the class $\Lambda(2^\w \times X)$.

Since $\phi\subseteq \Psi^{-1}(E) \subseteq \Psi^{-1}(H)$, we see that $\phi_h \subset S$, and therefore, by Lemma~\ref{budulec}$(3)$, there exists $x \in 2^\w$ such that  $h^{-1}(S_x) \notin \I$. By the definition of $S$ we have that $\Psi(\{(h,x)\}\times S_x) \subset \Psi_h(S) \subset H$. But $\Psi(\{(h,x)\}\times S_x) = S_x + t(h(Z),x,c(h))$ is just a translation of $S_x$, so a translate of $H$ contains $S_x$ and $h^{-1}(S_x) \notin \I$, contradicting that $h$ witnesses that $H$ is Haar-$\I$. 
\end{proof}
\end{proof}

We shall deduce from Theorem~\ref{brakpowlok} the following theorem.

\begin{theorem}\label{t:NoBorelHulls} Let $\I$ be a $\sigma$-ideal on a  compact metrizable space $Z=\bigcup\I\notin\I$. If $\I$ is Borel-on-Borel, then for every countable ordinal $\alpha$, the triple $(\I,\Delta^1_1,\Pi^0_\alpha)$ has the property $\mathbf{(H)}$. Consequently, for every non-locally compact Polish group $X$ containing a topological copy of $Z$, there exists a Borel Haar-1 subset $E\subset X$ which cannot be enlarged to a Haar $\I$-set $H\in\Pi^0_\alpha(X)$.
\end{theorem}

\begin{proof} First we check that the triple $(\I,\Delta^1_1,\Pi^0_\alpha)$ has the property $\mathbf{(H)}$. The condition $\mathbf{(H1)}$ holds since $Z\notin\I$. The condition $\mathbf{(H2)}$ follows from \cite[Theorem 22.3]{K}. The condition $\mathbf{(H3)}$ holds since the ideal $\I$ is Borel-on-Borel by the assumption. The condition $\mathbf{(H5)}$ follows from \cite[15.2]{K}.

It remains to check the condition $\mathbf{(H4)}$. Fix a Borel subset $A\subset C(Z,X)\times 2^\w\times X$ such that for every $(h,x)\in C(Z,X)\times 2^\w$ either  $A_{(h,x)}=\emptyset$ or $h^{-1}(A_{(h,x)})\notin\I$. Consider the continuous map $$\xi:Z\times C(Z,X)\times 2^\w\to C(Z,X)\times 2^\w\times X,\;\;\xi:(z,h,x)\mapsto (h,x,h(x)),$$ and observe that $A'=\xi^{-1}(A)$ is a Borel subset of $Z\times C(Z,X)\times 2^\w$ such that for every $(h,x)\in C(Z,X)\times 2^\w$ the set $A'_{(h,x)}=h^{-1}(A_{(h,x)})$ is either empty of $A'_{(h,x)}\notin\I$. By \cite[Theorem 18.6]{K}, the Borel set $A'$ has a Borel uniformization $\phi\subset A'$. We claim that the function $\xi{\restriction}\phi$ is injective.  Indeed, if $(z,h,x),(z',h',x')$ are distinct points of $\phi$, then $(h,x)\ne (h',x')$ and hence $\xi(z,h,x)=(h,x,h(x))\ne (h',x',h'(x'))=\xi(z',h',z')$. Now Theorem 15.2 in \cite{K} implies that $B:=\xi(\phi)$ is a Borel subset of $C(Z,X)\times 2^\w\times X$. It is clear that $B=\xi(\phi)\subset \xi(A')\subset A$. For any $(h,x)\in C(Z,X)\times 2^\w$ with $A_{(h,x)}\ne\emptyset$, we get $A'_{(h,x)}=h^{-1}(A_{(h,x)})\notin\I$ and hence $A'_{(h,x)}\ne\emptyset$ and $\phi_{(h,x)}$ coincides with a singleton $\{(z,h,x)\}$. Then $B_{(h,x)}=\{(h,x,h(z))\}$ is a singleton, too.
Therefore, the triple $(\I,\Delta^1_1,\Pi^0_\alpha)$ has the property $\mathbf{(H)}$ and by Theorem~\ref{brakpowlok}, every non-locally compact Polish Abelian group $X$ contains  a Borel Haar-1 set $E\subset X$ that cannot be enlarged to  Haar-$\I$ set $H\in\Pi^0_\alpha(X)$.
\end{proof}  

By Theorem~\ref{t:MN-on}, the ideals $\M$ and $\mathcal N$ on $2^\w$ are $\Delta^1_1$-on-$\Delta^1_1$. So, we can apply Theorem~\ref{t:NoBorelHulls} and obtain the following corollary, first proved in \cite{DV} and \cite{EVid}.

\begin{corollary}\label{c:hull-MN} Let $\alpha$ be a countable ordinal and $\I$ be one of the $\sigma$-ideals $\M$ or $\mathcal N$ on $2^\w$. Every non-locally compact Polish group $X$  contains a Borel Haar-1 set that cannot be enlarged to a Haar-$\I$ set of Borel class $\Pi^0_\alpha(X)$.
\end{corollary}

Similar results are true for coanalytic sets. 

\begin{theorem} Let $\I$ be one of the $\sigma$-ideals $\M$ or $\mathcal N$ on $2^\w$. Then the triple $(\I,\Pi^1_1,\Sigma^1_1)$ has the property $\mathbf{(H)}$. Consequently, every non-locally compact Polish group $X$ contains a coanalytic Haar-$1$ set that cannot be enlarged to an analytic Haar-$\I$ set in $X$.
\end{theorem}

\begin{proof} The conditions  $\mathbf {(H1)}$ for the triple $(\I,\Pi^1_1,\Sigma^1_1)$ trivially hold. The condition $\mathbf{(H2)}$ follows from \cite[Theorem 26.1]{K} and $\mathbf{(H3)}$ from Theorem~\ref{t:MN-on}, the condition $\mathbf{(H4)}$ can be established by the arguments presented \cite[\S36.F]{K}. The condition $\mathbf{(H5)}$ follow from well-known properties of coanalytic and analytic sets, see \cite[\S25.A]{K}. Therefore, the triple $(\I,\Pi^1_1,\Sigma^1_1)$ has the property $\mathbf{(H)}$ and by Theorem~\ref{brakpowlok}, every non-locally compact Polish group $X$  contains a coanalytic Haar-1 set that cannot be enlarged to an analytic Haar-$\I$ subset of $X$.
\end{proof}

On the other hand, for analytic Haar-$\I$ sets, we have the following theorem generalizing \cite[Proposition (i)]{S} and \cite[Proposition 8]{DVVR}.

\begin{theorem}\label{t:AB-hull} Let $\I$ be a semi-ideal on a compact metrizable space $Z=\bigcup\I\notin \I$. If $\I$ is $\Pi_1^1$-on-$\Sigma^1_1$, then every  analytic Haar-$\I$ set in a non-locally Polish group $X$  can be enlarged to a Borel  Haar-$\I$ set in $X$. 
\end{theorem}

\begin{proof}
Given an analytic Haar-$\I$ set $H\subset X$, find $h\in C(Z,X)$ such that $h^{-1}(g+H)\in\I$ for all $g\in X$. Consider the family $$\A=\{\Sigma^1_1(X):\forall g\in X\;\;h^{-1}(g+A)\in\I\}$$ of analytic Haar-$\I$ sets in $X$.

We claim that $\A$ is $\Pi^1_1$-on-$\Sigma^1_1$. Given any Polish space $Y$ and an analytic set $A\subset X\times Y$, we need to show that the set $\{y\in Y:A_y\in\A\}$ is coanalytic in $Y$. Consider the continuous map $$F:X\times Y\times Z\to X\times Y,\;\;F:(g,y,z)\mapsto(h(z)-g,y).$$ By \cite[14.4]{K}, the preimage $\tilde A=F^{-1}(A)$ is an analytic set in $X\times Y\times Z$. Observe that for every $(g,y)\in X\times Y$ we have
$$
\begin{aligned}
\tilde A_{(g,y)}=\{z\in Z:(g,y,z)\in\tilde A\}&=\{z\in Z:(h(z)-g,y)\in A\}\\
&=\{z\in Z:h(z)\in g+A_y\}=h^{-1}(g+A_y).
\end{aligned}
$$ 
Then  
$$L:=\{(g,y)\in X\times Y:h^{-1}(g+A_y)\notin \I\}=\{(g,y)\in X\times Y:\tilde A_{(g,y)}\notin\I\}\in \Sigma^1_1(X\times Y)$$ as the family $\I$ is $\Pi^1_1$-on-$\Sigma^1_1$. Consider the projection $\pr_Y:X\times Y\to Y$ and observe that the image $\pr_Y(L)$ is an analytic set and its complement
$$Y\setminus\pr_Y(L)=\{y\in Y:\forall g\in X\;\;h^{-1}(g+A_y)\in\I\}=\{y\in Y:A_y\in\A\}$$is coanalytic in $Y$, witnessing that the family $\A$  is $\Pi^1_1$-on-$\Sigma_1^1$. 

Now, since $H\in\A$, by (the dual form of) the First Reflection Theorem (see
\cite[Theorem 35.10]{K} and the remarks following it]), the set $H$ can be enlarged to some Borel set $B\in\A$, which is Haar-$\I$ by the definition of the family $\A$.
\end{proof}

Combining Theorems~\ref{t:MN-on} and Theorem~\ref{t:AB-hull}, we obtain the following corollary first proved in \cite[Proposition (i)]{S} and \cite[Proposition 8]{DVVR}.

\begin{corollary} Let $\I$ be one of the $\sigma$-ideals $\M$ or $\mathcal N$ on the Cantor cube $2^\w$. Every  analytic Haar-$\I$ set in a Polish group $X$ is contained in a Borel  Haar-$\I$ subset of $X$. 
\end{corollary}

\section{Cardinal characteristics of the semi-ideals $\HI$}\label{s:cc}

In this section given a semi-ideal $\I$ on a compact Hausdorff space $K\notin\I$ we evaluate the cardinal characteristics $\add(\HI)$ and $\cof(\HI)$ of the semi-ideal $\HI$ on a non-locally compact Polish group.

First observe that Theorem~\ref{t:NoBorelHulls} implies the following upper bound for the additivity of the semi-ideals $\HI$.

\begin{corollary} For any Borel-on-Borel $\sigma$-ideal $\I$ with $\bigcup\I = 2^\w\notin\I$, the semi-ideal $\HI$ on any non-locally compact Polish group $X$ has $\on{add}(\HI)\leq \omega_1$.
\end{corollary}

\begin{remark} By \cite{EVid}, \cite{DV}, \cite{EP}, the $\sigma$-ideals $\HN$ and $\HM$ on any non-locally compact Polish group have
$$\on{add}(\HN)=\w_1=\on{add}(\HM)\mbox{ \ and \ }\on{cof}(\HN)=\mathfrak c=\on{cof}(\HM).$$
In fact, the proof of the  equality $\on{cof}(\HN)=\mathfrak c=\on{cof}(\HM)$ presented in \cite{EP} works also for many semi-ideals $\HI$.
\end{remark}

First we detect semi-ideals $\I$ for which the semi-ideal $\HI$ has uncountable cofinality 

\begin{proposition}\label{p:cof} Let $\I$ be a semi-ideal on a compact Hausdorff space $K\notin\I$ such that $\I$ contains all finite subsets of $K$. For any non-compact Polish group $X$ admitting a continuous injective map $K\to X$, the semi-ideal $\HI$ has uncountable cofinality $\cof(\HI)$.
\end{proposition}

\begin{proof} To derive a contradiction, assume that $\cof(\HI)\le\w$ and find a countable cofinal subfamily $\{B_n\}_{n\in\w}\subset\HI$. For every $n\in\w$ find a continuous map $f_n:K\to X$ such that $f_n^{-1}(B_n+x)\in\I$ for all $x\in X$. Taking into account that $K\notin\I$, we conclude that $f_n^{-1}(B_n+x)\ne K$ for every $x\in X$ and hence $f_n(K)\not\subset B_n+x$.

Fix a complete invariant metric $\rho$ generating the topology of the Polish group. 
Since the Polish group $X$ is not compact, the metric $\rho$ is not totally bounded. Consequently, there exists $\e>0$ such that for any compact subset $C\subset X$ there exists $x\in X$ with $\inf_{c\in C}\rho(x,c)\ge \e$.  

For every $n\in\w$ we shall inductively choose a point $x_n\in X\setminus B_n$ such that $\rho(x_n,x_k)\ge \e$ for all $k<n$. To start the inductive construction, choose any point $x_0\in X\setminus B_0$. Assume that for some $n\in\IN$ the points $x_0,\dots,x_{n-1}$ have been chosen. The choice of $\e$ yields a point $y_n\in X$ such that $\rho(y_n,x)\ge \e$ for all $x\in \bigcup_{k<n}(x_k-f_n(K))$. By the choice of the map $f_n$, there exists a point $x_n\in (y_n+f_n(K))\setminus B_n$. Observe that for every $k<n$ we have $x_n-y_n\in f_n(K)$ and by the invariance of the metric $\rho$,
$$\rho(x_n,x_k)=\rho(y_n+(x_n-y_n),x_k)=\rho(y_n,x_k-(x_n-y_n))  \ge\e.$$
This completes the inductive step.

After completing the inductive construction, consider the closed discrete subset  $D=\{x_n\}_{n\in\w}$ of $X$. By our assumption, the group $X$ admits a continuous injective map $f:K\to X$. Taking into account that the set $D$ is closed and discrete in $X$, we conclude that for any $x\in X$ the intersection $f(K)\cap (D+x)$ is finite and so is the set $f^{-1}(D+x)$. Since the semi-ideal $\I$ contains all finite subsets of $K$, the map $f$ witnesses that the set $D$ is Haar-$\I$ in $X$. By the cofinality of the family $\{B_n\}_{n\in\w}$ in $\HI$, there exists $n\in\w$ such that $D\subset B_n$, which contradicts the choice of the point $x_n\in D\setminus B_n$. This contradiction finishes the proof of the strict inequality $\cof(\HI)>\w$.
\end{proof}

\begin{theorem} Let $\I$ be a semi-ideal on $2^\w\notin\I$ containing the family $\overline{\mathcal N}$ of closed subsets of Haar measure zero in $2^\w$. For any non-locally compact Polish group $X$ the semi-ideal $\HI$ has cofinality $\on{cof}(\HI)=\mathfrak c$.
\end{theorem}

\begin{proof} Taking into account that each subset $A\in\HI$ can be enlarged to a Borel set $B\in\HI$, we conclude that $\on{cof}(\HI)\le\mathfrak c$. By Proposition~\ref{p:cof}, $\cof(\HI)>\w$. If $\mathfrak c=\w_1$, then the inequalities $\w<\cof(\HI)\le\mathfrak c$ imply the equality $\cof(\HI)=\mathfrak c$. It remains to prove the inequality $\cof(\HI)\ge\mathfrak c$ under the assumption $\w_1<\mathfrak c$.

In this case we shall apply Theorem 2.10 of \cite{EP}. In this theorem Elekes and Po\'or constructed a Borel map $\varphi:2^\w\to\F(X)$ to the hyperspace of closed subsets of $X$ endowed with the Effros-Borel structure such that
\begin{itemize}
\item[(i)] $\varphi(x)\in \overline{\HN}$ for every $x\in 2^\w$ and
\item[(ii)] for any compact uncountable set $K\subset 2^\w$ the union $\bigcup_{x\in K}\varphi(x)$ is thick in $X$.
\end{itemize}
The inclusion $\overline{\mathcal N}\subset\I$ and Theorem~\ref{t:Haar}(1) imply $\overline{\HN}\subset\HI$. Therefore, $\varphi(x)\in\HI$ for any $x\in 2^\w$. Assuming that $\on{cof}(\HI)<\mathfrak c$, we could find a family $\mathcal B\subset\HI$ of Borel subsets of $X$ with $|\mathcal B|=\on{cof}(\HI)<\mathfrak c$ such that every Haar-$\I$ subset of $X$ is contained in some set $B\in\mathcal B$. In particular, for every $x\in 2^\w$ the closed Haar-$\I$ set $\varphi(x)$ is contained in some set $B\in\mathcal B$. Since $\w_1<\mathfrak c=|2^\w|$, by the Pigeonhole Principle, for some $B\in\mathcal B$ the set $\Phi_B:=\{x\in 2^\w:\varphi(x)\subset B\}$ has cardinality $|\Phi_B|>\w_1$. By Lemma 2.12 \cite{EP}, the set $\Phi_B$ is coanalytic. Since each coanalytic set in a Polish space is the union of $\w_1$ many Borel sets \cite[34.5]{K}, the coanalytic set $\Phi_B$ contains an uncountable Borel subset, which contains an uncountable compact subset by \cite[29.1]{K}. Now the property (ii) of the map $\varphi$ ensures that the set $B\supset \bigcup_{x\in\Phi_B}\varphi(x)$ is thick in $X$ and hence cannot be Haar-$\I$ in $X$.
\end{proof}  

\begin{problem} Given a semi-ideal $\I$ with $\bigcup\I = 2^\w\notin\I$, evaluate the cardinal characteristics $\on{cov}(\HI)$ and $\on{non}(\HI)$ of the $\sigma$-ideal $\I$ on a Polish group $X$.
\end{problem}

\begin{remark} The cardinal characteristics of the $\sigma$-ideals $\HN$ and $\HM$ on the Polish group $X=\IZ^\w$ were calculated in \cite{EP} (see also \cite{Ban}): 
$$\on{add}(\HN){=}\w_1,\,\on{cov}(\HN){=}\min\{\mathfrak b,\on{cov}(\mathcal N)\},\,\on{non}(\HN){=}\max\{\mathfrak d,\on{non}(\mathcal N)\},\,\on{cof}(\HN){=}\mathfrak c,$$
$$\on{add}(\HM)=\w_1,\;\;\on{cov}(\HM)=\on{cov}(\mathcal M),\;\;\on{non}(\HM)=\on{non}(\mathcal M),\;\;\on{cof}(\HM)=\mathfrak c.$$
\end{remark}

\section{Generically Haar-$\I$ sets in Polish groups}\label{s12}

Let $X$ be a Polish group. It is well-known that for a compact metrizable space $K$ the group $\C(K,X)$ of all continuous functions from $K$ to $X$ is Polish with respect to the compact-open topology which is generated by the sup-metric
$$\hat\rho(f,g)=\sup_{x\in K}\rho(f(x),g(x)),$$where $\rho$ is any complete invariant metric generating the topology of the Polish group $X$.

Let $\F(K,X)$ be the subspace of $\C(K,X)$ consisting of functions $f:K\to X$ with finite image $f(K)$. It is clear that $\F(K,X)$ is a subgroup of the topological group $\C(K,X)$. 

For a semi-ideal $\I$ on a compact metrizable space $K$ and a subset $A$ of a Polish group $X$,
consider the subset
$$W_{\I}(A)=\{f\in \C(K,X):\forall x\in X\;\;f^{-1}(A+x)\in\I\}\subset \C(K,X),$$
called the {\em witness set} for $A$. Observe that $W_\I(A)\ne\emptyset$ if and only if $A$ is Haar-$\I$.

\begin{proposition}\label{p:sumD} For any ideal $\I$ on a compact metrizable space $K$ and any subset $A$ of a Polish group $X$ we have the equality
$W_\I(A)=W_\I(A)+\F(K,X)$.
\end{proposition}

\begin{proof} Given two functions $w\in W_\I(A)$ and $f\in \F(K,X)$, we should prove that the function $g:=w+f$ belongs to $W_\I(A)$. Observe that $\U=\{f^{-1}(x):x\in X\}$ is a finite disjoint cover of $K$ by clopen sets. For every nonempty set $U\in \U$ the image $f(U)$ is a singleton. Consequently, for every $x\in X$ we have the equality
$$U\cap g^{-1}(A+x)=\{u\in U:w(u)\in A+x-f(u)\}=U\cap w^{-1}(A+x-f(U))\in\I$$and finally, $g^{-1}(A+x)=\bigcup_{U\in\U}U\cap g^{-1}(A+x)\in\I$ as $\I$ is an ideal.
\end{proof}

Observe that for any zero-dimensional compact metrizable space $K$ and any Polish group $X$ the subgroup $\F(K,X)$ of functions with finite range is dense in the function space $\C(K,X)$. This fact combined with Proposition~\ref{p:sumD} implies the following trichotomy.

\begin{theorem}\label{t:tricho} For any ideal $\I$ on a zero-dimensional compact metrizable space $K$ and any subset $A$ of a Polish group $X$ the set $W_\I(A)$ has one of the following mutually exclusive properties:
\begin{enumerate}
\item[\textup{1)}] $W_\I(A)$ is empty;
\item[\textup{2)}] $W_\I(A)$ meager and dense in $\C(K,X)$;
\item[\textup{3)}] $W_\I(A)$ is a dense Baire subspace of $\C(K,X)$.
\end{enumerate}
\end{theorem}



We recall that a semi-ideal $\I$ on $K$ is {\it $\Pi^1_1$-on-$\Sigma^1_1$}, if for any Polish space $Y$ and any analytic set $A\subset K\times Y$ the set $\{y \in Y : \{x\in K:(x,y)\in A\} \in \I \}$ is coanalytic in $Y$. 

By \cite[Theorem 29.22]{K} for any Polish space $K$ the ideal $\M_K$ of meager sets in $K$ is $\Pi^1_1$-on-$\Sigma^1_1$, and by \cite[Theorems 29.22, 29.26]{K} for any $\sigma$-additive Borel probability measure $\mu$ on a Polish space the $\sigma$-ideal $\N_{\mu}$ is $\Pi^1_1$-on-$\Sigma^1_1$.

\begin{proposition}\label{BP}
Let $A$ be an analytic set in a Polish group $X$. Assume that a semi-ideal $\I$ on a compact metrizable space $K$ is $\Pi^1_1$-on-$\Sigma^1_1$. Then the witness set $W_{\I}(A)$ is coanalytic in the function space $\C(K,X)$.
\end{proposition}

\begin{proof} The continuity of the functions
  $$\Phi: \C(K,X)\times X \times K \to \C(K,X)\times X \times X,\;\;\Phi:(h,x,a)\mapsto (h,x,h(a)),$$and $\delta: X \times X \to X$, $\delta:(x,y)\mapsto y-x$, implies that the set
$$
\begin{aligned}
B:&=\{(h,x,a) \in \C(K,X)\times X \times K: a \in h^{-1}(A+x) \}\\ &= \{ (h,x,a): h(a)-x\in A\}= \Phi^{-1}\big(\C(K,X) \times \delta^{-1}(A)\big)
\end{aligned}
$$
is analytic in $\C(K,X)\times X\times K$. Taking into account that the semi-ideal $\I$ is $\Pi_1^1$-on-$\Sigma^1_1$, we conclude that the set
$D= \{(h,x) \in \C(K,X) \times X: h^{-1}(A+x ) \notin \I \}$ is analytic in $\C(K,X)\times X$ and then the set $$\C(K,X)\setminus W_\I(A)=\{h\in \C(K,X):\exists x\in X\;h^{-1}(A+x)\notin\I\}=\pr(D)$$ is analytic in $\C(K,X)$. Here by $\pr:\C(K,X)\times X\to \C(K,X)$ we denote the natural projection.
\end{proof}

Since coanalytic sets in Polish spaces have the Baire property, Proposition~\ref{BP} and Theorem~\ref{t:tricho} imply the following trichotomy.

\begin{corollary}\label{c:tricho2} For any $\Pi^1_1$-on-$\Sigma^1_1$ ideal $\I$ on a zero-dimensional compact metrizable space $K$ and any analytic subset $A$ of a Polish group $X$ the set $W_\I(A)$ has one of the following mutually exclusive properties:
\begin{enumerate}
\item[\textup{1)}] $W_\I(A)$ is empty;
\item[\textup{2)}] $W_\I(A)$ meager and dense in $\C(K,X)$;
\item[\textup{3)}] $W_\I(A)$ is comeager in $\C(K,X)$.
\end{enumerate}
\end{corollary}

We recall that a subset $M$ of a topological space $X$ is called {\em comeager} if its complement $X\setminus M$ is meager in $X$. A subset $M$ of a Polish space $X$ is comeager if and only if $M$ contains a dense $G_\delta$-subset of $X$.

The above trichotomy motivates the following definition.

\begin{definition} Let $\I$ be a semi-ideal on a compact metrizable space $K$. A subset $A$ of a Polish group $X$ is called \index{subset!generically Haar-$\I$}\index{generically Haar-$\I$ subset}{\em generically Haar-$\I$} if its witness set $W_{\I}(A)$ is comeager in the function space $\C(K,X)$. By \index{$\GH\I$}$\GH\I$ we shall denote the semi-ideal generated by Borel generically Haar-$\I$ sets in $X$.
\end{definition}

Taking into account that a countable intersection of comeager sets in a Polish space is comeager, we can see the following fact.

\begin{proposition}\label{p:GHI-ideal} For any ideal (resp. $\sigma$-ideal) $\I$ on a compact metrizable space $K$ and any Polish group $X$ the family of all generically Haar-$\I$ sets in $X$ is an invariant ideal (resp. a $\sigma$-ideal).
\end{proposition}

Generically Haar-$\I$ sets nicely behave under homomorphisms.

\begin{proposition} Let $\I$ be a semi-ideal of subsets of a zero-dimensional compact metrizable space $K$. For any surjective continuous homomorphism $h:X\to Y$ between Polish groups and any generically Haar-$\I$ set $A$ in $Y$, the preimage $f^{-1}(A)$ is a generically Haar-$\I$ set in $X$.
\end{proposition}

\begin{proof} By \cite[Theorem 1.2.6]{BecKec}, the continuous surjective homomorphism $h:X\to Y$ is open. The homomorphism $h$ induces the homomorphism $h_\#:\C(K,X)\to \C(K,Y)$, $h_\#:\varphi\mapsto h\circ\varphi$. The zero-dimensional Michael Selection Theorem \cite[Theorem 2]{Mich} implies that the homomorphism $h_\#$ is open.
Consequently, for the comeager set $W_\I(A)$ of $\C(K,Y)$ the preimage $h_\#^{-1}(W_\I(A))$ is comeager in $\C(K,X)$. To see that the set $h^{-1}(A)$ is generically Haar-$\I$, it suffices to check that $h_\#^{-1}(W_\I(A))\subset W_\I(h^{-1}(A))$.

Indeed, for any function $\varphi\in h_\#^{-1}(W_\I(A))$ we get $h\circ\varphi\in W_\I(A)$, which means that for any $y\in Y$ the preimage $(h\circ \varphi)^{-1}(A+y)\in\I$. Then for any $x\in X$ we obtain
$$\varphi^{-1}(h^{-1}(A)+x)=\varphi^{-1}\big(h^{-1}(A+h(x))\big)=(h\circ \varphi)^{-1}(A+h(x))\in\I,$$ which means that $\varphi\in W_\I(h^{-1}(A))$.
\end{proof}

\begin{theorem}\label{t:-GHI} Let $\I$ be semi-ideal of subsets of a zero-dimensional compact metrizable space $K=\bigcup\I$. A subset $A$ of a Polish group $X$ is generically Haar-$\I$ if the difference $A-A$ is meager in $X$.
\end{theorem}

\begin{proof} Assuming that some set $A\subset X$ has meager difference $A-A$ in $X$, we shall show that $A$ is generically Haar-$\I$ in $X$. This is trivial if $A$ is empty. So we assume that $A$ is not empty. The set $A-A$, being meager is contained in a meager $F_\sigma$-set $F\subset X$. Consider the continuous map $\delta:X\times X\to X$, $\delta:(x,y)\mapsto y-x$, and observe that it is open. This implies that $\delta^{-1}(F)$ is a meager $F_\sigma$-set in $X\times X$ and its complement $G=(X\times X)\setminus \delta^{-1}(F)$ is a dense $G_\delta$-set in $X\times X$.
For a subset $B\subset X$ put $(B)^2=\{(x,y)\in B\times B:x\ne y\}\subset X\times X$.
By Mycielski-Kuratowski Theorem \cite[19.1]{K}, the family $\mathcal B=\{B\in\mathcal{K}(X):(B)^2\subset G\}$ is a dense $G_\delta$-set in the hyperspace $\mathcal{K}(X)$. Consider the continuous map $r:\C(K,X)\to \K(X)$, $r:f\mapsto f(K)$. By the continuity of $r$, the preimage $r^{-1}(\mathcal B)$ is a $G_\delta$-set in the function space $\C(K,X)$.

\begin{claim} The $G_\delta$-set $r^{-1}(\mathcal B)$ is dense in $\C(K,X)$.
\end{claim}

\begin{proof} Fix a complete metric $\rho$ generating the topology of the Polish group $X$. Given any nonempty open set $\U\subset \C(K,X)$, we should find a function $f\in\U$ with $f(K)\in \mathcal B$.

Since $K$ is zero-dimensional, we can find a function $f\in\U$ with finite range $f(K)$. Find $\e>0$ such that each function $g\in \C(K,X)$ with $\hat\rho(g,f)<\e$ belongs to $\U$. By the density of the set $\mathcal B$ in $\K(X)$, there exists a compact set $B\in \mathcal B$ such that for every $x\in f(K)$ there exists a point $\gamma(x)\in B$ with $\rho(x,\gamma(x))<\e$. Then $\gamma:f(K)\to B$ is a continuous map such that the function $g:=\gamma\circ f:K\to X$ has $\hat\rho(g,f)<\e$ and hence $g\in\U$.
Since $g(K)=\gamma(f(K))\subset B\in \mathcal B$, the function $g$ belongs to $\U\cap r^{-1}(\mathcal B)$.
\end{proof}

Since the Polish group $X$ contains a nonempty meager subset $A-A$, it is not discrete and hence contains no isolated points. By Lemma~\ref{E-in-C}, the set $\E(K,X)$ of injective continuous maps is a dense $G_\delta$-set in $\C(K,X)$. Then the intersection $\E(K,X)\cap r^{-1}(\mathcal B)$ also is a dense $G_\delta$-set in $\C(K,X)$. To see that $A$ is generically Haar-$\I$, it suffices to check that $\E(K,X)\cap r^{-1}(\mathcal B)\subset W_\I(A)$.

Take any function $f\in\E(K,X)\cap r^{-1}(\mathcal B)$ and observe that $f:K\to X$ is an injective function with $f(K)\in \mathcal B$. We claim that for every $x\in X$ the set $f^{-1}(A+x)$ contains at most one point and hence belongs to the semi-ideal $\I$ (as $\cup\I=K$). Assuming that $f^{-1}(A+x)$ contains two distinct points $a,b$, we conclude that $f(a),f(b)$ are two distinct points of $A+x$ and hence $f(a)-f(b)\in A-A$. On the other hand, $f(K)\in \mathcal B$ and hence $(f(b),f(a))\in(f(K))^2\subset G$ and $$f(a)-f(b)=\delta(f(b),f(a))\in \delta(G)=X\setminus F\subset X\setminus(A-A),$$which contradicts $f(a)-f(b)\in A-A$. This contradiction shows that $|f^{-1}(A+x)|\le1$ and hence $f^{-1}(A+x)\in\I$. So, $f\in W_\I(A)$.
\end{proof}

\begin{corollary}\label{c:-GHI} Let $\I$ be a semi-ideal of subsets of a zero-dimensional compact metrizable space $K=\bigcup\I$. If an analytic subset $A$ of a Polish group $X$ is not generically Haar-$\I$, then the difference $A-A$ is not meager in $X$ and the set $(A-A)-(A-A)$ is a neighborhood of zero in $X$.
\end{corollary}

\begin{proof} If an analytic subset $A\subset X$ is not generically Haar-$\I$, then the set $A-A$ is not meager in $X$ by Theorem~\ref{t:-GHI}. Since $A$ is analytic, the difference set $A-A$ is analytic too (being a continuous image of the product $A\times A$). By Corollary~\ref{c:PP}, the set $(A-A)-(A-A)$ is a neighborhood of zero in $X$.
\end{proof}

Corollary~\ref{c:-GHI} implies another corollary.

\begin{corollary}\label{c:St-}
Let $\I$ be a semi-ideal of subsets of a zero-dimensional compact metrizable space $K=\bigcup\I$. For any Polish group $X$ the semi-ideals $\mathcal{GHI}$ and $\HI$ have the weak Steinhaus property.
\end{corollary}

\begin{problem}\label{prob:sigma-St}
Let $\I$ be a semi-ideal of subsets of a zero-dimensional compact metrizable space $K=\bigcup\I$. Have the $\sigma$-ideals $\sigma\overline{\HI}$ and $\sigma\overline{\GHI}$ the weak Steinhaus property for every locally compact Polish group?
\end{problem}

\begin{remark} By Corollary~\ref{c:sN->S+} and Theorem~\ref{t:HM=M}, the answer to Problem~\ref{prob:sigma-St} is affirmative for the ideals $\N$ and $\mathcal M$.
\end{remark}

\begin{proposition}\label{p:rest} Let $\I$ be a semi-ideal on a zero-dimensional compact metrizable space $K$ and $Z$ be a nonempty closed subset of $K$. Any generically Haar-$\I$ set $A$ in a Polish group $X$ is generically Haar-$\I|Z$ for the semi-ideal $\I|Z=\{I\cap Z:I\in\I\}$ on $Z$.
\end{proposition}

\begin{proof} Since $A$ is generically Haar-$\I$, its witness set $W_\I(A)$ is comeager in the function space $\C(K,X)$. The zero-dimensionality of $K$ ensures that $Z$ is a retract of $K$. Consequently, the restriction operator $r:\C(K,X)\to \C(Z,X)$ is surjective and by \cite[4.27]{Luk} is open. By Lemma~\ref{l:M-coM}, the image $r(W_\I(A))$ is comeager in $\C(Z,X)$. It remains to observe that $r(W_\I(A))\subset W_{\I|Z}(A)$ and conclude that the set $A$ is generically Haar-$\I|Z$.
\end{proof}

Now we explore the relation of the semi-ideal $\mathcal{GHI}$ to other semi-ideals on Polish groups.

\begin{proposition}\label{p:GHI->M} Let $\I$ be a proper semi-ideal on a zero-dimensional compact metrizable space $K$.
Each generically Haar-$\I$ set $B$ with the Baire Property in a Polish group $X$ is meager.
\end{proposition}

\begin{proof}
For the proof by contradiction suppose that the set $B$ is not meager in $X$. Since $B$ has the Baire property in $X$, for some open set $U\subset X$ the symmetric difference $U\triangle B$ is meager in $X$. Since $B$ is not meager, the open set $U$ is not empty. The meager set $U\setminus B\subset U\triangle B$ is contained in some meager $F_\sigma$-set $F\subset U$. Then $G_\delta$-set $G=U\setminus F\subset U\cap B$ is dense in $U$.

The zero-dimensionality of $K$ implies that $\C(K,G)$ is a dense $G_\delta$-set in the open subset $\C(K,U)$ of the function space $\C(K,X)$. Since the witness set $W_\I(B)$ is comeager in $\C(K,X)$, the intersection $\C(K,G)\cap W_\I(B)$ is not empty and hence contains some function $f:K\to G\subset B$. Then $K=f^{-1}(B)\in\I$ which contradicts the properness of the semi-ideal $\I$.
\end{proof}

\begin{proposition}\label{p:cHI=>HM} Let $\I$ be a semi-ideal of sets with empty interiors on a compact metrizable space $K$. If an $F_\sigma$-set $A$ in a Polish group $X$ is generically Haar-$\I$, then it is generically Haar-$\M$, where $\M$ is the $\sigma$-ideal of meager sets in $K$.
\end{proposition}

\begin{proof} If $A$ is generically Haar-$\I$ in $X$, then the set $W_{\I}(A)=\{f\in \C(K,X):\forall x\in X\;\;f^{-1}(A+x)\in\I\}$ is comeager in the function space $\C(K,X)$. Observe that for every $f\in W_{\I}(A)$ and $x\in X$ the preimage $f^{-1}(A+x)$ is an $F_\sigma$-subset of $K$ that belongs to the semi-ideal $\I$.
Assuming that the $F_\sigma$-set $f^{-1}(A+x)$ is not meager in $K$, we can conclude that the set $f^{-1}(A+x)\in\I$ has nonempty interior in $K$, which contradicts the choice of the semi-ideal $\I$. This contradiction shows that $W_{\I}(A)\subset W_\M(A)$, which implies that the set $W_\M(A)$ is comeager in $\C(K,X)$ and $A$ is generically Haar-$\M$.
\end{proof}

Proposition~\ref{p:GHI->M} and Lemma~\ref{E-in-C} imply:

\begin{corollary} Let $\I$ be a semi-ideal of subsets of a zero-dimensional compact metrizable space $K$. For any Polish group $X$ we have the inclusions
$$\GH\I\subset \M\cap\EHI\subset\EHI\subset\HI.$$
\end{corollary}

\begin{theorem}\label{t:12.14} Let $\I$ be a semi-ideal on a zero-dimensional compact metrizable space $K$ such that $\I\cap\K(K)$ is a $G_\delta$-set in the hyperspace $\K(K)$. Then for any compact set $A$ in a Polish group $X$ the witness set $W_{\I}(A)$ is a $G_\delta$-set in $\C(K,X)$. Consequently, $A$ is Haar-$\I$ if and only if $A$ is generically Haar-$\I$.
\end{theorem}

\begin{proof} Since $\I\cap\K(K)$ is a $G_\delta$-set in $\K(K)$, the complement $\K(K)\setminus\I$ can be written as the union $\bigcup_{n\in\w}\K_n$ of an increasing sequence $(\K_n)_{n\in\w}$ of compact sets $\K_n$ in the hyperspace $\K(K)$. It is easy to check that for every $n\in\w$ the set ${\uparrow}\K_n=\{B\in\K(K):\exists D\in\K_n\;\;D\subset B\}$ is closed in $\K(K)$ and does not intersect the semi-ideal $\I$. Replacing each set $\K_n$ by ${\uparrow}\K_n$, we can assume that $\K_n={\uparrow}\K_n$ for all $n\in\w$.

Observe that $$\C(K,X)\setminus W_\I(A)=\bigcup_{n\in\w}\{f\in \C(K,X):\exists x\in X\;\;f^{-1}(A+x)\in\K_n\}.$$
To see that $W_\I(A)$ is a $G_\delta$-set in $\C(K,X)$, it remains to prove that for every $n\in\w$ the set $$\F_n=\{f\in \C(K,X):\exists x\in X\;\;f^{-1}(A+x)\in\K_n\}$$is closed in $\C(K,X)$.

Take any sequence $\{f_i\}_{i\in\w}\subset\F_n$, convergent to some function $f_\infty$ in $\C(K,X)$. The convergence of $(f_i)_{i\in\w}$ to $f_\infty$ implies that the subset $D=f_\infty(K)\cup\bigcup_{i\in\w}f_i(K)$ of $X$ is compact. Then the set $D-A\subset X$ is compact, too.

For every $i\in\w$ find a point $x_i\in X$ such that $f_i^{-1}(A+x_i)\in\K_n\subset \K(K)\setminus\I\subset\K(K)\setminus\{\emptyset\}$. It follows that the set $K_i=f_i^{-1}(A+x_i)$ is not empty and hence $x_i$ belongs to the compact set $f_i(K)-A\subset D-A$. Replacing the sequence $(f_i)_{i\in\w}$ by a suitable subsequence, we can assume that the sequence $(x_i)_{i\in\w}$ converges to some point $x\in D-A$ and the sequence $(K_i)_{i\in\w}$ converges to some compact set $K_\infty$ in the compact space $\K_n\subset\K(K)$. Then the sequence $(f_i(K_i))_{i\in\w}$ converges to $f_\infty(K_\infty)$ and hence $f_\infty(K_\infty)=\lim_{i\to\infty}f_i(K_i)\subset \lim_{i\to\infty}(A+x_i)=A+x$. Then $\K_n\ni K_\infty\subset f_\infty^{-1}(A+x)$ and hence $f_\infty^{-1}(A+x)\in {\uparrow}\K_n=\K_n$, which means that $f_\infty\in \F_n$ and the set $\F_n$ is closed in $\C(K,X)$.

Therefore the witness set $W_\I(A)$ of $A$ is a $G_\delta$-set in $\C(K,X)$. If the set $A$ is Haar-$\I$, then the witness set $W_\I(A)$ is not empty and hence dense in $\C(K,X)$ by Theorem~\ref{t:tricho}. Being a dense $G_\delta$-set, the witness set $W_\I(A)$ is comeager in $\C(K,X)$, which means that $A$ is generically Haar-$\I$.
\end{proof}

\begin{corollary}\label{c:s-comp}
Let $\I$ be a $\sigma$-continuous semi-ideal on a zero-dimensional compact metrizable space $K$ such that $\I\cap\K(K)$ is a $G_\delta$-set in the hyperspace $\K(K)$. Then for any $\sigma$-compact set $A$ in a Polish group $X$ the witness set $W_{\I}(A)$ is a $G_\delta$-set in the function space $\C(K,X)$. Consequently, $A$ is Haar-$\I$ if and only if $A$ is generically Haar-$\I$.
\end{corollary}

\begin{proof} Given a $\sigma$-compact set $A\subset X$, write it as the union $A=\bigcup_{n\in\w}A_n$ of an increasing sequence $(A_n)_{n\in\w}$ of compact sets $A_n\subset X$. By Theorem~\ref{t:12.14}, for every $n\in\w$ the witness set $W_\I(A_n)$ is a $G_\delta$-set in $\C(K,X)$. Since the semi-ideal $\I$ is $\sigma$-continuous, the witness set $W_\I(A)$ is equal to the $G_\delta$-set $\bigcap_{n\in\w}W_\I(A_n)$.

If the $\sigma$-compact set $A$ is Haar-$\I$, then the witness set $W_\I(A)$ is nonempty and hence dense in $\C(K,X)$ by Theorem~\ref{t:tricho}. Now we see that the dense $G_\delta$-set $W_\I(A)$ is comeager in $\C(K,X)$, which means that the set $A$ is generically Haar-$\I$.
\end{proof}

\begin{remark} The papers of Solecki \cite{Sol2011} and Zelen\'y \cite{Zeleny} contain many examples of semi-ideals $\I$ on compact metrizable spaces $K$ such that $\I\cap\K(K)$ is a $G_\delta$-set in $\K(K)$.
\end{remark}

Now we shall evaluate the Borel complexity of the set $\F(X)\cap\GHI$ in the space $\F(X)$ of all closed subsets of a Polish group $X$. We recall that the space $\F(X)$ is endowed with the Fell topology, whose Borel $\sigma$-algebra is called the Effros-Borel structure of the hyperspace $\F(X)$. It is well-known \cite[12.6]{K} that this structure is standard, i.e., it is generated by a suitable Polish topology on $\F(X)$. So, we can speak of analytic and coanalytic subsets of the standard Borel space $\F(X)$.

We shall need the following folklore fact.

\begin{lemma}\label{l:KF} For any Polish space $X$ the set $$P:=\{(K,F)\in\K(X)\times\F(X):K\cap F\ne\emptyset\}$$ is of type $G_\delta$ in $\K(X)\times\F(X)$.
\end{lemma}

\begin{proof} Let $\rho$ be a complete metric generating the topology of the Polish space $X$. For every $n\in\w$, fix a countable cover $\U_n$ of $X$ by open sets of diameter $<\frac1{2^n}$. Taking into account that
$$
\begin{aligned}
P:&=\{(K,F)\in\K(X)\times\F(X):K\cap F\ne\emptyset\}\\
&=\bigcap_{n\in\w}\bigcup_{U\in\U_n}\{(K,F)\in\K(X)\times\F(X):K\cap U\ne\emptyset\ne U\cap F\}
\end{aligned},
$$
we see that the set $P$ is of type $G_\delta$ in $\K(X)\times\F(X)$.
\end{proof}

\begin{proposition}\label{p:GHI-A} Let $\I$ be a semi-ideal on a compact metrizable space $K$ such that $\K(K)\cap \I$ is a coanalytic subset of the hyperspace $\K(K)$. Then for any Polish group $X$ the set $\F(X)\cap\GHI$ is coanalytic in the standard Borel space $\F(X)$.
\end{proposition}

\begin{proof} Observing that for any nonempty closed subset $C\subset K$ the map $$\Phi:\F(X)\times \C(K,X)\times X\to\F(X)\times \K(X),\;\;\Phi:(F,f,x)\mapsto (F,f(C)-x),$$ is continuous and taking into account Lemma~\ref{l:KF}, we conclude that the set
$$P_C:=\{(F,f,x)\in\F(X)\times \C(K,X)\times \K(X):F\cap(f(C)-x)\ne\emptyset\}$$
is of type $G_\delta$ in $\F(X)\times\C(K,X)\times X$.

We claim that the map
$$\Psi:\F(X)\times \C(K,X)\times X\to\F(K),\;\;\Psi:(F,f,x)\mapsto f^{-1}(x+F),$$
is Borel.
We need to check that for any open set $U\subset K$ and compact set $C\subset K$ the sets $\Psi^{-1}(U^+)$ and $\Psi^{-1}(C^-)$ are Borel in $\F(X)\times C(K,X)\times X$. Here $U^+:=\{F\in\F(X):F\cap U\ne\emptyset\}$ and $C^-:=\{F\in\F(X):F\cap C=\emptyset\}$.
To see that $\Psi^{-1}(C^-)$ is Borel, observe that
$$
\begin{aligned}
\Psi^{-1}(C^-)&=\{(F,f,x)\in \F(X)\times\C(K,X)\times X:C\cap f^{-1}(x+F)=\emptyset\}\\
&=\{(F,f,x)\in \F(X)\times\C(K,X)\times X:F\cap (f(C)-x)=\emptyset\}\\
&=(\F(X)\times\C(K,X)\times X)\setminus P_C,
\end{aligned}
$$
is an $F_\sigma$-set in $\F(X)\times\C(K,X)\times X$.

To see that $\Psi^{-1}(U^+)$ is Borel, write $U$ as the countable union $U=\bigcup_{n\in\w}K_n$ of compact subsets $K_n\subset U$ and observe that
$$
\begin{aligned}
\Psi^{-1}(U^+)&=\{(F,f,x)\in \F(X)\times\C(K,X)\times X:U\cap f^{-1}(x+F)\ne\emptyset\}\\
&=\bigcup_{n\in\w}\{(F,f,x)\in \F(X)\times\C(K,X)\times X:K_n\cap f^{-1}(x+F)\ne\emptyset\}\\
&=\bigcup_{n\in\w}\{(F,f,x)\in \F(X)\times\C(K,X)\times X:F\cap (f(K_n)-x)\ne\emptyset\}=\bigcup_{n\in\w}P_{K_n},
\end{aligned}
$$
is a countable union of $G_\delta$-sets.

Therefore, the map $\Psi:\F(X)\times\C(K,X)\times X\to \F(K)$ is Borel. Since the set $\K(K)\cap\I$ is coanalytic in $\K(K)$, the set $\F(K)\cap\I$ is coanalytic in $\F(K)=\K(K)\cup\{\emptyset\}$ and the preimage $\Psi^{-1}(\I)$ is coanalytic in the standard Borel space $\F(X)\times\C(K,X)\times X$. Then the set
$$
\begin{aligned}
A:&=\{(F,f)\in\F(X)\times \C(K,X):\exists x\in X\;f^{-1}(F+x)\notin\I\}\\
&=\{(F,f)\in\F(X)\times \C(K,X):\exists x\in X\;(F,f,x)\notin\Psi^{-1}(\I)\}
\end{aligned}
$$ is an analytic subspace of the standard Borel space $\F(X)\times\C(K,X)$, and its complement
$$W:=\{(F,f)\in\F(X)\times \C(K,X):\forall x\in X\;\;f^{-1}(F+x)\in\I\}$$
is a coanalytic subset of $\F(X)\times\C(K,X)$.

Now \cite[36.24]{K} implies that the set
$$\F(X)\cap\GHI=\big\{F\in\F(X):\{f\in \C(K,X):(F,f)\in W\}\mbox{ is comeager in }\C(K,X)\big\}$$ is coanalytic in the standard Borel space $\F(X)$.
\end{proof}

\begin{theorem}\label{t:dodos2} Let $\I$ be a proper $\sigma$-ideal on $2^\w$ such that $\overline{\N}\subset\I$ and $\I\cap \K(2^\w)$ is a coanalytic set in $\K(2^\w)$. Each non-locally compact Polish group $X$ contains a closed subset $F\subset X$, which is openly Haar-null, but not generically Haar-$\I$. Consequently, $\overline{\HN^\circ}\not\subset\GHI$ and $\overline{\GHI}\ne\overline{\HI}$ on $X$.
\end{theorem}

\begin{proof} By Proposition~\ref{p:GHI-A}, the set $\F(X)\cap\GHI$ is coanalytic in the standard Borel space $\F(X)$. On the other hand, by Corollary~\ref{c:CND},
there exists a continuous map $\Phi:\F(\w^\w)\to\F(X)$ such that $\Phi^{-1}(\HN^\circ)=\Phi^{-1}(\HI)=\mathrm{CND}$. Then the preimage $\Phi^{-1}(\GHI)$ is a coanalytic subset of the standard Borel space $\F(\w^\w)$. By a result of Hjorth \cite{Hjorth} (mentioned in \cite[p.210]{S01} and \cite[4.8]{TV}), the set $\mathrm{CND}$ is $\mathbf{\Sigma_1^1}$-hard and hence is not coanalytic. Consequently, $\Phi^{-1}(\GHI)\ne \mathrm{CND}=\Phi^{-1}(\HI)$. Taking into account that $\GHI\subset\HI$, we can find a closed non-dominating subset $D\in\mathrm{CND}\setminus\Phi^{-1}(\GHI)$.
 Then the closed set $F=\Phi(D)$ is openly Haar-null and Haar-$\I$, but not generically Haar-$\I$.
 \end{proof}

Corollary~\ref{c:s-comp} and Theorem~\ref{t:dodos2} imply the following characterization of locally compact Polish groups.

\begin{theorem}\label{t:HN=GHI} Let $\I$ be a proper $\sigma$-ideal on $2^\w$ such that $\overline{\N}\subset\I$ and $\I\cap\K(2^\w)$ is a $G_\delta$-set in $\K(2^\w)$.
For a Polish group $X$ the following conditions are equivalent:
\begin{enumerate}
\item[\textup{1)}] $X$ is locally compact;
\item[\textup{2)}] each closed Haar-null set in $X$ is generically Haar-$\I$;
\item[\textup{3)}] each closed Haar-$\I$ set in $X$ is generically Haar-$\I$.
\end{enumerate}
\end{theorem}

For a Polish group $X$ by $\sigma\K$ we shall denote the family of $\sigma$-compact sets in $X$.
It follows that for any proper $\sigma$-ideal $\I$ on a zero-dimensional compact metrizable space $K$ such that $\I\cap\K(K)$ is a $G_\delta$-set in $\K(K)$ and any Polish group we have the following diagram.
$$
\xymatrix{
\sigma\K\cap\HI\ar@{=}[r]&\sigma\K\cap\GH\I\ar[r]\ar[dd]&\sigma\overline{\GH\I}\ar[r]
&\sigma\overline{\EHI}\ar[r]\ar[ddl]&\sigma\overline{\HI}\ar[r]\ar[d]&
\sigma\overline{\HM}\ar[d]\\
&&&&\mathcal H\sigma\overline{\I}\ar[r]\ar[d]&\HM\\
&\GH\I\ar[r]&\M\cap\EHI\ar[r]&\EHI\ar[r]&\HI
}
$$
If the Polish group is locally compact, then this diagram simplifies to the following form:
\begin{equation}\label{diag1}
\xymatrix{
\sigma\overline{\GHI}\ar@{=}[r]&\sigma\overline{\EHI}\ar@{=}[r]\ar[d]&\sigma\overline{\HI}\ar[r]&\mathcal H\sigma\overline{\I}\ar[d]\\
&\GH\I\ar[r]&\M\cap\EHI\ar[r]&\M\cap\HI.}
\end{equation}

\begin{remark}\label{r:Zw} By Corollary~\ref{c:St-}, for any semi-ideal $\I$ on a zero-dimensional compact metrizable space $K=\bigcup\I$, the Borel subgroup $H\notin\sigma\overline{\HM}$ from Example~\ref{Banakh} is generically Haar-$\I$, witnessing that $\GHI\not\subset \sigma\overline{\HM}$ for the Polish group $X=\IZ^\w$.
\end{remark}

\section{Generically Haar-$\I$ sets for some concrete semi-ideals $\I$}
\label{s13}

In this section we shall study generically Haar-$\I$ sets for some concrete semi-ideals $\I$ on zero-dimensional compact metrizable spaces.
For such sets we reserve special names.

\begin{definition} A subset $A$ of a Polish group $X$ is called
\begin{itemize}
\item \index{subset!generically Haar-null}\index{generically Haar-null subset}{\em generically Haar-null} if $A$ is generically Haar-$\N$ for the $\sigma$-ideal $\N$ of sets of Haar measure zero in $2^\w$;
\item \index{subset!generically Haar-meager}\index{generically Haar-meager subset}{\em generically Haar-meager} if $A$ is generically Haar-$\M$ for the $\sigma$-ideal $\M$ of meager sets in $2^\w$;
\item \index{subset!generically Haar-countable}\index{generically Haar-countable  subset}{\em generically Haar-countable} if $A$ is generically Haar-$[2^\w]^{\le\w}$ for the $\sigma$-ideal $[2^\w]^{\le\w}$ of at most countable sets in the Cantor cube $2^\w$;
\item \index{subset!generically Haar-finite}\index{generically Haar-finite subset}{\em generically Haar-finite} if $A$ is generically Haar-$[2^\w]^{<\w}$ for the ideal $[2^\w]^{<\w}$ of all finite sets in $2^\w$;
\item \index{subset!generically Haar-$n$}\index{generically Haar-$n$ subset} {\em generically Haar-$n$} for $n\in\IN$ if $A$ is generically Haar-$[2^\w]^{\le n}$ for the semi-ideal $[2^\w]^{\le n}$ consisting of subsets of cardinality $\le n$ in $2^\w$;
\item \index{subset!generically null-finite}\index{generically null-finite subset}{\em generically null-finite} if $A$ is generically Haar-$[\w{+}1]^{<\w}$ for the ideal $[\w{+}1]^{<\w}$ of all finite sets in the ordinal $\w+1$ endowed with the order topology;
\item \index{subset!generically null-$n$}\index{generically null-$n$ subset}{\em generically null-$n$} for $n\in\IN$ if $A$ is generically Haar-$[\w{+}1]^{\le n}$ for the semi-ideal $[\w{+}1]^{\le n}$ consisting of all subsets of cardinality $\le n$ in $\w+1$.
\end{itemize}
For a Polish group $X$ by \index{$\GHN$}$\GHN$ and \index{$\GHM$}$\GHM$ we denote the $\sigma$-ideals consisting of subsets of Borel subsets of $X$ which are generically Haar-null and generically Haar-meager, respectively.
\end{definition}

For every subset $A$ of a Polish group $X$ and every $n\in\IN$ these notions relate as follows:
$$
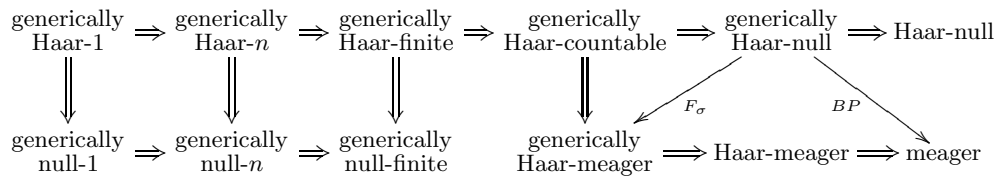
\small
\xymatrixcolsep{11pt}
\xymatrix{
\mbox{generically}\atop{\mbox{Haar-$1$}}\ar@{=>}[r]\ar@{=>}[d]&
\mbox{generically}\atop{\mbox{Haar-$n$}}\ar@{=>}[r]\ar@{=>}[d]&
\mbox{generically}\atop{\mbox{Haar-finite}}\ar@{=>}[r]\ar@{=>}[d]&
\mbox{generically}\atop{\mbox{Haar-countable}}\ar@{=>}[r]\ar@{=>}[d]&
\mbox{generically}\atop{\mbox{Haar-null}}\ar@{=>}[r]\ar_{BP}[rd]\ar^{F_\sigma}[ld]&\mbox{Haar-null}\\
\mbox{generically}\atop{\mbox{null-$1$}}\ar@{=>}[r]&
\mbox{generically}\atop{\mbox{null-$n$}}\ar@{=>}[r]&
\mbox{generically}\atop{\mbox{null-finite}}&
\mbox{generically}\atop{\mbox{Haar-meager}}\ar@{=>}[r]&
\mbox{Haar-meager}\ar@{=>}[r]&\mbox{meager}
}
$$
Non-trivial implications in this diagram are proved in Propositions~\ref{p:GHI->M}, and \ref{c:cHI=>HM}. By the simple arrow with an inscription we denote the implications (generically Haar-null set with the Baire Property is meager) and (generically Haar-null $F_\sigma$-set is generically Haar-meager) holding under some additional assumptions.

The above diagram and Theorem~\ref{t:NF=>HN+EHM} suggest the following open problem.

\begin{problem} Let $A$ be a generically null-finite Borel set in a Polish group $X$. Is $A$ generically Haar-null and generically Haar-meager in $X$?
\end{problem}

Theorem~\ref{t:-GHI} implies the following corollary.

\begin{corollary}\label{count} A subset $A$ of a Polish group $X$ is generically Haar-$1$ if $A-A$ is meager in $X$.
\end{corollary}

Proposition~\ref{p:cHI=>HM} implies the following corollary.

\begin{corollary}\label{c:cHI=>HM} Let $\I$ be a semi-ideal of sets with empty interior in $2^\w$. Each generically Haar-$\I$ $F_\sigma$-set in a Polish group is generically Haar-meager. In particular, each generically Haar-null $F_\sigma$-set in a Polish group is generically Haar-meager.
\end{corollary}

\begin{theorem}\label{t:sigma-gen} Let $n\in\IN$. A $\sigma$-compact subset of a Polish group is
\begin{enumerate}
\item[\textup{1)}] generically Haar-null if and only if it is Haar-null;
\item[\textup{2)}] generically Haar-meager if and only if it is Haar-meager;
\item[\textup{3)}] generically Haar-$n$ if and only if it is Haar-$n$;
\item[\textup{4)}] generically null-$n$ if and only if it is null-$n$.
\end{enumerate}
\end{theorem}

This theorem can be easily derived from Corollary~\ref{c:s-comp} and the following lemma.

\begin{lemma}\label{l:Gd} Let $K$ be a zero-dimensional compact metrizable space.
\begin{enumerate}
\item[\textup{1)}] The family $\{C\in \K(K):C$ is meager in $K\}$ is a $G_\delta$-set in the hyperspace $\K(K)$.
\item[\textup{2)}] For any probability measure $\mu\in P(K)$ the family $\{C\in \K(K):\mu(C)=0\}$ is a $G_\delta$-set in the hyperspace $\K(K)$.
\item[\textup{3)}] The family $\{C\in\K(K):|C|\le n\}$ is a closed (and hence $G_\delta$) set in $\K(K)$.
\end{enumerate}
\end{lemma}

\begin{proof} 1. Fix a countable base $\mathcal B$ of the topology of $K$, consisting of nonempty clopen subsets of $K$. For every $U\in\mathcal B$ consider the open subspace $$\widehat U=\bigcup_{\mathcal B\ni V\subset U}\{C\in\K(K):C\cap V=\emptyset\}$$ of $\K(K)$ and observe that
$$\{C\in\K(K):\mbox{$C$ is meager in $K$}\}=
\{C\in\K(K):\mbox{$C$ is nowhere dense in $K$}\}=\bigcap_{U\in\mathcal B}\widehat U$$is a $G_\delta$-set in $\K(K)$.
\smallskip

2. Fix a $\sigma$-additive Borel probability measure $\mu$ on $K$. The $\sigma$-additivity of the measure $\mu$ implies that the measure $\mu$ is \index{measure!regular}\index{regular measure}{\em regular} in the sense that for any $\e>0$ and any closed set $C\subset K$ there exists an open neighborhood $O_C\subset K$ of $C$ such that $\mu(O_C)<\mu(C)+\e$. The regularity of the measure $\mu$ implies that for every $\e>0$ the set $\{C\in\K(K):\mu(C)<\e\}$ is open in $\K(K)$. Then the set $$\{C\in\K(K):\mu(C)=0\}=\bigcap_{n=1}^\infty\{C\in\K(K):\mu(C)<\tfrac1n\}$$is of type $G_\delta$ in $\K(K)$.
\smallskip

3. The definition of the Vietoris topology implies that for every $n\in\IN$ the set $\{C\in\K(K):|C|\le n\}$ is closed (and hence $G_\delta$) in $\K(K)$.
\end{proof}

Lemma~\ref{l:Gd} and Theorem~\ref{t:HN=GHI} imply the following characterization of locally compact Polish groups.

\begin{theorem}\label{t:dodos} For a Polish group $X$ the following conditions are equivalent:
\begin{enumerate}
\item[\textup{1)}] $X$ is locally compact;
\item[\textup{2)}] each closed Haar-null set in $X$ is generically Haar-null;
\item[\textup{3)}] each closed Haar-meager set in $X$ is generically Haar-meager.
\item[\textup{4)}] each closed Haar-null set in $X$ is generically Haar-meager.
\end{enumerate}
\end{theorem}

\begin{remark} The equivalence $(1)\Leftrightarrow(2)$ in Theorem~\ref{t:dodos} was proved by Dodos \cite{D1}. In fact, our proof of Theorem~\ref{t:dodos} is a suitable modification of the proof of Corollary 9 in \cite{D1}.
\end{remark}

Now we prove characterizations of generically Haar-null and generically Haar-meager sets showing that our definitions of such sets are equivalent to the original definitions given in \cite{D2}, \cite{D3}.

\begin{theorem}\label{t:gen-char} A subset $A$ of a non-discrete Polish group $X$ is
\begin{enumerate}
\item[\textup{1)}] generically Haar-null if and only if the set $$T(A):=\{\mu\in P(X):\forall x\in X\;\;\mu(A+x)=0\}$$ is comeager in the space $P(X)$ of probability $\sigma$-additive Borel measures on $X$;
\item[\textup{2)}] generically Haar-meager if and only if the set $$K_\M(A):=\{K\in\K(X):\forall x\in X\;\;K\cap (A+x)\in\M_K\}$$ is comeager in the hyperspace $\K(X)$;
\item[\textup{3)}] generically Haar-countable if and only if the set $$K_{\le\w}(A):=\{K\in \K(X): \forall x\in X\;\;|K\cap (A+x)|\le \w\}$$ is comeager in $\K(X)$;
\item[\textup{4)}] generically Haar-finite if and only if the set $$K_{<\w}(A):=\{K\in \K(X): \forall x\in X\;\;|K\cap (A+x)|<\w\}$$ is comeager in $\K(X)$;
\item[\textup{5)}] generically Haar-$n$ for some $n\in\IN$ if and only if the set $$K_{\le n}(A):=\{K\in \K(X): \forall x\in X\;\;|K{\cap}(A{+}x)|\allowbreak\le n\}$$ is comeager in $\K(X)$.
\end{enumerate}
\end{theorem}

\begin{proof} Fix a complete metric $\rho$ generating the topology of the Polish group $X$. The metric $\rho$ induces the complete metric $\hat \rho(f,g)=\max_{x\in 2^\w}\rho(f(x),g(x))$ on the function space $\C(2^\w,X)$.
\smallskip

1. Consider the continuous map $$\Sigma:\C(2^\w,X)^\w\to P(X),\;\;\Sigma:(f_n)_{n\in\w}\mapsto \sum_{n=0}^\infty\frac1{2^{n+1}}Pf_n(\lambda),$$
where $\lambda$ is the Haar measure on the Cantor cube $2^\w$.

By Lemma~\ref{E-in-C}, the subspace $\E(2^\w,X)$ consisting of injective maps is dense $G_\delta$ in $\C(2^\w,X)$. Consequently, $\E(2^\w,X)^\w$ is dense $G_\delta$ in $\C(2^\w,X)^\w$. Let $D$ be a countable dense set in $X$. It is easy to see that the sets
$$
\begin{aligned}
&\mathcal D_1:=\{(f_n)_{n\in\w}\in\C(2^\w,X)^\w:\mbox{$\bigcup_{k\in\w}f_k(2^\w)$ is dense in $X$}\},\\
&\mathcal D_2:=\{(f_n)_{n\in\w}\in\C(2^\w,X)^\w:
\mbox{$D\cap\bigcup_{n\in\w}f_n(2^\w)=\emptyset$}\},\\
&\mathcal D_3:=\{(f_n)_{n\in\w}\in\C(2^\w,X)^\w:\mbox{the family $\big(f_n(2^\w)\big)_{n\in\w}$ is disjoint}\},\mbox{ and}\\
&\mathcal D:=\E(2^\w,X)^\w\cap\mathcal D_1\cap\mathcal D_2\cap\mathcal D_3
\end{aligned}
$$are dense $G_\delta$ in $\C(2^\w,X)^\w$.

Let $P_0(X)$ be the subspace of $P(X)$ consisting of all strictly positive continuous probability measures on $X$. Since the Polish group $X$ is non-discrete, the set $P_0(X)$ is dense in $P(X)$.
Observe that $\Sigma(\mathcal D)\subset P_0(X)$. The following claim implies that $\Sigma(\mathcal D)=P_0(X)$.

\begin{claim}\label{l:sur} For every measure $\mu\in P_0(X)$ and any closed nowhere dense $F\subset X$ with $\mu(F)=0$ there exists a sequence $\vec f=(f_n)_{n\in\w}\in\mathcal D$ such that $\Sigma(\vec f)=\mu$ and $\bigcup_{n\in\w}f_n(2^\w)\subset X\setminus F$.
\end{claim}

\begin{proof} Fix a countable family $(U_n)_{n\in\w}$ of nonempty open sets in $X\setminus F$ such that each nonempty open set $U\subset X$ contains some set $U_n$.

By induction we shall construct an increasing number sequence $(n_k)_{k\in\w}\in\IN^\w$ and a function sequence $(f_n)_{n\in\w}\in \E(2^\w,X)^\w$ such that for every $k\in\w$ the following conditions hold:
\begin{enumerate}
\item[(1)] $2^{n_{k+1}+1}\cdot\mu\big(U_k\setminus\bigcup_{i\le n_k}f_i(2^\w)\big)>1$;
\item[(2)] $f_{n_{k+1}}(2^\w)\subset U_k\setminus\big(D\cup F \cup\bigcup_{i\le n_k}f_i(2^\w)\big)$;
\item[(3)] $\mu(f_{n_{k+1}}(S))=\frac1{2^{n_{k+1}+1}}\lambda(S)$ for any Borel subset $S\subset 2^\w$;
\end{enumerate}
and for any $n\in\w$ with $n_k<n<n_{k+1}$ we get
\begin{itemize}
\item[(4)] $f_n(2^\w)\subset X\setminus \big(D\cup F\cup f_{n_{k+1}}(2^\w)\cup \bigcup_{i<n}f_i(2^\w)\big)$;
\item[(5)] $\mu(f_n(S))=\frac1{2^{n+1}}\cdot\lambda(S)$ for any Borel subset $S\subset 2^\w$.
\end{itemize}
To start the inductive construction, put $n_{-1}=-1$. Assume that for some $k\ge -1$ we have constructed a number $n_{k}$ and a sequence $(f_n)_{n\le n_k}$ satisfying the inductive assumptions (1)--(5). The density of $D$ in $X$ and the inductive assumptions (2) and (4) guarantee that the compact set $\bigcup_{n\le n_k}f_n(2^\w)$ is nowhere dense in $X$. Then the open set $U_k':=U_k\setminus \big(F\cup\bigcup_{n\le n_k}f_n(2^\w)\big)$ is not empty and hence has positive measure $\mu(U_k')$. So, we can choose a number $n_{k+1}>n_k$ such that $2^{n_{k+1}+1}\cdot\mu(U_k')>1$. The continuity of the measure $\mu$ guarantees that the dense Polish subspace $\Pi_k:=U_k'\setminus D$ of $U_k'$ has measure $\mu(\Pi_k)=\mu(U_k')$. Consider the probability measure $\mu'\in P(\Pi_k)$ defined by $\mu'(S):=\frac{\mu(S\cap \Pi_k)}{\mu(\Pi_k)}$ for a Borel subset $S\subset \Pi_k$. For the number $a=\frac1{2^{n_{k+1}+1}\cdot\mu(U_k')}<1$, Lemma~\ref{l:a} yields an injective continuous map $f_{n_{k+1}}:2^\w\to\Pi_k$ such that for any Borel subset $S\subset 2^\w$ we have $\mu'(f_{n_{k+1}}(S))=a\cdot\lambda(S)$ and hence $$\mu(f_{n_{k+1}}(S))=\mu(\Pi_k)\cdot\mu'(f_{n_{k+1}}(S))=\mu(U'_k)\cdot a\cdot\lambda(S)=\frac1{2^{n_{k+1}+1}}\cdot\lambda(S).$$
It is clear that the map $f_{n_{k+1}}$ satisfies the conditions (2), (3) of the inductive assumption.

Next, by finite induction, for every number $n$ in the interval $(n_k,n_{k+1})$ we shall construct an injective continuous map $f_n\in\E(2^\w,X)$ satisfying the conditions (4), (5) of the inductive construction. Assume that for some $n\in\w$ with $n_k<n<n_{k+1}$ the sequence of maps $(f_i)_{n_k<i<n}$ satisfying the conditions (4)--(5) has been constructed. The conditions (3) and (5) guarantee that $\mu(f_i(2^\w))=\frac1{2^{i+1}}$ for every $i<n$. By the condition (4), the family $\{f_{n_{k+1}}(2^\w)\}\cup\{f_i(2^\w)\}_{i<n}$ is disjoint. Since $\mu(D)=\mu(F)=0$, the Polish subspace
$X_n:=X\setminus\big(D\cup F\cup f_{n_{k+1}}(2^\w)\cup\bigcup_{i<n}f_i(2^\w)\big)$ of $X$ has measure $$\mu(X_n)=1-\mu(f_{n_{k+1}}(2^\w))-
\sum_{i<n}\mu(f_i(2^\w))=1-\tfrac1{2^{n_{k+1}+1}}-\sum_{i<n}\tfrac1{2^{i+1}}=
\tfrac1{2^n}-\tfrac1{2^{n_{k+1}+1}}>\tfrac1{2^{n+1}}.$$
Consider the probability measure $\mu_n\in P(X_n)$ defined by $\mu_n(S)=\frac{\mu(S\cap X_n)}{\mu(X_n)}$ for a Borel subset $S\subset X_n$. For the number $a_n=\frac1{2^{n+1}\cdot\mu(X_n)}<1$, Lemma~\ref{l:a} yields an injective map $f_n:2^\w\to X_n$ such that for every Borel subset $S\subset 2^\w$ we have
$\mu_n(f_n(S))=a_n\cdot \lambda(S)$ and hence $$\mu(f_n(S))=\mu(X_n)\cdot \mu_n(f_n(S))=\mu(X_n)\cdot a_n\cdot\lambda(S)=\frac1{2^{n+1}}\cdot\lambda(S),$$which means that the map $f_n$ satisfies the inductive assumptions (4) and (5).
This completes the inductive step.

After completing the inductive construction, we obtain a sequence $(f_n)_{n\in\w}\in \mathcal D$ such that $\mu(f_n(S))=\frac1{2^{n+1}}\lambda(S)$ for every Borel subset $S\subset 2^\w$. In particular, the $\sigma$-compact set $A=\bigcup_{n\in\w}f_n(2^\w)$ has measure $$\mu(A)=\sum_{n\in\w}\mu(f_n(2^\w))=\sum_{n\in\w}\frac1{2^{n+1}}=1.$$

 We claim that $\mu=\sum_{n\in\w}\frac1{2^{n+1}}P\!f_n(\lambda)$.
Indeed, for any Borel subset $B\subset X$ we have $$
\begin{aligned}
\mu(B)&=\mu(B\cap A)=\sum_{n\in\w}\mu(B\cap f_n(2^\w))=\sum_{n\in\w}\mu(f_n(f_n^{-1}(B)))\\
&=
\sum_{n\in\w}\frac1{2^{n+1}}{\cdot}\lambda(f_n^{-1}(B))=
\sum_{n\in\w}\frac1{2^{n+1}}{\cdot}P\!f_n(\lambda)(B).
\end{aligned}
$$
\end{proof}

\begin{claim} The map $\Sigma|\mathcal D:\mathcal D\to P_0(X)$ is open.
\end{claim}

\begin{proof} To show that the map $\Sigma|\mathcal D$ is open, fix any sequence $(f_n)_{n\in\w}\in\mathcal D$ and any open set $\W\subset\mathcal D\subset\C(2^\w,X)^\w$ containing $(f_n)_{n\in\w}$.
 Find $m\in\w$ and $\e>0$ such that each sequence $(g_n)_{n\in\w}\in\mathcal D$ with $\max_{n<m}\hat\rho(f_n,g_n)<3\e$ belongs to the open set $\W$. Find a finite disjoint open cover $\U$ of $2^\w$ such that for every $U\in\U$ and every $k<m$ the set $f_k(U)$ has diameter $<\e$ in $X$. We can assume that $\U$ is of the form $\U=\{U_s:s\in 2^l\}$ for some $l\in\w$ where $U_s=\{t\in 2^\w:t|l=s\}$ for $s\in 2^l$. In this case each set $U\in\U$ has Haar measure $\lambda(U)=\frac1{2^l}$.

For a subset $S\subset X$ by $B(S;\e)=\bigcup_{x\in S}B(x;\e)$ we denote the open $\e$-neighborhood of $S$ in the metric space $(X,\rho)$.
Since $(f_n)_{n\in\w}\in\mathcal D\subset\mathcal D_3$, the family $\big(f_k(2^\w)\big)_{k\in\w}$
is disjoint. So, we can find a positive $\delta<\e$ such that
\begin{itemize}
\item the family $\big(B(f_k(2^\w);\delta)\big)_{k<m}$ is disjoint, and
\item for every $k<m$ the family $\big(B(f_k(U);\delta)\big)_{U\in\U}$ is disjoint.
\end{itemize}

Consider the measure $\mu=\sum_{n\in\w}\frac1{2^{n+1}}Pf_n(\lambda)\in P_0(X)$.
The density of $\bigcup_{n\in\w}f_n(2^\w)$ in $X$ implies that for every $U\in\U$ and every $k\in\w$ the $\delta$-neighborhood $B(f_k(U);\delta)$ of the nowhere dense compact set $f_k(U)\subset X\setminus D$ has measure $\mu(B(f_k(U);\delta))>\frac1{2^{k+1}}\cdot\lambda(U)$.

Then
$$\mathcal V:=\bigcap_{U\in\U}\bigcap_{n<m}\big\{\nu\in P_0(X):\nu\big(B(f_n(U);\delta)\big)>\tfrac1{2^{n+1}}\cdot\lambda(U)\big\}$$is an open neighborhood of the measure $\mu$ in $P_0(X)$. We claim that $\V\subset \Sigma(\W)$.
Fix any measure $\nu\in\V$. For every $n<m$ and $U\in\U$ the $\delta$-neighborhood $B(f_n(U);\delta)$ has measure $\nu(B(f_n(U);\delta))>\frac1{2^{n+1}}\cdot\lambda(U)$. Consider the probability measure $\nu'$ on $X$ defined by $\nu'(S)=\frac{\nu(S\cap B(f_n(U);\delta))}{\nu(B(f_n(U);\delta))}$ for every Borel subset $S\subset X$.

Using Lemma~\ref{l:a}, for the number $a_{n,U}=\frac{\lambda(U)}{2^{n+1}\cdot\nu(B(f_n(U);\delta))}<1$ find an injective continuous map $g_{n,U}:U\to B(f_n(U);\delta)\setminus D$ such that for every Borel subset $S\subset U$ we get
$$\nu'(g_{n,U}(S))=a_{n,U}\cdot\frac{\lambda(S)}{\lambda(U)}=
\frac{1}{2^{n+1}\cdot\nu(B(f_n(U);\delta))}\cdot\lambda(S)$$and hence
\begin{equation}\label{eq:U}
\nu(g_{n,U}(S))=\nu(B(f_n(U);\delta))\cdot \nu'(g_{n,U}(S))=\frac{1}{2^{n+1}}\cdot\lambda(S).
\end{equation}

Observe that for every $x\in U$ we get $$\rho(f_n(x),g_{n,U}(x))\le\diam B(f_n(U);\delta)<2\delta+\diam f_n(U)<2\delta+\e\le 3\e.$$

By the choice of $\delta$, for every $n<m$ the family $\big(B(f_n(U);\delta)\big)_{U\in\U}$ is disjoint and so is the family $\big(g_{n,U}(U)\big)_{U\in\U}$. Then the map $g_n:2^\w\to X$ defined by $g_n|U=g_{n,U}$ for $U\in\U$, is injective and $$g_n(2^\w)\subset \bigcup_{U\in\U}B(f_n(U);\delta)=B(f_n(2^\w);\delta).$$
Moreover, $\hat\rho(g_n,f_n)=\max_{U\in\U}\max_{x\in U}\rho(g_{n,U}(x),f_n(x))<3\e$.

Since the family $\big(B(f_n(2^\w);\delta)\big)_{n<m}$ is disjoint, so is the family $(g_n(2^\w))_{n<m}$.
The equality (\ref{eq:U}) implies that $\nu(g_n(S))=\frac1{2^{n+1}}\cdot\lambda(S)$ for any $n<m$ and any Borel subset $S\subset 2^\w$. In particular, $\nu(g_n(2^\w))=\frac1{2^{n+1}}$. Consider the compact nowhere dense set $F=\bigcup_{n<m}g_n(2^\w)$ and observe that it has measure $\nu(F)=\sum_{n<m}\nu(g_n(2^\w))=\sum_{n<m}\frac1{2^{n+1}}=1-\frac1{2^{m}}$.

Now consider the continuous probability measure $\nu'\in P_0(X)$ defined by $\nu'(S)=\frac{\nu(S\setminus F)}{\nu(X\setminus F)}=2^m\cdot\nu(S\setminus F)$ and observe that $\nu'(F)=0$. The strict positivity of the measure $\nu$ and the nowhere density of $F$ in $X$ imply that the measure $\nu'$ is strictly positive and hence $\nu'\in P_0(X)$. By Claim~\ref{l:sur}, there exists a sequence $(g'_n)_{n\in\w}\in\mathcal D$ such that $\bigcup_{n\in\w}g_n'(2^\w)\subset X\setminus F$ and $\nu'=\sum_{n\in\w}\frac1{2^{n+1}}\cdot Pg_n'(\lambda)$.
Let $g_n:=g'_{n-m}$ for every $n\ge m$. Observe that the sequence $(g_n)_{n\in\w}$ belongs to the set $\mathcal D$ and $\max_{n< m}\hat\rho(g_n,f_n)<3\e$. So, $(g_n)_{n\in\w}\in\W$. We claim that $\nu=\sum_{n\in\w}\frac1{2^{n+1}}\cdot Pg_n(\lambda)$.
Indeed, for every Borel set $S\subset X$ we have
$$
\begin{aligned}
\nu(S)&=\nu(S\cap F)+\nu(S\setminus F)=\sum_{n<m}\nu(S\cap g_n(2^\w))+\nu(X\setminus F)\cdot\nu'(S)\\
&=\sum_{n<m}\nu(g_n(g_n^{-1}(S)))+\frac1{2^{m}}
\sum_{n\in\w}\frac1{2^{n+1}}\cdot Pg_n'(\lambda)(S)\\
&=\sum_{n<m}\frac1{2^{n+1}}\cdot\lambda(g_n^{-1}(S))+
\sum_{n\in\w}\frac1{2^{m+n+1}}\cdot Pg_{n+m}(\lambda)(S)=\sum_{n\in\w}\frac1{2^{n+1}}\cdot Pg_n(\lambda)(S).
\end{aligned}
$$
So, $\nu=\sum_{n\in\w}\frac1{2^{n+1}}\cdot Pg_n(\lambda)=\Sigma\big((g_n)_{n\in\w}\big)\in\Sigma(\W)$.
\end{proof}

 By \cite[8.19]{K}, the space $P_0(X)$ is Polish, being an open continuous image of the Polish space $\mathcal D$. Then $P_0(X)$ is a dense $G_\delta$-set in $P(X)$ and $T(A)$ is comeager in $P(X)$ if and only if $T(A)\cap P_0(X)$ is comeager in $P_0(X)$.
Since $(\Sigma|\mathcal D)^{-1}(T(A))=\mathcal D\cap W_{\N}(A)^\w$, we can apply Lemma~\ref{l:M-coM}(2) and conclude that the set $T(A)\cap P_0(X)$ is comeager in $P_0(X)$ if and only if $\mathcal D\cap W_{\N}(A)^\w$ is comeager in $\mathcal D$ if and only if $W_\N(A)^\w$ is comeager in $\C(2^\w,X)^\w$ if and only if $W_{\N}(A)$ is comeager in $\C(2^\w,X)$ if and only if $A$ is generically Haar-null in $X$.
\smallskip

2. Now we prove the second statement of Theorem~\ref{t:gen-char}. Consider the continuous map $r:\E(2^\w,X)\to\K(X)$, $r:f\mapsto f(2^\w)$. By Lemma~\ref{E-in-C}, the subspace $\E(2^\w,X)$ of injective maps is dense $G_\delta$ in $C(2^\w,X)$. The image $r(\E(2^\w,X))$ coincides with the subspace $\K_0(X)$ consisting of topological copies of the Cantor cube $2^\w$ in $X$.

\begin{claim} The surjective map $r:\E(2^\w,X)\to\K_0(X)$ is open. \end{claim}

\begin{proof} To show that $r$ is open, fix any function $f\in B(2^\w,X)$. Given any $\e>0$ we should find an open set $\W\subset\K_0(X)$ such that $f(2^\w)\in\W$ and for any compact set $K\in\W$ there exists an injective map $g\in \E(2^\w,X)$ such that $\hat\rho(g,f)<\e$ and $g(2^\w)=K$.

By the continuity of $f$, there exists a finite disjoint cover $\V$ of $2^\w$ by clopen sets such that $\diam f(V)<\frac12\e$ for every $V\in\V$. Choose a positive real number $\delta<\frac14{\e}$ such that $\delta\le\frac12\{\rho(f(x),f(y)):x\in U,\;y\in V\}$ for any distinct sets $U,V\in\V$. We claim that the open subset $$\W=\{K\in\K_0(X):K\subset\bigcup_{V\in\V}B(f(V);\delta)\}\cap\bigcap_{V\in\V}\{K\in\K_0(X):K\cap B(f(V);\delta)\ne\emptyset\}$$of $\K_0(X)$ has the required property.

 Indeed, take any set $K\in\W$. Since $\{K\cap B(f(V);\delta)\}_{V\in\V}$ is a disjoint open cover of $K$, for every $V\in\V$ the set $K_V:=K\cap B(f(V);\delta)$ is clopen in $K$. The definition of $\W$ ensures that $K_V$ is not empty. Since $K$ is homeomorphic to the Cantor cube, so is its clopen subspace $K_V$. Then we can find a homeomorphism $g_V:V\to K_V$.
Observe that for every $x\in V$ we have
$$\rho(g_V(x),f(x))\le \diam B(f(V);\delta)\le 2\delta+\diam f(V)<\e.$$

Since the family $\big(B(f(V);\delta)\big)_{V\in\V}$ is disjoint, the map $g:2^\w\to X$ defined by $g|V=g_V$ for $V\in\V$ is injective. It is clear that $\hat\rho(g,f)<\e$ and $g(2^\w)=K$.
\end{proof}

 By \cite[5.5.8]{Eng}, the space $\K_0(X)$ is Polish, being an open continuous image of a Polish space. By \cite[3.9.A]{Eng}, the Polish subspace $\K_0(X)$ is a dense $G_\delta$-set in $\K(X)$.

Observe that $$W_\M(A)\cap \E(2^\w,X)=r^{-1}(K_\M(A)\cap\K_0(X)).$$

If the set $K_\M(A)$ is comeager in $\K(X)$, then $K_\M(A)\cap \mathcal{K}_0(X)$ is comeager in the dense $G_\delta$-set $\mathcal{K}_0(X)$ of $\K(X)$. Then by Lemma~\ref{l:M-coM}(2), the set $r^{-1}(K_\M(A)\cap \mathcal{K}_0(X))\subset W_{\M}(A)$ is comeager in $\E(2^\w,X)$. Since $\E(2^\w,X)$ is a dense $G_\delta$-set in $\C(2^\w,X)$, the set $W_{\M}(A)$ is comeager in $\C(2^\w,X)$, which means that $A$ is generically Haar-$\M$.

If $A$ is generically Haar-$\M$, then its witness set $W_\M(A)$ is comeager in $\C(2^\w,X)$, the intersection $W_\M(A)\cap\E(2^\w,X)$ is comeager in $\E(2^\w,X)$, its image $r(W_\M(A)\cap \E(2^\w,X))=K_\M(A)\cap \K_0(X)$ is comeager in $\K_0(X)$ (by Lemma~\ref{l:M-coM}), and finally the set $\K_\M(A)$ is comeager in $\K(X)$.
\smallskip

3--5. The statements (3)--(5) of Theorem~\ref{t:gen-char} can be proved by analogy with the second statement.
\end{proof}

\begin{theorem}\label{t:MK} A nonempty subset $A$ of a Polish group $X$ is generically Haar-$n$ in $X$ for some $n\in\IN$ if the set $X{\cdot} A^{n+1}:=\{(x+a_0,\dots,x+a_n):x\in X,\;a_0,\dots,a_n\in A\}$ is meager in $X^{n+1}$.
\end{theorem}

\begin{proof} For a compact subset $K\subset X$ let $$K^{n+1}_*=\{(x_0,\dots,x_n)\in K^{n+1}:|\{x_0,\dots,x_n\}|=n+1\}$$ be the set of $(n+1)$-tuples consisting of pairwise distinct points of $K$.

Assume that the set $X{\cdot}A^{n+1}$ is meager in $X^{n+1}$. By Mycielski-Kuratowski Theorem \cite[19.1]{K}, the set $$W_*(A):=\{K\in\K(X):K^{n+1}_*\cap (X{\cdot}A^{n+1})=\emptyset\}$$ is comeager in the hyperspace $\K(X)$. We claim that for every $K\in W_*(A)$ and $x\in X$ the set $K\cap (x+A)$ has cardinality $\le n$. Assuming the opposite, we could find $n+1$ pairwise distinct points $a_0,\dots,a_n\in A$ such that $x+a_i\in K$ for all $i\le n$. Then $(x+a_0,\dots,x+a_n)\in K^{n+1}_*\cap X{\cdot} A^{n+1}$, which contradicts the inclusion $K\in W_*(A)$. This contradiction yields the inclusion $W_*(A)\subset \{K\in\K(X):\forall x\in X\;\;|K\cap (A+x)|\le n\}$. By Theorem~\ref{t:gen-char}(5), the set $A$ is generically Haar-$n$.
\end{proof}

Corollary~\ref{c:St-} implies:

\begin{corollary} For any Polish group the $\sigma$-ideals $\GHN$ and $\GHM$ have the weak Steinhaus property.
\end{corollary}

\begin{remark} The weak Steinhaus property of the $\sigma$-ideal $\GHN$ was proved by Dodos in \cite[Theorem A]{D2}.
\end{remark}

\begin{remark} For the ideals $\GHN$ and $\GHM$ on any locally compact Polish group,
the diagram (\ref{diag1}) of inclusions
obtained in the preceding section takes the following form (for two families $\I,\J$ the arrow $\I\to\J$ indicates that $\I\subset\J$).
$$
\xymatrix{
\sigma\overline{\GHN}\ar@{=}[r]&\sigma\overline{\HN}\ar@{=}[r]&
\sigma\overline{\N}\ar[r]\ar[d]&\mathcal H\sigma\overline{\N}\ar[d]\\
&&\GHN\ar[r]&\N\cap\M\ar[r]&\N\ar@{=}[r]&\HN
}
$$and
$$\sigma\overline{\GHM}=\sigma\overline{\HM}=\sigma\overline{\M}=\GHM=\EHM=\SHM=\HM=\M=\mathcal H\sigma\overline{\M}.$$
These inclusions and equalities can be easily derived from Theorems~\ref{t:sigma-gen}, \ref{t:HM=M}, Proposition~\ref{p:GHI->M} and Remark~\ref{rem:N}.
\end{remark}

\begin{problem}\label{Problem4} Is $\sigma\overline{\N}\ne\GHN\ne\N\cap\M$ in each non-discrete locally compact Polish group?
\end{problem}


\begin{remark} By Remark~\ref{r:Zw}, $\GHN\not\subset\sigma\overline{\HM}$ in the Polish group $\IZ^\w$. Since $\sigma\overline{\HN}\subset\sigma\overline{\HM}$, this implies the inequality $\GHN\ne\sigma\overline{\HN}$ for the Polish group $\IZ^\w$.
\end{remark}

\begin{problem}\label{prob:St-} Have the ideals $\GHM$ and $\GHN$ the Steinhaus property?
\end{problem}

\begin{problem}\label{Problem7} Is each generically Haar-null set in a Polish group Haar-meager?
\end{problem}

\section{Haar-$\I$ sets in countable products of finite groups}\label{s14}

In this section we check the smallness properties of some standard simple sets in countable products of finite groups. On this way we shall find some interesting examples distinguishing between Haar-$\I$ properties for various semi-ideals $\I$.

For a group $G$ let $G^*:=G\setminus\{\theta\}$ be the group with the removed neutral element.

\begin{example}\label{Hai} The compact metrizable group $X=(\IZ/6\IZ)^\w$ contains a closed subset $A$ such that
\begin{enumerate}
\item[\textup{1)}] $A+A$ is nowhere dense in $X$;
\item[\textup{2)}] $A-A=X$;
\item[\textup{3)}] $A$ is generically Haar-2 in $X$;
\item[\textup{4)}] $A$ cannot be written as a countable union of closed sets which are null-1 or Haar-thin in $X$.
\end{enumerate}
\end{example}

\begin{proof} Consider the cyclic group $C_6:=\IZ/6\IZ$ and observe that the set $B=\{6\IZ,1+6\IZ,3+6\IZ\}\subset C_6$ satisfies $B-B=C_6$ and $B+B=\{6\IZ,1+6\IZ,2+6\IZ,3+6\IZ,4+6\IZ\}\ne C_6$. Then for the countable power $A=B^\w$ the set $A+A$ is nowhere dense in $X=C_6^\w$ and $A-A=X$.

Observe that the set $C_6{\cdot}B^3=\{(x+b_0,x+b_1,x+b_2):x\in C_6,\;\;b_0,b_1,b_2\in B\}$ has cardinality $|C_6{\cdot}B^3|\le 6\cdot 3^3<6^3=|C_6^3|$, which implies that the set $$X{\cdot}A^3=\{(x+a_0,x+a_1,x+a_2):x\in X,\;\;a_0,a_1,a_2\in A\}$$ is nowhere dense in $X^3$. By Theorem~\ref{t:MK}, the set $A$ is generically Haar-2 in $X$.

Using Baire Theorem, Proposition~\ref{p:n1} and Theorem~\ref{t:thin}, we can show that the set $A$ cannot be written as a countable union of closed sets which are null-1 or Haar-thin in $X$.
\end{proof}

A subset $A$ of a group $G$ is called \index{subset!additive}\index{additive subset}{\em additive} if $A+A\subset A$. We recall that a subset $A$ is a group $G$ is {\em countably thick} if for any countable set $E\subset G$ there exists $x\in G$ such that $E+x\subset A$.

\begin{example}\label{e:Ruzsa} For any sequence $(G_n)_{n\in\w}$ of finite cyclic groups with $\sup_{n\in\w}|G_n|=\infty$ the compact topological group $G=\prod_{n\in\w}G_n$ contains an additive $\sigma$-compact subset $A\subset G$ such that
\begin{enumerate}
\item[\textup{1)}] $A$ is generically Haar-meager and generically Haar-null in $G$;
\item[\textup{2)}] $A$ is countably thick;
\item[\textup{3)}] $A$ is not null-finite and not Haar-thin;
\item[\textup{4)}] $A$ is not Haar-countable.
\end{enumerate}
\end{example}

\begin{proof} By a result of Haight \cite{Haight}, for every $k\in\IN$ there exists a number $m_k\in\IN$ such that for every integer number $q\ge m_k$ the cyclic group $C_q$ contains a subset $T\subset C_q$ having two properties:
\begin{itemize}
\item for any set $F\subset C_q$ of cardinality $|F|\le 2^k$ there exists $x\in C_q$ such that $F+x\subset T$;
\item the sum-set $T^{+k}:=\{x_1+\dots+x_{k}:x_1,\dots,x_{k}\in T\}$ is not equal to $C_q$.
\end{itemize}

Since $\sup_{n\in\w}|G_n|=\infty$, we can choose an increasing number sequence $(n_k)_{k\in\IN}$ such that $|G_{n_k}|\ge m_k$. For every $k\in\IN$, by the definition of the number $m_k$, there exists a subset $T_{n_k}\subset G_{n_k}$ such that 
 \begin{itemize}
\item[$(a_{k})$] for any set $F\subset G_{n_k}$ of cardinality $|F|\le 2^k$ there exists $x\in G_{n_k}$ such that $F+x\subset T_{n_k}$;
\item[$(b_k)$] the sum-set $T_{n_k}^{+k}:=\{x_1+\dots+x_{k}:x_1,\dots,x_{k}\in T_{n_k}\}$ is not equal to $G_{n_k}$.
\end{itemize}
For any $n\in \w\setminus\{n_k\}_{k\in\w}$ let $T_n:=G_n$. 

In the compact topological group $G=\prod_{n\in\w}G_n$ consider the closed subset $T:=\prod_{n\in\w}T_n$. The conditions $(b_k)$, $k\in\IN$, imply that for every $r\in\IN$ the sum-set $T^{+r}$ is nowhere dense in $G$. In the group $G$ consider the dense countable subgroup $$D:=\{(x_n)_{n\in\w}\in G:\exists n\in\w\;\forall m\ge n\;x_m=\theta_m\},$$ where $\theta_n$ is the neutral element of the group $G_n$.

It follows that $T^{+}:=\bigcup_{r=1}^\infty T^{+r}$ is a meager additive $\sigma$-compact subset of $G$ and so is the sum $A:=D+T^+$. It remains to prove that the set $A$ has properties (1)--(4).

First we fix some notation.  Observe that the topology of the compact group $X=\prod_{k\in\w}G_k$ is generated by the ultrametric
$$\rho(x,y)=\inf\{2^{-n}:n\in\w,\;x|n=y|n\}.$$ 
  For every $n\in\w$ let $\pi_n:G\to G_n$ be the projection onto the $n$th coordinate. 
\smallskip

1. Since the sum-set $A+A\subset A$ is meager in the compact topological group $G$, the (Borel) set $A$ has Haar measure zero by Theorem~\ref{Stein}. By Theorem~\ref{t:HN}(2), the set $A$ is Haar-null in $G$. By Theorem~\ref{t:sigma-gen}(1), $A$ is generically Haar-null. By Corollary~\ref{c:cHI=>HM}, the $F_\sigma$-set $A$ is generically Haar-meager in $G$.
\smallskip

2. To see that $A$ is countably thick, fix any countable set $C=\{c_n\}_{n\in\w}$ in  $G$. The conditions $(a_k)$, $k\in\IN$, imply the existence of an element $s\in G$  such that $$\{\pi_{n_k}(c_0),\dots,\pi_{n_k}(c_{2^k-1})\}+\pi_{n_k}(s)\in T_{n_k}$$ for any $k\in\IN$. We claim that $C+s\subset D+T\subset A$. Indeed, for any $k\in \w$ and $m>n_k$ we have $\pi_{m}(c_k)+\pi_m(s)\in T_m$ and hence $c_k+s\in D+T\subset A$.
\smallskip

3. The countable thickness of $A$ implies that $A-A=G$ and $A$ is not null-finite. Since $A-A=G$, the set $A$ is not Haar-thin by Theorem~\ref{t:thin}. 
\smallskip

4. To prove that $A$ is not Haar-countable, it suffices to show that for any uncountable compact subset $K\subset G$ there exists $x\in G$ such that $T\cap(K+x)$ is uncountable.
Replacing $K$ by a smaller subset, we can assume that the compact set $K$ is homeomorphic to the Cantor cube $2^\w$, see \cite[6.5]{K}. Following the lines of the proof of Theorem~6.2 in \cite{K}, we can construct a family $(U_s)_{s\in 2^{<\w}}$ of nonempty clopen subsets of $K$ such that for every $k\in\w$ and $s\in 2^k\subset 2^{<\w}$
\begin{enumerate}
\item[\textup{(i)}] $U_{s\hat{\;}0}$ and $U_{s\hat{\;}1}$ are disjoint clopen subsets of $U_s$;
\item[\textup{(ii)}] $\diam(U_s)<\frac1{2^{n_k}}$.
\end{enumerate}
Then $K'=\bigcup_{s\in 2^\w}\bigcap_{n\in\w}U_{s|n}$ is a compact subset of $K$, homeomorphic to $2^\w$.

The choice of the metric $\rho$ and the condition (ii) of the inductive construction ensure that for every $k\in\w$ the set $\pr_{n_k}(K')\subset G_{n_k}$ has cardinality $|\pr_{n_k}(K')|\le 2^k$. The conditions $(a_k)$ allow us to find an element $x\in G$ such that $\pr_{n_k}(x)+\pr_{n_k}(K')\subset T_{n_k}$ for every $k\in\w$. Then $x+K'\subset T$ and the set $(x+K)\cap A\supset x+K'$ is uncountable.
\end{proof}

\begin{example}\label{ex:13.11} Let $(G_k)_{k\in\w}$ be a sequence of non-trivial finite groups. For the closed subset $A:=\prod_{k\in\w}G_k^*$ of the compact Polish group $X=\prod_{k\in\w}G_k$ and a natural number $n\in\IN$ the following conditions are equivalent:
\begin{enumerate}
\item[\textup{1)}] $A$ is Haar-$n$;
\item[\textup{2)}] $A$ is generically Haar-$n$;
\item[\textup{3)}] $A$ is null-$n$;
\item[\textup{4)}] $A$ is generically null-$n$;
\item[\textup{5)}] $|G_k|\le n+1$ for infinitely many numbers $k\in\w$.
\end{enumerate}
\end{example}

\begin{proof} We shall prove the implications $(5)\Ra(2)\Leftrightarrow(1)\Ra(3)\Leftrightarrow(4)\Ra(5)$. Among these implications, $(1)\Ra(3)$ is trivial and the equivalences $(1)\Leftrightarrow(2)$ and $(3)\Leftrightarrow(4)$ follow from Theorem~\ref{t:sigma-gen}.

To prove that $(5)\Ra(2)$, assume that the set $\Omega=\{k\in\w:|G_k|\le n+1\}$ is infinite. We claim that for every $k\in\Omega$ the set $G_k\cdot (G_k^*)^{n+1}:=\{(x+a_0,\dots,x+a_n):x\in G_k,\;a_0,\dots,a_n\in G_k^*\}$ is not equal to $G_k^{n+1}$. This follows from the observation that for every $x\in G_k$ and $a_0,\dots,a_n\in G_k^*$ the set $\{x+a_0,\dots,x+a_n\}$ has cardinality $\le|G_k^*|<|G_k|$. On the other hand, the set $G_k^{n+1}$ contains vectors $(x_0,\dots,x_{n})$ of cardinality $\min\{|G_k|,n+1\}=|G_k|$.

Then the set $\prod_{k\in\Omega}G_k\cdot (G_k^*)^{n+1}$ is nowhere dense in $\prod_{k\in\Omega}G_k^{n+1}$ and the set $\prod_{k\in\w}G_k\cdot (G_k^*)^{n+1}=X{\cdot}A^{n+1}$ is nowhere dense in $\prod_{k\in\w}G_k^{n+1}=X^{n+1}$.
By Theorem~\ref{t:MK}, the set $A$ is generically Haar-$n$ in $X$.
\smallskip

To prove that $(3)\Ra(5)$, assume that the condition (5) does not hold. Then the set $F=\{k\in\w:|G_k|\le n+1\}$ is finite. In this case we shall show that the set $A$ is not null-$n$ in $X$. By Proposition~\ref{p:null-char}, it suffices to show that for any null sequence $(z_i)_{i\in\w}$ in $X$ there exists $x\in X$ such that the set $\{i\in\w:z_i\in A+x\}$ has cardinality $>n$. Since the sequence $(z_i)_{i\in\w}$ converges to zero, for every $k\in F$ there exists a number $i_k\in\w$ such that for every $i\ge i_k$ the element $z_i(k)$ is equal to the neutral element $\theta_k$ of the group $G_k$. Choose any subset $E\subset\w$ of cardinality $|E|=n+1$ with $\min E\ge \max_{k\in F}i_k$.
Next choose any function $x\in X=\prod_{k\in\w}G_k$ such that $x(k)\notin \{z_i(k):i\in E\}$ for any $k\in\w$. The choice of $x$ is possible since for any $k\in F$ we get $\{z_i(k)\}_{i\in E}=\{\theta_k\}\ne G_k$ and for any $k\in\w\setminus F$ the set $\{z_i(k):i\in E\}$ has cardinality $\le |E|=n+1<|G_k|$. The choice of $x$ ensures that for every $i\in E$ and $k\in\w$ we get $z_i(k)\ne x(k)$ and hence $z_i(k)\in G_k^*+x(k)$. Then $\{z_i\}_{i\in E}\subset A+x$, which implies that the set $\{i\in\w:z_i\in A+x\}\supset E$ has cardinality $\ge |E|=n+1$.
\end{proof}

\begin{example}\label{ex:A} Let $(G_k)_{k\in\w}$ be a sequence of non-trivial finite groups. For the closed subset $A:=\prod_{k\in\w}G_k^*$ of the compact Polish group $X=\prod_{k\in\w}G_k$ the following conditions are equivalent:
\begin{enumerate}
\item[\textup{1)}] $A$ is Haar-countable;
\item[\textup{2)}] $A$ is Haar-finite;
\item[\textup{3)}] $A$ is null-finite;
\item[\textup{4)}] $A$ is Haar-$n$ for some $n\in\IN$;
\item[\textup{5)}] $A$ is generically Haar-$n$ for some $n\in\IN$;
\item[\textup{6)}] $A$ is null-$n$ for some $n\in\IN$;
\item[\textup{7)}] $A$ is generically null-$n$ for some $n\in\IN$;
\item[\textup{8)}] $\lim_{k\to\infty}|G_k|\ne\infty$.
\end{enumerate}
\end{example}

\begin{proof} The equivalences $(4)\Leftrightarrow(5)\Leftrightarrow(6)\Leftrightarrow(7)
\Leftrightarrow(8)$ follow from the corresponding equivalences in Example~\ref{ex:13.11}.
The implications $(4)\Ra(2)\Ra(1,3)$ are trivial and $(3)\Ra(8)$ is proved in Example 2.4 of \cite{BJ}. So, it remains to prove $(1)\Ra(8)$.

Assume that $\lim_{k\to\infty}|G_k|=\infty$. To prove that $A$ is not Haar-countable, it suffices to show that for any uncountable compact subset $K\subset X$ there exists $x\in X$ such that $A\cap(K+x)$ is uncountable.

Since $\lim_{k\to\infty}|G_k|=\infty$, we can choose an increasing number sequence $(k_n)_{n\in\w}$ such that for every $n\in\w$ and $k\ge k_n$ the group $G_k$ has cardinality $|G_{k}|>2^n$. Observe that the topology of the compact group $X=\prod_{k\in\w}G_k$ is generated by the ultrametric
$$\rho(x,y)=\inf\{2^{-n}:n\in\w,\;x|n=y|n\}.$$

Replacing $K$ by a smaller subset, we can assume that the compact set $K$ is homeomorphic to the Cantor cube $2^\w$, see \cite[6.5]{K}. Following the lines of the proof of Theorem~6.2 in \cite{K}, we can construct a family $(U_s)_{s\in 2^{<\w}}$ of nonempty clopen subsets of $K$ such that for every $n\in\w$ and $s\in 2^n\subset 2^{<\w}$
\begin{enumerate}
\item[\textup{(1)}] $U_{s\hat{\;}0}$ and $U_{s\hat{\;}1}$ are disjoint clopen subsets of $U_s$;
\item[\textup{(2)}] $\diam(U_s)<\frac1{2^{k_{n+1}}}$.
\end{enumerate}
Then $K'=\bigcup_{s\in 2^\w}\bigcap_{n\in\w}U_{s|n}$ is a compact subset of $K$, homeomorphic to $2^\w$.

For every $k\in\w$ denote by $\pr_k:X\to G_k$, $\pr_k:x\mapsto x(k)$, the projection of the product $X=\prod_{n\in\w}G_n$ onto the $k$-th factor $G_k$.
The choice of the metric $\rho$ and the condition (2) of the inductive construction ensure that for every $n\in\w$ and $k\in[k_n,k_{n+1})$ the set $\pr_{k}(K')\subset G_{k}$ has cardinality $|\pr_{k}(K')|\le 2^n<|G_{k}|$. This allows us to find an element $x\in X$ such that $x(k)+\pr_k(K')\subset G_k^*$ for every $k\in\w$. Then $x+K'\subset\prod_{k\in\w}G_k^*=A$ and the set $(x+K)\cap A\supset x+K'$ is uncountable.
\end{proof}

\begin{example}\label{ex:hard} Let $(G_k)_{k\in\w}$ be a sequence of non-trivial finite groups, $X=\prod_{k\in\w}G_k$, $A:=\prod_{k\in\w}G_k^*$, $\alpha$ be any element of $A$ and $B:=\{x\in A:x(n)=\alpha(n)$ for infinitely many numbers $n\}$. The set $B$ has the following properties:
\begin{enumerate}
\item[\textup{1)}] $B$ is a dense $G_\delta$-set in $A$;
\item[\textup{2)}] $B-A$ has empty interior in $X$;
\item[\textup{3)}] $B$ is null-finite in $X$;
\item[\textup{4)}] $B$ is Haar-thin in $X$;
\item[\textup{5)}] for any continuous map $f:2^\w\to A$ there exists $x\in X$ such that $f^{-1}(B+x)\notin\sigma{\overline{\N}}$;
\item[\textup{6)}] if $\lim_{k\to\infty}|G_k|=\infty$, then $B$ is not Haar-countable and not null-$n$ for all $n\in\IN$;
\item[\textup{7)}] if $\prod_{k\in\w}\frac{|G_k|-1}{|G_k|}>0$, then $B$ is not Haar-$\sigma\overline{\N}$ in $X$ and consequently $B$ does not belong to the $\sigma$-ideals $\sigma\overline{\N}\subset\mathcal H\sigma\overline{\N}$.
\end{enumerate}
\end{example}

\begin{proof} 1. For every $n\in\w$ consider the closed nowhere dense set $A_n=\{x\in A:\forall k\ge n\;\;x(k)\ne \alpha(k)\}$ in $A$. Since $B=A\setminus\bigcup_{n\in\w}A_n$, the set $B$ is a dense $G_\delta$-set in $A$.
\smallskip

2. Observe that the set $$D=\{x\in X:\exists n\in\w\;\forall k\ge n\;\;x(k)=\alpha(k)\}$$ is countable and dense in $X$. We claim that $(B-A)\cap D=\emptyset$. In the opposite case, we could find points $x\in D$, $b\in B$, $a\in A$ such that $x=b-a$. By the definition of the set $D$, there exists $n\in\w$ such that $x(k)=\alpha(k)$ for all $k\ge n$. By the definition of $B$, there exists $k\ge n$ with $b(k)=\alpha(k)$. Then $a(k)=b(k)-x(k)=\alpha(k)-\alpha(k)\notin G_k^*$ and hence $a\notin A$, which is a desirable contradiction. So, $B-A$ does not intersect the dense set $D$ and hence $B-A=B-\bar B$ has empty interior.
\smallskip

3. Taking into account that the set $B-\bar B=B-A$ has empty interior, we can apply Theorem~\ref{t:St-NF} and conclude that the set $B$ is null-finite.
\smallskip

4. Since the set $B-B\subset B-A$ has empty interior, it is not a neighborhood of zero in $X$. Applying Theorem~\ref{t:thin}, we conclude that the set $B$ is Haar-thin.
\smallskip

5. Given any continuous map $f:2^\w\to A$, we shall find $x\in X$ such that $f^{-1}(B+x)\notin\sigma{\overline{\N}}$ in $2^\w$. We endow the Cantor cube with the standard product measure $\lambda$, called the Haar measure on $2^\w$.

For every $n\in \w$ consider the clopen set $Z_n=\{x\in X:x(n)=\theta(n)\}$ in $X$.
Here by $\theta$ we denote the neutral element of the group $X$. Separately we shall consider two cases.
\smallskip


I) For some $x\in X$ the sequence $\big(\lambda(f^{-1}(x+Z_n))\big)_{n\in\w}$ does not converge to zero. In this case we can find $\e>0$ and an infinite set $\Omega\subset\w$ such that for every $n\in\Omega$ the set $K_n:=f^{-1}(x+ Z_n)$ has Haar measure $\lambda(K_n)\ge\e$. Since the hyperspace $\K(2^\w)$ of compact subsets of $2^\w$ is compact, we can replace $\Omega$ by a smaller infinite set and assume that the sequence $(K_n)_{n\in\Omega}$ converges to some compact set $K_\infty\in\K(2^\w)$. By the regularity of the Haar measure $\lambda$ on $2^\w$, the set $K_\infty$ has Haar measure $\lambda(K_\infty)\ge \e$. Using the regularity of the Haar measure $\lambda$ on $2^\w$ and the convergence of the sequence $(K_n)_{n\in\Omega}$ to $K_\infty$, we can choose an infinite subset $\Lambda\subset\Omega$ such that the intersection $K:=\bigcap_{n\in\Lambda}K_n$ has positive Haar measure. Now take the element $y\in X$ defined by
$$y(n)=\begin{cases}
x(n)-\alpha(n)&\mbox{if $n\in\Lambda$};\\
\theta(n)&\mbox{otherwise}.
\end{cases}
$$
The inclusions $f(K)\subset f(K_n)\subset A\cap(x+Z_n)$ holding for all $n\in\Lambda$ ensure that $$f(K)\subset\prod_{n\in\Lambda}\{x(n)\}\times \prod_{n\in\w\setminus\Lambda}G^*_n=y+\prod_{n\in\Lambda}\{\alpha(n)\}\times \prod_{n\in\w\setminus\Lambda}G^*_n\subset y+B$$and hence the set $f^{-1}(y+B)\supset K$ does not belong to the ideals $\N\supset \sigma\overline{\N}$ in $2^\w$.
\smallskip

II) Next, consider the second (more difficult) case: $\lim_{n\to\infty}\lambda(f^{-1}(x+Z_n))=0$ for every $x\in X$. 
Our assumption implies that $$\lim_{n\to\infty}\lambda\big(f^{-1}(x\pm \alpha+Z_n)\big)=0$$ for every $x\in X$. Here $x\pm \alpha+Z_n:=(x+\alpha+Z_n)\cup(x-\alpha+Z_n)$.

For a closed subset $F\subset X$ let $\supp(\lambda|F)$ be the set of all points $x\in F$ such that for each neighborhood $O_x\subset 2^\w$ the intersection $O_x\cap F$ has positive measure $\lambda(O_x\cap F)$. The $\sigma$-additivity of the measure $\lambda$ implies that $\lambda(\supp(\lambda|F))=\lambda(F)$.

On the Cantor cube $2^\w$ we consider the metric $\rho(x,y):=\inf\{2^{-n}:n\in\w,\;x|n=y|n\}$.
For a point $x\in 2^\w$ and $\e>0$ by $B(x;\e)=\{y\in 2^\w:\rho(x,y)<\e\}$ we denote the open $\e$-ball centered at $x$ in the metric space $(2^\w,\rho)$.

By induction we shall construct a decreasing sequence $(F_n)_{n\in\w}$ of closed subsets of $2^\w$, a decreasing sequence $(\Omega_n)_{n\in\w}$ of infinite subsets of $\w$ and a sequence of points $(x_n)_{n\in \w}$ of $2^\w$ such that for every $n\in\w$ the following conditions are satisfied:
\begin{itemize}
\item[(i)] $\{x_k\}_{k\le n}\subset F_n$;
\item[(ii)] $\min_{k<n}\rho(x_n,x_k)=\max_{x\in F_{n-1}}\min_{k<n}\rho(x,x_k)$;
\item[(iii)] $\Omega_{n-1}\setminus \Omega_n$ is infinite and $\min\Omega_n>n$;
\item[(iv)] for every $k\le m\le n$ we have $\sum_{i\in\Omega_n}\lambda(f^{-1}(f(x_n)\pm\alpha+ Z_i))<\frac1{3^{n+1}}\lambda(B(x_k;\frac1{2^m})\cap F_{n-1})$;
\item[(v)] for every $k\le n$ and $m\in\Omega_n$ we have $\sum_{m< i\in\Omega_n}\lambda(f^{-1}(f(x_n)\pm \alpha+Z_i))<\frac12\lambda(W_{m,n}(x_k))$ where $W_{m,n}(x_k):=B(x_k;\frac1{2^m})\cap F_{n-1}\setminus \bigcup_{m\ge i\in\Omega_{n}}f^{-1}(f(x_n)\pm \alpha+Z_i)$;
\item[(vi)] $F_n=\supp(\lambda|E_n)$ where $E_n=F_{n-1}\setminus \bigcup_{i\in\Omega_n}f^{-1}(f(x_n)\pm \alpha+Z_i)$.
\end{itemize}

We start the inductive construction letting $F_{-1}=F_0=2^\w$, $\Omega_{-1}=\w$ and $x_0$  any point of $F_0$. Observe that for every $i\in\w$ the point $f(x_0)$ does not belong to the clopen set $f(x_0)\pm \alpha+Z_i$. This implies that for every $m\in\w$ the set $W_{m,0}(x_0)=B(x_0;\frac1{2^m})\cap F_{-1}\setminus\bigcup_{i\le m}f^{-1}(f(x_0)\pm \alpha+Z_i)$ has positive Haar measure.
Since the sequence $\big(\lambda(f^{-1}(f(x_0)\pm\alpha+Z_i))\big)_{i\in\w}$ tends to zero, we can choose an infinite set $\Omega_0\subset\w$ with infinite complement $\w\setminus\Omega_0$ such that
    $\sum_{i\in\Omega_0}\lambda(f^{-1}(f(x_0)\pm\alpha+ Z_i))<\frac13$ and
for every $m\in\Omega_0$ we have $\sum_{m< i\in\Omega_0}\lambda(f^{-1}(f(x_0)\pm\alpha+Z_i))<\frac12\lambda(W_{m,0}(x_0))$.
It is easy to see that the sets $F_0,\Omega_0$ and point $x_0$ satisfy the inductive conditions (i)--(vi) for $n=0$.

Now assume that for some $n\in\IN$ we have constructed closed subsets $F_0\supset\dots\supset F_{n-1}$, infinite sets $\Omega_0\supset\dots\supset\Omega_{n-1}$ and points $x_0,\dots,x_{n-1}\in F_{n-1}$ satisfying the inductive conditions (i)--(vi).
Choose any point $x_n\in F_{n-1}$ satisfying the inductive condition (ii).
Since $F_{n-1}=\supp(\lambda|F_{n-1})$ and the sequence $\big(\lambda(f^{-1}(f(x_n)\pm \alpha+Z_i))\big)_{i\in\w}$ converges to zero, we can choose an infinite subset $\Omega_n\subset \Omega_{n-1}$ satisfying the inductive conditions (iii)--(v). Put $F_n=\supp(\lambda|E_n)$ where $E_n=F_{n-1}\setminus\bigcup_{i\in\Omega_n}f^{-1}(f(x_n)\pm \alpha+Z_i)$.

We claim that for every $k\le n$ the point $x_k$ belongs to $F_n$. If $k<n$, then $x_k\in F_{n-1}=\supp(\lambda|F_{n-1})$ by the inductive assumption. Taking into account that $x_n\in F_{n-1}\subset F_k\subset X\setminus\bigcup_{i\in\Omega_k}f^{-1}(f(x_k)\pm \alpha+Z_i)$, we conclude that $f(x_n)\notin f(x_k)\pm\alpha+Z_i$ for every $i\in\Omega_k$ and hence $f(x_k)\notin f(x_n)\pm\alpha-Z_i=f(x_n)\pm\alpha+Z_i$ for every $i\in\Omega_k\supset\Omega_n$. Then $x_k\in E_n=F_{n-1}\setminus\bigcup_{i\in\Omega_n}f^{-1}(f(x_n)\pm\alpha+Z_i)$.

Since $x_n\notin\bigcup_{i\in\Omega_n}f^{-1}(f(x_n)\pm\alpha+Z_i)$, the point $x_n$ also belongs to $E_n$.

To see that for every $k\le n$ the point $x_k$ belongs to $F_n:=\supp(\lambda|E_n)$, it suffices to check that for every $m\in\Omega_n$ the neighborhood $B(x_k;\frac1{2^m})\cap E_n$ of $x_k$ in $E_n$ has positive Haar measure. For this observe that
$$
\begin{aligned}
\lambda(B(x_k;\tfrac1{2^m})\cap E_n)&=\lambda\big(B(x_k;\tfrac1{2^m})\cap F_{n-1}\setminus\bigcup_{i\in \Omega_n}f^{-1}(f(x_n)\pm\alpha+Z_i)\big)\\
&\ge \lambda\big(B(x_k;\tfrac1{2^m})\cap F_{n-1}\setminus\bigcup_{m\ge i\in\Omega_n}f^{-1}(f(x_n)\pm \alpha+Z_i)\big)\\
&-\sum_{m< i\in\Omega_n}\lambda\big(f^{-1}(f(x_n)\pm\alpha+Z_i)\big)\\
&=\lambda\big(W_{m,n}(x_k)\big)-\sum_{m< i\in\Omega_n}\lambda\big(f^{-1}(f(x_n)\pm\alpha+Z_i)\big)\\
&>\big(1-\tfrac12\big)\cdot \lambda(W_{m,n}(x_k))>0,
\end{aligned}
$$
by the inductive condition (v).

Thus the sets $F_n$, $\Omega_n$ and the point $x_n$ satisfy the inductive conditions (i)--(vi).

After completing the inductive construction, consider the closed set $F=\bigcap_{n\in\w}F_n$. We claim that the set $F$ has positive measure $\lambda(F)$.
The inductive condition (iv) implies that $$\lambda(F_n)=\lambda(E_n)\ge\lambda(F_{n-1})-\sum_{i\in \Omega_n}\lambda(f^{-1}(f(x_n)\pm \alpha+Z_i))>\lambda(F_{n-1})-\frac1{3^{n+1}}
$$
and hence $$\lambda(F_n)>\lambda(F_0)-\sum_{k=1}^{n}\frac1{3^{k+1}}>1-\sum_{k=1}^\infty \frac1{3^{k+1}}=\frac56.$$
Then $\lambda(F)=\inf_{n\in\w}\lambda(F_n)\ge\frac56$.
The inductive condition (i) guarantees that $\{x_n\}_{n\in\w}\subset F$.

Let us show that the set $\{x_n\}_{n\in\w}$ is dense in $F$. Assuming that $\{x_n\}_{n\in\w}$ is not dense in $F$, we can find a point $x\in F$ and positive number $\delta>0$ such that $\rho(x,x_n)\ge\delta$ for all $n\in\w$. Now the inductive condition (ii) ensures that $\rho(x_n,x_k)\ge \delta$ for all $n>k$, which contradicts the compactness of the Cantor cube $2^\w$.

Next, we show that for all numbers $k\le m$ the set $F\cap B(x_k;\tfrac1{2^m})$ has positive Haar measure. The $\sigma$-additivity of the measure $\lambda$ ensures that $\lambda(F\cap B(x_k;\tfrac1{2^m}))=\inf_{n>m}\lambda(F_n\cap B(x_k;\tfrac1{2^m}))$. By the inductive conditions (iii) and (iv), for every $n>m$ we have
$$
\begin{aligned}
\lambda\big(B(x_k;\tfrac1{2^m})\cap F_n\big)&=\lambda\big(B(x_k;\tfrac1{2^m})\cap E_n\big)\\
&=\lambda\big(B(x_k;\tfrac1{2^m})\cap F_{n-1}\setminus\bigcup_{i\in \Omega_n}f^{-1}(f(x_n)\pm \alpha+Z_i)\big)\\
&\ge\lambda\big(B(x_k;\tfrac1{2^m})\cap F_{n-1}\big)-\sum_{i\in\Omega_n}\lambda(f^{-1}(f(x_n)\pm \alpha+Z_i))\\
&>\lambda\big(B(x_k;\tfrac1{2^m})\cap F_{n-1}\big)\cdot
\big(1-\tfrac1{3^{n+1}}\big)
\end{aligned}
$$
and hence $$\lambda\big(B(x_k;\tfrac1{2^m})\cap F_n\big)>\lambda\big(B(x_k;\tfrac1{2^m})\cap F_m\big)\cdot\prod_{i=m+1}^{n}\big(1-\tfrac1{3^{i+1}}\big).$$
Then
$$\lambda\big(B(x_k;\tfrac1{2^m})\cap F\big)=\lim_{n\to\infty}\lambda\big(B(x_k;\tfrac1{2^m})\cap F_n\big)\ge \lambda\big(B(x_k;\tfrac1{2^m})\cap F_m\big)\cdot\prod_{i=m+1}^{\infty}\big(1-\tfrac1{3^{i+1}}\big)>0.$$

Therefore, $\{x_n\}_{n\in\w}\subset\supp(\lambda|F)$ and $F=\overline{\{x_n\}}_{n\in\w}=\supp(\lambda|F)$. By the continuity of the measure $\lambda|F$ its support $F$ has no isolated points. Then the dense set $\{x_n\}_{n\in\w}$ in $F$ also does not have isolated points.

For every $n\in\w$ let $y_n:=f(x_n)\in A$.
Now consider the function $y\in X$ defined by
$$y(i)=\begin{cases}\alpha(i)-y_n(i)&\mbox{if $i\in \Omega_n\setminus\Omega_{n+1}$ for some $n\in\w$};\\
\theta(i)&\mbox{if $i\notin\Omega_0$}.
\end{cases}
$$
We claim that for every $k\in\w$ the point $y_k+y$ belongs to the set $B$.
First we show that $y_k+y\in A$. The inclusion $y_k+y\in A$ will follow as soon as we check that $y_k(i)+y(i)\ne \theta(i)$ for all $i\in\w$. If $i\notin\Omega_0$, then
$y_k(i)+y(i)=y_k(i)\ne \theta(i)$ as $y_k\in f(2^\w)\subset A$. If $i\in\Omega_0$, then there exists a unique number $n\in\w$ such that $i\in\Omega_n\setminus\Omega_{n+1}$ and hence $y_k(i)+y(i)=y_k(i)+\alpha(i)-y_n(i)$. It follows from $y_k\notin y_n\pm\alpha+Z_i$ that $y_k(i)\ne y_n(i)-\alpha(i)$ and hence $y_k(i)+y(i)=y_k(i)+\alpha(i)-y_n(i)\ne\theta(i)$. This completes the proof of inclusion $y_k+y\in A$.
Next, observe that for any $i\in\Omega_k\setminus\Omega_{k+1}$ we get $y_k(i)+y(i)=y_k(i)+\alpha(i)-y_k(i)=\alpha(i)$, which implies that $y_k+y\in B$.

It follows that the $G_\delta$-set $f^{-1}(B-y)$ contains the set $\{x_k\}_{k\in\w}$. We claim that the set $f^{-1}(B-y)$ does not belong to the $\sigma$-ideal $\sigma\overline{\N}$. In the opposite case $f^{-1}(B-y)\subset \bigcup\mathcal Z$
for some countable family $\mathcal Z$ of closed sets of Haar measure zero in $2^\w$.
Applying the Baire Theorem to the dense Polish subspace $P=F\cap f^{-1}(B-y)$ of $F$, we can find a nonempty open set $U\subset F$ such that $U\cap P\subset Z$ for some $Z\in\mathcal Z$. Taking into account that $U\cap P$ is dense in $U$, we conclude that $U\subset \bar U=\overline{U\cap P}\subset \bar Z=Z$ and hence $U$ has measure $\lambda(U)\le\lambda(Z)=0$, which contradicts the equality $F=\supp(\lambda|F)$.
\smallskip

6. Assuming that $\lim_{k\to\infty}|G_k|=\infty$, we shall prove that the set $B$ is not Haar-countable. To derive a contradiction, assume that $B$ is Haar-countable and find an uncountable compact set $K\subset X$ such that for every $x\in X$ the set $(K+x)\cap B$ is countable. By Example~\ref{ex:A}, the closed subset $A$ of $X$ is not Haar-countable, so there exists $z\in X$ such that $(K+z)\cap A$ is uncountable. Replacing $K$ by $(K+z)\cap A$, we can assume that $K\subset A$. Since $K\subset A$ is uncountable, there exists a continuous injective map $f:2^\w\to K$. By Example~\ref{ex:hard}(5), for some $x\in X$ the preimage $f^{-1}(x+B)$ does not belong to the $\sigma$-ideal $\sigma\overline{\N}$ in $2^\w$. In particular, $f^{-1}(x+B)$ is uncountable. Then the intersection $K\cap (x+B)\supset f(f^{-1}(x+B))$ is uncountable as well. But this contradicts the choice of $K$.
\smallskip

Next, we prove that the set $B$ is not null-$n$ for every $n\in\IN$. To derive a contradiction, assume that for some $l\in\IN$ the set $B$ is null-$l$. Then we can find a null-sequence $(z_k)_{k\in\w}$ in $X$ such that $|\{i\in\w:y+z_i\in B\}|\le l$ for every $y\in X$.

Since $\lim_{k\to\infty}|G_k|=\infty$, for every $n\in\w$ there exists a number $k_n\in\w$ such that $|G_k|>n+2$ for all $k\ge k_n$. We can assume that $k_{n+1}>k_n$ for all $n\in\w$.

Replacing $(z_k)_{k\in\w}$ by a suitable subsequence, we can assume that for every $n\in\w$ and $k< k_{n+1}$ the element $z_n(k)$ is equal to the neutral element $\theta(k)$ of the group $G_k$.
Then for every $n\in\w$ and $k\in\w$ with $k_{n}\le k<k_{n+1}$ the set $\{z_i(k)\}_{i\in\w}$ is contained in $\{\theta(k)\}\cup\{z_i(k):i<n\}$ and hence has cardinality $|\{z_i(k)\}_{i\in\w}|\le 1+n<|G_k|$. Also, for any $k<k_0$ we get $|\{z_i(k)\}_{i\in\w}|=|\{\theta_k\}|=1<|G_k|$.
Consequently, $|\{z_i(k)\}_{i\in\w}|<|G_k|$ for all $k\in\w$, which allows us to choose a function $s\in X$ such that $s+\{z_i\}_{i\in\w}\subset A=\prod_{k\in\w}G_k^*$.

Choose any free ultrafilter $\U$ on the set $\w$. For two elements $x,y\in X$ we write $x=^* y$ if the set $\{k\in\w:x(k)=y(k)\}$ belongs to the ultrafilter $\U$. Let $[\w]^2=\{(i,j)\in\w\times\w:i<j\}$ and $\chi:[\w]^2\to\{-1,0,1\}$ be the function defined by
$$\chi(i,j)=\begin{cases}
1&\mbox{if $z_j=^*\alpha+z_i$};\\
-1&\mbox{if $z_j=^*-\alpha+z_i$};\\
0&\mbox{otherwise}.
\end{cases}
$$ By the Ramsey Theorem \cite{GRS}, there exists an infinite subset $\Omega\subset\w$ such that $\{\chi(i,j):i,j\in\Omega,\;i<j\}=\{c\}$ for some $c\in\{-1,0,1\}$. We claim that $c=0$. Assuming that $c\in\{-1,1\}$, choose any numbers $i<j<k$ in $\Omega$ and conclude that $\{\chi(i,j),\chi(j,k),\chi(i,k)\}=\{1\}$ implies that $c\cdot \alpha+z_i=^* z_k=^* c\cdot \alpha+z_j=^* c\cdot \alpha+(c\cdot \alpha+z_i)=2\cdot c\cdot \alpha+z_i$ and hence $\alpha=^*2\alpha$, which contradicts the choice of $\alpha\in A$. This contradiction shows that $c=0$.

Choose any finite subset $L\subset\Omega$ of cardinality $|L|>l$ and observe that the set $$U=\bigcap_{i,j\in L}\big\{k\in\w:z_i(k)-z_j(k)\notin\{\alpha(k),-\alpha(k)\}\big\}$$belongs to the ultrafilter $\U$ and hence is infinite.

Write $U$ as the union $U=\bigcup_{i\in L}U_i$ of pairwise disjoint infinite sets.
Consider the function $y\in X$ defined by the formula
$$y(k)=\begin{cases}
\alpha(k)-z_i(k)&\mbox{if $k\in U_i$ for some $i\in L$};\\
s(k)&\mbox{otherwise.}
\end{cases}
$$
We claim that for every $i\in L$ the element $y+z_i$ belongs to the set $B$. First we show that $y+z_i$ belongs to $A$. This will follow as soon as we check that $y(k)+z_i(k)\ne\theta(k)$ for every $k\in\w$. If $k\notin U$, then $y(k)+z_i(k)=s(k)+z_i(k)\ne\theta(k)$ (as $s+\{z_j\}_{j\in\w}\subset A$).
If $k\in U$, then $y(k)+z_i(k)=\alpha(k)-z_j(k)+z_i(k)$ for the unique number $j\in L$ with $k\in U_j$. The definition of the set $U\supset U_j\ni k$ guarantees that $z_i(k)-z_j(k)\ne-\alpha(k)$ and hence $y(k)+z_i(k)=\alpha(k)-z_j(k)+z_i(k)\ne\theta(k)$. This completes the proof of the inclusion $y+z_i\in A$.

To see that $y+z_i\in B$, observe that for every $k\in U_i$ we get $y(k)+z_i(k)=\alpha(k)-z_i(k)+z_i(k)=\alpha(k)$ and hence $y+z_i\in B$ by the definition of $B$.

Therefore $|\{i\in\w:y+z_i\in B\}|\ge|L|>l$, which contradicts the choice of the sequence $(z_i)_{i\in\w}$.
\smallskip

7. Assuming that $\prod_{k\in\w}\frac{|G_k|-1}{|G_k|}>0$, we shall prove that the set $B$ is not Haar-$\sigma\overline{\N}$ in $X$. Given any continuous map $f:2^\w\to X$, we should find $x\in X$ such that $f^{-1}(x+B)\notin\sigma\overline{\N}$.
Our assumption implies that the set $A$ has positive Haar measure and hence is not Haar-null. Then for some $x\in X$ the preimage $E=f^{-1}(x+A)$ has positive Haar measure in $2^\w$. Let $F=\supp(\lambda|E)$ be the support of the measure $\lambda|E$ (assigning to each Borel subset $S\subset 2^\w$ the number $\lambda(S\cap E)$). By \cite[2.12]{Akin99}, there exists a homeomorphism $h:2^\w\to F$ such that $\{h(N):N\in\mathcal N\}=\{M\subset F:\lambda(M)=0\}$. Now consider the map $g:2^\w\to A$ defined by $g(z)=-x+f\circ h(z)$ for $z\in 2^\w$. Since $g(2^\w)=-x+f(F)\subset -x+(x+A)=A$, we can apply Example~\ref{ex:hard}(5) and find $y\in X$ such that $g^{-1}(y+B)\notin \sigma\overline{\N}$. Then $f^{-1}(x+y+B)=h(g^{-1}(y+B))\notin\overline{\N}$ in $F$ and in $2^\w$.
\end{proof}

\begin{problem} Is the Haar-null $G_\delta$-set $B$ from Example~\ref{ex:hard} generically Haar-null in $X$?
\end{problem}




\section{Smallness properties of additive sets and wedges in Polish groups}\label{s16}

A subset $A$ of a group $X$ is called \index{subset!additive}\index{additive subset}{\em additive} if $A+A\subset A$. In other words, $A$ is a subsemigroup of $X$. Observe that for a nonempty additive set $A$ the set $A-A$ is a subgroup of $X$.

For example,  $[0,\infty)^\w$ is an additive thick set in the Polish group $\IR^\w$.

We recall that a subset $A$ of a topological group $X$ is called ({\em finitely}) {\em thick} if for each (finite) compact subset $K\subset X$ there is $x\in X$ such that $K+x\subset A$.

\begin{theorem}\label{t:additive1} For an additive analytic set $A$ in a Polish group $X$ the following conditions are equivalent:
\begin{enumerate} 
\item[\textup{1)}] the subgroup $A-A$ is not open in $X$;
\item[\textup{2)}] the subgroup $A-A$ is meager in $X$;
\item[\textup{3)}] $A$ is generically Haar-1;
\item[\textup{4)}] $A$ is null-$n$ for some $n\in\IN$;
\item[\textup{5)}] $A$ is not finitely thick in any open subgroup $H\subset X$, containing $A$.
\end{enumerate}
\end{theorem}

\begin{proof} $(1)\Ra(2)$ Observe that the space $A-A$ is analytic, being a continuous image of the analytic space $A\times A$. If $A-A$ is not meager, then by Corollary \ref{c:PP}, the set $(A-A)-(A-A)=(A+A)-(A+A)= A-A$ is a neighborhood of $\theta$ in $X$ and hence is open in $X$ (being a subgroup of $X$).
\smallskip

The implication $(2)\Ra(3)$ follows from Theorem~\ref{t:-GHI} and $(3)\Ra(4)$ follow immediately from the definitions.
\vskip3pt

$(4)\Ra(5)$ To derive a contradiction, assume that $A$ is null-$n$ for some $n\in\IN$ and $A$ is finitely thick in some open subgroup $H\subset X$ containing $A$.
Since $A$ is null-$n$, there exists a null-sequence $(x_k)_{k\in\w}$ in $X$ such that for every $s\in X$ the set $\{k\in\w:x_k+s\in A\}$ contains at most $n$ numbers. Since $\lim_{k\to\infty}x_k=\theta\in H$, there exists $l\in\w$ such that $x_k\in H$ for all $k\ge l$. Now consider the finite set $F=\{x_k:l\le k\le l+n\}$ in the group $H$. Since $A$ is finitely thick in $H$, there exists $s\in H$ such that $F+s\subset A$. Then the set $\{k\in\w:x_k+s\in A\}$ contains the set $\{l,\dots,l+n\}$ of cardinality $n+1$, which contradicts the choice of the sequence $(x_k)_{k\in\w}$.
\smallskip

To show that $(5)\Ra(1)$, we shall prove that the negation of $(1)$ implies the negation of $(5)$. Assume that the subgroup $A-A$ is open in $X$. In this case we shall show that $A$ is finitely thick in the open subgroup $H:=A-A$. Given any finite set $F\subset H$, write it as $F=\{x_1,\dots,x_n\}$ for $n=|F|$. 
Since $F\subset A-A$, each point $x_i\in F$ can be written as $x_i=a_i^+-a_i^-$ for some points $a_i^+,a_i^-\in A$. Then for the element $s=\sum_{i=1}^na_i^-\in A$ and every $j\in\{1,\dots,n\}$ we get
$$x_j+s=a_j^+-a_j^-+\sum_{i=1}^na^-_i=\sum_{i=1}^na_i^{\e_i}\in A$$where $\e_j=+$ and $\e_i=-$ for all $i\in\{1,\dots,n\}\setminus\{j\}$.  Therefore, $F+s\subset A$ and $A$ is finitely thick in $H$.
\end{proof}

\begin{corollary} For a closed additive subset $A$ of a Polish group $X$ the following conditions are equivalent:
\begin{enumerate}
\item[\textup{1)}] $A$ is Haar-meager;
\item[\textup{2)}] $A$ is strongly Haar-meager;
\item[\textup{3)}] $A$ is not Haar-open;
\item[\textup{4)}] $A$ is not thick in any open subgroup $H\subset X$ containing $A$.
\end{enumerate}
\end{corollary}

\begin{proof} The equivalence $(1)\Leftrightarrow(2)\Leftrightarrow(3)$ has been proved in Theorem~\ref{t:prethick}.
\smallskip

$(2)\Ra(4)$ Assuming that $A$ is strongly Haar-meager, find a compact subset $K\subset X$ such that $K\cap (A+x)$ is meager in $K$ for any $x\in X$. Replacing $K$ by a suitable shift, we can assume that $\theta\in K$. To derive a contradiction, assume that $A$ is thick in some open subgroup $H\subset X$, containing $A$.  It follows that the subgroup $H$ is also closed in $X$. Then the set $K\cap H$ is compact and open in $K$. Since $A$ is thick in $H$, there exists an element $x\in H$ such that $x+(K\cap H)\subset A$. Then $K\cap (A-x)$ contains the nonempty open subset $K\cap H$ of $K$ and is not meager in $K$.
\smallskip

$\neg(3)\Ra\neg(4)$ If $A$ is Haar-open, then $A$ is not Haar-1 and by Theorem~\ref{t:additive1}, $H:=A-A$ is an open subgroup in $X$. We claim that $A$ is thick in $H$. Given a compact subset $K\subset H$, we need to find $x\in H$ such that $K+x\subset A$. By Proposition~\ref{p:HO=>prethick}, there exists a finite subset $F\subset X$ such that $K\subset F+A$. Since $K\cup A\subset H$, we can replace $F$ by $F\cap H$ and assume that $F\subset H$. By Theorem~\ref{t:additive1}, $A$ is finitely thick in some open subgroup $G\subset X$, containing $A$.  Consequently, there exists a point $x\in G$ such that $x+F\subset A$. It follows that $x\in A-F\subset H$ and $x+K\subset x+F+A\subset A+A\subset A$, witnessing that $A$ is thick in the open subgroup $H$.
\end{proof}

\begin{problem}\label{prob:HN=>HN} Let $A$ be a closed additive subset of a Polish group $X$. Assume that $A$ is Haar-meager. Is $A$ Haar-null?
\end{problem} 

\begin{remark} By Theorem~\ref{t:Sum-HO}, the answer to Problem~\ref{prob:HN=>HN} is affirmative for Polish groups which are countable products of locally compact groups.
\end{remark}

Given an additive set $A$ in a topological group $X$ with $\theta\in\bar A$, let us consider the  \index{Addition Game}{\sf Addition Game} on $A$ played by two players I and II according to the following rules. The player I starts the addition game selecting a neighborhood $U_1\subset X$ of $\theta$ and the player II responds selecting a point $a_1\in U_1\cap A$. At the $n$th inning the player I selects a neighborhood $U_n\subset X$ of $\theta$ and the player II responds selecting a point $a_n\in U_n\cap A$. At the end of the game the player I is declared a winner if the series $\sum_{n=1}^\infty a_n$ converges to some point $a\in A$. In the opposite case the player II is a winner.

We shall say that an additive subset $A$ of a Polish group $X$ is \index{subset!winning additive}\index{winning additive subset}{\em winning-additive} if $\theta\in\bar A$ and the player I has a winning strategy in the {\sf Addition Game} on $A$. The winning strategy of I can be thought as a function $\yen$ assigning to each finite sequence $(a_1,\dots,a_n)$ of points of $A$ a neighborhood $\yen(a_1,\dots,a_n)$ of $\theta$  in $X$ such that for any infinite sequence $(a_n)_{n=1}^\infty\in A^\IN$ the series $\sum_{n=1}^\infty a_n$ converges to a point of $A$ if $a_{n}\in \yen(a_1,\dots,a_{n-1})$ for all $n\in\IN$. For $n=1$ the neighborhood $\yen(\;)$ is the first move $U_1$ of the player I according to the strategy $\yen$.

\begin{proposition}\label{p:Polish-winad} An additive set $A$ in a Polish group $X$ is winning-additive if $\theta\in\bar A$ and the space $A$ is Polish.
\end{proposition}

\begin{proof} Assume that $\theta\in\bar A$ and the space $A$ is Polish, which means that the topology of $A$ is generated by a complete metric $\rho\le 1$.

Now we describe the winning strategy $\yen$ of the player I in the {\sf Addition Game}. The strategy is a function $\yen:A^{<\w}\to\tau_\theta$ on the family $A^{<\w}=\bigcup_{n\in\w}A^n$ of finite sequences in $A$ with values in the set $\tau_\theta$ of neighborhoods of $\theta$ in $X$. 

For any $n\in\w$ and a sequence $(a_1,\dots,a_n)\in A^n$, consider the element $s_n:=\theta+\sum_{i=1}^na_i$ and choose a neighborhood $V\subset X$ of $\theta$ such that $A\cap(s_n+V)\subset\{x\in A:\rho(s_n,x)<\frac1{2^n}\}$.  Put $\yen(a_1,\dots,a_n):=V$. 

We claim that the strategy $\yen$ is winning. Indeed, fix any sequence $(a_n)_{n\in\IN}\in A^\IN$ such that $a_{n}\in \yen(a_1,\dots,a_{n-1})$ for any $n\in\IN$. Consider the sequence $(s_n)_{n=1}^\infty$ of partial sums $s_n=\theta+\sum_{k=1}^na_k$ of the series $\sum_{k=1}^\infty a_k$. Since $a_n\in\yen(a_1,\dots,a_{n-1})$, the element $s_{n}=s_{n-1}+a_{n}$ belongs to the ball $\{x\in A:\rho(s_{n-1},x)<\frac1{2^n}\}$ which implies that the sequence $(s_n)_{n=1}^\infty$ is Cauchy in the complete metric space $(A,\rho)$ and hence converges to some element of $A$.
\end{proof}


Another important notion involving the name of Steinhaus is defined as follows.

\begin{definition} A subset $A$ of a topological group $X$ is called \index{subset!Steinhaus at a point}{\em Steinhaus at a point} $x\in X$ if for any neighborhood $U\subset X$ of $x$ the set $(U\cap A)-(U\cap A)$ is a neighborhood of $\theta$ in $X$.
\end{definition}

\begin{theorem}\label{t:additive2} If a winning-additive subspace $A$ of a Polish group $X$ is Steinhaus at $\theta$, then $A$ is neither null-finite nor Haar-countable in $X$.
\end{theorem}

\begin{proof} Fix a complete invariant metric $\rho\le 1$, generating the topology  of the Polish group $X$. For an element $x\in X$ put $\|x\|:=\rho(x,0)$, and for $\e>0$ let $B(x;\e):=\{y\in X:\|x-y\|<\e\}$ be the open $\e$-ball around $x$ in the metric space $(X,\rho)$. 

Since $A$ is winning-additive, the player I have a winning strategy $\yen$, which is  a function $\yen:A^{<\w}\to \tau_\theta$ such that for any sequence $(a_n)_{n\in\w}\in A^\w$ the series converges to a point of $A$ if $a_n\in\yen(a_1,\dots,a_{n-1})$ for all $n\in\IN$. 

In the following two claims we shall prove that $A$ in not null-finite and not Haar-countable in $X$.

\begin{claim}\label{cl:null-fin} The set $A$ is not null-finite in $X$.
\end{claim}

\begin{proof} Given a null sequence $(x_n)_{n\in\w}$ in $X$, we should find a point $s\in X$ such that the set $\{n\in\w:x_n+s\in A\}$ is infinite.

By induction for every $k\in\IN$ we shall construct a neighborhood $U_k\subset X$ of $\theta$, a number $n_k\in\IN$, and points $a_k^+,a_k^-\in A$ such that the following conditions are satisfied:
\begin{enumerate}
\item[\textup{(1)}] $\theta\in U_k\subset B(\theta;\frac1{2^k})\cap\yen(a_1^{\e_1},\dots,a_{k-1}^{\e_{k-1}})$ for any $\e_1,\dots,\e_{k-1}\in\{+,-\}$;
\item[\textup{(2)}] $n_k>n_{k-1}$, $x_{n_k}=a_k^+-a_k^-$ and $a_k^+,a_k^-\in A\cap U_k$.
\end{enumerate}

We start the inductive construction letting $n_0=1$. 
Assume that for some $k\in\IN$ the points $a_i^+,a_i^-$, $1\le i<k$, have been constructed. Choose any neighborhood $U_k\subset X$ of $\theta$ that satisfies the inductive condition (1). Since $A$ is Steinhaus at $\theta$, the set $(U_k\cap A)-(U_k\cap A)$ is a neighborhood of $\theta$ and hence it contains some point $x_{n_k}$ with $n_k>n_{k-1}$. Then $x_{n_k}=a_k^+-a_k^-$ for some points $a_k^+,a_k^-\in A\cap U_k$. This completes the inductive step.
\smallskip

Since $a^-_k\in U_k\subset\yen(a_1^-,\dots,a^-_{k-1})$ for all $k\in\IN$, the series $\sum_{k=1}^\infty a_k^-$ converges to some point $s\in A$.

We claim that $x_{n_i}+s\in A$ for every $i\in\IN$. Indeed, 
$$x_{n_i}+s=(a_i^+-a_i^-)+\sum_{k=1}^\infty a_k^-=\sum_{k=1}^\infty a_k^{\e_k}$$ where $\e_i=+$ and $\e_k=-$ for all $k\in\IN\setminus\{i\}$.
Since $a_k^{\e_k}\in U_k\subset\yen(a_1^{\e_1},\dots,a_{k-1}^{\e_{k-1}})$, the sum $x_{n_i}+s$ of the series $\sum_{k=1}^\infty a_k^{\e_k}$ belongs to $A$.
Now we see that the set $\{n\in\w:x_n+s\in A\}\supset\{n_i\}_{i=1}^\infty$ is infinite, witnessing that the set $A$ is not null-finite.
\end{proof}

\begin{claim}\label{cl:null-fin} The set $A$ is not Haar-countable in $X$.
\end{claim}

\begin{proof} Given any continuous map $f:2^\w\to X$ we should find an element $s\in X$ such that $f^{-1}(A+s)$ is not countable. If for some $s\in X$ the fiber $f^{-1}(s)$ is uncountable, then for any $a \in A$ the preimage $f^{-1}(A+(s-a))\supset f^{-1}(s)$ is uncountable and we are done. So, we assume that for each $x\in X$ the set $f^{-1}(x)$ is countable and hence $K=f(2^\w)$ is a compact set without isolated points in $X$. Replacing $K$ by a suitable shift, we can assume that $\theta\in K$.

Consider the set $2^{<\w}:=\bigcup_{k\in\w}2^k$ of finite binary sequences and let $2^{<\w}_0:=2^0\cup\bigcup_{n\in\IN}2^n_0$ where $2^n_0=\{(s_0,\dots,s_{n-1})\in 2^n:s_{n-1}=0\}$.
 For a binary sequence $s=(s_0,\dots,s_{n-1})\in 2^n\subset 2^{<\w}$ by $|s|$ we denote the length $n$ of $s$. For a number $k<n$ let $s{\restriction}k:=(s_0,\dots,s_{k-1})$ be the restriction of $s$.
 
 Write the countable set $2^{<\w}$ as $2^{<\w}=\{\alpha_i\}_{i\in\IN}$ so that  $\{\alpha_i:2^n\le i<2^{n+1}\}=2^n$ for all $n\in\w$.
 For every $k\ge 2$ let ${\downarrow}k$ be the unique number such that $\alpha_{\downarrow k}=\alpha_k{\restriction}(|\alpha_k|-1)$. So, ${\downarrow}k$ is the predecessor of $k$ in the tree structure on $\IN$, inherited from $2^{<\w}$. 

For every $k\in\IN$ we shall inductively construct two neighborhoods $U_k,V_k$ of $\theta$, and points $x_i\in K$, $a_k^+,a_k^-\in U_k$ such that the following conditions are satisfied:
\begin{enumerate}
\item[\textup{(1)}] $\theta\in U_k\subset B(\theta;\frac1{2^k})\cap \yen(a_1^{\e_1},\dots,a_{k-1}^{\e_{k-1}})$ for any signs $\e_1,\dots,\e_{k-1}\in\{+,-\}$;
\item[\textup{(2)}] $V_k:=(U_k\cap A)-(U_k\cap A)$;
\item[\textup{(3)}] $x_1=\theta$ and $x_k\in K\cap (x_{{\downarrow}k}+V_k)$ if $k>1$;
\item[\textup{(4)}] $x_k=x_{{\downarrow}k}$ if $\alpha_k\in 2^{<\w}_0$ and $x_k\ne x_{{\downarrow}k}$ if $\alpha_k\notin 2^{<\w}_0$;
\item[\textup{(5)}]  $a_k^+,a_k^-\in A\cap U_k$ and $x_{k}-x_{{\downarrow}k}=a_k^+-a_k^-$.
\end{enumerate}

We start the inductive construction letting $x_0=\theta$. Assume that for some $k\in\IN$ and all $1\le i<k$ neighborhoods $U_i,V_i$ and points $x_i$, $a_i^+$, $a_i^-$ have been constructed so that the conditions (1)--(5) are satisfied.

Choose a neighborhood $U_k$ satisfying the condition (1). Since $A$ is Steinhaus at $\theta$, the set $V_k$ defined in the inductive condition (2) is a neighborhood of $\theta$ in $X$.
Taking into account that the set $K\ni \theta$ has no isolated points, we can find a point $x_k\in K$ satisfying the conditions (3), (4). The definition of the neighborhood $V_k\ni x_{k}- x_{{\downarrow}k}$ yields points $a_k^+,a_k^-\in A\cap U_k$ satisfying the condition (5). This completes the inductive step.
\smallskip

After completion the inductive construction, consider the series $\sum_{k=1}^\infty a_k^-$, which converges to a point $s\in A$ by the choice of the winning strategy $\yen$ and the conditions (1), (5).
We claim that $x_k+s\in A$ for every $k$. Find $n\in\w$ with $\alpha_k\in 2^n$ and for every $i\le n$ find a unique number $k_i\in\IN$ such that $\alpha_{k_i}=\alpha_k{\restriction}i$. It follows that $k_n=k$, $k_0=1$, and $k_i={\downarrow}k_{i+1}$ for $0\le i<n$.
Then $$s+x_k=s+\sum_{i=1}^n(x_{k_i}-x_{k_{i-1}})=s+\sum_{i=1}^n(x_{k_i}-x_{{\downarrow}k_i})=\sum_{j=1}^\infty a_j^-+\sum_{i=1}^n(a_{k_i}^+-a_{k_i}^-)=\sum_{j=1}^\infty a_j^{\e_j}$$
where $$
\e_j=
\begin{cases}+&\mbox{if $j=k_i$ for some $1\le i\le n$};\\
-&\mbox{otherwise.}
\end{cases}
$$
The inductive conditions  (1) and (5) guarantee that sum $x_k+s$ of the series 
$\sum_{j=1}^\infty a_j^{\e_j}$ belongs to the set $A$.

The inductive conditions (1)--(4) ensure that the set $\{x_k+s\}_{k\in\IN}\subset (K+s)\cap A$ has no isolated points. Then its closure $F$ in $K+s$ has no isolated points, too. Being a compact space without isolated points, $F$ has cardinality of continuum. We claim that $F\subset A$.

Consider the map $f:2^\w\to F$ assigning to each binary sequence $\beta\in 2^\w$ the limit of the sequence $(s+x_{k_i})_{i\in\w}$ where $k_i\in\IN$ is a unique number such that $\alpha_{k_i}=\beta{\restriction}i$. Observe that for every $i>1$ we have $x_{{\downarrow}k_i}=x_{k_{i-1}}$  and hence $\|x_{k_i}-x_{k_{i-1}}\|<\frac2{2^{k_i}}$ by the inductive assumptions (1)--(3). This implies that the sequence $(x_{k_i})_{i=1}^\infty$ is Cauchy in $X$ and hence it converges to some point $x_\beta\in X$ such that
\begin{equation}\label{eq:additive2}
\|x_\beta-x_{k_i}\|\le\sum_{j=i+1}^\infty\frac{2}{2^{k_j}}\le \frac4{2^{k_{i+1}}}
\end{equation}
for all $i\in\IN$. 

It follows that $x_\beta=\sum_{i=1}^\infty (x_{k_i}-x_{k_{i-1}})= \sum_{i=1}^\infty (x_{k_i}-x_{{\downarrow}k_i})=\sum_{i=1}^\infty (a_{k_i}^+-a_{k_i}^-)$.
Then $$f(\beta)=x_\beta+s=\sum_{i=1}^\infty(a_{k_i}^+-a_{k_i}^-)+\sum_{k=1}^\infty a_k^-=\sum_{k=1}^\infty a_k^{\e_k}$$ where 
$$\e_k=\begin{cases}
+&\mbox{if $k=k_i$ for some $i\in\IN$};\\
-,&\mbox{otherwise}.
\end{cases}
$$
The inductive conditions (1) and (5) imply that the series $\sum_{k=1}^\infty a_k^{\e_k}$ converges to some point of $A$. So, $f(\beta)=x_\beta+s=\sum_{k=1}^\infty a_k^{\e_k}\in A$. The inequality \eqref{eq:additive2} implies that the map $f:2^\w\to F$, $f:\beta\mapsto x_\beta+s$ is continuous. Then $f(2^\w)=F$ by the compactness of $f(2^\w)$ and density of $f(2^\w)$ in $F$. Finally we conclude that $F=f(2^\w)\subset A\cap(K+s)$, witnessing that $A$ is not Haar-countable.
\end{proof}
\end{proof}

Now we will find a useful condition on a subset $A$ of a group guaranteeing that $A$ is Steinhaus at $\theta$.

\begin{lemma}\label{l:SteinH} Let $Z$ be a subgroup of a topological group $X$. A subset $A=A+Z$ of a topological group $X$ is Steinhaus at some point $x\in X$ if and only if for any neighborhood $U\subset X$ of $x$ the set $(U\cap A)-(U\cap A)+Z$ is a neighborhood of $\theta$ in $X$.
\end{lemma}

\begin{proof} The ``only if'' part is trivial. To prove the ``if'' part, assume that 
for any neighborhood $V\subset X$ of $x$ the set $(V\cap A)-(V\cap A)+Z$ is a neighborhood of $\theta$ in $X$.

To prove that $A$ is Steinhaus at $x$, take any neighborhood $U\subset X$ of $x$.
By the continuity of the group operation on $X$, there exists a neighborhood $V\subset X$ of $\theta$ such that $V+V\subset U-x$. By our assumption, the set $((V+x)\cap A)-((V+x)\cap A)+Z$ is a neighborhood of $\theta$ and so is the set $$W:=V\cap \big(\left((V+x)\cap A\right)-\left((V+x)\cap A\right)+Z\big).$$ We claim 
that $W\subset (U\cap A)-(U\cap A)$. Indeed, for any $w\in W\subset V$, we can find points $a,b\in (V+x)\cap A$ and $z\in Z$ such that $w=a-b+z$. Then $z=w+b-a$ and $a+z= w+b \in (A+Z)\cap(V+V+x)\subset A\cap U$. Concluding, we obtain that
$$w=(a+z)-b\in (U\cap A)-\left((V+x)\cap A\right)\subset (U\cap A)-(U\cap A)$$
since $V+x \subset U$. 	
\end{proof}

A subset $M\subset\IN$ is called \index{subset!multiplicative}\index{multiplicative subset}{\em multiplicative} if $M\cdot M\subset M$. It is easy to check that for a multiplicative subset $M\subset \IN$ the set $$\tfrac\theta{M}:=\{\theta\}\cup\bigcup_{m\in M}\{x\in X:mx=\theta\}$$ is a subgroup of $X$.

\begin{definition} Let $M$ be a multiplicative subset of $\IN$. 
A nonempty subset $A$ of a topological group $X$ is called an \index{$M$-wedge}{\em $M$-wedge} if 
$A$ is additive, $A=A+\frac\theta{M}$, and 
 for any points $a,b\in A$ and neighborhood $U\subset X$ of $\theta$ there are points $u,v\in A\cap U$ such that $a=mu$ and $b=mv$ for some $m\in M$.
\end{definition}

\begin{definition} A nonempty subset $A$ of a topological group $X$ is called a \index{wedge}{\em wedge} in $X$ if $A$ is an $M$-wedge for some multiplicative subset $M\subset\IN$. 
\end{definition}

Note that $\theta \in \bar{A}$ for any wedge $A \subset X$. A wedge $A$ in a topological group is called {\em analytic} if its underlying topological space is analytic.

\begin{proposition}\label{p:wedge=>St} An analytic wedge $A$ in a Polish group $X$ is Steinhaus at $\theta$ if and only if $A-A$ is a neighborhood of $\theta$ in $X$.
\end{proposition}

\begin{proof} The ``only if'' part is trivial. To prove the ``if'' part, assume that $A-A$ is a neighborhood of $\theta$ in $X$. 
Since $A$ is additive, the set $A-A$ is an open subgroup of $X$.
Replacing $X$ by $A-A$, we can assume that $X=A-A$. Find a multiplicative subset $M\subset\IN$ for which the wedge $A$ is an $M$-wedge. Then $A=A+\frac\theta{M}$.

By Lemma~\ref{l:SteinH}, to show that $A$ is Steinhaus at $\theta$, it suffices to check that for any neighborhood $U\subset X$ of $\theta$, the set $(U\cap A)+(U\cap A)+\frac\theta{M}$ is a neighborhood of $\theta$ in $X$. Given a neighborhood $U\subset X$ of $\theta$, choose an open neighborhood $V\subset X$ of $\theta$ such that $V+V\subset U$. 

First we show that $X=\bigcup_{m\in M}mW$ where $W:= (V\cap A)-(V\cap A)$.
Indeed, for any element $x$ of $X=A-A$, we can find points $a,b\in A$ such that $x=a-b$. Since $A$ is an $M$-wedge, there are points $u,v\in A\cap V$ such that $a=mu$ and $b=mv$ for some $m\in M$. Then $x=a-b=m(u-v)\in mW$ and hence $A-A=\bigcup_{m\in M}mW$. 

By the Baire Theorem, for some $m\in M$ the (analytic) set $mW$ is not meager and by Corollary \ref{c:PP} the set $mW-mW=m(W-W)$ is a neighborhood of $\theta$. By the continuity of the map $\mathit m:X\to X$, $\mathit m:x\mapsto mx$, the set $W-W+\frac\theta{m}=\mathit m^{-1}(m(W-W))$ is a neighborhood of $\theta$. Now observe that $W-W=(V\cap A)+(V\cap A)-(V\cap A)-(V\cap A)$ and $(V\cap A)+(V\cap A)\subset (V+V)\cap A$. Then $$(U\cap A)-(U\cap A)+\tfrac\theta{M}\supset ((V+V)\cap A)-((V+V)\cap A)+\tfrac\theta{m}\supset W-W+\tfrac\theta{m}$$ is a neighborhood of $\theta$ in $X$.
\end{proof} 

Now we prove the main characterization theorem in this section.

\begin{theorem}\label{t:main-wedge} For any winning-additive analytic wedge $A$ in a Polish group $X$ the following conditions are equivalent:
\begin{enumerate}
\item[\textup{1)}] the subgroup $A-A$ is not open in $X$;
\item[\textup{2)}] $A-A$ is meager in $X$;
\item[\textup{3)}] $A$ is generically Haar-$1$;
\item[\textup{4)}] $A$ is Haar-thin;
\item[\textup{5)}] $A$ is null-finite;
\item[\textup{6)}] $A$ is Haar-countable.
\end{enumerate}
\end{theorem}

\begin{proof} The equivalence of the conditions (1)--(3) follows  from Theorem~\ref{t:additive1}, $(1)\Leftrightarrow(4)$ follows from Theorem~\ref{t:thin}, and the implications $(3)\Ra(5,6)$ are trivial. To finish the proof it suffices to show that the negation of (1) implies the negations of (5) and (6). So, assume that the subgroup $A-A$ is open in $X$.  By Proposition~\ref{p:wedge=>St}, the wedge $A$ is Steinhaus at $\theta$ and by Theorem~\ref{t:additive2}, $A$ is not null-finite and not Haar-countable.
\end{proof}

Theorem~\ref{t:main-wedge} and Proposition~\ref{p:Polish-winad} imply 

\begin{corollary}\label{c:main-wedge} For any Polish wedge $A$ in a Polish group $X$ the following conditions are equivalent:
\begin{enumerate}
\item[\textup{1)}] the subgroup $A-A$ is not open in $X$;
\item[\textup{2)}] $A-A$ is meager in $X$;
\item[\textup{3)}] $A$ is generically Haar-$1$;
\item[\textup{4)}] $A$ is Haar-thin;
\item[\textup{5)}] $A$ is null-finite;
\item[\textup{6)}] $A$ is Haar-countable.
\end{enumerate}
\end{corollary}

Now we apply Corollary~\ref{c:main-wedge} to convex wedges in Polish vector spaces. By a \index{Polish vector space}{\em Polish vector space} we understand a topological linear space over the field of real numbers, whose underlying topological space is Polish.

A subset $A$ of a vector space $X$ is called a \index{convex wedge}{\em convex wedge} if $av+bw\in A$ for any positive real numbers $s,t$ and any points $a,b\in A$. A convex wedge $A$ is called a \index{convex cone}{\em convex cone} if $A\cap(-A)=\{\theta\}$.

Since any convex wedge in a topological vector space $X$ is a wedge in the topological group $X$, Corollary~\ref{c:main-wedge} implies the following corollary.

\begin{corollary}\label{c:wedge} For a Polish convex wedge $A$ in a Polish vector space $X$ the following conditions are equivalent:
\begin{enumerate}
\item[\textup{1)}] the subgroup $A-A$ is not open in $X$;
\item[\textup{2)}] $A-A$ is meager in $X$;
\item[\textup{3)}] $A$ is generically Haar-$1$;
\item[\textup{4)}] $A$ is Haar-thin;
\item[\textup{5)}] $A$ is null-finite;
\item[\textup{6)}] $A$ is Haar-countable.
\end{enumerate}
\end{corollary}

\begin{remark} In Theorem~\ref{t:conv-Ban} we shall show that the equivalence of the conditions (1)--(6) in Corollary~\ref{c:wedge} remains true for any Polish convex subset of a Polish vector space.
\end{remark}
\newpage

\section{Smallness properties of mid-convex sets in Polish groups}\label{s17}

In this section we explore the smallness properties of mid-convex sets in (2-divisible) Polish groups. A group $X$ is called \index{group!2-divisible}\index{2-divisible group}{\em $2$-divisible} if for every $a\in X$ the set $\frac12a:=\{x\in X:x+x=a\}$ is not empty. In this case, for every nonempty subset $A\subset X$ and every $n\in\w$ the set $\frac1{2^n}A:=\{x\in X:2^n x\in A\}$ is not empty. Here $2^0 x:=x$ and $2^{n+1} x:=2^nx+2^nx$ for $n\ge 0$.
Observe that $\frac{\theta}{2^n}:=\{x\in X:2^nx=\theta\}$ is a subgroup of $X$ and so is the union $\frac{\theta}{2^\IN}:=\bigcup_{n\in\IN}\frac{\theta}{2^n}$. 

\begin{definition} A subset $A$ of a group $X$ is called \index{subset!mid-convex}\index{mid-convex subset}{\em mid-convex} if $\emptyset\ne\frac12(a+b)\subset A$ for any $a,b\in A$.
\end{definition}

A standard example of a mid-convex set is a convex set in any vector space. In this case the subgroup $\frac\theta{2}=\frac{\theta}{2^\IN}$ is trivial.



Let us establish some elementary properties of mid-convex sets in groups.

\begin{lemma}\label{l:sum} Let $n\in\w$ and $A$ be a mid-convex set is a group $X$. Then
\begin{enumerate}
\item[$(1_n)$] $\frac1{2^n}a_0+\sum_{k=1}^n\tfrac1{2^k}a_k\subset A$
 for any points $a_0,a_1,\dots,a_n\in A$; 
\item[$(2_n)$] $A=A+\frac\theta{2^n}$;
\item[$(3)$] $A=A+\frac\theta{2^\IN}$;
\item[$(4)$] $A-x$ is mid-convex for every $x\in X$. 
\end{enumerate}
\end{lemma}

\begin{proof} 1. The statement $(1_0)$ trivially holds as $\frac1{2^0}a_0=\{a_0\}\subset A$ for any $a_0\in A$. Assume that the statement $(1_n)$ holds for some $n\in\w$. To prove the statement $(1_{n+1})$, take any points $a_0,a_1,\dots,a_{n+1}\in A$. By the statement $(1_n)$, $\frac1{2^n}a_0+\sum_{k=1}^n\frac1{2^k}a_{k+1}\subset A$. The mid-convexity of $A$ ensures that 
\begin{multline*}
\tfrac1{2^{n+1}}a_0+\sum_{k=1}^{n+1}\tfrac1{2^k}a_k=\tfrac12a_1+\tfrac12\big(\tfrac1{2^n}a_0+\sum_{k=2}^{n+1}\tfrac1{2^{k-1}}a_k\big)\\
=\tfrac12a_1+\tfrac12\big(\tfrac1{2^n}a_0+\sum_{k=1}^{n}\tfrac1{2^{k}}a_{k+1}\big)\subset \tfrac12a_0+\tfrac12A\subset A.
\end{multline*}

2. The statement $(2_0)$ is true as $A+\frac\theta{2^0}=A+\{\theta\}=A$. Assume that for some $n\in\w$ the equality $A=A+\frac\theta{2^n}$ has been proved. Then for any $a\in A$ and $z\in\frac\theta{2^{n+1}}$ we get $a+2z\in A+\frac\theta{2^n}$ and $a+z\in \frac12(a+a+2z)\subset \frac12(A+A)\subset A$  by the mid-convexity of $A$. Therefore, $A+\frac\theta{2^{n+1}}\subset A=A+\{\theta\}\subset A+\frac{\theta}{2^{n+1}}$ and finally, $A=A+\frac\theta{2^{n+1}}$.
\smallskip

3. The third statement follows from the second one.
\smallskip

4. For any $x\in X$ and $a,b\in A$ the mid-convexity of $A$ yields $\emptyset\ne\frac12(a+b)\subset A$, which implies $\emptyset\ne\frac12((a-x)+(b-x))=\frac12(a+b)-x\subset A-x$. This means that the set $A-x$ is mid-convex in $X$.
\end{proof}


\begin{theorem}\label{t:main-convex-analytic} For an analytic mid-convex subset $A$ of a Polish group $X$ the following conditions are equivalent:
\begin{enumerate}
\item[\textup{1)}] $A-A$ is not a neighborhood of zero in $X$;
\item[\textup{2)}] $A-A$ is meager in $X$;
\item[\textup{3)}] $A$ is generically Haar-$1$;
\item[\textup{4)}] $A$ is Haar-thin;
\item[\textup{5)}] $A$ is null-$n$ for some $n\in\IN$.
\end{enumerate}
\end{theorem}

\begin{proof} We shall prove the implications $(1)\Ra(2)\Ra(3)\Ra(5)\Ra(1)\Leftrightarrow(4)$.
\smallskip

$(1)\Ra(2)$ If $A-A$ is not meager, then by the Piccard-Pettis Theorem \ref{c:PP} applied to the analytic set $A-A$, the set $B:=(A-A)-(A-A)=(A+A)-(A+A)$ is a neighborhood of zero. By the continuity of the homomorphism $\mathit 2:X\to X$, $\mathit 2:x\mapsto x+x$, the set $\frac12B$ is a neighborhood of zero. Taking into account that the set $A$ is mid-convex, we conclude that $\frac12(A+A)=A$ and hence the set $A-A=\frac12(A+A)-\frac12(A+A)=\frac12B$ is a neighborhood of zero.
\smallskip

The implication $(2)\Ra(3)$ follows from Theorem~\ref{t:-GHI} and $(3)\Ra(5)$ is trivial. The equivalence $(1)\Leftrightarrow(4)$ was proved in Theorem~\ref{t:thin}.
\smallskip

Now we prove the implication $(5)\Ra(1)$. Assuming that $A-A$ is a neighborhood of zero, we shall prove that $A$ is not  null-$n$ for every $n\in\IN$.
Given any $n\in\IN$ and null-sequence $(x_k)_{k\in\w}$ we should find an element $s\in X$ such that the set $\{k\in\w:x_k+s\in A\}$ contains at least $n$ numbers. 

Replacing $A$ by a suitable shift $A-a$, we can assume that $\theta\in A$. 
Taking into account that $A-A$ is a neighborhood of $\theta=\lim_{k\to\infty}x_k$ and replacing $(x_k)_{k\in\w}$ by a suitable subsequence, we can assume that $2^kx_k\in A-A$ for all $k\in\w$. So, for every $k\in\w$, we can find elements $a_k^+,a_k^-\in A$ such that $2^kx_k=a_k^+-a_k^-$. By the mid-convexity of $A\ni \theta$, there exists a point $b_k^-\in A$ such that $2^kb_k^-=a_k^-$ and hence $b_k^-\in\frac1{2^k}a_k^-$. 

Consider the element $s:=\sum_{k=1}^nb_k^-\subset\sum_{k=1}^n\frac1{2^k}a_k^-$ of the group $X$. We claim that  $x_i+s\in A$ for every $i\in\{1,\dots,n\}$. Indeed, by Lemma~\ref{l:sum},$$x_i+s\in \tfrac1{2^i}(a_i^+-a_i^-)+\sum_{k=1}^n\tfrac1{2^k}a_k^-=\sum_{k=1}^m\tfrac1{2^k}a_k^{\e_k}\subset \tfrac1{2^m}\theta+\sum_{k=1}^m\tfrac1{2^k}a_k^{\e_k}\subset A$$where $\e_i=+$ and $\e_k=-$ for $k\ne i$. 
\end{proof}

To establish more refined properties of mid-convex sets, we shall consider a convex modification of the {\sf Addition Game}, called the {\sf Convexity Game}. 

This game is played by two players, I and II, on a mid-convex subset $A$ of a topological group $X$ such that $\theta\in A$. The player I starts the game selecting a neighborhood $U_1\subset X$ of $\theta$ in $X$ and the player II responds selecting a point $a_1\in U_1\cap\frac12A$. At the $n$th inning the player $I$ chooses a neighborhood $U_n\subset X$ of $\theta$ in $X$ and the player II responds selecting a point $a_n\in U_n\cap\frac1{2^n}A$. At the end of the game the player I declared the winner if the series $\sum_{n=1}^\infty a_n$ converges to a point of the set $A$. In the opposite case, II wins the {\sf Convexity Game}.  

\begin{definition} A mid-convex set $A$ of a topological group $X$ is called \index{winning mid-convex subset}{\em winning} at a point $a\in A$ if the player I has a winning strategy in the {\sf Convexity Game} on $A-a$.
\end{definition}

\begin{proposition}\label{p:Polish-convgame} Each Polish $mid$-convex set $A$ in a Polish group $X$ is winning at each point $a\in A$.
\end{proposition}

\begin{proof} The space $A$, being Polish, admits a complete metric $\rho$ generating the topology of $A$. 

Now we describe a winning strategy $\yen$ of the player I on the mid-convex set $A-a$ where $a\in A$ is an arbitrary point. Replacing $A$ by its shift $A-a$, we can assume that $a=\theta$. For every $n\in\w$ and a sequence $(a_1,\dots,a_{n-1})\in A^n$ the player I considers the sum $s_n=\theta+\sum_{i=1}^{n-1}a_i$ and if $s_n\notin A$, then I puts $\yen(a_1,\dots,a_n):=X$. If $s_n\in A$, then the player I chooses a neighborhood $V\subset X$ of $\theta$ such that $A\cap (s_n+V)\subset \{a\in A:\rho(a,s_n)<\frac1{2^n}\}$ and put $\yen(a_1,\dots,a_{n-1}):=V$.

Let us show that the strategy $\yen$ is winning in the {\sf Convexity game} on $A=A-\theta$. Given any sequence $(a_n)_{n\in\IN}\in A^\IN$ we should prove that the series $\sum_{n=1}^\infty a_n$ converges to a point of $A$ if $a_n\in \yen(a_1,\dots,a_{n-1})\cap\frac1{2^n}A$ for every $n\in\IN$.

Lemma~\ref{l:sum} implies that for every $n\in\IN$ the element $s_n:=\sum_{k=1}^n a_k\in\sum_{k=1}^n\frac1{2^k}A\subset \frac1{2^n}\theta+\sum_{k=1}^n\frac1{2^k}A\subset A$. In this case the choice of $\yen(a_1,\dots,a_{n-1})\ni a_n$ ensures that $\rho(s_n,s_{n+1})=\rho(s_n,s_n+a_{n+1})<\frac1{2^n}$, which implies that the sequence $(s_n)_{n=1}^\infty$ is Cauchy in the complete metric space $(A,\rho)$ and hence converges to some element $a\in A$, equal to the sum of the series $\sum_{n=1}^\infty a_n$.
\end{proof}

We recall that a subset $A$ of a topological group $X$ is {\em Steinhaus} at a point $x\in X$ if for each neighborhood $U\subset X$ of $x$, the set $(U\cap A)-(U\cap A)$ is a neighborhood of $\theta$ in $X$.

The following lemma is the most difficult technical result of this section.

\begin{lemma}\label{l:main-convex} Assume that a mid-convex set $A$ in a Polish group $X$ is both winning and Steinhaus at some point $a\in A$. Then $A$ is not null-finite and not Haar-scattered in $X$. Moreover, if the subgroup $\frac\theta{2^\IN}$ is closed in $X$, then $A$ is not Haar-countable.
\end{lemma}

\begin{proof} Replacing $A$ by its shift $A-a$, we can assume that $A$ is winning and Steinhaus at $\theta$. 

Observe that the subgroup $\langle A\rangle$, generated by the mid-convex set $A\ni \theta$, is open and 2-divisible. So, we can replace $X$ by the open subgroup $\langle A\rangle$ and assume that the Polish group $X$ is 2-divisible. Then the continuous homomorphism $\mathit 2:X\to X$, $\mathit 2:x\mapsto x+x$, is surjective and hence open (by the Open Mapping Principle \ref{c:OMP}).

Fix a complete invariant metric $\rho\le 1$, generating the topology  of the Polish group $X$. For an element $x\in X$ put $\|x\|:=\rho(x,0)$, and for $\e>0$ let $B(x;\e):=\{y\in X:\|x-y\|<\e\}$ be the open $\e$-ball around $x$ in the metric space $(X,\rho)$. Let $\tau_\theta$ denote the family of open neighborhoods of $\theta$ in $X$.

Since the set $A$ is winning at $\theta$, the player I has a winning strategy $\yen$ in the {\sf Convexity Game} on $A$. This strategy is a function $\yen:A^{<\w}\to\tau_\theta$ assinging to each finite sequence $(a_1,\dots,a_{n})\in A^{<\w}$ an open neighborhood $\yen(a_1,\dots,a_n)$ of $\theta$ in $X$ such that for any sequence $(a_n)_{n\in\IN}\in A^\IN$ the series $\sum_{n=1}^\infty a_n$ converges to a point of $A$ provided $a_n\in\yen(a_1,\dots,a_{n-1})\cap\frac1{2^n}A$ for all $n\in\IN$.

In the following two claims we shall prove that the set $A$ in not null-finite and not Haar-scattered in $X$.

\begin{claim}\label{cl:null-fin} The set $A$ is not null-finite in $X$.
\end{claim}

\begin{proof} Given a null sequence $(x_n)_{n\in\w}$ in $X$, we should find a point $s\in X$ such that the set $\{n\in\w:x_n+s\in A\}$ is infinite.

By induction for every $k\in\w$ we shall construct two neighborhoods $U_k$ and $V_k$ of $\theta$ in $X$, a number $n_k\in\IN$, and points $b_k^+,b_k^-,a_k^+,a_k^-\in A$ such that the following conditions are satisfied:
\begin{enumerate}
\item[\textup{(1)}] $\theta\in U_k\subset B(\theta;\frac1{2^k})\cap\yen(a_1^{\e_1},\dots,a_{k-1}^{\e_{k-1}})$ for any signs $\e_1,\dots,\e_{k-1}\in\{+,-\}$;
\item[\textup{(2)}] $V_k:=(A\cap 2^kU_k)-(A\cap 2^kU_k)$;
\item[\textup{(3)}] $2^kx_{n_k}\in V_k$ and $n_k>n_{k-1}$;
\item[\textup{(4)}]  $b_k^+,b_k^-\in A\cap 2^kU_k$ and $2^kx_{n_k}\in b_k^+-b_k^-$;
\item[\textup{(5)}] $a_k^-,a_k^+\in U_k\cap\frac1{2^k}A\subset A$ and $2^ka_k^+=b_k^+$, $2^ka_k^-=b_k^-$.
\end{enumerate}

We start the inductive construction letting $n_0=1$.
Assume that for some $k\in\IN$, the numbers  $n_0,\dots,n_{k-1}$, and points $b_1^+,\dots,b_{k-1}^+$ and $a_1^-,\dots,a_{k-1}^-$ have been constructed.

Choose a neighborhood $U_k$ satisfying the condition (1). By the oppenness of the homomorphism $\mathit 2^k:X\to X$, the set $\mathit 2^k U_k$ is a neighborhood of $\theta$ in $X$. Since  $A$ is Steinhaus at $\theta$, the set
$V_k:=(A\cap 2^kU_k)-(A\cap 2^kU_k)$  is a neighborhood of $\theta$ in $X$. Since $\lim_{n\to\infty}2^kx_n=\theta\in V_k$, we can find a number $n_k>n_{k-1}$ such that $2^kx_{n_k}\in V_k$. By the definition of the set $V_k\ni 2^kx_{n_k}$, there are two points $b_k^+,b_k^-\in A\cap 2^kU_k$ such that $2^kx_{n_k}\in b_k^+-b_k^-$. Since $b_k^+,b_k^-\in A\cap 2^kU_k$, there are two points $a_k^+,a_k^-\in U_k\cap\frac1{2^k}A\subset A$ such that $2^ka_k^+=b_k^+$ and $2^ka_k^-=b_k^-$. Then $2^kx_{n_k}= b_k^+-b_k^-=2^k(a_k^+-a_k^-)$ and hence $x_{n_k}\in a_k^+-a_k^-+\frac\theta{2^k}$. This completes the inductive step.

After completing the inductive construction, consider the series $\sum_{k=1}^\infty a_k^-$ and observe that it converges to some element $s\in X$ as $\sum_{k=1}^\infty\|a_k^-\|\le\sum_{k=1}^\infty\frac1{2^k}=1$.

We claim that $x_{n_i}+s\in A$ for every $i\in\IN$. Indeed,
$$x_{n_i}+s\in a_i^+-a_i^-+\tfrac\theta{2^i}+\sum_{k=1}^\infty a_k^-=\tfrac\theta{2^i}+\sum_{k=1}^\infty a_k^{\e_k}$$where $\e_i=+$ and $\e_k=-$ for all $k\in\IN\setminus\{i\}$. The inductive conditions (1) and (5) guarantee that the series $\sum_{k=1}^\infty a_k^{\e_k}$ converges to some point of $A$. Then $x_{n_i}+s\in \frac\theta{2^i}+\sum_{k=1}^\infty a_k^{\e_k}\subset\frac\theta{2^i}+A=A$.

 Now we see that the set $\{n\in\w:x_n+s\in A\}\supset\{n_i\}_{i=1}^\infty$ is infinite, witnessing that the set $A$ is not null-finite.
\end{proof}

\begin{claim}\label{cl:null-fin} The set $A$ is not Haar-scattered in $X$. If the subgroup $\frac\theta{2^\IN}$ is closed in $X$ or $A$ is Polish, then $A$ is not Haar-countable.
\end{claim}

\begin{proof} Given any continuous map $f:2^\w\to X$ we should find an element $s\in X$ such that $f^{-1}(A+s)$ is not scattered (or uncountable). If for some $s\in X$ the fiber $f^{-1}(s)$ is uncountable, then $f^{-1}(s)$ is not scattered and the preimage $f^{-1}(A+s)\supset f^{-1}(s)$ is not scattered, too. So, we assume that for each $x\in X$ the set $f^{-1}(x)$ is countable and hence $K=f(2^\w)$ is a compact set without isolated points in $X$. 
Replacing $K$ by a suitable shift, we can assume that $\theta\in K$.

Write the set $2^{<\w}:=\bigcup_{k\in\w}2^k$ of finite binary sequences as $2^{<\w}=\{\alpha_i\}_{i\in\IN}$ so that $2^n=\{\alpha_i:2^n\le i<2^{n+1}\}$ for all $n\in\w$. For a number $k<n$ let $\alpha{\restriction}k:=(\alpha_0,\dots,\alpha_{k-1})$ be the restriction of $\alpha$.

 For every $k\ge 2$ let ${\downarrow}k$ be the unique number such that $\alpha_{\downarrow k}=\alpha_k{\restriction}(|\alpha_k|-1)$. So, ${\downarrow}k$ is the predecessor of $k$ in the tree structure on $\IN$, inherited from the tree $2^{<\w}=\{\alpha_i\}_{i\in\IN}$. 

For every $k\in\w$ we shall inductively construct three neighborhoods $U_k,V_k,W_k$ of $\theta$, and points $x_i\in K$, $a_k^+,a_k^-\in U_k$ such that the following conditions are satisfied:
\begin{enumerate}
\item[\textup{(1)}] $\theta\in U_k\subset B(\theta;\frac1{2^k})\cap\yen(a_1^{\e_1},\dots,a_{k-1}^{\e_{k-1}})$ for any signs $\e_1,\dots,\e_{k-1}\in\{+,-\}$;
\item[\textup{(2)}] $V_k:=(A\cap 2^kU_k)-(A\cap 2^kU_k)$;
\item[\textup{(3)}] $W_k:=B(\theta;\frac1{2^k})\cap\frac1{2^k}V_k$;
\item[\textup{(4)}] $x_1=\theta$ and $x_k\in K\cap (x_{{\downarrow}k}+W_k)\setminus\{x_{{\downarrow}k}\}$ if $k>1$;
\item[\textup{(5)}]   $2^k(x_{k}-x_{{\downarrow}k})=b_k^+-b_k^-$ and $b_k^+,b_k^-\in A\cap 2^kU_k$;
\item[\textup{(6)}] $a_k^+,a_k^-\in U_k\cap\frac1{2^k}A\subset A$ and  $2^ka_k^+=b_k^+$, $2^ka^-_k=b_k^-$;
\item[\textup{(7)}] $x_k-x_{{\downarrow}k}\in a_k^+-a_k^-+\frac\theta{2^k}$. 
\end{enumerate}

We start the inductive construction letting $\{x_\alpha\}_{\alpha\in 2^0}=\{\theta\}$. Assume that for some $k\in\IN$ and all $i<k$  neighborhoods $U_i,V_i,W_i$ and points $x_i,b_i^+,b_i^-,a_i^+,a_i^-$ have been constructed so that the conditions (1)--(7) are satisfied.

Choose a  neighborhood $U_k$ satisfying the condition (1). The openness of the homomorphism $\mathit 2^{k}:X\to X$, $\mathit 2^{k}:x\mapsto 2^{k}x$, implies that $2^kU_k$ is a neighborhood of $\theta$ in $X$. Since $A$ is Steinhaus at $\theta$, the set $V_k:=(A\cap 2^kU_k)-(A\cap 2^kU_k)$ is a neighborhood of $\theta$ in $X$. 
By the continuity of the homomorphism $\mathit 2^{k}:X\to X$, the set $W_k:=B(\theta;\frac1{2^k})\cap\frac1{2^k}V_k$ is a neighborhood of $\theta$ in $X$.

Taking into account that the set $K\ni \theta$ has no isolated points, we can choose a point $x_k\in K$ satisfying the condition (4). Since $x_k-x_{{\downarrow}k}\in\frac1{2^k}V_k$, there are points $b_k^+,b_k^-\in A\cap 2^kU_k$ such that 
$2^k(x_k-x_{{\downarrow}k})=b_k^+-b_k^-$ and points $a_k^+,a_k^-\in U_k$ such that $2^ka_k^+=b_k^+$ and $2^ka_k^-=b_k^-$. It follows that $a_k^{+}\in\frac1{2^k}b_k^+\in \frac1{2^k}A\subset A$ and by analogy $a_k^-\in \frac1{2^k}A\subset A$.  
Observe that $2^k(x_k-x_{{\downarrow}k})=b_k^+-b_k^-=2^k(a_k^+-a_k^-)$ and hence $x_k-x_{{\downarrow}k}=a_k^+-a_k^-+\frac\theta{2^k}$, so the condition (7) is satisfied. This completes the inductive step.
\smallskip

After completing the inductive construction, consider the series $\sum_{k=1}^\infty a_k^-$. The choice of the winning strategy $\yen$ and the conditions (1) and (6) guarantee that the series $\sum_{k=1}^\infty a_k^-$ converges to a point $s\in A$. So, $s=\lim_{m\to\infty}s_m$ where $s_m=\sum_{k=1}^ma_k^-$.

We claim that $x_k+s\in A$ for every $k\in\IN$. Given any $k\in\IN$, find a number $n\in\w$ with $\alpha_k\in 2^n$ and for every $i\le n$ find a unique number $k_i\in\IN$ such that $\alpha_{k_i}=\alpha_k{\restriction}i$. It follows that $k_n=k$, $k_0=1$, and $k_i={\downarrow}k_{i+1}$ for $0\le i<n$. Observe that for every $m\ge k_i$ we get 
\begin{multline*}
s_m+x_k=s_m+\sum_{i=1}^n(x_{k_i}-x_{k_{i-1}})=s_m+\sum_{i=1}^n(x_{k_i}-x_{{\downarrow}k_i})\\
\in\sum_{j=1}^m a_j^-+\sum_{i=1}^n(a_{k_i}^+-a_{k_i}^-+\tfrac\theta{2^{k_i}})\subset\tfrac\theta{2^k}+\sum_{j=1}^m a_j^{\e_j}
\end{multline*}
where $$
\e_j=
\begin{cases}+&\mbox{if $j=k_i$ for some $1\le i\le n$};\\
-&\mbox{otherwise}.
\end{cases}
$$
It follows that $s_m+x_k=z_m+\sum_{j=1}^ma_j^{\e_j}$ for some $z_m\in\frac\theta{2^k}$. 

The inductive conditions  (1) and (6) guarantee that the series 
$\sum_{j=1}^\infty a_j^{\e_j}$ converges to a point $s'$ of the set $A$. Then the sequence $(z_m)_{m=k}^\infty$ converges to $s+x_k-s'$. Since the subgroup $\frac\theta{2^k}$ is closed in $X$, the limit point $s+x_k-s'=\lim_{m\to\infty}z_m$ belongs to $\frac\theta{2^k}$. Then $s+x_k=(s+x_k-s')+s'\in\frac\theta{2^k}+A=A$ and we are done.

The inductive conditions (3) and (4) ensure that the set $\{x_k+s\}_{k=1}^\infty\subset (K+s)\cap A$ has no isolated points.  So, $A$ is not Haar-scattered. If $A$ is Polish, the set $A$ is not Haar-countable by Proposition~\ref{p:scattered}.

Now assuming that the subgroup $\frac\theta{2^\IN}$ is closed in $X$, we shall show that $A$ contains the closure $F$ of the set $\{x_k+s\}_{k\in \IN}$ in $K+s$.
This closure has no isolated points and hence is a compact uncountable set (by the Baire Theorem).

Consider the map $f:2^\w\to F$ assigning to each binary sequence $\beta\in 2^\w$ the limit of the sequence $(s+x_{k_i})_{i\in\w}$ where for every $i\in\IN$, $k_i\in\IN$ is a unique number such that $\alpha_{k_i}=\beta{\restriction}i$. Observe that for every $i>1$ we have $x_{{\downarrow}k_i}=x_{k_{i-1}}$  and hence $\|x_{k_i}-x_{k_{i-1}}\|<\frac1{2^{k_i}}$ by the inductive assumptions (3) and (4). This implies that the sequence $(x_{k_i})_{i=1}^\infty$ is Cauchy in $X$ and hence converges to some point $x_\beta\in X$ such that 
\begin{equation}\label{eq:convex2}
\|x_\beta-x_{k_i}\|\le\sum_{j=i+1}^\infty\frac{1}{2^{k_j}}\le \frac2{2^{k_{i+1}}}
\end{equation}
for all $i\in\IN$. 

By the inductive condition (7), for every $i\in\IN$ there exists a point $z_i\in\frac\theta{2^{k_i}}\subset \frac\theta{2^\IN}$ such that $x_{k_i}-x_{k_{i-1}}=x_{k_i}-x_{{\downarrow}k_i}=a_{k_i}^+-a_{k_i}^-+z_i$. Then 
$$\|z_i\|\le \|x_{k_i}-x_{{\downarrow}k_i}\|+\|a_{k_i}^+\|+\|a-{k_i}^-\|<\frac1{2^{k_i}}+\frac2{2^{k_i}}$$and hence the series $\sum_{i=1}^\infty z_i$ converges to some point $z_\infty$ of the closed subgroup $\frac\theta{2^\IN}$. 

It follows that $x_\beta=\sum_{i=1}^\infty (x_{k_i}-x_{k_{i-1}})= \sum_{i=1}^\infty (x_{k_i}-x_{{\downarrow}k_i})=\sum_{i=1}^\infty (a_{k_i}^+-a_{k_i}^-+z_i)=z_\infty+\sum_{i=1}^\infty(a_{k_i}^+-a_{k_i}^-)$.
Then $$f(\beta)=x_\beta+s=z_\infty+\sum_{i=1}^\infty(a_{k_i}^+-a_{k_i}^-)+\sum_{k=1}^\infty a_k^-=z_\infty+\sum_{k=1}^\infty a_k^{\e_k}$$ where 
$$\e_k=\begin{cases}
+&\mbox{if $k=k_i$ for some $i\in\IN$};\\
-&\mbox{otherwise}.
\end{cases}
$$
The inductive conditions (1) and (6) imply that the series $\sum_{k=1}^\infty a_k^{\e_k}$ converges to some point of $A$. Then, $f(\beta)=x_\beta+s=z_\infty+\sum_{k=1}^\infty a_k^{\e_k}\in \frac\theta{2^\IN}+A=A$.

The inequality \eqref{eq:convex2} implies that the map $f:2^\w\to F$, $f:\beta\mapsto x_\beta+s$, is continuous. Then $f(2^\w)=F$ by the compactness of $f(2^\w)$ and density of $f(2^\w)$ in $F$. Finally we conclude that $F=f(2^\w)\subset A\cap(K+s)$, witnessing that $A$ is not Haar-countable.
 \end{proof}
\end{proof}

\begin{corollary}\label{c:main-convex} Assume that a mid-convex $G_\delta$-set $A$ in a Polish group $X$ is Steinhaus at some point $a\in A$. Then $A$ is not null-finite and not Haar-countable in $X$.
\end{corollary}

\begin{proof} By Proposition~\ref{p:Polish-convgame}, the mid-convex $G_\delta$-set $A$, being a Polish space,  is winning at the point $a\in A$. By Lemma~\ref{l:main-convex}, $A$ is not null-finite and not Haar-scattered. By Proposition~\ref{p:scattered}, $A$ is not Haar-countable.
\end{proof}


\begin{lemma}\label{l:specpoint} A mid-convex analytic subset $A$ of a Polish group $X$ is Steinhaus at some point $a\in\bar A$ if and only if the set $A-A$ is a neighborhood of zero.
\end{lemma}

\begin{proof} The ``only if'' part is trivial. To prove the ``if'' part, assume that $A-A$ is a neighborhood of $\theta$. Replacing $A$ by a suitable shift of $A$, we can assume that $\theta\in A$. In this case the subgroup $\langle A\rangle$ of $X$ generated by $A$ is mid-convex and hence 2-divisible. Since $A-A$ is a neighborhood of zero in $X$, the subgroup $\langle A\rangle$ is open in $X$. Replacing $X$ by the subgroup $\langle A\rangle$, we can assume that the group $X$ is 2-divisible. Then the homomorphism $\mathit 2:X\to X$, $\mathit 2:x\mapsto x+x$, is surjective and  open (by the Open Mapping Principle \ref{c:OMP}).

Fix a complete invariant metric $\rho\le 1$, generating the topology  of the Polish group $X$. For an element $x\in X$ put $\|x\|:=\rho(x,\theta)$.
\smallskip

To prove the lemma, we shall construct inductively a sequence of points $(a_n)_{n\in\w}$ in $A$ such that for every $n\in \w$ the point $a_{n+1}$ belongs to the set $A_n:=A\cap B(a_n;\frac1{2^n})$ and the set $A_n-A_n+\frac\theta{2^n}$ is a neighborhood of zero in $X$.

We start the inductive construction choosing any point $a_0\in A$ and observing that $A_0=A$ and $A_0-A_0+\frac\theta{2^0}=A-A$ is a neighborhood of zero in $X$.

Assume that for some $n\in\IN$ we have constructed a point $a_n\in A$ such that $A_n-A_n+\frac\theta{2^n}$ is a neighborhood of zero, where $A_n:=A\cap B(a_n;\frac1{2^n})$. Choose $\e>0$ so small that $B(\theta;3\e)\subset (A_n-A_n+\frac\theta{2^n})\cap\mathit 2(B(\theta;\frac1{2^{n+1}}))$ (the choice of $\e$ is possible since the homomorphism $\mathit 2:X\to X$ is open). Let $\U$ be any countable cover of the set $A_n=\bigcup\U$ by analytic subspaces  of diameter $<\e$. Since $B(\theta;\e)\subset A_n-A_n+\frac\theta{2^n}=\bigcup_{U,V\in\U}(U-V+\frac\theta{2^n})$, we can apply the Baire Theorem and find two sets $U,V\in\U$ such that  the (analytic) set $M=B(0;\e)\cap(U-V+\frac\theta{2^n})$ is not meager in $X$. By the Piccard-Pettis Theorem~\ref{c:PP}, the set $M-M$ is a neighborhood of zero in $X$. The continuity of the map $\mathit 2:X\to X$ ensures that the set $\frac12(M-M)$ is a neighborhood of zero. Since $B(\theta;\e)\cap(U-V+\frac\theta{2^n})\ne\emptyset$, there are points $a_{n+1}\in U,$ $b_{n+1}\in V$ and $z\in\frac\theta{2^n}$ such that $\|a_{n+1}-b_{n+1}+z\|<\e$. 
We claim that $\frac12(M-M)\subset A_{n+1}-A_{n+1}+\frac\theta{2^{n+1}}$ where 
$A_{n+1}:=A\cap B(a_{n+1};\frac1{2^{n+1}})$.

Given any point $x\in\frac12(M-M)$, find points $a,b\in M=B(\theta;\e)\cap(U-V+\frac\theta{2^n})$ such that $2x=a-b$.
Next, find points $u_a,u_b\in U$, $v_a,v_b\in V$, and $z_a,z_b\in\frac\theta{2^n}$  such that $a=u_a-v_a+z_a$ and $b=u_b-v_b+z_b$.  Then $2x=a-b=(u_a-v_a+z_a)-(u_b-v_b+z_b)=(u_a+v_b+z_a)-(v_a+u_b+z_b)$ and hence $x\in \frac12(u_a+v_b+z_a)-\frac12(v_a+u_b+z_b)$.

Observe that $\|(u_a+v_b-z)-2a_{n+1}\|\le\|u_a-a_{n+1}\|+\|v_b-z-a_{n+1}\|< \e+\|v_b-b_{n+1}\|+\|b_{n+1}-z-a_{n+1}\|<3\e$. Since $B(\theta;3\e)\subset\mathit 2\big(B(\theta;\frac1{2^{n+1}})\big)$, there exists an element $s\in B(\theta;\frac1{2^{n+1}})$ such that $2s=(u_a+v_b-z)-2a_{n+1}$. Then $a_{n+1}+s\in\frac12(u_a+v_b-z)\subset \frac12(U+V+\frac\theta{2^n})\subset\frac12(A+A+\frac\theta{2^n})\subset A$ by the mid-convexity of $A=A+\frac\theta{2^n}$ and $\frac12(u_a+v_b-z)\subset\frac\theta{2}+a_{n+1}+s\subset \frac\theta{2}+A\cap B(a_{n+1};\frac1{2^{n+1}})=\frac\theta{2}+A_{n+1}$.
By analogy we can show that $\frac12(v_a+u_b-z)\subset 
\frac\theta{2}+A_{n+1}$.

Then $x\in \frac12(u_a+v_b+z_a)-\frac12(v_a+u_b+z_b)=\frac12(u_a+v_b-z)-\frac12(v_a+u_b-z)+\frac12(z_b-z_a)\subset A_{n+1}-A_{n+1}+\frac\theta{2^{n+1}}$ and hence $A_{n+1}-A_{n+1}+\frac\theta{2^{n+1}}\supset \frac12(M-M)$ is a neighborhood of zero in $X$.
This completes the inductive construction.
\smallskip

Now consider the sequence $(a_n)_{n\in\w}$ and observe that it is Cauchy (as $\|a_{n+1}-a_n\|<\frac1{2^n}$ for all $n\in\w$), and hence it converges to some point $a\in \bar A$ such that $$\|a-a_n\|\le\sum_{k=n}^\infty\|a_{k+1}-a_k\|<\sum_{k=n}^\infty\tfrac1{2^k}=\frac1{2^{n+1}}$$for every $n\in\IN$.

We claim $A$ is Steinhaus at the point $a$. Since $A=A+\frac\theta{2^\IN}$, it suffices to check that for any neighborhood $U\subset X$ of $a$ the set $(U\cap A)-(U\cap A)+\frac\theta{2^\IN}$ is a neighborhood of $\theta$ in $X$ (see Lemma~\ref{l:SteinH}). 

Given any neighborhood $U\subset X$ of $a$, find $n\in\w$  such that $B(a;\frac{3}{2^n})\subset U$ and observe that 
$U\cap A\supset B(a;\tfrac3{2^n})\cap A\supset B(a_n;\tfrac1{2^n})\cap A=A_n$.  Then the set $(U\cap A)-(U\cap A)+\frac\theta{2^\IN}\supset A_n-A_n+\frac\theta{2^n}$ is a neighborhood of $\theta$ in $X$.
\end{proof}

Lemma~\ref{l:specpoint} implies

\begin{corollary}\label{c:cl-fp}
A closed mid-convex set $A$ in a Polish group $X$ is Steinhaus at some point $a\in A$ if and only if $A-A$ is a neighborhood of $\theta$ in $X$.
\end{corollary}

For closed mid-convex sets we can extend the characterization given in Theorem~\ref{t:main-convex-analytic} with two more properties.

\begin{theorem} For a closed mid-convex subset $A$ of a Polish group $X$ the following conditions are equivalent:
\begin{enumerate}
\item[\textup{1)}] $A-A$ is not a neighborhood of $\theta$ in $X$;
\item[\textup{2)}] $A-A$ is meager in $X$;
\item[\textup{3)}] $A$ is generically Haar-1;
\item[\textup{4)}] $A$ is Haar-thin;
\item[\textup{5)}] $A$ is null-finite;
\item[\textup{6)}] $A$ is Haar-countable.
\end{enumerate}
\end{theorem}

\begin{proof} The equivalence of the statements (1)--(4) was proved in Theorem~\ref{t:main-convex-analytic}. The implications $(3)\Ra(5,6)$ are trivial. To finish the proof, it suffices to show that the negation of (1) implies the negations of (5) and (6). So, assume that $A-A$ is a neighborhood of $\theta$ in $X$. By Corollary~\ref{c:cl-fp}, $A$ is Steinhaus at some point $a\in A$ and by Proposition~\ref{p:Polish-convgame}, $A$ is winning mid-convex at the point $a$. By Lemma~\ref{l:main-convex}, $A$ is not Haar-null and not Haar-scattered. By Proposition~\ref{p:scattered}, $A$ is not Haar-countable.
\end{proof}

We shall say that a subset $A$ of a topological group $X$ is  \index{subset!$\midd$-convex}\index{$\midd$-convex subset}{\em $\midd$-convex} if  $A$ is mid-convex in $X$ and for any points $a,b\in A$ and any neighborhood $U\subset A$ of $a$ there exists a point $x\in U$ such that $2^n(x-a)=b-a$ for some $n\in\IN$. 

For example, any convex set in a topological vector space is $\midd$-convex.

\begin{lemma}\label{l:St-conv} For a $\midd$-convex analytic set in a Polish group $X$ the following conditions are equivalent:
\begin{enumerate}
\item[\textup{1)}] $A$ is Steinhaus at some point of $A$;
\item[\textup{2)}] $A$ is Steinhaus at any point of $A$;
\item[\textup{3)}] $A-A$ is a neighborhood of $\theta$;
\item[\textup{4)}] $A-A$ is non-meager.
\end{enumerate}
\end{lemma}

\begin{proof} 
 $(1)\Ra(2)$ Assuming that $A$ is Steinhaus at some point  $a\in A$, we shall prove that $A$ is Steinhaus at any point $a'\in A$. Replacing $A$ by $A-a'$, we can assume that $a'=\theta$. 

Since $A-A$ is a neighborhood of $\theta$, the subgroup $\langle A\rangle$ generated by the mid-convex set $A\ni\theta$ is open and 2-divisible. Replacing $X$ by the open subgroup $\langle A\rangle$, we can assume that $X$ is 2-divisible. In this case, the homomorphism $\mathit 2:X\to X$ is surjective and open (by the Open Mapping Principle \ref{c:OMP}).

By Lemma~\ref{l:SteinH}, the Steinhaus property of the mid-convex set $A=A+\frac\theta{2^\IN}$ will follow as soon as we check that for every neighborhood $U\subset X$ of $\theta$, the set $(U\cap A)-(U\cap A)+\frac\theta{2^\IN}$ is a neighborhood of $\theta$ in $X$. Given a neighborhood $U\subset X$ of $\theta$, use the $\midd$-convexity of $A$ and find a point $x\in U\cap A$ such that $2^nx=a$ for some $n\in\IN$. Since the homomorphism $\mathit 2^n:X\to X$ is open, the set $2^nU$ is a neighborhood of $a$. Since $A$ is Steinhaus at $a$, the set $(A\cap 2^nU)-(A\cap 2^nU)$ is a neighborhood of $\theta$. The continuity of the homomorphism $\mathit 2^n$ ensures that the set
$$\tfrac1{2^n}\big((A\cap 2^nU)-(A\cap 2^nU)\big)\subset (U\cap\tfrac1{2^n}A)-(U\cap\tfrac1{2^n}A)+\tfrac\theta{2^n}\subset (U\cap A)-(U\cap A)+\tfrac\theta{2^\IN}$$
is a neighborhood of $\theta$ in $X$.
\smallskip

The implications $(2)\Ra(3)\Ra(4)$ are trivial. So, it remains to prove that $(4)\Ra(1)$. Assume that the set $A-A$ is not meager in $X$ and consider the map $\delta:A\times A\to A-A$, $\delta(x,y)\mapsto x-y$.
Let $\W$ be the family of open sets $W\subset A\times A$ whose image $\delta(W)$ is meager in $X$. The space $\cup\W$, being Lindel\"of,  coincides with the union $\bigcup\W'$ of some countable subfamily $\W'$ of $\W$. This implies that the image $\delta(\cup\W)=\bigcup_{W\in\W'}\delta(W)$ is meager in $X\times X$. Since the set $\delta(A\times A)=A-A$ is not meager, there is a pair $(b,c)\in (A\times A)\setminus \bigcup\W$. By the definition of $\W$ for any neighborhoods $U_b,U_c\subset A$ of $b,c$, the image $\delta(U_b\times U_c)=U_b-U_c$ is not meager in $X$. 

Now we shall prove that the set $A$ is Steinhaus at the point $c$. Replacing $A$ by the shift $A-c$, we can assume that there exists a point $a\in A$ such that for any open sets $U_\theta\ni\theta$ and $U_a\ni a$ in $A$, the set $U_\theta-U_a$ is not meager in $X$.

By Lemma~\ref{l:SteinH}, it suffices to show that for any open neighborhood $U\subset X$ of $\theta$ the set $(U\cap A)-(U\cap A)+\frac\theta{2^\IN}$ is a neighborhood of $\theta$ in $X$.  Since the set $A\ni \theta$ is mid-convex, so is its group hull $\langle A\rangle$ in $X$. Since the analytic set $A-A$ is not meager, the set $(A-A)-(A-A)$ is a neighborhood of zero, which implies that the subgroup $\langle A\rangle$ is open in $X$. Replacing $X$ by the subgroup $\langle A\rangle$, we can assume that the Polish group $X$ is 2-divisible, which implies that the homomorphism $\mathit 2:X\to X$ is surjective and open (by the Open Mapping Principle \ref{c:OMP}).

By the openness of the homomorphism $\mathit 2$, the set $2U=\{u+u:u\in U\}$ is an open neighborhood of $\theta$ in $X$. Next, choose an open neighborhood $W\subset X$ of $\theta$ such that $W+W\subset 2U$. Since the set $A$ is $\midd$-convex, there exist a point $x\in W$ such that $2^nx=a$ for some $n\in\IN$.  By the openness of the homomorphism $\mathit 2^n:X\to X$ the analytic set $V:=A\cap 2^nW$ is an open neighborhood of the points $\theta$ and $a$ in $A$. The choice of the points $\theta$ and $a$ ensures that the analytic set $V-V$ is not meager in $X$. By the Piccard-Pettis Theorem~\ref{c:PP}, the set $V_\pm:=(V-V)-(V-V)=(V+V)-(V+V)$ is a neighborhood of $\theta$ in $X$. By the continuity of the homomorphism $\mathit 2^{n+1}$, the set $\frac1{2^{n+1}}V_\pm$ is a neighborhood of $\theta$ in $X$. Observe that $\frac1{2^{n+1}}V_{\pm}=\frac1{2^{n+1}}(V+V)-\frac1{2^{n+1}}(V+V)$. 

To finish the proof it suffices to show that $\frac1{2^{n+1}}(V+V)\subset (U\cap A)+\frac\theta{2^\IN}$. Given any point $x\in\frac1{2^{n+1}}(V+V)$, find two points $v_1,v_2\in V$ such that $2^{n+1}x=v_1+v_2$. Since $v_1,v_2\in V=A\cap 2^nW$, there are points $w_1,w_2\in W$ such that $2^nw_i=v_i$ for $i\in\{1,2\}$.
It follows that $w_i\in\frac1{2^n}v_i\subset\frac1{2^n}A\subset A$ (here we use that $\theta\in A$) and $w_1+w_2\in W+W\subset 2U$ and hence $w_1+w_2=u+u$ for some $u\in U$. It follows that $u\in\frac12(w_1+w_2)\subset\frac12(A+A)=A$. Now observe that 
\begin{multline*}
x\in\tfrac1{2^{n+1}}(v_1+v_2)=\tfrac1{2^{n+1}}(2^n(w_1+w_2))=\tfrac12(w_1+w_2)+\tfrac\theta{2^{n+1}}\\
=\tfrac12(u+u)+\tfrac\theta{2^{n+1}}= u+\tfrac\theta{2^{n+1}}\subset U+\tfrac\theta{2^\IN}
\end{multline*} 
and hence $u\in A\cap (U+\frac\theta{2^\IN})=(A\cap U)+\frac\theta{2^\IN}$ as $A=A+\frac\theta{2^\IN}$. Finally, $x\in u+\frac\theta{2^\IN}\subset (U\cap A)+\frac\theta{2^\IN}$.
\end{proof}

\begin{theorem}\label{t:analytic-midd} Assume that a $\midd$-convex analytic set $A$ in a Polish group is winning at some point $a\in A$. Then the following conditions are equivalent:
\begin{enumerate}
\item[\textup{1)}] $A-A$ is not a neighborhood of zero in $X$;
\item[\textup{2)}] $A-A$ is meager in $X$;
\item[\textup{3)}] $A$ is generically Haar-$1$;
\item[\textup{4)}] $A$ is Haar-thin;
\item[\textup{5)}] $A$ is null-finite;
\item[\textup{6)}] $A$ is Haar-scattered.
\end{enumerate}
If the subgroup  $\frac\theta{2^\IN}$ is closed in $X$, then the conditions \textup{(1)--(6)} are equivalent to  
\begin{itemize}
\item[\textup{7)}] $A$ is Haar-countable.
\end{itemize}
\end{theorem}

\begin{proof}  The equivalence of the statements (1)--(4) was proved in Theorem~\ref{t:main-convex-analytic}. The implications $(3)\Ra(5,6)$ and $(6)\Ra(7)$ are trivial. To finish the proof, it suffices to show that the negation of (1) implies the negations of (5), (6) (and (7) if $\frac\theta{2^\IN}$ is closed in $X$).

 So, assume that $A-A$ is a neighborhood of $\theta$ in $X$. By Corollary~\ref{l:St-conv}, $A$ is Steinhaus at the point $a$. By Lemma~\ref{l:main-convex}, $A$ is not Haar-null and not Haar-scattered. Moreover, this lemma implies that $A$ is not Haar-countable if the subgroup $\frac\theta{2^\IN}$ is closed in $X$.
 \end{proof}
 
For $\midd$-convex $G_\delta$-sets in a Polish group we can prove the following  characterziation. 
 
\begin{theorem}\label{t:convex-Gdelta} For a $\midd$-convex $G_\delta$-set $A$ in a Polish group $X$ the following conditions are equivalent:
\begin{enumerate}
\item[\textup{1)}] $A-A$ is not a neighborhood of zero in $X$;
\item[\textup{2)}] $A-A$ is meager in $X$;
\item[\textup{3)}] $A$ is generically Haar-$1$;
\item[\textup{4)}] $A$ is Haar-thin;
\item[\textup{5)}] $A$ is null-finite;
\item[\textup{6)}] $A$ is Haar-scattered;
\item[\textup{7)}] $A$ is Haar-countable.
\end{enumerate}
\end{theorem}

\begin{proof} By Proposition~\ref{p:Polish-convgame}, the mid-convex set $A$, being Polish, is winning at any point $a\in A$. By Theorem~\ref{t:analytic-midd}, the conditions (1)--(7) are equivalent. The equivalence $(6)\Leftrightarrow(7)$ follows from Proposition~\ref{p:scattered}.
\end{proof}

\section{Smallness properties of convex sets in Polish vector spaces}\label{s19}

In this section we shall apply the results of the preceding section to study smallness properties of convex sets in Polish vector space. By a \index{Polish vector space}{\em Polish vector space} we understand a Polish topological linear space over the field of real numbers.

Since convex sets in topological vector spaces are $\midd$-convex, we can apply Theorems~\ref{t:analytic-midd}, \ref{t:convex-Gdelta} to obtain the following characterization.

\begin{theorem}\label{t:conv-Ban}  For an analytic convex subset $A$ of a Polish vector space $X$, the following conditions are equivalent:
\begin{enumerate}
\item[\textup{1)}] $A-A$ is not a neighborhood of zero in $X$;
\item[\textup{2)}] $A-A$ is meager in $X$;
\item[\textup{3)}] $A$ is generically Haar-$1$;
\item[\textup{4)}] $A$ is Haar-thin;
\item[\textup{5)}] $A$ is null-$n$ for some $n\in\IN$.
\end{enumerate}
If $A$ is $G_\delta$-set in $X$, then the conditions \textup{(1)--(5)} are equivalent to
\begin{itemize}
\item[\textup{6)}] $A$ is null-finite;
\item[\textup{7)}] $A$ is Haar-countable.
\end{itemize}
\end{theorem}

Natural examples of winning mid-convex sets arise from considering $\infty$-convex sets in topological vector spaces.

A subset $A$ of a topological vector space $X$ is called \index{subset!$\infty$-convex}{$\infty$-convex subset}{\em $\infty$-convex} if for any bounded sequence $(x_n)_{n=1}^\infty$ in $A$ the series $\sum_{n=1}^\infty\frac1{2^n}x_n$ converges to some point of $A$ (cf. \cite{BLM}, \cite{BBK}). We recall that a subset $B$ of a topological vector space $X$ is {\em bounded} if for any neighborhood $U\subset X$ of $\theta$ there exists $n\in\IN$ such that $B\subset nU$.

Each $\infty$-convex subset $C$ of a topological vector space is convex. Indeed, for any points $x,y\in C$ and any $t\in[0,1]$ we can find a subset $\Omega\subset\IN$ such that $t=\sum_{n\in\Omega}\frac1{2^n}$ and observe that $tx+(1-t)y=\sum_{n=1}^\infty\frac1{2^n}x_n\in C$ where $x_n=x$ for $n\in\Omega$ and $x_n=y$ for $n\in\IN\setminus \Omega$.

\begin{proposition}\label{p:infty-conv} Each $\infty$-convex subset $A$ of a metric linear space $X$ is winning mid-convex at each point $a\in A$.
\end{proposition}  

\begin{proof} Replacing $A$ by the shift $A-a$, we can assume that $a=\theta$. Fix an invariant metric $\rho$ generating the topology of $X$ and define a winning strategy of the player I in the {\sf Convexity Game} on $A$ letting $\yen(a_1,\dots,a_n):=B(\theta;\frac1{4^n})$ for every $n\in\IN$ and every $(a_1,\dots,a_n)\in A^n$. To show that the strategy $\yen$ is winning, take any sequence $(a_n)_{n\in\IN}\in A^\IN$ such that $a_{n}\in\frac1{2^n}A\cap\yen(a_1,\dots,a_{n-1})\subset B(\theta,\frac1{4^{n-1}})$ for every $n\in\IN$. For every $n\in\IN$ the point $2^na_n\in A$ has $\rho(2^na_n,\theta)\le 2^n\cdot\rho(a_n,0)<2^n\frac1{4^n}=\frac1{2^n}$, which implies that the sequence $(2^na_n)_{n=1}^\infty$ converges to $\theta$ and hence is bounded in $X$. By the $\infty$-convexity of $A$, the series $\sum_{n=1}^\infty\frac1{2^n}b_n=\sum_{n=1}^\infty a_n$ converges to some point of the set $A$.
\end{proof}

Combining Theorem~\ref{t:analytic-midd} with Proposition~\ref{p:infty-conv}, we obtain the following characterization of smallness properties of $\infty$-convex sets in Polish vector spaces.

\begin{theorem}\label{t:convex4} For an analytic $\infty$-convex subset of a Polish vector space $X$ the following conditions are equivalent:
\begin{enumerate}
\item[\textup{1)}] $A-A$ is not a neighborhood of zero in $X$;
\item[\textup{2)}] $A-A$ is meager in $X$;
\item[\textup{3)}] $A$ is generically Haar-$1$;
\item[\textup{4)}] $A$ is Haar-thin;
\item[\textup{5)}] $A$ is null-finite;
\item[\textup{6)}] $A$ is Haar-countable.
\end{enumerate}
\end{theorem}

This corollary can be compared with the following characterization of closed convex Haar null-sets in reflexive Banach spaces, proved in \cite{Mat} and \cite{MSt}.

\begin{theorem}[Matou\v skov\'a, Stegall]\label{t:Eva} For a closed convex set $A$ in a reflexive Banach space $X$ the following conditions are equivalent:
\begin{enumerate}
\item[\textup{1)}] $A$ is Haar-null in $X$;
\item[\textup{2)}] $A$ is Haar-meager in $X$;
\item[\textup{3)}] $A$ is nowhere dense in $X$.
\end{enumerate}
On the other hand, each non-reflexive Banach space contains a nowhere dense  closed convex cone $C\subset X$ which is thick in the sense that for each compact subset $K\subset X$ there exists $x\in X$ such that $K+x\subset C$. The thickness of $C$ implies that $C$ is not Haar-null and not Haar-meager.
\end{theorem}

\begin{problem} Are the properties of Haar-nullness and Haar-meagerness equivalent for analytic convex (additive) sets in Banach spaces?
\end{problem}

Now we apply Theorem~\ref{t:conv-Ban} to closed cones in  classical Banach spaces.

\begin{example} For any $p\in[1,+\infty)$, the positive convex cone $\ell_p^+=\{(x_n)_{n\in\w}\in\ell_p:\forall n\in\w\;\;x_n\ge 0\}$ in the Banach space $\ell_p$ has the following properties:
\begin{enumerate}
\item[\textup{1)}]  $\ell_p^+=\ell_p^++\ell_p^+$  is closed and nowhere dense in $\ell_p$;
\item[\textup{2)}] $\ell^+_p-\ell^+_p=\ell_p$;
\item[\textup{3)}] $\ell_p^+$ is not null-finite and not Haar-countable;
\item[\textup{4)}] $\ell_p^+$ is Haar-null and Haar-meager.
\end{enumerate}
\end{example}

\begin{proof} The statements (1), (2) are obvious and (3) follows from Corollary~\ref{c:wedge}. The statement (4) was proved in \cite[Example 4.3]{BN} (see also \cite[Lemma2]{Mat}).
\end{proof}

The positive cone $c_0^+$ in the Banach space $c_0$ has different properties.

\begin{example} The positive convex cone $c_0^+=\{(x_n)_{n\in\w}\in c_0:\forall n\in\w\;\;x_n\ge 0\}$ in the Banach space $c_0$ has the following properties:
\begin{enumerate}
\item[\textup{1)}] $c_0^+=c_0^++c_0^+$  is closed and nowhere dense in $c_0$;
\item[\textup{2)}] $c^+_0-c^+_0=c_0$;
\item[\textup{3)}] $c^+_0$ is thick and hence not Haar-meager and not Haar-null in $c_0$.
\end{enumerate}
\end{example}

\begin{proof} The statements (1) and (2) are obvious. To see that $c_+^0$ is thick, observe that each compact set $K\subset c_0$ is contained in the ``cube'' $\{(x_n)_{n\in\w}\in c_0:\forall n\in\w\;|x_n|\le a_n\}$ for a suitable sequence $a=(a_n)_{n\in\w}\in c_0^+$ (see Example 3.2 in \cite{BN}). Then $K+2a\subset c_0^+$.
\end{proof}

\begin{remark} For the Banach space $c$ of convergent sequences the convex cone $c^+=\{(x_n)_{n\in\w}:\forall n\in\w\;|x_n|\ge0\}$ is closed and has nonempty interior in $c$.
\end{remark}

Finally, we apply the results obtained in this paper to evaluate the smallness properties of the set $M$ (and $M_c$)  consisting of continuous functions $f:[0,1]\to\IR$ that are monotone (and constant) on some nonempty open set $U_f\subset [0,1]$, depending on $f$. The sets $M_c\subset M$ are considered as subsets of the  Banach space $C[0,1]$ of all continuous real-valued functions on $[0,1]$. The Banach space $C[0,1]$ is endowed with the norm $\|f\|=\max_{x\in[0,1]}|f(x)|$.

By Examples 6.1.2 and 6.2.1 in \cite{EN}, the set $M$ is Haar-null and Haar-meager in $C[0,1]$. We shall establish more refined (and strange) properties of the sets $M$ and $M_c$.

\begin{example}\label{ex:M} The sets $M_c\subset M$ in the Banach space $X:=C[0,1]$ have the following properties:
\begin{enumerate}
\item[\textup{1)}] $M$ is a countable union of closed convex sets in $X$;
\item[\textup{2)}] $M$ is (generically) Haar-countable in $X$;
\item[\textup{3)}] $M_c$ is countably thick in $X$;
\item[\textup{4)}] $M_c$ is not null-finite in $X$;
\item[\textup{5)}] $M_c$ is not Haar-scattered in $X$;
\item[\textup{6)}] $M_c$ is not Haar-thin in $X$.
\end{enumerate} 
\end{example}

\begin{proof} 1,2. Fix a countable base $\mathcal B$ of the topology of $[0,1]$, consisting of nonempty open connected subsets of $[0,1]$. For every $B\in\mathcal B$ let $M_B^+$ (resp. $M_B^-$) be the set of all functions $f\in C[0,1]$ that are non-decreasing (resp. non-increasing) on the interval $B$. It is clear that $M_B^+=-M_B^-$ and $M_B^-=-M_B^+$ are closed convex subsets of $C[0,1]$. It is well-known that the set $M_B^--M_B^-=M_B^+-M_B^+$ consists of functions $f\in C[0,1]$ that have bounded variation on the closed interval $\bar B$. So, $M_B^--M_B^-=M_B^+-M_B^+$ is not a neighborhood of zero in $C[0,1]$. By Theorem~\ref{t:conv-Ban}, the closed convex sets $M_B^-$ and $M_B^+$ are generically Haar-1 and hence generically Haar-countable in $C[0,1]$. By Proposition~\ref{p:GHI-ideal}, the set $M=\bigcup_{B\in\mathcal B}(M_B^+\cup M_B^-)$ is generically Haar-countable in $C[0,1]$.
\smallskip

3. Given any countable set $D=\{f_n\}_{n\in\w}$ in $C[0,1]$, for every $n\in\w$ use the continuity of $f_n$ at the point $b_n=\frac1{2^{n}}$ and find a real number $a_n$ such that $\frac1{2^{n+1}}<a_n<b_n$ and $|f_n(x)-f(b_n)|<\frac1{2^n}$ for every $x\in[a_n,b_n]$. Let $g\in C[0,1]$ be any continuous function such that $g(x)=f_n(b_n)-f_n(x)$ for every $n\in\w$ and $x\in[a_n,b_n]$.
It is clear that $g+f_n$ is constant on each interval $[a_n,b_n]$, so $g+D\subset M_c$, which means that the set $M_c$ is countably thick in $X=C[0,1]$.
\smallskip

4,5,6. The countable thickness of $M_c$ implies that $M_c$ is not null-finite and not Haar-scattered. This property also implies that $M_c-M_c=X$. By Theorem~\ref{t:thin}, the set $M_c$ is not Haar-thin.
\end{proof} 

\section{Steinhaus properties of various semi-ideals}\label{s20}

In this section we present two tables summing up the results and open problems related to (strong or weak) Steinhaus properties of various semi-ideals in Polish groups. In these tables we assume that
\begin{itemize}
\item either $\I$ is a semi-ideal on the Cantor cube $2^\w$ such that $\I\subset\sigma\overline{\N}$ and $\bigcup\I=2^\w$;
\item or $\I$ is a semi-ideal on an infinite countable space $K$ such that $\bigcup\I=K$ and $\I\subset[K]^{\le n}$ for some $n\in\IN$.
\end{itemize}
For a Polish group $X$ by $\NF$ we denote the semi-ideal of null-finite sets in $X$. By \cite{Kwela}, the semi-ideal $\NF$ is not an ideal.

\begin{table}[h]
\centering
\begin{tabular}{|c|c|c|c|}
\hline
Semi-ideal& strong Steinhaus&Steinhaus&weak Steinhaus\\
\hline
\!$\M{=}\HM{=}\GHM{=}\sigma\overline{\M}$\!&Yes [\ref{t:PP}]&Yes [\ref{t:PP}]&Yes [\ref{t:PP}]\\
$\N=\HN$&Yes [\ref{Stein}]&Yes [\ref{Stein}]&Yes [\ref{Stein}]\\
$\M\cap\N$&No [\ref{r:MN+}]&Yes [\ref{c:MN-}]&Yes [\ref{c:MN-}]\\
$\HT=\mathcal{EHT}$&\!No [\ref{r:MN+}+\ref{p:thin=>HN+HM}]\!&Yes [\ref{t:thin}]&Yes [\ref{t:thin}]\\
$\GHN$&No [\ref{e:Ruzsa}]&{\bf ?} [\ref{prob:St-}]&Yes [\ref{c:St-}]\\
$\sigma\overline{\N}$& No [\ref{r:MN+}]&No [\ref{e:sN+-}]&Yes [\ref{c:sN->S+}]\\
$\mathcal{HI}$&No [\ref{ex:hard}]&No [\ref{ex:hard}]&\!Yes [\ref{c:St-}]\\
$\mathcal{GHI}$&No [\ref{ex:hard}]&No [\ref{ex:hard}]&\!Yes [\ref{c:St-}]\\
$\NF$&No [\ref{r:MN+}+\ref{t:NF=>HN+EHM}]&{\bf ?} [\ref{prob:NF-St}]&\!Yes [\ref{c:St-}]\\
$\sigma\overline{\HI}=\sigma\overline{\mathcal{GHI}}$&No [\ref{ex:hard}]&No [\ref{ex:hard}]&{\bf ?} [\ref{prob:sigma-St}]\\
\hline
\end{tabular}
\vskip5pt
\caption{Steinhaus properties of semi-ideals in locally compact Polish groups}
\end{table}


\begin{table}[h]
\centering
\begin{tabular}{|c|c|c|c|}
\hline
Semi-ideal&strong Steinhaus&Steinhaus&weak Steinhaus\\
\hline
\!$\M=\sigma\overline{\M}$\!&Yes [\ref{t:PP}]&Yes [\ref{t:PP}]&Yes [\ref{t:PP}]\\
$\HN$&No [\ref{t:NoSt+}]&Yes [\ref{t:HN}]&Yes [\ref{t:HN}]\\
$\HM$&No [\ref{t:NoSt+}]&Yes [\ref{t:Jab}]&Yes [\ref{t:Jab}]\\
$\EHM$&No [\ref{t:NoSt+}]&Yes [\ref{c:EHT-}]&Yes [\ref{c:EHT-}]\\
$\HT{=}\mathcal{EHT}$&No [\ref{t:NoSt+}]&Yes [\ref{t:thin}]&Yes [\ref{t:thin}]\\
$\GHM$&No [\ref{t:NoSt+}]&{\bf ?} [\ref{prob:St-}]&Yes [\ref{c:St-}]\\
$\GHN$&No [\ref{t:NoSt+}]&{\bf ?} [\ref{prob:St-}]&Yes [\ref{c:St-}]\\
$\NF$&No [\ref{t:NoSt+}]&{\bf ?} [\ref{prob:NF-St}]&Yes [\ref{c:St-}]\\
$\mathcal{HI}$&No [\ref{t:NoSt+}]&No [\ref{ex:hard}]&Yes [\ref{c:St-}]\\
$\mathcal{GHI}$&No [\ref{t:NoSt+}]&No [\ref{ex:hard}]&Yes [\ref{c:St-}]\\
$\sigma\overline{\HM}$& No [\ref{t:NoSt+}]&No [\ref{Banakh}]&No [\ref{Banakh}]\\
$\sigma\overline{\HN}$& No [\ref{t:NoSt+}]&No [\ref{Banakh}]&No [\ref{Banakh}]\\
$\sigma\overline{\HI}$&No [\ref{t:NoSt+}]&No [\ref{Banakh}]&No [\ref{Banakh}]\\
$\sigma\overline{\mathcal{GHI}}$&No [\ref{t:NoSt+}]&No [\ref{Banakh}]&No [\ref{Banakh}]\\
\hline
\end{tabular}
\vskip5pt
\caption{Steinhaus properties of semi-ideals in non-locally compact Polish groups}
\end{table}


\section{Acknowledgement}

The authors would like to express their sincere thanks to S\l awomir Solecki for clarifying the situation with the unpublished result of Hjorth \cite{Hjorth} (on the $\mathbf{\Sigma_1^1}$-hardness of the family $\mathrm{CND}$ of closed non-dominating sets) and sending us a copy of the handwritten notes of Hjorth.

\newpage
\printindex


\begin{thebibliography}{99}

\bibitem{Akin99} E.~Akin, {\em Measures on Cantor space}, Topology Proc. {\bf 24} (1999), 1--34.





\bibitem{Ban} T.~Banakh, {\em Cardinal characteristics of the ideal of Haar null sets}, Comment. Math. Univ. Carolin. {\bf 45}:1 (2004), 119--137.


\bibitem{BBK} T.~Banakh, B.~Bokalo, N.~Kolos, {\em On $\infty$-convex sets in spaces of scatteredly continuous functions}, Topology Appl. {\bf 169} (2014), 33--44.


\bibitem{BJ} T.~Banakh, E.~Jab\l o\'nska, {\em Null-finite sets in topological groups and their applications}, Israel J. Math. {\bf 230}:1 (2019), 361--386.


\bibitem{BLM} T.~Banakh, W.E.~Lyantse, Ya.V.~Mykytyuk. {\em $\infty$-Convex sets and their applications to the proof of certain classical theorems of functional analysis}, Mat. Stud. {\bf 11}:1 (1999), 83--84.


\bibitem{BRZ} T.~Banakh, R.~Ra\l owski, S.~\.Zeberski, {\em Classifying invariant $\sigma$-ideals with analytic base on good Cantor measure spaces}, Proc. Amer. Math. Soc. {\bf 144}:2 (2016) 837--851.


\bibitem{BFN} A. Bartoszewicz, M. Filipczak, T. Natkaniec, {\em On Smital properties}, Topology Appl. {\bf 158} (2011), 2066--2075.

\bibitem{BaJu} T.~Bartoszy\'nski, H.~Judah, {\em Set theory. On the structure of the real line}, A K Peters, Ltd., Wellesley, MA, 1995.

\bibitem{BecKec} H. Becker, A. Kechris, {\em The Descriptive Set Theory
of Polish Group Actions}, Cambridge Univ. Press, Cambridge 1996.

\bibitem{Bog} V.~Bogachev, {\em Measure Theory}, Springer, 2007. 

\bibitem{BN} J.M.~Borwein, D.~Noll, {\em Second order differentiability of convex functions in Banach spaces}, Trans. Amer. Math. Soc. {\bf 342} (1994) 43--81.

\bibitem{GN} S. G\l \c ab, P. Borodulin-Nadzieja, {\it Ideals with bases of unbounded Borel hierarchy}, MLQ Math. Log. Q. {\bf 57}:6 (2011), 582--590.

\bibitem{Ch}
J.P.R. Christensen, {\em On sets of Haar measure zero in Abelian Polish groups}, Israel J. Math. {\bf 13} (1972), 255--260.

\bibitem{D}
U.B. Darji, {\em On Haar meager sets}, Topology Appl. {\bf 160} (2013), 2396--2400.

\bibitem{D1} P. Dodos, {\em Dichotomies of the set of test measures of a Haar null set}, Israel J. Math. {\bf 144} (2004), 15--28.

\bibitem{D2} P. Dodos, {\em On certain regularity properties of Haar-null sets}, Fund. Math. {\bf 181} (2004), 97--109.

\bibitem{D3} P. Dodos, {\em The Steinhaus property and Haar-null sets}, Bull. Lond. Math. Soc. {\bf 41} (2009), 377--384.


\bibitem{DV} M. Dole\v{z}al, V. Vlas\v{a}k, {\em Haar meager sets, their hulls, and relationship to compact sets}, J. Math.
Anal. Appl. {\bf 446}:1 (2017), 852--863.


\bibitem{DVVR} M. Dole\v{z}al, V. Vlas\v{a}k, B. Vejnar, M. Rmoutil, {\em Haar meager sets revisited}, J. Math. Anal. Appl. {\bf 440:2} (2016), 922--939.



\bibitem{EN} M.~Elekes, D.~Nagy, {\em Haar null and Haar meager sets: a survey and new results}, Bull. Lond. Math. Soc. (to appear); (https://users.renyi.hu/~emarci/Haar\_null\_survey.pdf).

\bibitem{ENPV} M.~Elekes, D.~Nagy, M.~Po\'or, Z.~Vidny\'anszky, {\em A Haar meager set that is not strongly Haar meager}, Israel J. Math. (to appear);  (https://arxiv.org/abs/1806.11524).

\bibitem{EP} M.~Elekes, M.~Po\'or, {\em Cardinal invariants of Haar null and Haar meager sets}, preprint (https://arxiv.org/abs/1908.05776).

\bibitem{EVid} M. Elekes, Z. Vidny\'anszky, {\em Haar null sets without $G_\delta$ hulls}, Israel J. Math. {\bf 209} (2015), 199--214.


\bibitem{EVidN}  M. Elekes, Z. Vidny\'anszky, {\em  Naively Haar null sets in Polish groups}, J. Math. Anal. Appl. {\bf 446}:1 (2017), 193--200.


\bibitem{Eng} R.~Engelking, {\em General Topology}, Heldermann Verlag, Berlin, 1989.

\bibitem{End} R.~Engelking, {\em Theory of dimensions, finite and infinite}, Heldermann Verlag, Berlin, 1995.

\bibitem{GRS} R.~Graham, B.~Rothschild, J.~Spencer, {\em Ramsey Theory}, John Wiley and Sons, New York, 1990.

\bibitem{Haight} J.A.~Haight, {\em An $F_\sigma$ semigroup of zero measure which contains a translate of every countable set}, Mathematika {\bf 31}:2 (1984), 272--281. 

\bibitem{Hjorth} G.~Hjorth, {\em Complexity of non-dominating sets}, unpublished notes, (1997), 4 pp.


\bibitem{J}
E. Jab{\l}o\'{n}ska, {\em Some analogies between Haar meager sets and Haar null sets in abelian Polish groups}, J. Math. Anal. Appl. {\bf 421} (2015), 1479--1486.

\bibitem {K} A. Kechris, {\it Classical Descriptive Set Theory}, Springer, New York 1998.


\bibitem{Kwela} A. Kwela, {\em Haar-smallest sets}, (arXiv:1711.09753).


\bibitem{Luk} G.~Luk\'acs, {\em Compact-like topological groups}, Helderman Verlag, 2009.

\bibitem{Mat97}  E.~Matou\v skov\'a, {\em Convexity and Haar null sets}, Proc. Amer. Math. Soc. {\bf 125}:6 (1997) 1793--1799.

\bibitem{Mat} E.~Matou\v skov\'a, {\em Translating finite sets into convex sets}, Bull. London. Math. Soc. {\bf 33} (2001) 711-714.

\bibitem{MSt} E.~Matou\v skov\'a, C. Stegall, {\em A characterization of reflexive Banach Spaces}, Proc. Amer. Math. Soc. {\bf 124}:4 (1996) 1083-1090.

\bibitem{MZ} E. Matou\u{s}kov\'{a}, M. Zelen\'{y}, {\em A note on intersections of non--Haar null sets}, Colloq. Math. {\bf 96} (2003), 1--4.



\bibitem{Mich} E. Michael, {\em Selected Selection Theorems}, The Amer. Math. Montly {\bf 63}:4 (1956), 233--238.

\bibitem{Pat} A. ~Paterson, {\em Amenability}, Amer. Math. Soc., 1988.

\bibitem{Pet} B.J. Pettis, {\em Remarks on a theorem of E. J. McShane}, Proc. Amer. Math. Soc. {\bf 2} (1951), 166--171.

\bibitem{Pic} S. Piccard, {\em Sur les ensembles de distances des ensembles de points d'un espace Euclidien},
Mem. Univ. Neuch\^{a}tel, vol. 13, Secr\'{e}tariat Univ., Neuch\^{a}tel, 1939.



\bibitem{S} S. Solecki, {\em On Haar null sets}, Fund. Math. {\bf 149}:3 (1996), 205--210.

\bibitem{S01} S.~Solecki, {\em Haar null and non-dominating sets}, Fund. Math. {\bf 170}:1-2 (2001), 197--217.

\bibitem{Sol2011} S.~Solecki, {\em $G_\delta$ ideals of compact sets}, J. Eur. Math. Soc. {\bf 13} (2011), 853--882.


\bibitem{Stein} H. Steinhaus, {\em Sur les distances des points des ensembles de mesure positive}, Fund. Math. {\bf 1} (1920), 99--104.

\bibitem{TV} S.~Todor\v cevi\'c, Z.~Vidny\'anszky, {\em A complexity problem for Borel graphs}, preprint\newline (https://arxiv.org/pdf/1710.05079).

\bibitem{Weil} A. Weil, {\em L'int\'egration dans les groupes topologiques}, Actualit\'es Scientifiques et Industrielles 1145, Hermann, 1965.


\bibitem{Zeleny} M. Zelen\'y, {\em Calibrated thin $\Pi_1^1$ $\sigma$-ideals are $G_\delta$}, Proc. Amer. Math. Soc. {\bf 125} (1997), 3027--3032.

\end{thebibliography}
\end{document}